\documentclass[10pt, a4paper]{article}
\usepackage{natbib}
\usepackage{graphicx}
\usepackage{color}
\usepackage{amsmath}
\usepackage{amssymb}
\usepackage{amscd}
\usepackage{bbm}
\usepackage{tikz}
\usepackage{framed}
\usetikzlibrary{patterns}
\usepackage{hyperref}
\hypersetup{
    bookmarks=true,         
    colorlinks=true,       
    linkcolor=red,          
    citecolor=green,        
    filecolor=magenta,      
    urlcolor=cyan           
}
\usepackage[most]{tcolorbox}
\usepackage[utf8]{inputenc}
\usepackage{amsthm}
\usepackage{bbm} 
\usepackage{glossaries}
\usepackage[linesnumbered,lined,boxed,commentsnumbered]{algorithm2e}

\usepackage{marginnote}

\newtheorem{Theorem}{Theorem}[section]
\newtheorem{Assumption}[Theorem]{Assumption}

\newtheorem{Lemma}[Theorem]{Lemma}
\newtheorem{Proposition}[Theorem]{Proposition}
\newtheorem{Corollary}[Theorem]{Corollary}
\newtheorem{Example}[Theorem]{Example}

\theoremstyle{remark}
\newtheorem{Remark}[Theorem]{Remark}

\numberwithin{equation}{section}

\newcounter{BigConst}                     
\newcommand{\nC}{                   
    \refstepcounter{BigConst}             
    \ensuremath{C_{\theBigConst}}
    }
\newcommand{\oC}[1]{\ensuremath{C_{\ref*{#1}}}}  

\newcounter{SmallConst}                     
\newcommand{\nc}{                   
    \refstepcounter{SmallConst}             
    \ensuremath{c_{\theSmallConst}}
    }
\newcommand{\oc}[1]{\ensuremath{c_{\ref*{#1}}}}  

\newcounter{gamma}                     

\newcounter{kappa}                     

\newcounter{delta}                     

\newcounter{theta}                     

\newcounter{L}                     

\newcounter{eps}                     

\newcommand{\inr}[1]{\bigl< #1 \bigr>}

\newcommand{\norm}[1]{\left\|#1\right\|}%
\newcommand{\vertiii}[1]{{\left\vert\kern-0.25ex\left\vert\kern-0.25ex\left\vert #1 
    \right\vert\kern-0.25ex\right\vert\kern-0.25ex\right\vert}}

\newcommand{\vertiiii}[1]{{\left\vert\kern-0.25ex\left\vert\kern-0.25ex\left\vert\kern-0.25ex\left\vert #1 
    \right\vert\kern-0.25ex\right\vert\kern-0.25ex\right\vert\kern-0.25ex\right\vert}}

\newcommand\ove[2]{\overset{#2}{#1}\;}

\newcommand\eps{\varepsilon}
\newcommand\krad{\eta}

\DeclareMathOperator*{\argmax}{argmax}
\DeclareMathOperator*{\argmin}{argmin}

\DeclareMathOperator{\diag}{diag}
\DeclareMathOperator{\Span}{span}

\DeclareMathOperator*{\diam}{diam}

\DeclareMathOperator*{\Range}{Range}

\DeclareMathOperator*{\sign}{sign}

\def\ds1{\textrm{1\kern-0.25emI}} 

\newcommand{\1}{\ensuremath{\mathbbm{1}}}

\newcommand{\be}{{\boldsymbol{e}}}

\DeclareMathOperator{\Var}{\mathrm{Var}}

\newcommand \E{\mathbb{E}}
\newcommand \R{\mathbb{R}}

\newcommand \cA{{\cal A}}

\newcommand \cD{{\cal D}}

\newcommand \cI{{\cal I}}

\newcommand \cL{{\cal L}}

\newcommand \cN{{\cal N}}

\newcommand \cR{{\cal R}}

\newcommand \bE{{\mathbb E}}

\newcommand \bP{{\mathbb P}}

\newcommand \bR{{\mathbb R}}

\newcommand \bX{{\mathbb X}}
\newcommand \bY{{\mathbb Y}}

\newcommand{\bbeta}{{\boldsymbol{\beta}}}

\newcommand{\bv}{\boldsymbol{v}}

\newcommand{\vol}{\mathrm{vol}}

\newcommand{\vmu}{{\boldsymbol{\mu}}}
\newcommand{\vnu}{{\boldsymbol{\nu}}}

\newcommand{\vbeta}{{\boldsymbol{\beta}}}

\newcommand{\vlambda}{{\boldsymbol{\lambda}}}

\newcommand{\vv}{{\boldsymbol{v}}}
\newcommand{\vh}{{\boldsymbol{h}}}

\newcommand{\vs}{{\boldsymbol{s}}}
\newcommand{\vzero}{{\boldsymbol{0}}}
\newcommand{\vx}{{\boldsymbol{x}}}

\newcommand{\vu}{{\boldsymbol{u}}}

\newcommand{\vy}{{\boldsymbol{y}}}
\newcommand{\vxi}{{\boldsymbol{\xi}}}

\newcommand{\vdelta}{{\boldsymbol{\delta}}}

\newcommand{\vzeta}{{\boldsymbol{\zeta}}}

\newcommand{\PP}{\mathbb{P}}
\newcommand{\EE}{\mathbb{E}}

\newcommand{\RR}{\mathbb{R}}
\newcommand{\XX}{\mathbb{X}}
\newcommand{\YY}{\mathbb{Y}}

\newcommand{\SSS}{S_2}

\newcommand{\Id}{\mathrm{Id}}
\newcommand{\sgn}{\operatorname{sgn}}

\newcommand{\bxi}{{\boldsymbol{\xi}}}

\DeclareMathOperator*{\Tr}{Tr}

\newcommand{\mysymbol}[3]{%
\newglossaryentry{#1}{%
      name={\ensuremath{#2}},%
      text={\ensuremath{#2}},%
      description={#3},%
      sort={#1}%
    }%
\expandafter\newcommand\expandafter{\csname smb#1\endcsname}{\gls{#1}}%
\expandafter\newcommand\expandafter{\csname #1\endcsname}{\ensuremath{#2}}%
}

\newcommand{\op}{\text{op}}

\newcommand{\vG}{{\boldsymbol{G}}}

\usepackage[normalem]{ulem}
\usepackage[english]{babel}
\usepackage{enumitem}
\usepackage{mathtools}

\usetikzlibrary{arrows.meta, shapes.geometric, positioning, calc}
\newcommand{\Log}{\operatorname{Log}}

\begin{document}
\title{Benign Overfitting Property of the Minimum \texorpdfstring{$\ell_q$}{ell q}-norm Interpolant Estimator in Regression Problem via Feature Space Decomposition}
\author{Rados{\l}aw Adamczak\thanks{University of Warsaw, Institute of Mathematics, ul. Banacha 2, 02-097, Warszawa, Poland. Email: \href{mailto:radamcz@mimuw.edu.pl}{radamcz@mimuw.edu.pl}}\and
Guillaume Lecu{\'e}\thanks{ESSEC, business school, 3 avenue Bernard Hirsch, 95021 Cergy-Pontoise, France. Email: \href{mailto:lecue@essec.edu}{lecue@essec.edu}}\and
Zong Shang\thanks{CREST-ENSAE, Institut Polytechnique de Paris, 5, avenue Henry Le Chatelier 91120 Palaiseau, France. Email: \href{mailto:zong.shang@ensae.fr}{zong.shang@ensae.fr}}\and
Marta Strzelecka\thanks{University of Warsaw, Institute of Mathematics, ul. Banacha 2, 02-097, Warszawa, Poland. Email: \href{mailto:martast@mimuw.edu.pl}{martast@mimuw.edu.pl}}
}

\date{16 October 2026}
\maketitle

\begin{abstract}
This paper establishes high-probability, non-asymptotic upper bounds for the population excess risk of minimum $\ell_q$-norm interpolating estimators in linear regression, valid for all $1<q<\infty$ under three scenarios: convex concentration, mild moment assumptions and log-concavity. Consequently, we derive sufficient conditions for these estimators to exhibit benign overfitting. Our analysis hinges on a Feature Space Decomposition (FSD) framework, which naturally highlights the self-regularizing properties of minimum norm interpolants as well as the role of alignment between the signal and the features for benign overfitting. From a technical perspective, our approach circumvents the traditionally used Convex Gaussian Min-max Theorem (CGMT) by instead leveraging the Dvoretzky-Milman theorem. We emphasize that the FSD method has the potential to refine standard uniform convergence approaches. Furthermore, the isomorphic Dvoretzky-Milman theorem established herein -- proven under the three scenarios mentioned above -- for the $\ell_q$-norm is of independent mathematical interest.
\end{abstract}

\tableofcontents


\section{Introduction}

Over the past two decades, deep neural networks have achieved unprecedented empirical success across a multitude of complex domains, including computer vision, natural language processing, and reinforcement learning \cite{krizhevsky2012imagenet,vaswani2017attention,silver2016mastering}. Despite these remarkable practical achievements, a rigorous theoretical understanding of why such massively overparameterized models generalize well to unseen data remains one of the central challenges in modern machine learning. Classical statistical learning theory suggests that models with enough capacity to perfectly memorize the training data should suffer from severe overfitting and poor generalization \cite{zhang2017understanding}. However, recent literature has demonstrated a counterintuitive phenomenon known as the `double descent' risk curve \cite{belkin2019reconciling,nakkiran2019deep}, which reveals that increasing model complexity beyond the interpolation threshold can actually drive the test error downward. A principled framework to understand this regime is the study of \textit{benign overfitting} \cite{bartlettBenignOverfittingLinear2020,bartlett2021deep,belkin2021fit,muthukumar2020harmless,negrea2020defense,hastie2022surprises,zhang2017understanding,shang2026feature}, which characterizes how the implicit regularization of gradient-based optimization algorithms directs over-parameterized models toward specific minimum norm interpolating estimators that do not harm generalization. While this phenomenon has been well documented for minimum $\ell_2$ \cite{bartlettBenignOverfittingLinear2020,hastie2022surprises,tsiglerBenignOverfittingRidge2023,tsiglerBenignOverfittingGeometry2025,lecueGeometricalViewpointBenign2024}  and $\ell_1$ \cite{wang2022tight,muthukumar2020harmless}  norm interpolants, the precise behavior of benign overfitting under the broader geometry of general $\ell_q$-norms \cite{donhauser2022fast,koehler2021uniform} remains an open question.

\paragraph{Algorithmic foundations.} To establish the algorithmic foundations of our study, we build upon the framework from \cite{gunasekar2018characterizing,azizan2019stochastic}, which formally links the geometry of optimization algorithms to the implicit regularization of the limiting interpolating estimators \cite{bartlett2021deep}. For underdetermined models where multiple global minima exist to perfectly fit the training data, \cite{gunasekar2018characterizing,azizan2019stochastic} demonstrated that the choice of optimization algorithm inherently steers the solution towards specific minimum-norm interpolants without the need for explicit penalty terms in the objective function. Specifically, they analyzed the trajectory of Mirror Descent (MD) parameterized by a strictly convex and differentiable potential function $\psi(\bbeta)$ (see Section~4 in \cite{azizan2019stochastic}). When optimizing the unregularized empirical risk, the MD updates minimize the loss while remaining strictly tied to the geometry induced by the corresponding Bregman divergence $D_\psi$: for a system of linear constraints $\bX  \bbeta = \vy$ and an initialization at the origin ($\bbeta_0 = 0$), the algorithm is proven to converge to $\arg\min_{\bX \bbeta = \vy} D_\psi(\bbeta, \bbeta_0)$, which reduces to minimizing the potential $\psi(\bbeta)$. By selecting the strictly convex and differentiable potential functions $\psi(\bbeta) = (1/2)\|\bbeta\|_q^2$ when $1<q\leq2$ and $\psi(\bbeta) = (1/q)\|\bbeta\|_q^q$ when $q\geq2$, the Bregman divergence  regularizes the optimization path, guaranteeing that the limit point of the algorithm is the minimum $\ell_q$-norm interpolating estimator (see Proposition~8 and the remark below from \cite{azizan2019stochastic}). This fundamental result provides the algorithmic justification for the analysis of the generalization properties of the class of minimum $\ell_q$-norm interpolation estimators.

While the algorithmic trajectory converging to the minimum $\ell_q$-norm interpolant is firmly established, the statistical properties of the resulting estimator -- specifically its capacity to exhibit benign overfitting -- represent a frontier in the literature. To understand this phenomenon, most of the literature has considered the linear regression model\footnote{The generalization property of deep neural networks is increasingly understood to be a product of both their architectural design and the implicit regularization induced by their training algorithms. To isolate the contribution of the algorithmic side of this phenomenon, we place ourselves within the framework of linear models (i.e. the simplest `neural network architecture'), allowing for a mathematical analysis of how optimization dictates generalization disentangled from the architecture contribution.} where one observes $N$ i.i.d.\ couples $\cD:=\{(X_i, Y_i)_{i=1}^N\}$ following the linear model $Y = \inr{\bbeta^*, X} + \xi$ where $X$ is the centered vector of features living in $\bR^p$ and $\xi$ is an independent centered noise with variance $\sigma_\xi^2$. Our aim is to study the generalization properties of the minimum $\ell_q$-norm interpolant estimator
     \begin{align}\label{eq:def_interpolant_estimator}
            \hat\bbeta = \argmin_{\bbeta\in\bR^p}\left( \norm{\bbeta}_q: \forall i\in[N],\, \langle X_i,\bbeta\rangle = Y_i \right)
        \end{align} for all $1<q<\infty$ that is to study their excess risk 
\begin{equation*}
    \bE[(Y - \inr{\hat\bbeta,X})^2|\cD] - \bE[(Y-\inr{\bbeta^*, X})^2] = \norm{\Sigma^{1/2}(\hat\bbeta - \bbeta^*)}_2^2
\end{equation*} where $\Sigma = \bE X^\top X$ is the covariance matrix of $X$. Bounding the generalization error of $\hat{\bbeta}$ allows us to unpack the mechanics by which implicit geometric regularization enables generalization and so harmless interpolation.  

\paragraph{Literature review.} Historically, state-of-the-art analyzes have focused predominantly on the $\ell_2$ and $\ell_1$ geometries.  The generalization bound of the minimum $\ell_2$-norm interpolating estimator (IE) -- the unique element solution to the underdetermined linear system $\bX\bbeta=\vy$ where $\bX^\top = [X_1|\cdots|X_N]$ and $\vy=(Y_i)_{i=1}^N$ with the smallest $\ell_2$-norm -- is well-characterized and so it is possible to identify the necessary and sufficient conditions for benign overfitting of this estimator \cite{tsiglerBenignOverfittingGeometry2025,lecueGeometricalViewpointBenign2024}: as $N,p\to \infty$,
\begin{equation}
    \label{eq:NSC_BO_ell2}
    \sigma_\xi \sqrt{\frac{|J_*|}{N}} + \norm{\bbeta^*_{J_*}}_2\sqrt{\frac{{\rm Tr}(\Sigma_{J_*^c})}{N}} + \norm{\Sigma^{1/2}_{J_*^c} \bbeta^*_{J_*^c}}_2 + \sigma_\xi \frac{\sqrt{N {\rm Tr}(\Sigma_{J_*^c}^2)}}{{\rm Tr}(\Sigma_{J_*^c})} \longrightarrow 0,
\end{equation}where $k^*=\min(k: N \lesssim {\rm Tr}(\Sigma_{k+1:p}) / \sigma_{k+1})$, $J_*=\{1, \ldots, k^*\}$, $\Sigma_J = P_J \Sigma P_J$, $\bbeta^*_J= P_J \bbeta^*$, $P_J$ is the projection matrix onto the span of $\Sigma_J$, and $\sigma_1\ge \dots \ge \sigma_p$ is the spectrum of $\Sigma$. In particular, benign overfitting for the minimum $\ell_2$-norm IE occurs when the data covariance matrix $\Sigma$ exhibits a heavy-tailed eigenvalue decay, allowing the estimator to absorb noise in harmless, low-variance directions \cite{bartlettBenignOverfittingLinear2020,hastie2022surprises} and when an `alignment property'  of the target signal $\bbeta^*$ with the covariance matrix holds \cite{lecueGeometricalViewpointBenign2024}. Similarly, the minimum $\ell_1$-norm IE has been analyzed first through the lens of basis pursuit in compressed sensing \cite{MR2914271,MR2591836,MR2559727,muthukumar2020harmless,MR4474492} and \cite{ju2020overfitting,donhauser2022fast,kur2026a}, where benign overfitting is governed by the underlying sparsity of the ground-truth signal.

However,  extending these guarantees to the general minimum $\ell_q$-norm IE for $1 < q < +\infty$ presents fundamental theoretical hurdles, in particular, beyond the Gaussian and isotropic case (i.e. when $X\sim\cN(0,I_p)$). Unlike the $\ell_2$-norm, which benefits from the orthogonal projections of a Hilbert space, or the $\ell_1$-norm, which induces exact sparsity, the strict but non-Euclidean geometry of general $\ell_q$-balls lacks these simplifying structural properties. Consequently, current literature on minimum $\ell_q$-norm IE is less furnished.

Recent literature has begun to uncover the intricate generalization properties of minimum $\ell_q$-norm interpolating estimators, revealing that classical intuition regarding inductive bias does not straightforwardly translate to the interpolation regime. A foundational contribution in this direction is the work \cite{donhauser2022fast}, which investigated the statistical rates of estimators satisfying the noisy linear interpolation constraint $\bX\bbeta=\vy$ under the assumption of a $1$-sparse ground truth. In traditional regularized learning, it is expected that perfectly aligning the estimator's inductive bias with the underlying signal structure -- such as utilizing the $\ell_1$-norm to promote sparsity -- yields optimal generalization. However, \cite{donhauser2022fast} demonstrated a trade-off: while a strong inductive bias successfully encourages alignment with the true signal, it simultaneously exacerbates the detrimental effects of noise interpolation. Specifically, they proved that the minimum $\ell_q$-norm interpolating estimator $\hat\bbeta$ for $1+ \frac{\kappa_2}{\log\log p}<q <2$ satisfies 
\begin{equation}
    \label{eq:Konstantine_rate}
    \max\bigg(\frac{\sigma_\xi^{4-2q} (q')^q p^{2q-2} }{N^q}, \frac{\sigma_\xi^2 N}{p} \bigg) \lesssim \norm{\hat\bbeta - \bbeta^*}_2^2 \lesssim \max \bigg( \frac{\sigma_\xi^{4-2q} (q')^q p^{2q-2} }{N^q}, \frac{\sigma_\xi^2 N \exp(\kappa_5 q^{'}) }{p q^{'}}  \bigg) 
\end{equation}with probability at least $1-\kappa_6 p^{-\kappa_7}$ when $N \log^{\kappa_3} N \lesssim p \lesssim N^{q'/2} \log^{-\kappa_4} N$, $X_i\ove{\sim}{i.i.d.} \cN(0, I_p)$, $\xi_i\ove{\sim}{i.i.d.} \cN(0, \sigma_\xi^2)$ (ind. of the $X_i$'s)  and $\bbeta^*=(1,0, \ldots, 0)=\be_1$, the first ($1$-sparse) canonical vector of $\bR^p$. In particular, when $p=N^2$ and $q=  1+ \frac{\kappa_2}{\log\log p}$, \cite{donhauser2022fast} recovers the  minimax convergence rate for $1$-sparse targets approaching $\mathcal{O}(\sigma_\xi^2/N)$ (up to log terms), whereas the (strongly biased toward sparsity) minimum $\ell_1$-norm IE suffers from a slower logarithmic $\sigma_\xi^2/\log N$ rate. This result establishes that intermediate $\ell_q$ geometries can vastly outperform their $\ell_1$ counterparts by better balancing signal structure recovery and noise absorption even for a $1$-sparse signal. The sharp rate \eqref{eq:Konstantine_rate} strongly relies on the Gaussian and isotropic assumption in two ways: via the CGMT \cite{thrampoulidis2015regularized,thrampoulidis2015gaussian,gordon1988milman} and the tight concentration inequality for $\ell_{q'}$-norm of isotropic Gaussian vector from \cite{paouris2017random}. It also relies on the $1$-sparse assumption on $\bbeta^*$ and requires $q\geq 1+ \frac{\kappa_2}{\log\log p}$. Three limitations raised in \cite{donhauser2022fast} that we aim to address in this work.

In \cite{koehler2021uniform}, the authors investigated the generalization properties of minimum norm IE  through the lens of uniform convergence restricted to the set of interpolating solutions for any norm $\norm{\cdot}$. Their theoretical setup also assumes that the data is generated from a Gaussian distribution and their proofs rely also on the Convex Gaussian Minimax Theorem (CGMT) as in \cite{donhauser2022fast}. Sufficient conditions for benign overfitting of min norm IE have been obtained in \cite{koehler2021uniform}: there exists a sequence of covariance splits $\Sigma = \Sigma_J\oplus \Sigma_{J^c}$ such that, for any $\eta>0$, as $N,p\to \infty$, 
\begin{equation}
    \label{eq:sufficient_conditions_KZSS20}
    \frac{ {\rm rank}(\Sigma_J)}{N}\to 0, \frac{\norm{\bbeta^*} \ell_*}{\sqrt{N}}\to 0, \sup_{\norm{\bbeta}\leq1}\norm{\Sigma_{J^c}^{1/2}\bbeta}_2 = o\big(\ell_*\big),  \sqrt{N} \bE\norm{\Sigma^{1/2}_{J^c}v^*}_2 = o\big(\ell_*\big) \mbox{ and } \bP\big(\norm{P_{J^c} \bv^*}^2>1+\eta\big)\to 0,
\end{equation}where $\ell_*:=\ell_*(\Sigma_{J^c}^{1/2}B)$ is the Gaussian mean width of $\Sigma_{J^c}^{1/2}B$ (see \eqref{eq:Gauss_width_def} below for the definition), $B$ is the unit ball of $\|\cdot\|$, $\bv^*\in \argmin\big(\norm{\Sigma^{1/2}\bv}:\bv\in\partial\norm{\Sigma^{1/2}G}_*\big)$,
$\norm{\cdot}_*$ is the dual norm of $\norm{\cdot}$, $G\sim\cN(0,I_p)$ and $P_{J^c}$ is the projection onto the space spanned by $\Sigma_{J^c}$. In \cite{koehler2021uniform}  two questions are raised: Are the sufficient conditions from \eqref{eq:sufficient_conditions_KZSS20} necessary? and Is it possible to extend the analysis to non-Gaussian features?

Very recently \cite{kur2026uniform} circumvents the Gaussian assumption and bypassed the CGMT  by leveraging high-dimensional probability and the local theory of Banach spaces. Specifically, by analyzing the structural properties of the norm's unit ball -- such as $2$-uniform convexity and isotropic position -- \cite{kur2026uniform} recovered  the upper bound of the rate \eqref{eq:Konstantine_rate} (up to extra log terms)  under symmetric, i.i.d.\ sub-Gaussian covariates still for the estimation of a $1$-sparse signal and for the ranges of $q,N$ and $p$: $1+ \frac{\kappa_2\log\log\log p}{\log\log p}<q <2$  and $N \log^{\kappa_3} p \lesssim p \lesssim N^{q'/2} \log^{-\kappa_4} N$.

\paragraph{The feature space decomposition method.}
To overcome the inherent limitations of classical uniform convergence bounds -- which require the empirical and population risks to remain close, a property fundamentally violated in the interpolation regime -- we employ the Feature Space Decomposition (FSD) method. It was originally noticed in \cite{bartlettBenignOverfittingLinear2020,tsiglerBenignOverfittingRidge2023} that the analysis of the minimum $\ell_2$-norm IE should be split into two subspaces; the FSD was then refined and systemized by \cite{koehler2021uniform,lecueGeometricalViewpointBenign2024,lecueFeatureSpaceDecomposition2025,gavrilopoulosGeometricalAnalysisKernel2025}. It has been applied to ridge regression \cite{tsiglerBenignOverfittingRidge2023,chengDimensionFreeRidge2022,barzilaiGeneralizationKernelRegression2024,gavrilopoulosGeometricalAnalysisKernel2025} and spectral algorithm \cite{lecueFeatureSpaceDecomposition2025}, obtaining matching upper and lower bounds for the excess risk of a large class of spectral methods (that includes ridge regression, gradient descent, gradient flow and PCR) for every $\Sigma$, $\bbeta^*$, and tuning parameter, thereby providing a complete understanding of their excess risk (up to absolute constants). 
The core idea behind the FSD is that a large part of an estimator in over-parametrized models does not actually estimate the signal. As a result the FSD identifies this part of the feature space where no estimation is happening and the other part where estimation happens. Hence, FSD decomposes the feature space into a direct sum $\bR^p= V_J \oplus V_{J^c}$ of two subspaces: an estimation subspace $V_J$ and a noise absorption subspace $V_{J^c}$ and decompose the (square root of the) excess risk as
\begin{equation}
    \label{eq:excess_risk_decomp_FSD}
    \norm{\Sigma^{1/2}(\hat\bbeta - \bbeta^*)}_2\leq  \norm{\Sigma_J^{1/2}(\hat\bbeta_J - \bbeta_J^*)}_2 + \norm{\Sigma_{J^c}^{1/2}\hat\bbeta_{J^c}}_2 + \norm{\Sigma_{J^c}^{1/2} \bbeta_{J^c}^*}_2.
\end{equation}In this decomposition, we see that $\hat\bbeta_{J^c}$ is the none estimation part of $\hat\bbeta$. Beyond being merely an analytical tool, FSD provides a geometrical understanding of implicit regularization. Specifically, on the estimation subspace $V_J$, the projected minimum $\ell_q$-norm interpolating estimator effectively behaves as a self-regularized Empirical Risk Minimizer (ERM) -- such as the square-root LASSO or ridge regression. Because this estimation part maintains a relatively small $\ell_q$-norm, classical uniform convergence arguments can be applied safely. Conversely, the estimator uses the features of the orthogonal subspace $V_{J^c}$ entirely to absorb the noise; here, the projection is noise-driven, naturally exhibits a large $\ell_q$-norm and does not estimate $\bbeta_{J^c}^*$. This geometrical decomposition explicitly reveals that, for the minimum $\ell_q$-norm interpolating estimator to generalize well, benign overfitting fundamentally requires an alignment of the target signal with the data's covariance matrix, i.e. that $\norm{\Sigma_{J^c}^{1/2} \bbeta_{J^c}^*}_2\to 0$. By restricting uniform convergence strictly to the lower-dimensional $V_J$ and employing a new Dvoretzky-Milman type bound on $ V_{J^c}$, the FSD method circumvents traditional barriers in over-parametrized models to yield high-probability non-asymptotic upper bounds for the population excess risk. To date, the FSD method has remained strictly confined to Euclidean geometries. Extending this framework to arbitrary $\ell_q$-norm spaces represents a significant theoretical leap, for which the generalization analysis of minimum $\ell_q$-norm interpolating estimators serves as a catalyst and testbed.

Building upon this insight, our work provides an analysis of the minimum $\ell_q$-norm interpolant across the entire continuous family $1 < q < +\infty$ beyond the Gaussian case and beyond isotropicity.  Our aim is to explain the implicit decomposition of the feature space into an estimation part and a noise absorption part in order to have a clear geometric picture of the benign overfitting over a continuous family of  geometries by getting  a sharp characterization of the risk curves and implicit bias for arbitrary $1<q<\infty$. However, moving beyond the Gaussian setting inherently required by the CGMT demands fundamentally new mathematical tools, particularly for analyzing the noise absorption subspace. This is precisely where new Dvoretzky-Milman type theorems -- a foundational result in the Asymptotic Geometric Analysis (AGA) literature -- plays a pivotal role. Furthermore, moving beyond the traditional Euclidean setting forces us to address non-linear operators, requiring again reliance on the Dvoretzky-Milman type theorems and convex analysis to complete our mathematical machinery. Consequently, the novel techniques developed to overcome these challenges significantly enrich the analytical toolbox of the FSD method.

\paragraph{The Dvoretzky-Milman type theorems.} To manage the complex geometry of the noise absorption subspace $V_{J^c}$, our Feature Space Decomposition (FSD) method relies on the Dvoretzky-Milman type theorems, a cornerstone of the Asymptotic Geometric Analysis (AGA). Originally established by Dvoretzky \cite{dvoretzky1961some}, who proved that every infinite-dimensional Banach space possesses finite-dimensional subspaces that are almost isometric to Euclidean space, the result was later revolutionized by V. Milman \cite{milman1971new}. Milman's probabilistic proof introduced the \textit{Dvoretzky dimension} and demonstrated that this phenomenon is entirely driven by the concentration of measure on the sphere, establishing a deep, intrinsic link between high-dimensional geometry and probability. In the context of our statistical framework, the Dvoretzky-Milman type theorems play an instrumental role. To circumvent the Gaussian assumptions required by the Convex Gaussian Min-max Theorem (CGMT), we deploy several isomorphic counterparts of the Dvoretzky-Milman theorem in $\ell^q$ beyond the Gaussian case. This provides the tools we need to analyze the behavior of $\hat\bbeta$ on the noise absorption part of the feature space.  By doing so, we demonstrate that the integration of the FSD method with this geometric result provides a powerful framework and tool to establish the universality of the benign overfitting phenomenon.

\paragraph{Beyond the linear model.} We acknowledge that neural networks are certainly more complicated than the `toy example' of linear regression. However, the results here constitute a fundamental first step towards understanding benign overfitting in neural networks. As nonlinear functions, neural networks are often amenable to characterization as linear models; for example, the Reproducing Kernel Hilbert Space \cite{jacotNeuralTangentKernel2018,ghorbaniLinearizedTwolayersNeural2021} and Reproducing Kernel Banach Space \cite{bachBreakingCurseDimensionality2017,bartolucciUnderstandingNeuralNetworks2023,bartolucciNeuralReproducingKernel2024} perspectives correspond to the minimum RKHS (RKBS) norm interpolant estimator(s), respectively. Among these, the linear model characterized by RKBS for shallow neural networks is currently more promising compared to the linear model characterized by the neural tangent kernel's RKHS.

\paragraph{Beyond minimum norm interpolating estimators.} The implications of our generalization bounds, the technical tools introduced, and the broader Feature Space Decomposition (FSD) framework extend well beyond the regime of interpolating estimators. Minimum norm IE may serve as the `theoretical' limit for classical statistical procedures or early stopped algorithms; for instance, the minimum $\ell_2$-norm IE is the limit of ridge regression as the regularization parameter vanishes, and similar limits govern the convergence points of optimization trajectories like early-stopping gradient descent based algorithms. Consequently, many practical algorithms and explicitly regularized estimators lie on a continuous path that converges toward an interpolating estimator. Because they approach this limit, these procedures fundamentally inherit the geometric properties and self-regularization mechanics of the IE itself. Viewing these early-stopped or regularized models through the lens of their limiting interpolating behavior therefore provides a geometric understanding of their generalization capabilities. Ultimately, the mathematical machinery developed here to manage the geometry of IE can be systematically adapted to these broader, non-interpolating settings -- a theoretical transferability demonstrated for the minimum $\ell_2$-norm IE and ridge regression in \cite{tsiglerBenignOverfittingGeometry2025,gavrilopoulosGeometricalAnalysisKernel2025}.

\paragraph{Our Contributions:}

{\it{The FSD beyond Euclidean and  Gaussian settings.}} We introduce the Feature Space Decomposition (FSD) framework as a novel analytical approach in a supervised regression problems beyond the Euclidean geometry and beyond Gaussian isotropic features. Transcending the limitations of traditional uniform convergence arguments, FSD provides a versatile methodology applicable to general supervised learning settings yielding high-probability upper bounds on an estimator's population excess risk. Ultimately, this approach modernizes and refines one of the most fundamental tools in mathematical statistics and shows that it is relevant even in non-Euclidean and non-Gaussian settings.

\textit{Geometrical interpretation of the FSD.} Furthermore, the application of the FSD method yields profound new insights into the geometry of benign overfitting, particularly within the strict, non-Euclidean setup induced by general $\ell_q$-norms. A central contribution of our work is the ability to handle general, non-Gaussian data distributions. By liberating the analysis from the rigid Gaussian assumptions historically required by the Convex Gaussian Min-max Theorem (CGMT), our mathematical framework demonstrates the universality of the benign overfitting phenomenon. We establish that harmless interpolation is not merely an artifact of Gaussianity, but a robust, universal consequence of the underlying optimization geometry.

\textit{Isomorphic $\ell_2^N\to \ell_{q'}^{p}$ Dvoretzky-Milman type theorems beyond Gaussian.} Implementing the Feature Space Decomposition (FSD) method requires entirely new technical tools, particularly to analyze the noise absorption part of the feature space. Classical uniform convergence arguments fall short onto this subspace, revealing that the Dvoretzky-Milman theorem plays a critical role in controlling this component. Consequently, to establish our universality results, we had to develop novel, isomorphic Dvoretzky-Milman type theorems for the $\ell_{q'}$-norm that extends beyond the standard Gaussian case. Establishing these results under general moment assumptions for random matrices with independent columns as well as in the convex concentration and log-concave settings forms a standalone contribution to the Asymptotic Geometric Analysis, directly motivated by the demands of the FSD method.

\textit{Revealing a regularized ERM behind $\hat\bbeta$.} By investigating minimum $\ell_q$-norm interpolating estimators through the lens of the FSD method, we identify a `hidden' estimator that fundamentally corresponds to classical regularized empirical risk minimization (RERM) procedures. This discovery provides a precise characterization of how $\hat{\bbeta}$ behaves within the estimation subspace $V_J$ and provides new insights on the self-regularization phenomenon in over-parameterized models.

\textit{Sufficient conditions for BO of the minimum $\ell_q$-norm interpolating estimators.} We provide sufficient conditions for the benign overfitting phenomenon to hold for minimum $\ell_q$-norm interpolating estimators. These conditions are expressed in terms of the geometry of the feature space. We obtain such conditions under weak moment assumption, in the convex concentration setting, and in the log-concave case. In particular in the weak moment setting, for $1<q\leq2$, we obtain the following sufficient conditions:
\begin{Theorem}[See Corollary~\ref{cor:BO_weak_moment} below for $1<q\leq2$]\label{thm:BO_intro}
    Let $1<q\leq 2$ and assume that along a sequence of problems $(N,p,\Sigma,\bbeta^*,\sigma_\xi)_{N\ge1}$ there exists a coordinate splitting $\bR^p=V_J\oplus^\perp V_{J^c}$ with $|J|\to\infty$, satisfying:
    \begin{itemize}
        \item[$(A_1)$] \emph{(Weak moments on the noise.)}  $\xi$ is independent of $X$, centered, with  $\bE\xi^2=\sigma_\xi^2>0$ and $\bE\xi^4\le M^4\sigma_\xi^4$;
        \item[$(A_2)$] \emph{(Weak moments on the design.)} $X=\Sigma^{1/2}\vzeta$, where $\Sigma=\diag(\sigma_1,\dots,\sigma_p)$ with $\sigma_1\ge\dots\ge\sigma_p>0$\footnote{The ordering of the $\sigma_j$'s plays no role in the proofs, but we keep it since it is always available and it is the usual convention.} and $\vzeta=(\zeta_1,\dots,\zeta_p)$ is isotropic with independent coordinates satisfying $\bE|\zeta_j|^{2q'+\varepsilon}\le M^{2q'+\varepsilon}$ for every $j\le p$;
        \item[$(A_3)$] \emph{(Sample size window.)} $C|J|\le N\le \kappa_{\mathrm{DM}}\,d_*$, where $C$ and $\kappa_{\mathrm{DM}}$ depend only on $q$ and $M$;
    \end{itemize}with universal constants $M,\varepsilon$ (independent of $N$) and assume $\sigma_\xi=O(1)$. Then $\hat\bbeta$ exhibits the BO when:
    \begin{equation}\label{eq:BO_weak_moment_intro}
\frac{|J|}{N}\to0,\qquad \|\bbeta_J^*\|_q\frac{\ell_*}{\sqrt N}\to0,\qquad \|\Sigma_{J^c}^{1/2}\bbeta_{J^c}^*\|_2\to0,\qquad \frac{N\sigma_\xi^2}{R_*}\to0,
    \end{equation}
where
\[\ell_*=\ell_*(\Sigma_{J^c}^{1/2}B_q^p) \sim_q \bigl(\sum_{j\in J^c}\sigma_j^{q'/2}\bigr)^{1/q'},
R_*= \frac{\ell_*^{2q'}}{\Tr(\Sigma_{J^c}^{q'})} \sim_q \frac{\bigl(\sum_{j\in J^c}\sigma_j^{q'/2}\bigr)^2}{\sum_{j\in J^c}\sigma_j^{q'}},
\mbox{ and } d_*=\Bigl(\frac{\ell_*(\Sigma_{J^c}^{1/2}B_q^{p})}{\diam(\Sigma_{J^c}^{1/2}B_q^{p})}\Bigr)^2.
\]
\end{Theorem}
In particular, Conditions~\ref{eq:BO_weak_moment_intro} improves the one from \eqref{eq:sufficient_conditions_KZSS20} (obtained in the Gaussian case) even though they hold under a $L_4$ moment assumption on the noise and a $L_{2q'+\varepsilon}$ moment assumption on the design. They exactly match the optimal conditions in the case $q=2$ recalled in \eqref{eq:NSC_BO_ell2}.

\paragraph{Structure of the paper.}
 In Sections~\ref{sec:main_results}, \ref{sec:main_results_weak_moment_optimal} and \ref{sec:main_results_other_cases}, we provide upper bounds on the excess risk of $\hat\bbeta$ from which we deduce sufficient conditions to exhibit the benign overfitting properties under three setups: the convex concentration setting, the weak moment conditions, and the log-concave case. In Sections~\ref{sect:proof-conv-conc} and \ref{sec:proofs_other_cases}, we provide the proofs of these three results. Section~\ref{sec:noise_absorption_part} shows how to estimate the noise absorption part in the weak moment condition setting. In the final Section~\ref{sec:proof_DM_q} we provide  all the proofs of the  Dvoretzky-Milman isomorphic results in the three scenarios.

\paragraph{Notation.} For any $1\leq q \leq \infty$, let $\|x\|_q=(\sum_{i}|x_i|^q)^{1/q}$ denote the $\ell_q$-norm of a vector $x=(x_i)_i$. We denote by $\ell_q^p$ the Banach space $\bR^p$ endowed with the $\ell_q$-norm, by  $B_q^p$ its unit ball, $S_q^p$ its unit sphere -- except for $S_2^{p-1}$ which is preserved for the unit sphere of $\ell_2^p$. Moreover, $q'\in[1,\infty]$ is the H{\"o}lder conjugate to $q$, so $\frac 1q+\frac 1{q'}=1$. We denote $[p] = \{1, \ldots, p\}$ and $(\be_j)_{j=1}^p$ is the canonical basis of $\bR^p$. Let $J\subset [p]$. We denote $V_J=\Span\{\be_j\colon j\in J\}$,  the orthogonal projection onto $V_J$ is $P_J$ and we denote $B_q^J=P_JB_q^p$, $S_q^J=P_JS_q^p$, and $S_2^J=P_JS_2^{p-1}$.
We denote by $I_{V_J}:\vv \mapsto \vv$ the identical operator on $V_J$. 
We use $\lesssim$ (or $\gtrsim$) to denote inequality up to multiplicative constants, and $\sim$ to denote that both $\lesssim$ and $\gtrsim$ hold. We write $\lesssim_{q,r}, \sim_{C,\kappa}$ etc., if the underlying constant depends on the parameters given in the subscripts. For a deterministic vector $\vv\in\bR^p$ and a P.S.D. matrix $A\in\bR^{p\times p}$, we denote by $\cN(\vv,A)$ the standard Gaussian random vector whose mean is $\vv$ and covariance matrix is $A$. 

If $u$ and $v$ are two vectors, then $u\odot v$ denotes the elementwise product of $u$ and $v$, ie $u\odot v = (u_i v_i)_{i=1}^p$, $|u|=(|u_i|)_{i=1}^p$ denotes the elementwise absolute value of $u$ and ${\rm sign}(u)={\rm sign}(u_i)_{i=1}^p$ denotes the elementwise sign of $u$, where ${\rm sign}(0)=0$. For instance, when all $u_i\neq0$, the gradient of the $\ell_q$-norm to the power $q$ is given by $\nabla (u\to \|u\|_q^q) = q {\rm sign}(u) \odot |u|^{q-1}$ and -- if $1<q<\infty$ --  the unique solution to $\max\left(\inr{u, \theta} : \norm{\theta}_{q'} \leq 1\right)$ is given by $\theta = {\rm sign}(u) \odot |u|^{q-1}/\norm{u}_q^{q-1}$.

For a vector $\bbeta\in\bR^p$, we denote by $\beta_j = \langle\bbeta,\be_j\rangle$ its $j$-th coordinate.
For any symmetric convex body $K\subset\bR^p$, let
\begin{equation}
    \label{eq:Gauss_width_def}
    \ell_*(K)=\bE\sup_{\vv\in K}\langle\vv,G\rangle= \bE \|G\|_{K^{\circ}}
\end{equation}
 denote the Gaussian mean width of $K$, where $G$ is a standard Gaussian random vector in $\bR^p$, $K^\circ = \{\vv\in\bR^p\colon \langle\vv,\vu\rangle\leq 1\text{ for all } \vu\in K\}$ is  the polar body of $K$, and $\|\cdot\|_{K^\circ}$ is the norm whose unit ball is $K^\circ$. 
 Let
 \[
 \diam(K)=\max\{\|\vv\|_2\colon\vv\in K\} = \max\{\|\vu\|_{K^\circ}\colon\vu\in B_2^p\}
 \]
  be the half of the $\ell_2$ diameter of $K$ and let
  \[
  d_*(K)=\Bigl(\frac{\ell_*(K^\circ)}{\diam(K^\circ)}\Bigr)^2
  \]
   be the Dvoretzky dimension of $K$. 
   
   We say that a random vector $\vzeta\in \bR^p$ is centered if $\bE \vzeta =0$ and we say that it is isotropic if it is centered and $\operatorname{Cov}(\vzeta)=\bE[\vzeta\otimes\vzeta] = I_{\bR^p}$. For any matrix $A\in\bR^{p\times p}$, let $\Tr(A)$ be its trace, and if $A\in \bR^{p\times N}$, then let $\norm{A:\ell_q\to \ell_2}= \|A\|_{q\to 2}=\sup(\|A\vv\|_2:\,\|\vv\|_q=1)$ be the $\ell_q^N\to\ell_2^p$ operator norm of $A$. Sometimes the $\ell_2\!\to\!\ell_2$ operator norm is denoted by $\|\cdot\|_{\op}$. Finally, denote $\max\{1,\ln(x)\}$ by $\Log(x)$.

\section*{Acknowledgment and declaration of generative AI}

During the preparation and development of this manuscript, OpenAI GPT-5.6 Sol was guided in the derivation of the results in Sections 7.1 and 7.3. The tool was also used to improve language clarity and readability. All derived logical steps, intermediate lemmas, and final calculations were independently verified, rewritten, and rigorously validated by the authors, who take full responsibility for the mathematical correctness of all stated results.

Research of RA was partially supported by The National Science Center, Poland, grant IMPRESS-U number 2023/05/Y/ST1/00188.



\section{Main results in the convex concentration case}
\label{sec:main_results}
In this section, we provide a high probability excess risk bound for $\hat\bbeta$ when $X$ is assumed to satisfy a \textit{convex concentration property}. In this setting the covariance matrix of $X$ is not assumed to be diagonal; the only structural requirement linking $\Sigma$ to the canonical basis of $\bR^p$ -- the basis with respect to which the $\ell_q$-norm decomposes -- is the two-block structure recorded in Remark~\ref{rem:block-diagonal} below, which is a consequence of the independence of $X_J$ and $X_{J^c}$.

\begin{Assumption}\label{assumption-1}
Assume that there exist an isotropic random vector $\vzeta\in \bR^p$, a constant $\rho>0$, and a $p\times p$ positive definite matrix $\Sigma$ such that $X=\Sigma^{1/2}\vzeta$ and
 for every i.i.d.\ copies $Y_1, Y_2, \dots$ of $\vzeta$, every integer $N$,   every convex 1-Lipschitz function $f \colon \RR^{Np} \to \RR$, and every $t > 0$,
\begin{align}
\label{eq:assumption-1}
\PP\Bigl(\bigl|f(Y_1,\ldots,Y_N) - \EE f(Y_1,\ldots,Y_N)\bigr| \ge t\Bigr) \le 2 \exp\Big(-\frac{t^2}{\rho^2}\Big).
\end{align}
Assume also that the noise $\xi\in \bR$  is centered, has variance $\sigma_{\xi}^2$ and is sub-Gaussian with constant $\rho \sigma_{\xi}$, i.e., it satisfies
\[
	 \PP(|\xi|\ge t) \le 2e^{-t^2/(\sigma_\xi^2\rho^2)} \quad \text{for every } t>0.
\]
\end{Assumption}

Inequality~\ref{eq:assumption-1} is weaker than dimension-free concentration for all Lipschitz functions. In particular, it follows by the Gaussian concentration that Assumption~\ref{assumption-1} is satisfied for $\vzeta \sim \cN(0,\Id)$ and $\xi\sim \cN(0,\sigma_\xi^2)$ with $\rho=\sqrt2$.
 The famous convex distance inequality due to Talagrand \cite{MR1361756} implies that if the entries of $\vzeta$ are independent, with $\|\zeta_i\|_\infty \le K$, then the inequality \eqref{eq:assumption-1} holds with $\rho$ depending only on $K$. More generally,  \cite{grst-Kantorovich-duality} proved that \eqref{eq:assumption-1} holds with a constant $\rho$ independent of $N$ if and only if $\vzeta$ satisfies the $\overline{T}_2$ transportation inequality. For vectors $\vzeta$ with independent coordinates this inequality has been characterized in terms of the distribution of coefficients in \cite{grsst-characterization}. Moreover, \cite{Samson2000} showed this property for vectors with dependent coordinates under certain $\Phi$-mixing conditions. We refer the reader to these references for a detailed exposition of the convex concentration property.

Let us first introduce some notation before turning to our statistical bound under Assumption~\ref{assumption-1}. Given a $p\times p$ symmetric matrix $A$,  the  Dvoretzky dimension of $A^{-1}K^\circ$ is
\[
d_*(A^{-1}K^\circ)=\Bigl(\frac{\ell_*(A K)}{\diam(A K)}\Bigr)^2.
\]

For a $p\times p$ matrix $\Sigma=(s_{ij})_{i,j\le p}$ and a set $J\subset [p]$ we write $\Sigma_J$ for a $p\times p $ matrix $(s_{ij}\1_{\{i,j\in J\}})_{i,j\le p}$, and if we write $\Sigma_J^{-1}$, we mean $(t_{ij}\1_{\{i,j\in J\}})_{i,j\le p}$, where $(t_{ij}\1_{\{i,j\in J\}})_{i,j\in J}$ is the inverse of $(s_{ij}\1_{\{i,j\in J\}})_{i,j\in J}$; note that we will be working only with positive definite matrices $\Sigma$, so $(s_{ij}\1_{\{i,j\in J\}})_{i,j\in J}$ will always be invertible.
We will use the following abbreviations to state our results:
\[
\ell_* = \ell_*(\Sigma_{J^c}^{1/2}B_q^p) \quad \mbox{ and } \quad  d_*= d_*(\Sigma_{J^c}^{-1/2}B_{q'}^p).
\]
Besides the Dvoretzky dimension $d_*$, the sharper results of Section~\ref{sec:main_results_weak_moment_optimal} involve a second, always of not smaller order, notion of effective rank of the noise absorption subspace.

For $\bbeta_J\in V_J$ we define
\begin{align}\label{eq:def_vertiii}
    \vertiii{\bbeta_J} = \sup\Bigl\{\langle \bbeta_J,\vu\rangle:\, \vu\in \frac{
    \sqrt{N}}{\ell_*} B_q^J \cap \Sigma_J^{-1/2}B_2^J\Bigr\}.
\end{align}

The main theorem of this section is stated below.
The proof of Theorem~\ref{thm:conv-conc} can be found in Subsection~\ref{proof-sketch} and Section~\ref{sect:proof-conv-conc}. Note that this theorem covers the classical Gaussian setting, when $X\sim \mathcal{N}(0,\Sigma)$ and $\xi\sim \mathcal{N}(0,\sigma_{\xi}^2)$. It holds even when $\Sigma$ is not diagonal, unlike in the weak moments and log-concave cases.

\begin{Theorem}[Convex concentration case]\label{thm:conv-conc}
Suppose $\xi$ is independent of $X$, $J\subset[p]$ is such that $X_J$ is independent of $X_{J^c}$,  and  that $\xi$ and $X$ satisfy Assumption~\ref{assumption-1}.
Let $\hat\bbeta$ be the minimum $\ell_q$-norm interpolant estimator in the linear regression problem with $1<q <\infty$. Then there exist  constants $C=C(q,\rho)\ge 1$, $c=c(q,\rho)\in(0,1)$ and $ \kappa_{DM}= \kappa_{DM}(q,\rho)\in(0,1)$ depending only on $q$ and $\rho$, such that whenever $C|J| \le N\le \kappa_{DM}d_*$,
with probability at least
\begin{equation}\label{eq:probability-conv-conc}
1- 2\exp\bigl(-c \min\{ |J|, d_* \}\bigr),
\end{equation}
we have
\begin{equation}\label{eq:aim-conv-conc}
	\|\langle\hat\bbeta-\bbeta^*,X\rangle\|_{L^2(\mu_X)} \le C \biggl( r(V_J,V_{J^c}) + \sqrt{\frac{N}{d_*}}\,\sigma_\xi +  \|\Sigma_{J^c}^{1/2}\bbeta_{J^c}^*\|_2 \biggr),
\end{equation}
 where
        \begin{align}\label{eq:def_r_star_final}
  r =   r(V_J,V_{J^c}) = C \begin{cases}
       \Bigl(\frac{|J|}{N}\Bigr)^{\frac{1}{2(q-1)}} S
        +   \Bigl( \frac{\ell_*}{\sqrt{N}}\Bigr)^{\frac{q}{q-1}}B^{\frac{1}{q-1}}
        & \text{if } q\geq 2,\\
       \Bigl(\frac{|J|}{N}\Bigr)^{\frac{1}{2(q-1)}} S
               + \Bigl( \frac{\ell_*}{\sqrt{N}}\Bigr)^{\frac{q}{q-1}}  B^{\frac{1}{q-1}}
        + \|\bbeta_{J}^\ast\|_q \frac{\ell_\ast}{\sqrt{N}}
       +\sqrt{\frac{|J|}{N}} S
        +    \Bigl( \frac{\ell_*}{\sqrt{N}}\Bigr)^{q}  BS^{2-q}
           &\text{if } 1<q<2,
    \end{cases}
    \end{align}
    \[
    S=\sigma_\xi  +  \|\Sigma_{J^c}^{1/2}\bbeta_{J^c}^*\|_2, \qquad B=  \vertiii{\bbeta_J^*\odot|\bbeta_J^*|^{\odot(q-2)}},
    \]
   and  $ \bbeta_J^*\odot|\bbeta_J^*|^{\odot(q-2)} =  \sign(\bbeta_J^*)\odot|\bbeta_J^*|^{\odot(q-1)} $ is the vector in $V_J$ with $j$-th coordinate, $j\in J$, equal to $\beta_j^*|\beta_j^*|^{q-2}= \sign \beta_j^* \cdot |\beta_j^*|^{q-1}$.

\end{Theorem}

We remark that since the constant $C$ appears already in \eqref{eq:aim-conv-conc}, its presence in the definition \eqref{eq:def_r_star_final} of $r(V_J,V_{J^c})$ is not necessary. We introduce it there in order to simplify the proofs and avoid multiple adjustments of constants in formulations of auxiliary facts and objects.

Let us record the only compatibility between $\Sigma$ and the canonical basis of $\bR^p$ that is used in the convex concentration case.

\begin{Remark}[Two-block structure of $\Sigma$]\label{rem:block-diagonal}
Since $X$ is centered and $X_J$ is independent of $X_{J^c}$, we have $\bE\langle X,\be_i\rangle\langle X,\be_j\rangle=0$ for every $i\in J$ and $j\in J^c$. Hence $\Sigma=\Sigma_J+\Sigma_{J^c}$ is block-diagonal with respect to the splitting $\bR^p=V_J\oplus^\perp V_{J^c}$, and so is every power of $\Sigma$; in particular $P_I\Sigma^{\pm 1/2}=\Sigma_I^{\pm 1/2}$ for $I\in\{J,J^c\}$. Within each of the two blocks the matrix $\Sigma$ is arbitrary, so in Theorem~\ref{thm:conv-conc} the covariance matrix need not be diagonal. The two-block structure is used  in the proof of Proposition~\ref{prop:probab-omega-multi} only. Let us stress that the splitting $\|\bbeta\|_q^q=\|\bbeta_J\|_q^q+\|\bbeta_{J^c}\|_q^q$, which is the algebraic engine of the FSD method for the $\ell_q$-norm (see \eqref{eq:beta_J-minimizer}), forces $V_J$ and $V_{J^c}$ to be \emph{coordinate} subspaces; on the other hand the FSD inequality \eqref{eq:aim-triangle} is valid for an arbitrary $\Sigma$, because $\|\Sigma^{1/2}P_I\vv\|_2=\|\Sigma_I^{1/2}\vv\|_2$ for every $\vv\in\bR^p$ and $I\in\{J,J^c\}$. In the weak moment and the log-concave cases we will assume in addition that $\Sigma$ is diagonal (see Assumptions~\ref{assumption-2} and \ref{assumption-3}); there diagonality is used only to obtain the Dvoretzky--Milman estimates of Section~\ref{sec:proof_DM_q}.
\end{Remark}

\subsection{Sufficient conditions for benign overfitting in the convex concentration case}
\label{sec:BO_conditions_conv_conc}

We say that the minimum $\ell_q$-norm interpolant estimator has the \emph{benign overfitting} property along a sequence of problems $(N,p,\Sigma,\bbeta^*,\sigma_\xi)_{N\ge 1}$ with $\sigma_\xi>0$ and $p>N$ if its excess risk $\|\Sigma^{1/2}(\hat\bbeta-\bbeta^*)\|_2^2$ tends to zero in probability while $\hat\bbeta$ interpolates the data, so that its empirical risk is identically zero. Theorem~\ref{thm:conv-conc} yields the following sufficient conditions. Recall that we call a splitting $\bR^p=V_J\oplus^\perp V_{J^c}$ an \emph{admissible FSD} when $X_J$ is independent of $X_{J^c}$.

\begin{Corollary}\label{cor:BO_conv_conc}
Let $1<q<\infty$ and let $C,\kappa_{DM}$ be the constants from Theorem~\ref{thm:conv-conc}. Grant Assumption~\ref{assumption-1} and assume that $\xi$ is independent of $X$ and that $\sigma_\xi=O(1)$. Assume that along a sequence of problems as above there exists an admissible FSD $\bR^p=V_J\oplus^\perp V_{J^c}$ such that $|J|\to \infty$, $C|J|\le N\le \kappa_{DM}d_*$ and
\begin{equation}
    \label{eq:BO_conditions_conv_conc}
     \frac{|J|}{N}\longrightarrow 0, \qquad
     \|\bbeta_J^*\|_q\frac{\ell_*}{\sqrt N}\longrightarrow 0, \qquad
     \|\Sigma_{J^c}^{1/2}\bbeta_{J^c}^*\|_2 \longrightarrow 0, \qquad
      \frac{N\sigma_\xi^2}{d_*}\longrightarrow 0 .
\end{equation}
Then $\hat\bbeta$ has the benign overfitting property.
\end{Corollary}

\begin{proof}
    By H{\"o}lder's inequality and the definition \eqref{eq:def_vertiii} of $\vertiii{\cdot}$ we have
    \begin{equation}
        \label{eq:crude_bound_B}
        B=\vertiii{\bbeta_J^*\odot|\bbeta_J^*|^{\odot(q-2)}} \le \frac{\sqrt N}{\ell_*}\bigl\|\bbeta_J^*\odot|\bbeta_J^*|^{\odot(q-2)}\bigr\|_{q'} = \frac{\sqrt N}{\ell_*}\|\bbeta_J^*\|_q^{q-1}.
    \end{equation}
    Since $(q-1)(q'-1)=1$, plugging \eqref{eq:crude_bound_B} into \eqref{eq:def_r_star_final} gives
    \[
    \Bigl( \frac{\ell_*}{\sqrt{N}}\Bigr)^{q'}B^{q'-1} \le \|\bbeta_J^*\|_q\frac{\ell_*}{\sqrt N}
    \qquad \text{and} \qquad
    \Bigl( \frac{\ell_*}{\sqrt{N}}\Bigr)^{q}  BS^{2-q} \le \Bigl(\|\bbeta_J^*\|_q\frac{\ell_*}{\sqrt N}\Bigr)^{q-1}S^{2-q}.
    \]
    As $S=\sigma_\xi+\|\Sigma_{J^c}^{1/2}\bbeta_{J^c}^*\|_2=O(1)$ by assumption, each of the five terms in \eqref{eq:def_r_star_final} tends to zero under \eqref{eq:BO_conditions_conv_conc}, i.e. $r(V_J,V_{J^c})\to0$. The assertion now follows from \eqref{eq:aim-conv-conc} together with \eqref{eq:probability-conv-conc}, since $N\le \kappa_{DM}d_*$ forces $d_*\to\infty$ as $N\to\infty$, so the probability in \eqref{eq:probability-conv-conc} tends to one.
\end{proof}

The four conditions in \eqref{eq:BO_conditions_conv_conc} have an interpretation within the FSD framework. The first one says that the estimation subspace $V_J$ has a dimension negligible with respect to the sample size, which is what makes a uniform convergence argument possible on $V_J$. The second one is a \emph{self-regularization} condition: by \eqref{eq:beta_J-minimizer} and \eqref{eq:def-omega-DM-cA}, $\hat\bbeta_J$ behaves like an empirical risk minimizer penalized by $\ell_*^q\norm{\cdot}_q^q$, and $\|\bbeta_J^*\|_q\ell_*/\sqrt N$ is precisely the bias that such a penalization creates. The third one is the \emph{alignment} condition between the signal and the covariance structure: the features of $V_{J^c}$ are spent on interpolating the noise, so whatever part of the signal lives in $V_{J^c}$ is simply lost. The last one quantifies the \emph{price of overfitting}: the noise is absorbed by $V_{J^c}$ at a cost measured by the Dvoretzky dimension $d_*$ of that subspace. As we will see in Section~\ref{sec:main_results_weak_moment_optimal}, the last condition is not optimal and can be improved into $N\sigma_\xi^2=o(R_*)$, where $R_*\gtrsim_q d_*$ is the effective rank defined as follows.  For a diagonal $\Sigma=\diag(\sigma_1,\dots,\sigma_p)$ we set
\begin{equation}
    \label{eq:def_effective_rank_R_star}
    R_* = R_*(V_{J^c}) = \frac{\ell_*^{2q'}}{\Tr(\Sigma_{J^c}^{q'})} = \frac{\ell_*(\Sigma_{J^c}^{1/2}B_q^p)^{2q'}}{\sum_{j\in J^c}\sigma_j^{q'}} \sim_q \frac{\bigl(\sum_{j\in J^c}\sigma_j^{q'/2}\bigr)^2}{\sum_{j\in J^c}\sigma_j^{q'}},
\end{equation}
the last equivalence being a consequence of Lemma~\ref{lemma:estimate_Gaussian_mean_width}. As explained in Section~\ref{sec:app-lemma}, $R_*$ is (up to a constant depending on $q$) the quantity $R_{\norm{\cdot}_q}(\Sigma_{J^c})=\ell_*^2/(\bE\|\Sigma_{J^c}^{1/2}\bv^*\|_2)^2$ from \eqref{eq:def_effective_rank}, which is the effective rank appearing in the fourth condition in \eqref{eq:sufficient_conditions_KZSS20}; for $q=2$ it is the effective rank $R_{k}=(\Tr\Sigma_{J^c})^2/\Tr(\Sigma_{J^c}^2)$ of \cite{bartlettBenignOverfittingLinear2020}, while $d_*$ is their $r_k = \Tr(\Sigma_{J^c})/\sigma_{k+1}$. The inequality $d_*\lesssim_q R_*$ is always true: it is exactly \eqref{eq:Holder_R_versus_d} below.

\subsection{Comparison with the literature}
\label{sec:comparison_literature}

Let us compare Corollary~\ref{cor:BO_conv_conc} with the two main references on minimum $\ell_q$-norm interpolants, \cite{koehler2021uniform} and \cite{donhauser2022fast}. Both work with a Gaussian design and rely on the CGMT, whereas Corollary~\ref{cor:BO_conv_conc} only requires the convex concentration property from Assumption~\ref{assumption-1} (and Corollary~\ref{cor:BO_weak_moment} below only requires weak moments).

\paragraph{Comparison with \cite{koehler2021uniform}.} The sufficient conditions for benign overfitting obtained in \cite{koehler2021uniform} are recalled in \eqref{eq:sufficient_conditions_KZSS20}; specialized to $\norm{\cdot}=\norm{\cdot}_q$ they read
\[
\frac{|J|}{N}\to0, \qquad \frac{\|\bbeta^*\|_q\ell_*}{\sqrt N}\to0, \qquad \diam\bigl(\Sigma_{J^c}^{1/2}B_q^p\bigr) = o(\ell_*), \qquad
\sqrt N\,\bE\norm{\Sigma^{1/2}_{J^c}\bv^*}_2 = o(\ell_*),
\]
together with $\bP(\norm{P_{J^c}\bv^*}_q^2>1+\eta)\to0$ for every $\eta>0$. The dictionary with \eqref{eq:BO_conditions_conv_conc} is the following. The first conditions coincide, since ${\rm rank}(\Sigma_J)=|J|$ for a positive definite $\Sigma$. Our second condition is weaker than theirs, because it only involves the projection $\bbeta_J^*$ of the signal onto the estimation subspace. Their second and third conditions together imply our alignment condition whenever $N=O(d_*)$: indeed $\|\Sigma_{J^c}^{1/2}\bbeta_{J^c}^*\|_2 \le \|\bbeta^*\|_q\diam(\Sigma_{J^c}^{1/2}B_q^p) = \|\bbeta^*\|_q\ell_*d_*^{-1/2}=o(\sqrt{N/d_*})$. Their third condition is exactly $d_*\to\infty$, and it is implied by -- but strictly weaker than -- our Dvoretzky--Milman requirement $N\le \kappa_{DM}d_*$; this is the price we pay for dispensing with the CGMT. Finally, their fourth condition is equivalent to $N=o(R_*)$ (see \eqref{eq:def_effective_rank_R_star}), which is precisely the improved price of overfitting obtained in Theorem~\ref{thm:weak-moment-optimal} below, and \emph{not} what follows from Theorem~\ref{thm:conv-conc}: this is the main motivation for the refined analysis of the noise absorption part carried out in Section~\ref{sec:noise_absorption_part}. Their last condition has no counterpart in our analysis. Summarizing: up to the Dvoretzky--Milman condition $N\lesssim_{q,\rho} d_*$, the conditions of \cite{koehler2021uniform} imply ours, our bounds are non-asymptotic, and they hold far beyond the Gaussian case.

\paragraph{Comparison with \cite{donhauser2022fast}.} In \cite{donhauser2022fast} the design is standard Gaussian, $\Sigma=I_p$, the signal is $1$-sparse, $\bbeta^*=\be_1$, $1<q<2$, and the rate \eqref{eq:Konstantine_rate} is established in the regime
\[
N\log^{\kappa_3}N\lesssim p \lesssim N^{q'/2}\log^{-\kappa_4}N .
\]
By Lemma~\ref{lemma:estimate_Gaussian_mean_width} below, if $1<q<2$ and $\Sigma=I$, then $\ell_*\sim_q|J^c|^{1/q'}$ and $\diam(\Sigma_{J^c}^{1/2}B_q^p)=1$, so that $d_*\sim_q|J^c|^{2/q'}$; consequently their upper restriction on $p$ reads $d_*\lesssim N\log^{-2\kappa_4/q'}N$. In other words, \cite{donhauser2022fast} describes exactly the regime that is \emph{complementary} to the Dvoretzky--Milman regime $N\le \kappa_{DM}d_*$ in which our analysis takes place. This is not an artifact of the proofs but a genuine geometric phenomenon: once $N\lesssim d_*$, the $\ell_{q'}$-norm on the random $N$-dimensional subspace $\Range(\bX_{J^c}^\top)$ is, up to constants, a multiple of the Euclidean norm, so the $\ell_q$ inductive bias can no longer take advantage of the sparsity of $\bbeta^*$. Example~\ref{ex:isotropic} below confirms this: on an isotropic design and in the Dvoretzky--Milman regime, no benign overfitting occurs for a signal of constant norm, whatever its sparsity, and the two analyses therefore describe two different mechanisms -- theirs the sparsity-driven one, ours the one driven by the decay of the covariance profile together with alignment. The two pictures do match on the quantity they share, namely the price paid for absorbing the noise: on an isotropic design the sharp bound of Section~\ref{sec:noise_absorption_part} gives $\|\Sigma_{J^c}^{1/2}\hat\bbeta_{J^c}\|_2^2\lesssim_q N\sigma_\xi^2/p$, which matches -- up to a constant depending on $q$ only -- the second term $\sigma_\xi^2N/p$ in the lower bound of \eqref{eq:Konstantine_rate}, whereas the crude bound \eqref{eq:noise_absorption_part_intro} would only give the much larger $N\sigma_\xi^2/d_*\sim_q N\sigma_\xi^2p^{-2/q'}$.

\subsection{Examples of covariance profiles}
\label{sec:examples_profiles}

For completeness, we include estimates for $\ell_*(AB_q^p)$ and $\diam(AB_q^p)$ in the case when the matrix $A$ is diagonal. We will usually apply it to  $A=\Sigma_{J^c}^{1/2}$ . The proof of the following lemma is standard (see \eqref{eq:c0'} and \eqref{eq:c0} below).
\begin{Lemma}\label{lemma:estimate_Gaussian_mean_width}
    Suppose that $q>1$ and $A$ is diagonal with the diagonal $a=(a_1,\dots,a_p)$ and $a_j\ge 0$. Then there exist  constants $\nc\label{c_estimate_GMM}=\oc{c_estimate_GMM}(q)\in(0,1)$ and $\nC\label{C_estimate_GMM} = \oC{C_estimate_GMM}(q)>1$, depending only on $q$,  such that
    \[
    \oc{c_estimate_GMM}\|a\|_{q'}\leq \ell_*(A B_q^p)\leq \oC{C_estimate_GMM}\|a\|_{q'}.
    \]
    Moreover,
    \[
    \diam(A B_q^p)=
    \begin{cases}
        \max_{j\le p} a_j &\text{if }q\le 2,\\
         \|a\|_{\frac{2q'}{2-q'}} &\text{if }q> 2
    \end{cases}
    =
    \begin{cases}
        \|a\|_\infty &\text{if }q\le 2,\\
        \|a\|_{\frac{2q}{q-2}}  &\text{if }q> 2.
    \end{cases}
    \]
\end{Lemma}

Throughout the rest of this subsection $\Sigma=\diag(\sigma_1,\dots,\sigma_p)$ with $\sigma_1\ge\dots\ge\sigma_p>0$ and the FSD is given by $J=\{1,\dots,k\}$, so that $V_{J^c}$ is spanned by the $p-k$ least informative features. Lemma~\ref{lemma:estimate_Gaussian_mean_width} applied to $A=\Sigma_{J^c}^{1/2}$ yields
\begin{equation}
    \label{eq:profile_quantities}
    \ell_* \sim_q \Bigl(\sum_{j>k}\sigma_j^{q'/2}\Bigr)^{1/q'}, \qquad
    d_* \sim_q \frac{\bigl(\sum_{j>k}\sigma_j^{q'/2}\bigr)^{2/q'}}{\displaystyle\begin{cases}\sigma_{k+1} & \text{if } q\le 2,\\[1mm] \bigl(\sum_{j>k}\sigma_j^{\frac{q}{q-2}}\bigr)^{\frac{q-2}{q}} & \text{if } q>2,\end{cases}}
    \qquad
    R_* \sim_q \frac{\bigl(\sum_{j>k}\sigma_j^{q'/2}\bigr)^{2}}{\sum_{j>k}\sigma_j^{q'}}.
\end{equation}
H{\"o}lder's inequality with the exponents $\frac{1}{2-2/q'}$ and $\frac{1}{2/q'-1} = \frac{q}{q-2}$ when $q>2$, and the trivial bound $\sum_j\sigma_j^{q'}\le \max_j\sigma_j^{q'/2}\sum_j\sigma_j^{q'/2}$ together with $\max_j\sigma_j^{q'/2}\le \sum_j\sigma_j^{q'/2}$ when $q\le2$, give in both cases
\begin{equation}
    \label{eq:Holder_R_versus_d}
    \sum_{j>k}\sigma_j^{q'} \le \Bigl(\sum_{j>k}\sigma_j^{q'/2}\Bigr)^{2-\frac{2}{q'}} \diam\bigl(\Sigma_{J^c}^{1/2}B_q^p\bigr)^2,
    \qquad\text{hence}\qquad d_*\lesssim_q R_*.
\end{equation}
The gap between the two effective ranks can be very large, which is exactly what the analysis of Section~\ref{sec:noise_absorption_part} exploits.

\begin{Example}[isotropic profile]\label{ex:isotropic}
    Let $\Sigma=I_p$ and $m=p-k=|J^c|$. By \eqref{eq:profile_quantities},
    \[
    \ell_*\sim_q m^{1/q'}, \qquad d_*\sim_q \begin{cases} m^{2/q'} & \text{if } q\le 2,\\ m & \text{if } q>2,\end{cases}
    \qquad R_*\sim_q m .
    \]
    For $q\le 2$ we have $2/q'=2(q-1)/q\le 1$, so $R_*/d_*\sim_q m^{(2-q)/q}$ is polynomially large in $m$ as soon as $q<2$, whereas for $q>2$ the two effective ranks are of the same order. Isotropy is, however, the least favourable profile: in the Dvoretzky--Milman regime $N\le\kappa_{DM}d_*$ we get
    \[
    \frac{\ell_*}{\sqrt N} \gtrsim_q \frac{m^{1/q'}}{\sqrt{d_*}} \sim_q \begin{cases} 1 & \text{if }q\le2,\\ m^{\frac1{q'}-\frac12} & \text{if } q>2,\end{cases}
    \]
    which is bounded away from zero in both cases (and diverges when $q>2$, since then $1/q'>1/2$). Hence, the second condition in \eqref{eq:BO_conditions_conv_conc} forces $\|\bbeta_J^*\|_q\to0$, and the third one forces $\|\bbeta_{J^c}^*\|_2\to0$: for an isotropic design and in the Dvoretzky--Milman regime, our sufficient conditions show that benign overfitting can occur for a signal that vanishes asymptotically, irrespective of its sparsity. For $\bbeta^*=0$ -- pure noise interpolation -- take $J=\emptyset$, then $m=p$ and $\hat\bbeta=\cA[\bxi]$. The excess risk is at most $C\sigma_\xi^2\frac{N}{p}$ by Proposition~\ref{prop:objective_cA_xi}, instead of $C\sigma_\xi^2\frac{N}{d_*}=C\sigma_\xi^2Np^{-2/q'}$ for $q\leq 2$; and $C\sigma_\xi^2N/p$ for $q>2$ given by Theorem~\ref{thm:conv-conc}; as discussed in Subsection~\ref{sec:comparison_literature}, the former is sharp.
\end{Example}

\begin{Example}[power-law profiles]\label{ex:power_law}
    Let $\sigma_i = \Log^\beta(1+i)\,i^{-\alpha}$ with $\alpha>0$ and $\beta\in\bR$, as in \cite{bartlettBenignOverfittingLinear2020}, and set
    \[
    \gamma = \frac{\alpha q'}{2}, \qquad b= \frac{\beta q'}{2},
    \]
    so that $\sigma_j^{q'/2}= \Log^b(1+j)j^{-\gamma}$. The critical decay is $\gamma=1$, that is $\alpha=2/q'=2(q-1)/q$.

    \emph{(i) Supercritical decay, $\alpha>2/q'$.} Then $\sum_{j>k}\sigma_j^{q'/2}\sim_{\gamma,b}k^{1-\gamma}\Log^b k$ and all the logarithmic factors in \eqref{eq:profile_quantities} cancel out, leaving
    \[
    d_*\sim_{q,\alpha,\beta} k^{\min\{2/q',\,1\}} .
    \]
    Consequently the requirement $C|J|=Ck\le N\le \kappa_{DM}d_*$ of Theorem~\ref{thm:conv-conc} cannot be met at all when $q<2$, and forces $N\sim_{q,\alpha,\beta}k$ when $q\ge2$, which contradicts the first condition in \eqref{eq:BO_conditions_conv_conc}. Thus our sufficient conditions do not establish benign overfitting for these covariance profiles for the minimum $\ell_q$-norm interpolant: the noise absorption subspace is then too small in the Dvoretzky sense to absorb the noise.

    \emph{(ii) Critical decay, $\alpha=2/q'$, and $1<q\le 2$.} Assume first $b<-1$, i.e., $\beta<-2/q'$. The tail $\sum_{j>k}j^{-1}\Log^b(1+j)\sim_b\Log^{b+1}k$ and so
    \[
    \ell_*\sim_{q,\beta}\Log^{\frac{b+1}{q'}}k, \qquad d_*\sim_{q,\beta}\bigl(k\Log k\bigr)^{2/q'}, \qquad R_*\sim_{q,\beta}k\Log^2k .
    \]
    For $q=2$ this is the setting of \cite{bartlettBenignOverfittingLinear2020}: $d_*\sim_\beta k\Log k$, $R_*\sim_\beta k\Log^2k$, and the choice $k\sim N/\Log N$ satisfies all the conditions of Corollary~\ref{cor:BO_weak_moment}, with $|J|/N\sim1/\Log N$, $N/R_*\sim 1/\Log N$ and $\ell_*/\sqrt N\sim \Log^{(\beta+1)/2}(N)N^{-1/2}$; benign overfitting holds with the logarithmic rate $\Log^{-1/2}N$. For $q<2$ one has $2/q'<1$, hence $d_*=o(k)$ and the conditions are void. Assume now $b>-1$. Then $\sum_{j>k}\sigma_j^{q'/2}\sim_b \Log^{b+1}p$ whenever $\Log p\ge 2\Log k$, so that $d_*\sim_{q,\beta}k^{2/q'}\Log^{\frac{2(b+1)}{q'}}(p)\Log^{-\beta}k$, and the Dvoretzky--Milman condition becomes feasible for $q<2$ as well, but only for a number of features exponentially large in $N$, namely $\Log p\gtrsim N^{\frac{q'}{2(b+1)}(1-\frac{2}{q'})}$. Benign overfitting then holds, with logarithmic rates.

    \emph{(iii) Subcritical decay, $0<\alpha<1/q'$ (i.e. $0<\gamma<1/2$), and $1<q\le 2$.} All the sums in \eqref{eq:profile_quantities} are now dominated by their largest index, and for $p\ge k^2$ we get, up to factors $\Log^{O_{\beta}(1)}p$,
    \[
    \ell_*\sim_{q,\alpha} p^{\frac{1-\gamma}{q'}}, \qquad
    d_*\sim_{q,\alpha} p^{\frac{2(1-\gamma)}{q'}}k^{\alpha}, \qquad
    R_*\sim_{q,\alpha} p .
    \]
    Choose $\bbeta^*\in V_J$ with $\|\bbeta^*\|_q=O(1)$, so that the alignment condition holds exactly, and let $\sigma_\xi=O(1)$, $p=\lceil N^a\rceil$ and $k=\lceil N^c\rceil$. Writing $\tau = \frac{2a(1-\gamma)}{q'}$, the conditions of Corollary~\ref{cor:BO_weak_moment} read (again up to logarithmic factors) $c<1$, $\tau+c\alpha>1$, $\tau<1$ and $a>1$. They are simultaneously satisfiable: the interval
    \[
    \max\Bigl\{1,\ \frac{(1-\alpha)q'}{2(1-\gamma)}\Bigr\} < a < \frac{q'}{2(1-\gamma)}
    \]
    is nonempty because $\alpha>0$ and $\frac{q'}{2(1-\gamma)}>\frac{q'}{2}\ge1$ for $q\le2$, and then any $c\in\bigl(\frac{1-\tau}{\alpha},1\bigr)$ will do. For such a choice the excess risk of $\hat\bbeta$ satisfies, up to logarithmic factors,
    \[
    \|\Sigma^{1/2}(\hat\bbeta-\bbeta^*)\|_2 \lesssim_{q,\alpha,\beta} N^{-\delta}, \qquad \delta = \tfrac12\min\bigl\{1-c,\ (q-1)(1-\tau),\ a-1\bigr\}>0 ,
    \]
    so benign overfitting holds with a polynomial rate and a number of features polynomial in the sample size. The remaining range $1/q'\le \alpha\le 2/q'$ is treated in the same way, the only change being that $\sum_{j>k}\sigma_j^{q'}$, and hence $R_*$, also depends on $k$.
\end{Example}

Summarizing Examples~\ref{ex:isotropic} and \ref{ex:power_law}: for the minimum $\ell_q$-norm interpolant with $1<q\le 2$, our sufficient conditions for benign overfitting require a covariance profile decaying at a rate $\alpha\in(0,2/q']$, the critical decay $2/q'=2(q-1)/q$ being an increasing function of $q$. The stronger the inductive bias -- that is, the closer $q$ is to $1$ -- the flatter the admissible profiles, and the larger the number of features needed; at the same time a completely flat profile is excluded by Example~\ref{ex:isotropic}.
For $q=2$ we recover the classical picture of \cite{bartlettBenignOverfittingLinear2020}, with critical decay $\alpha=1$.

\subsection{Overview of the proof of Theorem~\ref{thm:conv-conc}}
\label{proof-sketch}

Below we present the overview of the proof of Theorem~\ref{thm:conv-conc} with a particular emphasis on the FSD methodology (in particular, we detail some interpretations and point of view that are specific to the FSD method). The proofs of the main results under the weak moment and log-concave settings (see Theorems \ref{thm:independent} and \ref{thm:log-concave} below) follow the same strategy. The detailed discussion of the steps which are different in these cases is provided in Section~\ref{sec:proofs_other_cases}.

Let us start with some notation. Recall that $X=\Sigma^{1/2}\vzeta$, where $\vzeta$ is an isotropic vector in $\bR^p$ and $\Sigma$ is the covariance matrix of $X$.
Let $Y_1,\dots ,Y_N$ be i.i.d.\ random vectors in $\bR^p$ with the same distribution as $\vzeta$.
Then $\bX$ -- the matrix with row vectors $X_1, \dots,X_N$ --  has the same distribution as the matrix with row vectors $\Sigma^{1/2}Y_1,\dots, \Sigma^{1/2}Y_N$.
This means that for any $\vbeta \in \bR^p$,
\begin{equation} \label{eq:bX-to-bY}
\bX\vbeta = \bY \Sigma^{1/2}\vbeta, \quad \text{or equivalently, }\  \bX =  \bY \Sigma^{1/2},
\end{equation}
 where $\bY$ is the matrix with rows $Y_1,\dots, Y_N$.

Let us now develop the FSD method for the minimum $\ell_q$-norm interpolant  estimator $\hat\bbeta$. The first crucial observation is that the projection $\hat\bbeta_J$ of $\hat\bbeta$ on $V_J=\bR^J$ is the solution to the following regularized minimization problem:
	\begin{align} \label{eq:beta_J-minimizer}
            \hat\bbeta_{J}
          &  = \argmin_{\bbeta_J\in V_J}\Bigl( \norm{ \cA[\vy - \bX_{J}\bbeta_{J}] }_q^q + \norm{\bbeta_{J}}_q^q \Bigr),
        \end{align}
        where, for any $\vmu \in \operatorname{Im} \bX_{J^c}$ (where $\operatorname{Im}$ stays for the image of the linear operator), we define
	\begin{equation}\label{eq:def-cA}
	             \cA[\vmu] = \argmin_{\vnu\in V_{J^c}}\bigl(\norm{\vnu}_q \colon \bX_{J^c}\vnu=\vmu\bigr).
        \end{equation}
Indeed, $ \hat\bbeta =\argmin_{\bbeta \in \bR^p} \bigl(\norm{\bbeta}_q^q \colon \bX \bbeta=\vy\bigr)$ where $\bX$ is linear, and $\|\bbeta\|_q^q=\|\bbeta_J\|_q^q+\|\bbeta_{J^c}\|_q^q$, so
        \begin{equation} \label{eq:beta_J^c-minimizer}
                   \hat\bbeta_{J^c}  = \cA[\vy - \bX_J \hat\bbeta_J]
        \end{equation}
        and  \eqref{eq:beta_J-minimizer} holds.
        Since $q>1$,  both minimization problems  \eqref{eq:beta_J-minimizer} and   \eqref{eq:def-cA} have a unique solution, because the $\ell_q$-norm is strictly convex in that case. Moreover, if $\mu$ is not in the image of $\bX_{J^c}$, then we formally define $\|\cA[\mu]\|_q= \infty$. Note that except the case $q=2$ the map $\cA$ is nonlinear and there are no explicit expressions for the solutions. It is however possible to bypass this issue by using duality argument and the Dvoretzky-Milman type results from Section~\ref{sec:proof_DM_q}.

Let us now introduce the dual problem (and its solution) to the (primal) minimization problem \eqref{eq:def-cA}: for all $\vmu \in \bR^N$,
\begin{equation}\label{eq:dual_problem_cA}
  \vlambda^*[\vmu] \coloneqq \argmax\Bigl(\inr{\vmu, \vlambda}: \norm{\bX_{J^c}^\top \vlambda}_{q^\prime}\leq1\Bigr). 
\end{equation}
As before, since $q>1$, the dual problem has a unique solution when  $\mu\neq 0$ and $\bX_{J^c}^T$ is injective. Since we will be working on the set $\Omega_{\mathrm{DM}}$ (defined below), the operator $\bX_{J^c}^T$ will always be injective, so $\vlambda^*[\vmu]$ will be uniquely defined as long as  $\vmu\neq0$. Moreover, using a duality argument (see Section~\ref{sect:proof-conv-conc} for more details) we get that then for any $\vmu \in \operatorname{Im} \bX_{J^c}$,
\begin{equation}\label{eq:duality-statement}
    \norm{\cA[\vmu]}_q = \min\Bigl\{\norm{\vnu}_q:\bX_{J^c}\vnu   = \vmu,\, \vnu\in V_{J^c}\Bigr\}
    = \max\bigl\{\inr{\vmu, \vlambda}: \norm{\bX_{J^c}^\top \vlambda}_{q^\prime}\leq1\bigr\}
    = \inr{\vmu, \vlambda^*[\vmu]},
\end{equation}
which justifies why  we call \eqref{eq:dual_problem_cA} the dual problem to \eqref{eq:def-cA}.

 The first step of the FSD method used to prove Theorem~\ref{thm:conv-conc}  is the triangle inequality based on the FSD $\bR^p = V_J \oplus^\perp V_{J^c}$:
\begin{equation}\label{eq:aim-triangle}
\|\langle\hat\bbeta-\bbeta^*,X\rangle\|_{L^2(\mu_X)}
= \| \Sigma^{1/2}(\hat\bbeta-\bbeta^\ast) \|_2
\le \| \Sigma^{1/2}_J(\hat\bbeta_J-\bbeta_J^\ast) \|_2 + \|\Sigma^{1/2}_{J^c}\hat\bbeta_{J^c}\|_2 +\|\Sigma^{1/2}_{J^c}\bbeta_{J^c}^\ast\|_2.
\end{equation}The third term appearing in \eqref{eq:aim-triangle} is deterministic, it is the `alignment term' between $\bbeta^*$ and $\Sigma$, it appears as it is in the final upper bound \eqref{eq:aim-conv-conc}. The heart of the proof is to provide high probability upper bounds for the first two terms in \eqref{eq:aim-triangle} that we call the `estimation part' $\| \Sigma^{1/2}_J(\hat\bbeta_J-\bbeta_J^\ast) \|_2$  and the `noise absorption part' $\|\Sigma^{1/2}_{J^c}\hat\bbeta_{J^c}\|_2$. In the FSD approach, we always start with the estimation part  since it is also used for the noise absorption part. 

Our strategy is to identify events (revealing crucial properties of $\bxi$, $\bX_J$ and $\bX_{J^c}$) onto which we can derive an upper bound on the estimation part and, later on, on the noise absorption part. This derivation is solely algebraic and may be used for other `stochastic scenarios' than the one considered in this paper. By doing so we disentangle the stochastic argument from the `algebraic part' of  the proof. All that remains is to show that the events onto which we develop our argument hold with large probability.

The event $\Omega_0$ identified to be sufficient for our analysis -- whose probability is not smaller than \eqref{eq:probability-conv-conc} --  is the intersection of the following four events:
\begin{align}
\Omega_{\mathrm{noise}}
&= \Bigl\{ C_{\mathrm{noise}}^{-1} \sigma_\xi
\le \frac{\|\bxi\|_2}{\sqrt{N}}
\le  C_{\mathrm{noise}}\sigma_\xi,
\  C_{\mathrm{noise}}^{-1}  \|\Sigma_{J^c}^{1/2} \bbeta_{J^c}^\ast \|_2
 \le \frac{\|\bX_{J^c} \bbeta_{J^c}^\ast \|_2}{\sqrt{N}}
  \le C_{\mathrm{noise}} \|\Sigma_{J^c}^{1/2} \bbeta_{J^c}^\ast \|_2
 \Bigr\}, \label{eq:def-omega-noise} \\ 
\Omega_{\mathrm{RIP}}  & = \Bigl\{ C_{\mathrm{RIP}}^{-1}  \|\Sigma_J^{1/2}\bbeta\|_2 \le \frac{\|\bX_{J} \bbeta \|_2}{\sqrt{N}}\le  C_{\mathrm{RIP}} \|\Sigma_J^{1/2}\bbeta\|_2 \quad \text{for all } \bbeta\in V_J \Bigr\}, \label{eq:def-omega-RIP}\\ 
\Omega_{\mathrm{DM}} & =  \Bigl\{ C_{\mathrm{DM}}^{-1} \ell_* \|\vlambda\|_2 \le \|\bX_{J^c}^\top \vlambda\|_{q'}\le  C_{\mathrm{DM}} \ell_* \|\vlambda\|_2 \quad \text{for all } \vlambda \in\bR^N \Bigr\}, \label{eq:def-omega-DM}\\ 
\Omega_{\mathrm{multi}} & = \bigl\{ \| \Sigma_J^{-1/2}\bX_J^\top \vlambda^*[\vy-\bX_J\bbeta_J^*]\|_2 \le C_{\mathrm{multi}}\sqrt{|J|}\, \|\vlambda^*[\vy-\bX_J\bbeta_J^*]\|_2\bigr\}, \label{eq:def-omega-multi}
\end{align}
where $\vlambda^*[\vy-\bX_J\bbeta_J^*]$ is defined in \eqref{eq:dual_problem_cA}.
The constants $C_{\mathrm{noise}}, C_{\mathrm{RIP}},C_{\mathrm{DM}},C_{\mathrm{multi}}\ge 1$ depend only on $q$ and $\rho$ from Assumption~\ref{assumption-1}; we will determine their values later. We will prove separately that each of these events holds with large probability, see Propositions~\ref{prop:probab-omega-noise}-\ref{prop:probab-omega-multi} below. In $\Omega_{\mathrm{RIP}}$ and $\Omega_{\mathrm{DM}}$, \textit{RIP} stands for the \textit{Restricted Isometry Property} in reminiscence to the classical RIP property used in Compressed Sensing from \cite{candes2005decoding} -- even though $\Omega_{\mathrm{RIP}}$ assumes a `unrestricted' isomorphy property of $\bX_J$ over the entire space $V_J$ -- and \textit{DM} stands for Dvoretzky-Milman since it is a $\ell_2^N\to \ell_{q'}^p$ isomorphic property that we assume $\bX_{J^c}^\top$ has for $\Omega_{\mathrm{DM}}$ to hold. The event $\Omega_{\mathrm{RIP}}$ is classical in mathematical statistics and machine learning, but $\Omega_{\mathrm{DM}}$ is not. That is why we had to prove it holds with large probability in our three scenarios and that we devote an entire section to this task, i.e.,  Section~\ref{sec:proof_DM_q}. Let us now explain one major role the  Dvoretzky-Milman isomorphic property play in the FSD method for the minimal $\ell_q$-norm IE.

It is a priori  difficult to work with the random non-linear operator $\cA$.
However, in our approach, only its $\ell_q$-norm plays a role. The Dvoretzky-Milman $\ell_{2}\to \ell_{q'}$ isomorphy property holding on $\Omega_{\mathrm{DM}}$  permits  to `exchange' $\|\cA[\vmu]\|_q$ with $\|\vmu\|_2/\ell_*$. More precisely,  it follows from the the duality property \eqref{eq:duality-statement} that
\begin{equation}\label{eq:def-omega-DM-cA}
\Omega_{\mathrm{DM}} \subset  \biggl\{ C_{\mathrm{DM}}^{-1}\frac{\norm{\vmu}_2}{\ell_*} \leq \norm{\cA[\vmu]}_q \leq C_{\mathrm{DM}} \frac{\norm{\vmu}_2}{ \ell_*} \quad \text{for all } \vmu\in\bR^N  \biggr\}.
\end{equation}It means that, even though $\cA$ is a non linear operator (unless $q=2$), it behaves like a $\ell_2\to \ell_q$ isomorphy on the event $\Omega_{\mathrm{DM}}$. This observation (applied to $\vmu = \vy - \bX_{J}\bbeta_{J}$) together with \eqref{eq:beta_J-minimizer} tells that $\hat\bbeta_J$ behaves like the RERM $\argmin_{\bbeta_J}\Bigl( \norm{\vy - \bX_{J}\bbeta_{J}}_2^q + \ell_*^q\norm{\bbeta_{J}}_q^q \Bigr)$; in particular, it coincides with a ridge estimator for $q=2$ (as in \cite{lecueGeometricalViewpointBenign2024}) and the square root LASSO \cite{belloni2011square} when $q=1$. As a consequence, we will analyze $\hat\bbeta_J$ as a RERM and show in Proposition~\ref{prop:hatbeta_in_set} that, on the event $\Omega_0$, 
\begin{equation}\label{eq:hatbeta-lies-set}
	\hat\bbeta_J \in \bbeta_J^* + r\Sigma^{-1/2}_JB_2^{J},
\end{equation}
where $r=r(V_J,V_{J^c})$ is given by \eqref{eq:def_r_star_final}. Property \eqref{eq:hatbeta-lies-set} yields in particular that  
\begin{equation}\label{eq:est-part-control}
\| \Sigma^{1/2}_J(\hat\bbeta_J-\bbeta_J^\ast) \|_2 \le r
\end{equation}
on $\Omega_0\subset \Omega_{\mathrm{RIP}}$.
 This achieves the `estimation part' of the FSD based proof.

The second and final step of the FSD method is to handle the noise absorption term $\|\Sigma^{1/2}_{J^c}\hat\bbeta_{J^c}\|_2$ in \eqref{eq:aim-triangle}. We first observe that by using the triangle inequality in \eqref{eq:aim-triangle} we give up on any hope that $\hat\bbeta_{J^c}$ could estimate the part of the signal $\bbeta^*_{J^c}$ it was a priori designed to estimate. We can clearly see here the price for overfitting the noise: the projected features $\bX_{J^c}$ are not used for signal estimation but for noise interpolation. As a consequence, $\|\Sigma^{1/2}_{J^c}\hat\bbeta_{J^c}\|_2$ is not an estimation error and thus cannot be handled using classical estimation tools. That is a good thing because we do not have such tools available on the high dimensional space $V_{J^c}$ such as a RIP but \emph{only} a Dvoretzky-Milman $\ell_2\to \ell_{q'}$ isomorphy property on  $\Omega_{\mathrm{DM}}$.

Let us now provide a first `quick' analysis on how to proceed to handle the noise absorption part that works in the three stochastic scenarios; a better control of  the noise absorption part is provided in Section~\ref{sec:noise_absorption_part} that works only under the 'weak moment' setup since its proof requires the independence of the coordinates of the design vector (a property we don't assume in the convex concentration and log-concave cases)   

First, as we said above, the estimation part helps us to handle the noise absorption part. Indeed, the convergence rate derived from \eqref{eq:hatbeta-lies-set} on $V_J$  together with the RIP holding on  $\Omega_{\mathrm{RIP}}$ lead to
\begin{equation}\label{eq:X_J(beta-beta*)-estimate}
\|\bX_J(\hat\bbeta_J -\bbeta_J^\ast)\|_2 
\leq 2C_{\mathrm{RIP}} \sqrt N r
\end{equation} 
on $\Omega_0$.
Next, using \eqref{eq:beta_J^c-minimizer} and \eqref{eq:def-omega-DM-cA}, we get that, on $\Omega_{\mathrm{DM}}$,
\begin{align*}
 \|\Sigma^{1/2}_{J^c}\hat\bbeta_{J^c}\|_2
& =  \|\Sigma^{1/2}_{J^c} \cA[\vy - \bX_J \hat\bbeta_J] \|_2
 \le  \sup_{\bbeta_{J^c} \in B_q^{J^c}} \|\Sigma^{1/2}_{J^c} \bbeta_{J^c}\|_2\cdot  \| \cA[\vy - \bX_J \hat\bbeta_J] \|_q
\\ & \le C_{\mathrm{DM}} \frac{\diam (\Sigma^{1/2}_{J^c} B_q^{J^c}) }{\ell_\ast} \| \vy - \bX_J \hat\bbeta_J\|_2
\le C_{\mathrm{DM}}\,  d_\ast^{-1/2} \bigl( \| \vy - \bX_J \bbeta_J^*\|_2 + \|  \bX_J (\hat\bbeta_J - \bbeta_J^*)\|_2\bigr)
\\ & =C_{\mathrm{DM}}\,  d_\ast^{-1/2}  \| \vxi + \bX_{J^c} \bbeta_{J^c}^\ast\|_2 + C_{\mathrm{DM}} \,  d_\ast^{-1/2}\|  \bX_J (\hat\bbeta_J - \bbeta_J^*)\|_2,
\end{align*}
where in the last equation we used  the fact that in our model $\vy=\bX\bbeta^\ast+\vxi= \bX_J\bbeta_J^\ast +\bX_{J^c}\bbeta_{J^c}^\ast+\vxi$, where $\vxi$ has i.i.d.\ coordinates distributed as $\xi$.
Recall that $N\le \kappa_{\mathrm{DM}}d_*$ where $\kappa_{\mathrm{DM}}\le 1$, so $N\le d_\ast$. This together with \eqref{eq:X_J(beta-beta*)-estimate} shows that the second term above is bounded by $C_\mathrm{DM}C_{\mathrm{RIP}}r$.
In order to bound the first term note that on the event $\Omega_{\mathrm{noise}}$,
\[
	 \| \vxi + \bX_{J^c} \bbeta_{J^c}^*\|_2 \le \|\vxi\|_2 + \|  \bX_{J^c} \bbeta_{J^c}^*\|_2
	 \le C_{\mathrm{noise}}\sqrt{N} (\sigma_\xi + \|\Sigma^{1/2}_{J^c}\bbeta_{J^c}^\ast\|_2 ),
\]
so
\begin{equation}
    \label{eq:noise_absorption_part_intro}
	 \|\Sigma^{1/2}_{J^c}\hat\bbeta_{J^c}\|_2  \le C_{\mathrm{DM}}C_{\mathrm{noise}} \biggl( \sigma_\xi \sqrt{\frac{N}{d_*}} +  \|\Sigma^{1/2}_{J^c}\bbeta_{J^c}^\ast\|_2\biggr) + C_{\mathrm{DM}} C_{\mathrm{RIP}} r,
\end{equation}
where we again used  $N \le d_\ast$.

Combining the control of the estimation part from \eqref{eq:est-part-control} together with the control on the noise absorption part from \eqref{eq:noise_absorption_part_intro} shows that  \eqref{eq:aim-conv-conc}  holds on $\Omega_0 = \Omega_{\mathrm{noise}} \cap  \Omega_{\mathrm{RIP}} \cap \Omega_{\mathrm{DM}} \cap  \Omega_{\mathrm{multi}}$. Therefore, in order to prove Theorem~\ref{thm:conv-conc} we only have to justify \eqref{eq:duality-statement}, show that \eqref{eq:hatbeta-lies-set} holds on $\Omega_0$, and prove that all four events $\Omega_{\mathrm{noise}}, \Omega_{\mathrm{RIP}}, \Omega_{\mathrm{DM}}$, and $\Omega_{\mathrm{multi}}$ hold with large probabilities. We postpone all these proofs to Sections~\ref{sect:proof-conv-conc} and \ref{sec:proof_DM_q}.

Note that we will use the event $\Omega_{\mathrm{multi}}$  only to prove that  \eqref{eq:hatbeta-lies-set} holds on $\Omega_0$, and we will use the assumption that $X_J$ and $X_{J^c}$ are independent only to prove that $\Omega_{\mathrm{multi}}$ has high probability. Let us remark that it will follow by classical concentration inequalities that $\Omega_{\mathrm{noise}}, \Omega_{\mathrm{RIP}}, \Omega_{\mathrm{multi}}$ hold with probability close to one. On the other hand, the proof that $\Omega_{\mathrm{DM}}$ holds with high probability requires deriving new Dvoretzky--Milman-type theorems for matrices with independent columns, i.e., in a different setting than the one considered usually in the literature. Note that we do not need to be as precise as in the classical Dvoretzky--Milman theorem. We only need to give two sided estimates for the $q$-th norm of a random linear transformation of vectors $\vlambda\in \bR^N$ up to a constant, which may be much bigger than $1$. In other words, we only need to provide an \textit{isomorphic}, instead of an \textit{isometric}, version of the Dvoretzky--Milman theorem for $\ell_q$-norms.  Section~\ref{sec:proof_DM_q} is entirely devoted to the Dvoretzky--Milman isomorphic results under our three scenarios. Once again, it is important in mathematical statistics to prove a phenomenon in various settings in order to achieve the \textit{universality} of this phenomenon, in particular, to break free from the Gaussian hypothesis.

\section{Main result with the improved control of the noise absorption part in the weak moment case}
\label{sec:main_results_weak_moment_optimal}

The three results stated in Theorems~\ref{thm:conv-conc}, \ref{thm:independent}, and \ref{thm:log-concave} handle the noise absorption part with the `quick' estimate \eqref{eq:noise_absorption_part_intro}, which charges the interpolation of the noise at the price $\sigma_\xi\sqrt{N/d_*}$. As already announced, this price can be improved under the extra assumption that the coordinates of the design vector are independent, an assumption granted in the 'weak moment' case. In the refined analysis provided in Section~\ref{sec:noise_absorption_part} of the noise assumption part,  the Dvoretzky dimension $d_*$ can be replaced by the larger effective rank $R_*$ from \eqref{eq:def_effective_rank_R_star}. We now combine the two ingredients -- the control of the estimation part in the weak moment setting, obtained exactly as in Theorem~\ref{thm:independent}, and the control of the noise absorption part from Corollary~\ref{cor:noise_absorption} -- into our sharpest result in the 'weak moment' case. It holds under weak moment assumptions on \emph{both} the design vector and the noise. In particular, we do not need to assume that the noise is Gaussian.

\subsection{Assumptions}
\label{subsec:assumptions_optimal}

Let $1<q<\infty$ and recall that $q'=q/(q-1)$. We work under the following three assumptions, in which $M\ge1$ and $\varepsilon>0$ are fixed.

\begin{itemize}
    \item[$(A_1)$] \emph{(Weak moments on the noise.)} The noise $\xi$ is independent of $X$, centered, with variance $\bE\xi^2=\sigma_\xi^2>0$ and a bounded kurtosis: $\bE\xi^4\le M^4\sigma_\xi^4$.
    \item[$(A_2)$] \emph{(Weak moments on the design.)} $X=\Sigma^{1/2}\vzeta$, where $\Sigma=\diag(\sigma_1,\dots,\sigma_p)$ is positive definite with $\sigma_1\ge\dots\ge\sigma_p>0$ and $\vzeta=(\zeta_1,\dots,\zeta_p)$ has independent isotropic coordinates satisfying $\bE|\zeta_j|^{\max\{4,2q'\}+\varepsilon}\le M^{\max\{4,2q'\}+\varepsilon}$ for every $j\le p$.
       \item[$(A_3)$] \emph{(Sample size window.)} $C|J|\le N\le C^{-1}\kappa_{\mathrm{DM}}\,d_*\Log(|J^c|)^{-2\cdot\1_{\{q>2\}}}$, where $C$ and $\kappa_{\mathrm{DM}}$ depend only on $q$ and $M$, and $\kappa_{\mathrm{DM}}$ is the constant from Proposition~\ref{prop:probab-omega-DM-independent}.
\end{itemize}

Let us record four consequences of assumptions $(A_1)$--$(A_3)$.
First, Assumptions $(A_1)$ and $(A_2)$ together are exactly Assumption~\ref{assumption-2}, which implies Assumption~\ref{assum:beyond_gaussian} with $M'=M$, needed in Section~\ref{sec:noise_absorption_part} to perform a more delicate analysis of the noise absorption part than the one presented previously in Section~\ref{proof-sketch}.  Second, since $\Sigma$ is diagonal and the coordinates of $\vzeta$ are independent, $X_J$ is automatically independent of $X_{J^c}$, so that no extra requirement about the independence has to be imposed in Theorem~\ref{thm:weak-moment-optimal} below: every coordinate splitting is an admissible FSD. Third, $(A_2)$ and $(A_3)$ imply Assumption~\ref{assum:LBDM} for every subset $I\subset J^c$ satisfying $d_*(\Sigma_I^{-1/2}B_{q'}^p)\ge C^{-1} d_*(\Sigma_{J^c}^{-1/2}B_{q'}^p)$ (because of items~\emph{1.} and~\emph{2.} of Theorem~\ref{thm:DM-independent} applied with $q:=q'$ and $A=\Sigma_I^{1/2}$). As we will see in Section~\ref{sec:noise_absorption_part}, these properties allow us to apply a better control of the noise absorption part presented therein. Hence, $(A_1)$--$(A_3)$ give us access at the same time to the four events constituting $\Omega_0$ (see \eqref{eq:def-omega-noise}--\eqref{eq:def-omega-multi} for the definition) and to Corollary~\ref{cor:noise_absorption}, which sums up the way we control the noise absorption part in Section~\ref{sec:noise_absorption_part}. 

\subsection{The main result}

Throughout this subsection we abbreviate
\begin{equation}
    \label{eq:def_r_L}
    L = \Log(|J^c|)^{\frac{q}{q-1}\1_{\{q>2\}}}
    \qquad\text{and}\qquad
    r_L = L\, r(V_J,V_{J^c}),
\end{equation}
where $r(V_J,V_{J^c})$ is defined in \eqref{eq:def_r_star_final}; thus $r_L$ is exactly the rate achieved on the estimation part in the weak moment setting, see \eqref{eq:aim-independent}. We also set
\begin{equation}
    \label{eq:def_Theta_q}
    \Theta_q = (\sigma_\xi+r_*)^{1-\theta_q}r_*^{\theta_q},
    \qquad
    \theta_q = \begin{cases} \frac{1}{3-q} & \text{if } 1<q<2,\\[1mm] \frac{1}{q-1} & \text{if } q\ge 2,\end{cases}
    \qquad
    r_* = C\bigl(r_L+\|\Sigma_{J^c}^{1/2}\bbeta_{J^c}^*\|_2\bigr).
\end{equation}

\begin{Theorem}[Weak moment case, improved noise absorption]\label{thm:weak-moment-optimal}
Let $1<q<\infty$, $J\subset [p]$, $M>0$, and $\hat\bbeta$ be the minimum $\ell_q$-norm interpolant estimator in the linear regression problem. Then there exist constants $C=C(q,M)\ge1$, $c=c(q,M)\in(0,1)$, $\varepsilon_0=\varepsilon_0(q,M)\in(0,1)$ and $\kappa_{\mathrm{DM}}=\kappa_{\mathrm{DM}}(q,M)\in(0,1)$, depending only on $q$ and $M$, such that if $(A_1)$--$(A_3)$ hold, then for every $\tau\in(0,1)$, with probability at least
\begin{equation}
    \label{eq:probability-weak-moment-optimal}
    1-\tau-\frac{C}{\min\{N^{\varepsilon/2},|J|\}}
    -  \begin{cases}
2e^{-c  N } + Cd_\ast^{-(q'-2)/4} &\text{if}\quad 1< q< 2,\\
 Cd_\ast^{-c\min\{\eps,\eps_0\}} &\text{if}\quad q=2, \\
 Cd_\ast^{-c\min\{\eps,\eps_0\}} +2\Log(|J^c|) e^{- cd_\ast/\Log^{2}(|J^c|)}  &\text{if}\quad  2< q,
\end{cases}
\end{equation}
we have
\begin{equation}
    \label{eq:aim-weak-moment-optimal}
	\|\langle\hat\bbeta-\bbeta^*,X\rangle\|_{L^2(\mu_X)} \le C\Biggl( r_L + \sigma_\xi\sqrt{\frac{N}{\tau R_*}} + \Theta_q + \|\Sigma_{J^c}^{1/2}\bbeta_{J^c}^*\|_2 \Biggr),
\end{equation}
where $R_*$ is the effective rank from \eqref{eq:def_effective_rank_R_star} and $r_L,\Theta_q$ are given by \eqref{eq:def_r_L} and \eqref{eq:def_Theta_q}.
\end{Theorem}

We postpone the proof to Section~\ref{sec:proof-thm-3-1}. Let us now compare \eqref{eq:aim-weak-moment-optimal} with \eqref{eq:aim-independent}. The estimation term $r_L$ and the alignment term $\|\Sigma_{J^c}^{1/2}\bbeta^*_{J^c}\|_2$ are unchanged. The noise absorption term $\sigma_\xi\sqrt{N/d_*}$ of \eqref{eq:aim-independent} has been replaced by $\sigma_\xi\sqrt{N/(\tau R_*)}$, which by \eqref{eq:Holder_R_versus_d} is smaller as soon as $\tau\gtrsim_q d_*/R_*$, and by a polynomial factor in the situations described in Examples~\ref{ex:isotropic} and~\ref{ex:power_law}; the price to pay is the polynomial, instead of exponential, deviation parameter $\tau$, which comes from the use of Markov's inequality in Proposition~\ref{prop:objective_cA_xi}. The additional term $\Theta_q$ is a geometric mean of $\sigma_\xi+r_*$ and $r_*$; it vanishes together with $r_*$, hence it never restricts the set of covariance profiles for which benign overfitting takes place. See Corollary~\ref{cor:BO_weak_moment} below for more details.

\subsection{Sufficient conditions for benign overfitting in the weak moment case}
\label{subsec:BO_optimal}

We now identify the conditions under which Theorem~\ref{thm:weak-moment-optimal} produces benign overfitting. They should be compared with \eqref{eq:BO_conditions_conv_conc}: the only -- but decisive -- difference (up to logarithmic factors) is that the Dvoretzky dimension $d_*$ in the last condition is replaced by the effective rank $R_*$. Recall however, that here we work in a different setting and that \eqref{eq:BO_conditions_conv_conc} are sufficient conditions for BO to hold without the independence assumption.

\begin{Corollary}\label{cor:BO_weak_moment}
Let $1<q<\infty$ and let $C,\kappa_{\mathrm{DM}}$ be the constants from Theorem~\ref{thm:weak-moment-optimal}. Set $\delta_0=\min\{\frac{1}{2(q-1)},\frac12\}$ and $\delta_1=\min\{1,q-1\}$. Grant $(A_1)$ and $(A_2)$ with constants $M,\varepsilon$ independent of $N$, and assume $\sigma_\xi=O(1)$. Assume that along a sequence of problems $(N,p,\Sigma,\bbeta^*,\sigma_\xi)_{N\ge1}$ there exists a coordinate splitting $\bR^p=V_J\oplus^\perp V_{J^c}$ satisfying $(A_3)$, with $|J|\to\infty$, and
\begin{equation}
    \label{eq:BO_conditions_weak_moment}
    L\Biggl[ \Bigl(\frac{|J|}{N}\Bigr)^{\delta_0} + \Bigl(\|\bbeta_J^*\|_q\frac{\ell_*}{\sqrt N}\Bigr)^{\delta_1}\Biggr] \longrightarrow 0,
    \qquad
    \|\Sigma_{J^c}^{1/2}\bbeta_{J^c}^*\|_2\longrightarrow 0,
    \qquad
    \frac{N\sigma_\xi^2}{R_*}\longrightarrow 0,
\end{equation}
with $L$ as in \eqref{eq:def_r_L}. Then $\hat\bbeta$ has the benign overfitting property.
\end{Corollary}

\begin{proof}
    Set $u=|J|/N$ and $v=\|\bbeta_J^*\|_q\ell_*/\sqrt N$. Exactly as in the proof of Corollary~\ref{cor:BO_conv_conc}, the bound \eqref{eq:crude_bound_B} on $B$ turns each of the five terms of \eqref{eq:def_r_star_final} into a power of $u$ or of $v$ multiplied by a power of $S=\sigma_\xi+\|\Sigma_{J^c}^{1/2}\bbeta_{J^c}^*\|_2=O(1)$: the exponents of $u$ are $\frac{1}{2(q-1)}$ and, when $1<q<2$, also $\frac12$, while the exponents of $v$ are $1$ and, when $1<q<2$, also $q-1$. Since $u\le C^{-1}<1$ by $(A_3)$ and since the first condition in \eqref{eq:BO_conditions_weak_moment} forces $v\le1$ eventually, the smallest exponents dominate and
    \[
    r(V_J,V_{J^c}) \le C\bigl(u^{\delta_0} + v^{\delta_1}\bigr).
    \]
    The first condition in \eqref{eq:BO_conditions_weak_moment} therefore says that $r_L=Lr(V_J,V_{J^c})\to0$, which with the second condition gives $r_*\to0$. Moreover, $\sigma_\xi$ being bounded, $\Theta_q=(\sigma_\xi+r_*)^{1-\theta_q}r_*^{\theta_q}\le Cr_*^{\theta_q}\to0$. Choosing $\tau=\tau_N=\sqrt{N\sigma_\xi^2/R_*}$, which tends to zero by the third condition, the noise absorption term in \eqref{eq:aim-weak-moment-optimal} becomes $(N\sigma_\xi^2/R_*)^{1/4}\to0$. Finally the probability \eqref{eq:probability-weak-moment-optimal} tends to one, because $\tau_N\to0$,  $|J|\to\infty$ and  $(A_3)$ forces $d_*\ge N/\kappa_{\mathrm{DM}}\to\infty$. Hence, the conclusion follows from Theorem~\ref{thm:weak-moment-optimal}.
\end{proof}

For $1<q\le2$, we recover the result announced in Theorem~\ref{thm:BO_intro}: we have $L=1$, $\delta_0=1/2$ and $\delta_1=q-1$, so \eqref{eq:BO_conditions_weak_moment} reads
\[
\frac{|J|}{N}\to0,\qquad \|\bbeta_J^*\|_q\frac{\ell_*}{\sqrt N}\to0,\qquad \|\Sigma_{J^c}^{1/2}\bbeta_{J^c}^*\|_2\to0,\qquad \frac{N\sigma_\xi^2}{R_*}\to0,
\]
which is \eqref{eq:BO_conditions_conv_conc} with $d_*$ replaced by $R_*$, obtained here under weak moment assumptions on both the noise and the design vector. As explained in Subsection~\ref{sec:comparison_literature}, the fourth condition above is exactly the fourth condition of \cite{koehler2021uniform} recalled in \eqref{eq:sufficient_conditions_KZSS20}, so that Theorem~\ref{thm:weak-moment-optimal} recovers -- and extends far beyond the Gaussian case -- the sufficient conditions of \cite{koehler2021uniform}, at the price of the Dvoretzky--Milman window $(A_3)$. The covariance profiles satisfying \eqref{eq:BO_conditions_weak_moment} are described in Examples~\ref{ex:isotropic} and~\ref{ex:power_law}: benign overfitting fails for an isotropic profile with a non-vanishing signal, and for a power-law profile $\sigma_i=\Log^\beta(1+i)i^{-\alpha}$ it occurs precisely when $0<\alpha\le2/q'$, the critical decay $2/q'=2(q-1)/q$ being an increasing function of $q$.


\section{Proof of Theorem~\ref{thm:conv-conc}, the main result in the convex concentration case}\label{sect:proof-conv-conc}

As mentioned in Section~\ref{proof-sketch}, Theorem~\ref{thm:conv-conc} will follow if we prove \eqref{eq:duality-statement}, show that \eqref{eq:hatbeta-lies-set} holds on $\Omega_0$, and show that all four events $\Omega_{\mathrm{noise}}, \Omega_{\mathrm{RIP}}, \Omega_{\mathrm{DM}}$, and $\Omega_{\mathrm{multi}}$ hold with large probabilities. Let us begin with the proof of \eqref{eq:duality-statement}. Even though it can be directly derived from the von Neuman-Sion minimax theorem, we provide a self-contained proof for it.

The Lagrange multiplier theorem yields that there exists $\vlambda \in \bR^N$ such that for the minimizer $\vnu=\cA[\vmu]$ of \eqref{eq:def-cA}, we have $q\vnu\odot|\vnu|^{\odot(q-2)}=\bX_{J^c}^\top\vlambda$, so
\begin{equation}\label{eq:maximizer-inclusion}
	\frac{\vnu\odot|\vnu|^{\odot(q-2)}}{\|\vnu\|_q^{q-1}} \in \operatorname{Im} \bX_{J^c}^\top;
\end{equation}
we will use this inclusion later on.
Since $q'(q-1)=q$, $\|\bX_{J^c}^\top\vlambda \|_{q'}= q\|\vnu\|_q^{q-1}$, so
\begin{equation}\label{eq:duality-proof}
	 \max\bigl\{\inr{\vmu, \tilde\vlambda}: \|\bX_{J^c}^\top \tilde\vlambda\|_{q^\prime}\leq1\bigr\}
	 \ge  \frac{\langle \vmu, \vlambda\rangle}{q\|\vnu\|_q^{q-1}}
	 = \frac{\langle \bX_{J^c}\vnu, \vlambda\rangle}{q\|\vnu\|_q^{q-1}}
	 = \frac{\langle \vnu, \bX_{J^c}^\top\vlambda\rangle}{q\|\vnu\|_q^{q-1}}
	  = \|\vnu\|_q =  \|\cA[\vmu]\|_q .
\end{equation}
On the other hand, if $\vmu =  \bX_{J^c} \vnu$ and $\| \bX_{J^c}^\top \vlambda\|_{q'} \le 1$, then
\[
	\langle\vmu, \vlambda \rangle
	= \langle \bX_{J^c}\vnu, \vlambda\rangle
	=  \langle \vnu, \bX_{J^c}^\top\vlambda\rangle
	\le \|\vnu\|_q \| \bX_{J^c}^\top \vlambda\|_{q'}
	\le \|\vnu\|_q,
\]
so
\[
\min\Bigl\{\norm{\vnu}_q:\bX_{J^c}\vnu   = \vmu,\, \vnu\in V_{J^c}\Bigr\}
    \ge \max\bigl\{\inr{\vmu, \vlambda}: \norm{\bX_{J^c}^\top \vlambda}_{q^\prime}\leq1\bigr\},
\]
which together with \eqref{eq:duality-proof} yields \eqref{eq:duality-statement}.

In Subsection~\ref{proof-sketch}  we reduced the proof of  Theorems~\ref{thm:conv-conc} to the following five propositions. Their proofs are given in consecutive subsections.

\begin{Proposition}
\label{prop:probab-omega-noise}
If $\vxi=(\xi_1, \dots, \xi_N)$ has i.i.d.\ coordinates with the same distribution as $\xi$, and $X$ and $\xi$ satisfy Assumption~\ref{assumption-1}, then there exists an absolute constant $C_{\mathrm{noise}}$  such that, $\bP(\Omega_{\mathrm{noise}}) \ge 1- 2e^{-cN/\rho^4}$ where $\Omega_{\mathrm{noise}}$ is defined in \eqref{eq:def-omega-noise} and $c$ is the universal constant.
\end{Proposition}

\begin{Proposition}
\label{prop:probab-omega-RIP}
If  $X$  satisfies Assumption~\ref{assumption-1}, then there exist positive universal constants $C,c$, such that if $N > C\rho^4 |J|$, then $\bP(\Omega_{\mathrm{RIP}}) \ge 1- 2e^{-c N/\rho^4}$, where $\Omega_{\mathrm{RIP}}$ is defined as in \eqref{eq:def-omega-RIP} with $C_{\mathrm{RIP}} = 2$.
\end{Proposition}

The following proposition is a result in the spirit of the Dvoretzky-Milman theorem. Its proof is postponed to Section~\ref{sec:proof_DM_q}.

\begin{Proposition}
\label{prop:probab-omega-DM}
If $X$  satisfies Assumption~\ref{assumption-1} and $q>1$, then there exist  constants $C_{\mathrm{DM}}$, $\kappa_{\mathrm{DM}}$, and $c$ depending only on $\rho$ and $q$, such that $\bP(\Omega_{\mathrm{DM}}) \ge 1- 2e^{-cd_\ast}$
provided that $N\le \kappa_{\mathrm{DM}} d_\ast$, where $\Omega_{\mathrm{DM}}$ is defined in \eqref{eq:def-omega-DM}. 
\end{Proposition}

\begin{Proposition}
\label{prop:probab-omega-multi}
If $\vxi=(\xi_1, \dots, \xi_N)$ has i.i.d.\ coordinates with the same distribution as $\xi$, $\bX_{J}$ is independent of $\bX_{J^c}$ and $\vxi$, and $X$ and $\xi$ satisfy Assumption~\ref{assumption-1}, then there exists an absolute constant $C_{\mathrm{multi}}$  such that $\bP(\Omega_{\mathrm{multi}}) \ge 1- 2e^{-c|J|/\rho^2}$ where $\Omega_{\mathrm{multi}}$ is defined in \eqref{eq:def-omega-multi}, and $c$ is a  universal constant.
\end{Proposition}

Recall that
\begin{equation}\label{eq:def-Omega_0}
\Omega_0 = \Omega_{\mathrm{noise}}\cap \Omega_{\mathrm{RIP}}\cap \Omega_{\mathrm{DM}} \cap \Omega_{\mathrm{multi}}.
\end{equation}
The immediate corollary of Propositions~\ref{prop:probab-omega-noise}-\ref{prop:probab-omega-multi} is that $\bP (\Omega_0 ) \ge 1- 2\exp\bigl(-c \min\{ |J|, d_* \}\bigr)$ provided that $c|J|\le N\le \kappa_{\mathrm{DM}} d_\ast$, with $c$ and $\kappa_{\mathrm{DM}}$ depending only on $q$ and $\rho$. These four propositions handle all the `stochastic side' of the proof of Theorem~\ref{thm:conv-conc}. Once we know that $\Omega_0$ holds with large probability, we place ourselves on that event and turn to the `algebraic side' of the proof and show the following which concludes the proof of Theorem~\ref{thm:conv-conc}, as discussed in Section~\ref{proof-sketch}.

\begin{Proposition}
\label{prop:hatbeta_in_set}
Inclusion \eqref{eq:hatbeta-lies-set} holds on the event $\Omega_0$.
\end{Proposition}


\subsection{Proof of Proposition~\ref{prop:probab-omega-noise}: \texorpdfstring{$\Omega_{\mathrm{noise}}$ holds with high probability}{noise bounds hold with high probability}}
\label{sec:proof-4.1}

\begin{proof}[Proof of Proposition~\ref{prop:probab-omega-noise}]
Note that \eqref{eq:bX-to-bY} implies that
 \begin{equation}\label{eq:prop-probab-omega-noise-1}
\bE \|\bX_{J^c} \bbeta_{J^c}^\ast \|_2^2 = \bE \|\bY_{J^c} \Sigma_{J^c}^{1/2} \bbeta_{J^c}^\ast \|_2^2
= \bE \sum_{i=1}^N |\langle Y_i , \Sigma_{J^c}^{1/2} \bbeta_{J^c}^\ast \rangle |^2
= N\, \bE |\langle \vzeta , \Sigma_{J^c}^{1/2} \bbeta_{J^c}^\ast \rangle |^2
= N \|\Sigma_{J^c}^{1/2} \bbeta_{J^c}^\ast\|_2^2,
\end{equation}
where the last equality follows by the isotropy of $\vzeta$.
 Similarly, $\|\bX_{J^c} \bbeta_{J^c}^\ast \|_2 = f(Y_1, \dots, Y_N)$, where  $f(y_1, \dots, y_N)=(\sum_{i} |\langle y_i,  \Sigma_{J^c}^{1/2} \bbeta_{J^c}^\ast \rangle |^2 )^{1/2}$.
 Note that the Lipschitz constant of $f\colon \bR^{pN}\to \bR$ is upper bounded by
 \[
 	\sup_{\|y\|_2\le 1} \Bigl(\sum_{i=1}^N \|y_i\|_2^2 \| \Sigma_{J^c}^{1/2} \bbeta_{J^c}^\ast\|_2^2 \Bigr)^{1/2} =  \|\Sigma_{J^c}^{1/2} \bbeta_{J^c}^\ast\|_2.
 \]

 Therefore, the convex concentration assumption \eqref{eq:assumption-1}, applied with $t=\frac 12 \bE  \|\bX_{J^c} \bbeta_{J^c}^\ast \|_2$ yields that
 \begin{equation}\label{eq:prop-probab-omega-noise-2}
 \bP \Bigl( \frac 12 \bE \|\bX_{J^c} \bbeta_{J^c}^\ast \|_2 \le  \|\bX_{J^c} \bbeta_{J^c}^\ast \|_2 \le \frac 32 \bE\|\bX_{J^c} \bbeta_{J^c}^\ast \|_2 \Bigr)
 \ge 1-  2\exp\Bigl( -\frac{(\bE  \|\bX_{J^c} \bbeta_{J^c}^\ast \|_2)^2}{4\rho^2 \|\Sigma_{J^c}^{1/2} \bbeta_{J^c}^\ast\|_2^2 }\Bigr).
 \end{equation}
 Moreover, by integrating by parts the following inequality (which also follows by \eqref{eq:assumption-1})
 \[
 \bP \bigl( \|\bX_{J^c} \bbeta_{J^c}^\ast \|_2 > t+ \bE \|\bX_{J^c} \bbeta_{J^c}^\ast \|_2 \bigr) \le 2e^{-t^2/(\rho^2 \|\Sigma_{J^c}^{1/2} \bbeta_{J^c}^\ast\|_2^2)},
 \]
 we get
 \[
 \sqrt{N}\|\Sigma_{J^c}^{1/2} \bbeta_{J^c}^\ast\|_2 \mathop{=}^{\eqref{eq:prop-probab-omega-noise-1}} \bigl(\bE \|\bX_{J^c} \bbeta_{J^c}^\ast \|_2^2 \bigr)^{1/2} \le  C_1\bE \|\bX_{J^c} \bbeta_{J^c}^\ast \|_2 +C_1 \rho \|\Sigma_{J^c}^{1/2} \bbeta_{J^c}^\ast\|_2 ,
 \]
where $C_1$ is the universal constant. Note that we may assume without loss of generality that $N\ge 4C_1^2\rho^2$, because otherwise we may take $c$ in the assertion small enough to make the probability bound trivial. Hence, the previous bound implies that
\[
\bE \|\bX_{J^c} \bbeta_{J^c}^\ast \|_2 \ge \frac{\sqrt{N}}{2C_1}  \|\Sigma_{J^c}^{1/2} \bbeta_{J^c}^\ast\|_2  .
\]
Moreover, $\bE \|\bX_{J^c} \bbeta_{J^c}^\ast \|_2 \le  \bigl(\bE \|\bX_{J^c} \bbeta_{J^c}^\ast \|_2^2 \bigr)^{1/2}  = \sqrt{N} \|\Sigma_{J^c}^{1/2} \bbeta_{J^c}^\ast\|_2$.
This together  with \eqref{eq:prop-probab-omega-noise-1} and \eqref{eq:prop-probab-omega-noise-2}, implies that
for some absolute constant $c_1\in (0,1)$,
  \begin{equation}\label{eq:proof-noise-convconc-11}
 \bP \Bigl( c_1\sqrt{N} \|\Sigma_{J^c}^{1/2} \bbeta_{J^c}^\ast\|_2 \le  \|\bX_{J^c} \bbeta_{J^c}^\ast \|_2 \le \frac 32 \sqrt{N} \|\Sigma_{J^c}^{1/2} \bbeta_{J^c}^\ast\|_2 \Bigr)
 \ge 1- 2\exp\Bigl( -\frac{c_1N}{\rho^2}\Bigr).
\end{equation}

Now we move to two-sided bounds for the noise $\xi$.
Recall that $\xi_1, \dots, \xi_N$ are i.i.d.\ ($\rho\sigma_{\xi}$)-sub-Gaussian random variables with variance $\sigma_\xi^2$. Thus, $X_i=\xi_i^2-\sigma_\xi^2$ are centered, independent ($C\rho^2\sigma_\xi^2$)-sub-exponential random variables.
Thus, Bernstein's inequality (see, e.g. \cite[Theorem~2.8.1]{vershyninHighDimensionalProbabilityIntroduction2018}) implies that for any $t>0$,
 \[
 	\PP\bigl( \bigl|\|\vxi\|_2^2 - \sigma_\xi^2 N \bigr| >t \bigr) \le 2\exp \Bigl(-c_2 \min\Bigl\{\frac{t^2}{N\sigma_\xi^4 \rho^4}, \, \frac{t}{\sigma_\xi^2\rho^2} \Bigr\}\Bigr),
 \]
 where $c_2\in (0,1)$ is the universal constant.
  Applying this with $t=N\sigma_\xi^2/2$ we obtain
    \begin{equation}\label{eq:proof-noise-convconc-12}
 \bP \Bigl(\frac 12 \sqrt{N} \sigma_\xi \ \le  \|\vxi\|_2 \le 2 \sqrt{N} \sigma_\xi \Bigr)
 \ge 1- 2\exp\Bigl( -\frac{c_3 N }{\max\{\rho^{4},\,\rho^2\}}\Bigr),
\end{equation}
where $c_3$ is another absolute constant.

By isotropy and \eqref{eq:assumption-1}, $\rho^2\geq\frac{1}{2}$, consequently, $\max\{\rho^4,\rho^2\}\leq2\rho^4$. Hence, we get the assertion of Proposition~\ref{prop:probab-omega-noise} by adjusting the constants and using \eqref{eq:proof-noise-convconc-11}, \eqref{eq:proof-noise-convconc-12}, and the union bound.
\end{proof}


\subsection{Proof of Proposition~\ref{prop:probab-omega-RIP}: \texorpdfstring{$\Omega_{\mathrm{RIP}}$ holds with high probability}{RIP holds with high probability}}

The proof relies on well-known estimates of singular values of matrices with independent subgaussian rows. We will rely on the presentation from \cite{vershyninHighDimensionalProbabilityIntroduction2018}.

\begin{proof}[Proof of Proposition~\ref{prop:probab-omega-RIP}]
Recall that $X = \Sigma^{1/2}\vzeta$, where $\vzeta$ is an isotropic random vector, which satisfies the convex concentration property \eqref{assumption-1}. In particular, for all $x \in V_J$ and $t \ge 0$, $\PP(|\langle \vzeta_J,x\rangle| \ge t) \le 2\exp(-t^2/(\rho^2\|x\|_2^2))$, which by \cite[Proposition 2.5.2]{vershyninHighDimensionalProbabilityIntroduction2018} implies that for some universal constant $C_1$,
\begin{displaymath}
\|\langle \vzeta_J,x\rangle\|_{\psi_2} \le C_1 \rho \|x\|_2,
\end{displaymath}
where for a random variable $Z$, $\|Z\|_{\psi_2}$ denotes the Orlicz norm corresponding to the function $\psi_2(t) = e^{t^2} - 1$, i.e., $\|Z\|_{\psi_2} = \inf\{ a > 0\colon \E \psi_2(Z/a) \le 1\}$. Upon indentifying $V_J$ with $\R^{|J|}$ in a standard way, this means that $\bY_J$ satisfies the assumptions of \cite[Theorem 4.6.1]{vershyninHighDimensionalProbabilityIntroduction2018}, which states that for all $t \ge 0$, with probability at least $1 - 2\exp(-t^2)$, for all $x \in V_J$,
\begin{displaymath}
\bigl( \sqrt{N} - C_2\rho^2 (\sqrt{|J|} + t) \bigr) \|x\|_2 \le \|\bY_J x\|_2 \le \bigl( \sqrt{N} + C_2\rho^2 (\sqrt{|J|} + t)\bigr) \|x\|_2,
\end{displaymath}
where $C_2$ is another universal constant.
Let now $C_3 = 16C_2^2\rho^4$. If $N \ge C_3|J|$, then by choosing $t = \sqrt{N/C_3}$ we obtain that with probability at least $1 - 2\exp(-N/C_3)$, for all $x \in V_J$,
\begin{displaymath}
\frac{1}{2}\sqrt{N} \|x\|_2 \le \|\bY_J x\|_2 \le \frac{3}{2}\sqrt{N}\|x\|_2.
\end{displaymath}
Substituting $x = \Sigma_{J}^{1/2}\beta$ and recalling that $\bY_J x = \bX_J \beta$ (see \eqref{eq:bX-to-bY}) gives the assertion of Proposition~\ref{prop:probab-omega-RIP}.
with $C = 16C_2^2$ and $c = 1/C$.
\end{proof}


\subsection{Proof of Proposition~\ref{prop:probab-omega-DM}: \texorpdfstring{$\Omega_{\mathrm{DM}}$ holds with high probability}{DM holds with high probability}} \label{sect:proof-DM-proposition}

Proposition~\ref{prop:probab-omega-DM} follows by Proposition~\ref{thm:DM} below, applied with $q\coloneqq q'$, $m\coloneqq |J^c|$, $A\coloneqq\Sigma_{J^c}^{1/2}$, $d_\ast(A)= d_\ast$ and $\ell_q(A)= \ell_\ast$. 

\subsection{Proof of Proposition~\ref{prop:probab-omega-multi}: \texorpdfstring{$\Omega_{\mathrm{multi}}$ holds with high probability}{Omega multi holds with high probability}}
\label{sec:proof-4.4}

\begin{proof}[Proof of Proposition~\ref{prop:probab-omega-multi}]

To deal with $\| \Sigma_J^{-1/2}\bX_J^\top \vlambda^*[\vy-\bX_J\bbeta_J^*]\|_2$ from the definition of  $\Omega_{\mathrm{multi}}$ (see \eqref{eq:def-omega-multi}) note first that  $\Sigma_J^{-1/2}\bX_J^\top = \bY_J^\top$ (this equality follows by \eqref{eq:bX-to-bY} and by the block-diagonality of $\Sigma$ recorded in Remark~\ref{rem:block-diagonal}, which gives $P_J\Sigma^{1/2}=\Sigma_J^{1/2}$) is independent of $ \vlambda^*[\vy-\bX_J\bbeta_J^*] = \vlambda^*[\vxi+\bX_{J^c}\bbeta_{J^c}^\ast] $. Moreover, for a fixed $\vlambda \in\bR^N$, the random vector $ \bY_J^\top \vlambda$ in $\bR^J$ satisfies
\[
	\bE\|\bY_J^\top \vlambda\|_2^2 = \bE  \sum_{j\in J} (\bY_J^\top \vlambda)_j^2 = \bE \sum_{j\in J} \bigl( \sum_{i=1}^N Y_i(j)\lambda_i  \bigr)^2 = |J|\cdot \|\vlambda\|_2^2.
\]
By using the convex concentration assumption and proceeding similarly as in the proof of Proposition~\ref{prop:probab-omega-noise},
we obtain that for any $\vlambda \in \bR^N$,
\[
	\bP \biggl( \frac 1{C} \sqrt{|J|} \|\vlambda\|_2 \le  \|\bY_J^\top \vlambda\|_2 \le\frac 32\sqrt{|J|} \|\vlambda\|_2 \biggr) \ge 1- 2 \exp\Bigl( -\frac{c|J|}{\rho^2}\Bigr).
\]
We apply this inequality conditionally for a random vector $\vlambda = \vlambda^*[\vy-\bX_J\bbeta_J^*]$ independent of $\bY_J$ to obtain the assertion.
\end{proof}


\subsection{Proof of Proposition~\ref{prop:hatbeta_in_set}\texorpdfstring{: $\hat\bbeta_J \in \bbeta_J^* + r\Sigma^{-1/2}_JB_2^{J}$}{} }

The proof presented in this subsection is purely deterministic. In particular, it works also in the independent and the log-concave settings up to some minor modifications due to the presence of extra $\log$ terms in the Dvoretzky-Milman results.

As it is classical in the convergence analysis of RERM, we prove two statistical guarantees for $\hat\bbeta_{J}$: one for the ellispoid norm associated with $\Sigma$, the covariance matrix of the design vector, and another one derived from the regularization norm, i.e. the $\ell_q$-norm in our case.  Therefore, we will prove a stronger result than the one of  Proposition~\ref{prop:hatbeta_in_set} since only the convergence rate for  the ellipsoid norm is stated in Proposition~\ref{prop:hatbeta_in_set}. More precisely, below we show that
\begin{equation}\label{eq:hatbeta_in_set-stronger}
	\hat\bbeta_J \in \bbeta_J^* +  K_{\mathrm{model}}(\krad)  \cap  r\Sigma^{-1/2}_JB_2^{J},
\end{equation}
where $\krad = r \sqrt{N} / \ell_*$ and
\begin{equation}\label{eq:def-Kmodel}
      K_{\mathrm{model}}(\krad) = K_{\mathrm{model}}(\krad, J) =
    \begin{cases}
        \{ \vv\in V_J:\, \sum_{j\in J}\alpha_q(|\beta_j^*|, v_j)\leq\krad^q \} & \text{when } 1<q<2, \\
 \krad B_q^J & \text{when } q\ge 2,
 \end{cases}
\end{equation}
where, for any $x\geq 0$ and $y\in\bR$, $\alpha_q(x,y)$ is introduced in \cite{bubeckHomotopyMethod$ell_p$2018} as
\begin{align}\label{eq:def_gamma_q}
    \alpha_q(x,y) = \begin{cases}
        \frac{q}{2} x^{q-2} y^2, & \mbox{ if } |y|\leq x\\
        |y|^q + \bigl(\frac{q}{2} - 1\bigr)x^q, &\mbox{ otherwise.}
    \end{cases}
\end{align}

Before we move to the proof of \eqref{eq:hatbeta_in_set-stronger}, we show two properties of function $\alpha_q$ defining the statistical model $K_{\mathrm{model}}(\krad)$. The first one shows that $\alpha_q$ plays the role of a second order lower estimate on the increments of the convex function $t\mapsto |t|^q$.

\begin{Lemma}\label{lemma:lq-difference}
    If $1<q<\infty$, then for every $s,t\in\bR$,
\begin{align}\label{eq:differential}
    |t|^q-|s|^q 
    \geq qs |s|^{q-2} (t-s) + \frac{q-1}{q2^q}\alpha_q (|s|,\, t-s).
    \end{align}
\end{Lemma}

\begin{proof}
	\cite[Lemma~4.{1}]{MR3909555} asserts that for any $x$ and $\Delta$,
    \begin{align*}
        |x+\Delta|^q \geq |x|^q + q|x|^{q-2}x\Delta + \frac{q-1}{q2^q}\alpha_q(|x|,\Delta).
    \end{align*}
    To obtain \eqref{eq:differential} it is enough to apply this inequality with $x=s$ and $\Delta=t-s$.
\end{proof}

If $q\ge 2$, then $ \alpha_q(x,y)\ge |y|^q$. For $q<2$ this is not true and thus we will use the inequality from the next lemma to deal with the second term on the right hand-side of \eqref{eq:differential}.

\begin{Lemma}\label{lemma:adil_et_al}
    If $1<q<2$, then for every $x\in \bR_+^d$ and $y\in \bR^d$ we have
\begin{equation}\label{eq:alpha_lower_q_below_2-xy}
     \sum_{i=1}^d \alpha_q(x_i,y_i)  \ge \frac 1{2^{2/q}} \min\bigl\{ \|y\|_q^q,\, \|y\|_q^2 \|x\|_q^{q-2} \bigr\}.
\end{equation}
\end{Lemma}

\begin{proof}
The definition of $\alpha_q$ yields that
\begin{align*}
    2 \sum_{i=1}^d \alpha_q(x_i ,y_i)
    \ge \sum_{i=1}^d \min\bigl\{ |y_i|^2|x_i|^{q-2}, |y_i|^q \bigr\}
    = \sum_{i=1}^d |y_i|^2 \bigl(\max \bigl\{ |x_i|, |y_i| \bigr\}\bigr)^{q-2}.
\end{align*}
 H{\"o}lder's inequality with exponents $2/q$ and $(2/q)'$ implies that
 \begin{align*}
     \sum_{i=1}^d |y_i|^q
    \le \Bigl( \sum_{i=1}^d |y_i|^2  \bigl(\max \bigl\{ |x_i|, |y_i| \bigr\}\bigr)^{q-2}\Bigr)^{q/2}
    \Bigl( \sum_{i=1}^d  \bigl(\max \bigl\{ |x_i|, |y_i| \bigr\}\bigr)^{-  (q(q-2)/2)\cdot (2/q)'}\Bigr)^{1-q/2} .
\end{align*}
Since $(q(q-2)/2)\cdot (2/q)' = -q$ and $\max\{|x_i|^q,|y_i|^q \} \le |x_i|^q +|y_i|^q$, we get
\begin{align*}
     \sum_{i=1}^d \alpha_q(|x_i|,y_i)
   & \ge \frac 12 \Bigl( \sum_{i=1}^d |y_i|^q  \Bigr)^{2/q} \Bigl( \sum_{i=1}^d |x_i|^q + \sum_{i=1}^d |y_i|^q \Bigr)^{1-2/q}
   \\ &\ge  \frac {1}{2^{2/q}} \min\biggl\{ \sum_{i=1}^d |y_i|^q  ,\, \Bigl( \sum_{i=1}^d |y_i|^q  \Bigr)^{2/q} \Bigl( \sum_{i=1}^d |x_i|^q \Bigr)^{1-2/q} \biggr\},
\end{align*}
which completes the proof.
\end{proof}

We will also use the following lemma providing some facts about operators $\cA$ and $\vlambda^\ast$ and the relation between them.

\begin{Lemma}\label{lemma:operator_A}
If $1\le q<\infty$, 
then maps $\cA\colon \bR^N\to\bR^p$ and $\vlambda^*\colon\bR^N\to\bR^N$ satisfy the following properties:
    \begin{enumerate}
        \item  \label{item:cA-zero}  for any $\alpha>0$ and $\vmu\in\bR^N$, $\cA[\alpha\vmu]= \alpha\cA[\vmu]$ and $\vlambda^*[\alpha\vmu]=\vlambda^*[\vmu]$; moreover, $\cA[\vmu]=\vzero$ if and only if $\vmu=\vzero$;
        
        \item  \label{item:cA-symmetry}
        for all $\vmu\in\bR^N, \cA[-\vmu]=-\cA[\vmu]$;
        
         \item \label{item:lq-cA-norm}
        $\|\cA[\cdot]\|_q$ is a norm;
        
        \item \label{item:calculation_subdifferential}
        if $\vmu \in \bR^N $ and $q>1$, then         
        \begin{align}\label{eq:sub_gradient_applied}
            \bigl|\cA[\vmu]\bigr|^{\odot(q-2)}\odot \cA[\vmu] = \|\cA[\vmu]\|_q^{q-1} \bX_{J^c}^\top\vlambda^*[\vmu].
        \end{align}

        \item\label{item:inequality_subdifferential_ineq}
        if $\vmu, \vh \in \bR^N $ and $q>1$, then
        \[
               \|\cA[\vmu] \|_q^q -  \|\cA[\vmu+\vh] \|_q^q 
               \le -q \|\cA[\vmu] \|_q^{(q-1)}  \langle \vh,\, \vlambda^*[\vmu]\rangle.
        \]
    \end{enumerate}
\end{Lemma}

\begin{proof}
\begin{enumerate}
    \item
    When $\vmu=\vzero$, then $\argmin(\|\vnu\|_q: \bX_{J^c}\vnu=\vzero)=\vzero$. On the other side, since $\vmu=\bX_{J^c}\cA[\vmu]$, when $\cA[\vmu]=\vzero$, we know that $\vmu=\vzero$. It is clear that $\cA[\alpha\vmu]= \alpha\cA[\vmu]$ and  $\vlambda^*[\alpha\vmu]=\vlambda^*[\vmu]$.

    \item This point is clear by the definition of $\cA[\cdot]$.

    \item For any $\vv_1,\vv_2\in\bR^N$, $\|\cA[\vv_1+\vv_2]\|_q\leq \|\cA[\vv_1]+\cA[\vv_2]\|_q$. This is because 
    $\vv_1+\vv_2 = \bX_{J^c}(\cA[\vv_1] + \cA[\vv_2])$, so $\cA[\vv_1]+\cA[\vv_2]$ belongs to  $\{\vnu: \bX_{J^c}\vnu = \vv_1+\vv_2\}$, the feasible set in the definition of $\cA[v_1+v_2]$. Since $\cA[\vv_1+\vv_2]$ has the smallest $\|\cdot\|_q$ norm on this feasible set, we get that $\norm{\cA[v_1+v_2]}_q\leq \norm{\cA[v_1] + \cA[v_2]}_q$. Hence, the sub-additivity of $\|\cA[\cdot]\|_q$ follows from the triangle inequality.
    Thus, the assertion of this item follows by items \ref{item:cA-zero} and \ref{item:cA-symmetry}.

    \item
    Recall that \eqref{eq:maximizer-inclusion}   says  that
    \[
        \frac{|\cA[\vmu]|^{\odot(q-1)}\odot \sign(\cA[\vmu])}{\|\cA[\vmu]\|_{q}^{q-1}} \in \operatorname{Im} \bX_{J^c}^\top.
    \]

    On the other hand, in the dual maximization problem
    \begin{align*}
        \max \bigl(\langle \vmu, \vlambda \rangle \colon \|\bX_{J^c}^\top \vlambda\|_{q'}\le 1\bigr)
        =   \max \Bigl(\langle \bX_{J^c}(\cA[\vmu]), \vlambda \rangle \colon \|\bX_{J^c}^\top \vlambda\|_{q'}\le 1\Bigr)
        =\max \Bigl(\langle \cA[\vmu], \bX_{J^c}^\top\vlambda \rangle \colon \|\bX_{J^c}^\top \vlambda\|_{q'}\le 1\Bigr),
    \end{align*}
    the latter maximum is achieved only by $\vlambda$ satisfying
     \[
       \bX_{J^c}^\top\vlambda= \frac{|\cA[\vmu]|^{\odot(q-1)}\odot \sign(\cA[\vmu])}{\|\cA[\vmu]\|_{q}^{q-1}},
    \]
    since the right hand-side is  the unique solution to $\max \Bigl(\langle \cA[\vmu], \theta \rangle \colon \|\theta\|_{q'}\le 1\Bigr)$ and it is known to lie in the image of $\bX_{J^c}^\top$. This yields the assertion.

    \item By the convexity of the function $x\mapsto |x|^q$ we have for any $\vx, \vs \in \bR^p$,
    \[
        \|x\|_q^q -  \|x+s\|_q^q 
        = -\sum_{i=1}^p \bigl(|x_i+s_i|^q - |x_i|^q \bigr) 
        \le- q\sum_{i=1}^p  |x_i|^{q-1}\sign(x_i) s_i
        = -q \langle |x|^{\odot(q-2)}\odot x, s \rangle, 
    \]
    so, by using item \ref{item:calculation_subdifferential} and the relation $\bX_{J^c}(\cA[\vmu+\vh] - \cA[\vmu] ) = (\vmu+\vh)-\vmu = \vh$, we get 
    \begin{align*}
           \|\cA[\vmu] \|_q^q  - \|\cA[\vmu+\vh] \|_q^q 
         & \le- q \langle \bigl|\cA[\vmu]\bigr|^{\odot(q-2)}\odot \cA[\vmu], \cA[\vmu+\vh] - \cA[\vmu] \rangle
         \\ & =- q \|\cA[\vmu] \|_q^{q-1} \langle \bX_{J^c}^\top\vlambda^*[\vmu] ,  \cA[\vmu+\vh] - \cA[\vmu] \rangle
         \\ & = -q \|\cA[\vmu] \|_q^{q-1} \langle \vlambda^*[\vmu] ,  \bX_{J^c}(\cA[\vmu+\vh] - \cA[\vmu]) \rangle 
         \\ &= -q \|\cA[\vmu] \|_q^{q-1} \langle \vlambda^*[\vmu], \vh\rangle .
    \end{align*}
    \end{enumerate}
        This finishes the proof.
\end{proof}

Let us introduce the following norm which will be used in the proof of Proposition~\ref{prop:hatbeta_in_set}: for all $\vs\in V_J$,
\[
\|\vs\|_{\krad,r} = \sup\bigl\{ \langle \vs,\vu\rangle \colon \vu \in  K_{\mathrm{model}}(\krad)\cap r\Sigma_J^{- \frac 12}B_2^J \bigr\}.
\]
Recall that $\krad = r \sqrt N / \ell_*$. Moreover, if $q\ge 2$ then $ K_{\mathrm{model}}(\krad) = \krad B_q^J$, and so
\begin{equation}\label{eq:norm-rho_r-vs-viii}
	\|\vs\|_{\krad,r} = r \vertiii{\vs} \quad \text{if } q\ge 2,
\end{equation}
where $\vertiii{\cdot}$ was defined in \eqref{eq:def_vertiii}.
If $q<2$ then we rely on the following lemma to compare the norms $\|\cdot\|_{\krad,r}$ and $\vertiii{\cdot}$.

\begin{Lemma}\label{lemma:norm_rho-r-upper-bound}
	For every $1< q<2$,
\[
     K_{\mathrm{model}}(\krad)\subset 4 \max\{\krad,\|\vbeta_J^*\|_q\} B_q^J
\]
and, as a consequence, for any $\vs \in V_J$, if $\|\vbeta_J^*\|_q\le \krad$, then $\|\vs\|_{\krad,r} \le 4r \vertiii{\vs}$.
\end{Lemma}

\begin{proof}
Fix $\bv\in K_{\mathrm{model}}(\krad)$. Then, by the definition of $K_{\mathrm{model}}(\krad)$ and inequality \eqref{eq:alpha_lower_q_below_2-xy} we know that $\bv\in V_J$ and
\[
\krad^q\ge  \sum_{j\in J} \alpha_q(|\beta_j^*|,v_j)  \ge \frac 1{2^{2/q}} \min\bigl\{ \|\bv\|_q^q,\, \|\bv\|_q^2 \|\vbeta_J^*\|_q^{q-2} \bigr\}.
\]
Now we consider two cases. If $\|\bv\|_q\ge \|\vbeta_J^*\|_q$, then $\krad^q2^{2/q}\ge\|\bv\|_q^q$, so $\bv\in 4\krad B_q^J\subset4 \max\{\krad,\|\vbeta_J^*\|_q\}B_q^J$.
If on the other hand $\|\bv\|_q< \|\vbeta_J^*\|_q$, then $\bv \in\|\vbeta_J^*\|_q B_q^J\subset  \max\{\krad,\|\vbeta_J^*\|_q\} B_q^J$.

The \emph{as a consequence} part of the assertion follows by the definitions of $\|\cdot\|_{\krad,r}$, $\vertiii{\cdot}$ and $\krad = r \sqrt N / \ell_*$.
\end{proof}

\begin{proof}[Proof of Proposition~\ref{prop:hatbeta_in_set}]
	The case $r=0$ is trivial, since then $\bbeta^*=0$ and $\sigma_\xi=0$  (see the definition of $r=r(V_J,V_{J^c})$, i.e., \eqref{eq:def_r_star_final}), so $\vy=0$ and, as a consequence, $\hat\bbeta=0=\bbeta^*$ and \eqref{eq:hatbeta-lies-set} indeed holds.
	Let us now assume that $r>0$. During the entire proof we assume that we work on the event $\Omega_0$. Our aim is to show that \eqref{eq:hatbeta_in_set-stronger} holds.

	For $\bbeta_J \in V_J$ denote $\cL_{\bbeta_J} = \ell_{\bbeta_J} - \ell_{\bbeta_J^*}$, where $\ell_{\bbeta_J} = \norm{\bbeta_{J}}_q^q + \norm{ \cA[\vy - \bX_{J}\bbeta_{J}] }_q^q $.
	Equation \eqref{eq:beta_J-minimizer} implies that
\begin{align*}
    \hat\bbeta_{J} = \argmin_{\bbeta_J\in V_J}  \ell_{\bbeta_J},
\end{align*}
so $\cL_{\hat\bbeta_J}\le 0$.
Therefore, in order to prove \eqref{eq:hatbeta_in_set-stronger}, it suffices to show that $\cL_{\bbeta_J}$ is positive for every $\bbeta_J \notin \bbeta_J^* +  K_{\mathrm{model}}(\krad)  \cap  r\Sigma^{-1/2}_JB_2^{J}$.
 By item \ref{item:lq-cA-norm} of Lemma~\ref{lemma:operator_A} we know that $\bbeta_J\mapsto \cL_{\bbeta_J}$ is convex.
 Moreover, $\bbeta_J^* \in \bbeta_J^* +  K_{\mathrm{model}}(\krad)  \cap  r\Sigma^{-1/2}_JB_2^{J}$ and $\cL_{\bbeta_J^\ast}=0$.  Hence, it suffices to prove that  $\cL_{\bbeta_J}$ is positive for every $\bbeta_J \in \bbeta_J^* + \partial(  K_{\mathrm{model}}(\krad)  \cap  r\Sigma^{-1/2}_JB_2^{J})$.

To facilitate our analysis, we adopt notation that mirrors the classical framework of RERM (see for instance \cite{koltchinskiiOracleInequalitiesEmpirical2011}):
 \begin{align*}
    P_N\cL_{\bbeta_{J}}=\norm{ \cA[\vy - \bX_{J}\bbeta_{J}] }_q^q - \norm{ \cA[\vy - \bX_{J}\bbeta_{J}^*] }_q^q  
\qquad  \text{ and } \qquad  \cR_{\bbeta_{J}}=\norm{\bbeta_{J}}_q^q - \norm{\bbeta_{J}^*}_q^q.
\end{align*}In particular, the `$P_N$' notation does not stand for the empirical measure since our loss function $\bbeta_J\to \norm{ \cA[\vy - \bX_{J}\bbeta_{J}] }_q^q$ is not an empirical mean. However, in the classical analysis of RERM, it is the case and, as we will see below, thanks to the isomorphic property of $\cA$ on $\Omega_{\rm DM}$ -- see \eqref{eq:def-omega-DM-cA} -- this will be `almost' the case for our $P_N\cL_{\bbeta_{J}}$ as well.

Then $\cL_{\bbeta_J} =  P_N\cL_{\bbeta_{J}} + \cR_{\bbeta_{J}}$ and to prove that $\cL_{\bbeta_J} $ is positive we will lower bound each term separately.
Note also that the boundary of the set  $\bbeta_J^* +  K_{\mathrm{model}}(\krad)  \cap  r\Sigma^{-1/2}_JB_2^{J}$ consists of two parts: $\bbeta_J^* +  \partial K_{\mathrm{model}}(\krad)  \cap  r\Sigma^{-1/2}_JB_2^{J}$ and $\bbeta_J^* +  K_{\mathrm{model}}(\krad)  \cap  r\Sigma^{-1/2}_JS_2^{J}$.
We will lower bound $P_N\cL_{\bbeta_{J}}$ and $ \cR_{\bbeta_{J}}$ differently on each part of the boundary.

Let us first fix $\bbeta_J \in \bbeta_J^* + \partial(  K_{\mathrm{model}}(\krad)  \cap  r\Sigma^{-1/2}_JB_2^{J})$ and give some primary lower bounds on $P_N\cL_{\bbeta_{J}}$ and $ \cR_{\bbeta_{J}}$.
Applying \eqref{eq:differential} with $s=\beta_j^*$ and $t = \beta_j$ for each $j\in J$ and taking the sum over all $j\in J$, we obtain that
\begin{align*}
    \cR_{\bbeta_J} = &\norm{\bbeta_J}_q^q - \norm{\bbeta_J^*}_q^q = \sum_{j\in J}|\beta_j|^q - | \beta_j^* |^q\geq q\sum_{j\in J}\beta_j^*|\beta_j^*|^{q-2}(\beta_j-\beta_j^*) + \frac{q-1}{q2^q}\sum_{j\in J}\alpha_q(|\beta_j^*|,\, \beta_j-\beta_j^*).
\end{align*}
Since $\bbeta_J - \bbeta_J^*\in K_{\mathrm{model}}(\krad) $,
\begin{align}\label{eq:lower_cR_1}
            \cR_{\bbeta_J}&\geq -q \bigl\| \bbeta_J^*\odot|\bbeta_J^*|^{\odot(q-2)} \bigr\|_{\krad,r} + \frac{q-1}{q2^{q}}\sum_{j\in J}\alpha_q(|\beta_j^*|, \, \beta_j-\beta_j^*).
        \end{align}
Similarly, if we denote the $j$-th coordinate of $\cA[\vmu]$ by $\cA_j[\vmu]$,
\begin{align*}
     P_N\cL_{\bbeta_{J}} &= \sum_{j\in J^c} \bigl|\cA_j[\vy-\bX_J\bbeta_J]\bigr|^q - \bigl|\cA_j[\vy-\bX_J\bbeta_J^*]\bigr|^q
     \\
    &\geq q\Bigl< \cA[\vy-\bX_J\bbeta_J^*]\odot\bigl|\cA[\vy-\bX_J\bbeta_J^*]\bigr|^{\odot(q-2)},\, \cA[\vy-\bX_J\bbeta_J] - \cA[\vy-\bX_J\bbeta_J^*] \Bigr>
    \\ 
  &  \qquad + \frac{q-1}{q2^{q}} \sum_{j\in J^c}\alpha_q\Bigl(  \bigl| \cA_j[\vy-\bX_J\bbeta_J^*]\bigr| , \,\cA_j[\vy-\bX_J\bbeta_J] - \cA_j[\vy-\bX_J\bbeta_J^*]\Bigr).
\end{align*}
This, together with item \ref{item:calculation_subdifferential} of Lemma~\ref{lemma:operator_A} and the fact that $\bX_{J^c} \cA[\vmu]=\vmu$, yields
\begin{align}
     P_N\cL_{\bbeta_{J}}
    &\ge - q \|\cA[\vy-\bX_J\bbeta_J^\ast]\|_q^{q-1}  \Bigl| \bigl< \bX_{J^c}^\top\vlambda^*[\vy-\bX_J\bbeta_J],\, \cA[\vy-\bX_J\bbeta_J] - \cA[\vy-\bX_J\bbeta_J^*] \bigr>\Bigr|
    \notag \\
  &  \qquad +\frac{q-1}{q2^{q}} \sum_{j\in J^c}\alpha_q\Bigl(  \bigl| \cA_j[\vy-\bX_J\bbeta_J^*]\bigr| , \,\cA_j[\vy-\bX_J\bbeta_J] - \cA_j[\vy-\bX_J\bbeta_J^*]\Bigr)
    \notag \\
    & =   - q \|\cA[\vy-\bX_J\bbeta_J^*]\|_q^{q-1}  \Bigl| \bigl<\vlambda^*[\vy-\bX_J\bbeta_J^*],\, \bX_J\bigl(\bbeta_J-\bbeta_J^*\bigr) \bigr>\Bigr|
  \notag \\ &
  \qquad +\frac{q-1}{q2^{q}} \sum_{j\in J^c}\alpha_q\Bigl(  \bigl| \cA_j[\vy-\bX_J\bbeta_J^*]\bigr| , \,\cA_j[\vy-\bX_J\bbeta_J] - \cA_j[\vy-\bX_J\bbeta_J^*]\Bigr)
 \label{eq:lower_cL_2}
\end{align}
Recall that we work on $\Omega_0\subset\Omega_{\mathrm{multi}}$ (see the definition \eqref{eq:def-omega-multi}), and $\Sigma_J^{1/2}\bigl(\bbeta_J-\bbeta_J^*\bigr)  \in r B_2^J$, so
\begin{align*}
 \Bigl| \bigl<\vlambda^*[\vy-\bX_J\bbeta_J^*],\, \bX_J\bigl(\bbeta_J-\bbeta_J^*\bigr) \bigr>\Bigr|
 & =  \Bigl| \bigl< \Sigma_J^{-1/2}\bX_J^\top \vlambda^*[\vy-\bX_J\bbeta_J^*],\, \Sigma_J^{1/2}\bigl(\bbeta_J-\bbeta_J^*\bigr) \bigr>\Bigr|
\\ & \le \bigl\|  \Sigma_J^{-1/2}\bX_J^\top \vlambda^*[\vy-\bX_J\bbeta_J^*]\bigr\|_2 \cdot \bigl\|  \Sigma_J^{1/2}\bigl(\bbeta_J-\bbeta_J^*\bigr) \bigr\|_2
\\ & \le r C_{\mathrm{multi}}\sqrt{|J|} \, \|\vlambda^*[\vy-\bX_J\bbeta_J^*]\|_2.
\end{align*}
The definition of $\vlambda^*$ (see \eqref{eq:dual_problem_cA}) and the fact that we work on $\Omega_0\subset\Omega_{\mathrm{DM}}$ yield that
\begin{equation}\label{eq:use-of-the-upper-bound-in-DM}
	\|\vlambda^*[\vy-\bX_J\bbeta_J^*]\|_2 \le \frac{C_{\mathrm{DM}}}{\ell_*},
\end{equation}
while the fact that we work on $\Omega_0\subset\Omega_{\mathrm{DM}}\cap\Omega_{\mathrm{noise}}$, together with \eqref{eq:def-omega-DM-cA} and equation $\vy=\bX_J\bbeta_J^*+ \bX_{J^c}\bbeta_{J^c}^* +\vxi$, implies that
\begin{equation}\label{eq:norm-cA-DM-RIP'}
	\|\cA[\vy-\bX_J\bbeta_J^\ast]\|_q
	 \le C_{\mathrm{DM}} \frac{\| \vy-\bX_J\bbeta_J^\ast\|_2}{\ell_*}
	 = C_{\mathrm{DM}} \frac{\| \bX_{J^c}\bbeta_{J^c}^\ast +\vxi\|_2}{\ell_*}
	 \le C_{\mathrm{DM}} C_{\mathrm{noise}} \frac{\sqrt{N}(\sigma_\xi +\|\Sigma_{J^c}^{1/2} \bbeta_{J^c}^\ast \|_2) }{\ell_*} .
\end{equation}
Therefore, \eqref{eq:lower_cL_2}  implies that
\begin{align}
     P_N\cL_{\bbeta_{J}}
    &\ge - qC_{\mathrm{DM}}^2C_{\mathrm{multi}}C_{\mathrm{noise}} \frac{r\sqrt{|J|}}{\ell_*^q}  \bigl(\sqrt N (\sigma_\xi +\|\Sigma_{J^c}^{1/2} \bbeta_{J^c}^\ast \|_2)  \bigr)^{q-1}
  \notag \\ &
  \qquad +\frac{q-1}{q2^{q}} \sum_{j\in J^c}\alpha_q\Bigl(  \bigl| \cA_j[\vy-\bX_J\bbeta_J^*]\bigr| , \,\cA_j[\vy-\bX_J\bbeta_J] - \cA_j[\vy-\bX_J\bbeta_J^*]\Bigr).
 \label{eq:lower_cL_3}
\end{align}

Now let us consider two cases.

\paragraph{Case $q\ge 2$.} Then $\alpha_q(x,y) \ge |y|^q$ for any $x\ge 0$ and $y\in \bR$. Therefore,
\begin{equation}
 \sum_{j\in J^c}\alpha_q\Bigl(  \bigl| \cA_j[\vy-\bX_J\bbeta_J^*]\bigr| , \,\cA_j[\vy-\bX_J\bbeta_J] - \cA_j[\vy-\bX_J\bbeta_J^*]\Bigr) \ge \bigl\|  \cA[\vy-\bX_J\bbeta_J] - \cA[\vy-\bX_J\bbeta_J^*] \bigr\|_q^q.
 \label{eq:lower_cL_4}
 \end{equation}
Moreover, by the definition of $\cA[\bX_J(\bbeta_J^*-\bbeta_J)]$ and the fact that $\bX_{J^c} (\cA[\vy-\bX_J\bbeta_J] - \cA[\vy-\bX_J\bbeta_J^*]) = (\vy-\bX_J\bbeta_J) - (\vy-\bX_J\bbeta_J^*) = \bX_J(\bbeta_J^*-\bbeta_J)$, we obtain that $\norm{\cA[\vy-\bX_J\bbeta_J] - \cA[\vy-\bX_J\bbeta_J^*]}_q^q\geq \norm{ \cA[\bX_J(\bbeta_J^\ast - \bbeta_J)] }_q^q$. This and \eqref{eq:lower_cL_4} imply that
\begin{equation*}
 \sum_{j\in J^c}\alpha_q\Bigl(  \bigl| \cA_j[\vy-\bX_J\bbeta_J^*]\bigr| , \,\cA_j[\vy-\bX_J\bbeta_J] - \cA_j[\vy-\bX_J\bbeta_J^*]\Bigr)
 \ge \bigl\|  \cA[\bX_J \bigl(\bbeta_J-\bbeta_J^*\bigr)] \bigr\|_q^q.
\end{equation*}
Recall that we work on $\Omega_0\subset\Omega_{\mathrm{DM}}\cap\Omega_{\mathrm{RIP}}$, so the previous inequality and \eqref{eq:def-omega-DM-cA} imply
\begin{equation}
 \sum_{j\in J^c}\alpha_q\Bigl(  \bigl| \cA_j[\vy-\bX_J\bbeta_J^*]\bigr| , \,\cA_j[\vy-\bX_J\bbeta_J] - \cA_j[\vy-\bX_J\bbeta_J^*]\Bigr)
 \ge\biggl( \frac{\sqrt N}{C_{\mathrm{DM}}C_{\mathrm{RIP}} \ell_*} \biggr)^{q} \bigl\| \Sigma_J^{1/2}\bigl( \bbeta_J-\bbeta_J^* \bigr)\bigr\|_2^q.
 \label{eq:lower_cL_5}
\end{equation}

\textbf{In the case when $\bbeta_J - \bbeta_J^*\in K_{\mathrm{model}}(\krad)\cap r\Sigma^{-1/2}S_2^J $} the right-hand side of \eqref{eq:lower_cL_5} is equal to $\bigl(\frac{r\sqrt N}{C_{\mathrm{DM}}C_{\mathrm{RIP}} \ell_*} \bigr)^{q}$. Hence, we use 
\eqref{eq:lower_cR_1}  (together with the non-negativeness of its second term), \eqref{eq:lower_cL_3}, together with \eqref{eq:lower_cL_5}, and \eqref{eq:norm-rho_r-vs-viii} to obtain that in order to prove that $\cL_{\bbeta_J}> 0$ it suffices to show that
\begin{equation}\label{eq:condition-r-1}
	 q r\vertiii{ \bbeta_J^*\odot|\bbeta_J^*|^{\odot(q-2)} } + qC_{\mathrm{DM}}^2C_{\mathrm{multi}}C_{\mathrm{noise}} \frac{r\sqrt{|J|}}{\ell_*^q}  \bigl(\sqrt N (\sigma_\xi +\|\Sigma_{J^c}^{1/2} \bbeta_{J^c}^\ast \|_2)  \bigr)^{q-1}
 <\frac{q-1}{q2^{q}} \biggl( \frac{r\sqrt N}{C_{\mathrm{DM}}C_{\mathrm{RIP}} \ell_*} \biggr)^{q} .
\end{equation}
Since $r>0$, we may adjust the constant $C$ in the definition of $r$ (see \eqref{eq:def_r_star_final}) so that this inequality holds.

\textbf{In the case when $\bbeta_J - \bbeta_J^*\in \partial K_{\mathrm{model}}(\krad)\cap r\Sigma^{-1/2}B_2^J$} the second term of \eqref{eq:lower_cR_1} is equal to $\frac{q-1}{q2^q}\eta^q = \frac{q-1}{q2^q}(\frac{r\sqrt{N}}{\ell_*})^q $.
Now we use 
\eqref{eq:lower_cR_1}, together with the definition of $K_{\mathrm{model}}(\krad)$ (see \eqref{eq:def-Kmodel}) and equation $\krad = r \sqrt{N}/\ell_*$, \eqref{eq:lower_cL_3}  (together with the non-negativeness of its second term), and \eqref{eq:norm-rho_r-vs-viii}, to obtain that in order to prove that $\cL_{\bbeta_J}> 0$ it suffices to show that
\[
	 q r\vertiii{ \bbeta_J^*\odot|\bbeta_J^*|^{\odot(q-2)} } + qC_{\mathrm{DM}}^2C_{\mathrm{multi}}C_{\mathrm{noise}} \frac{r\sqrt{|J|}}{\ell_*^q}  \bigl(\sqrt N (\sigma_\xi +\|\Sigma_{J^c}^{1/2} \bbeta_{J^c}^\ast \|_2)  \bigr)^{q-1}
 < \frac{q-1}{q2^{q}} \biggl(\frac{r\sqrt N}{\ell_*} \biggr)^q,
\]
which follows by \eqref{eq:condition-r-1} (since $C_{\mathrm{DM}},C_{\mathrm{RIP}}\ge 1$).

\paragraph{Case $1<q<2$.} We proceed similarly as in the case $q\ge 2$, but we use \eqref{eq:alpha_lower_q_below_2-xy} instead of $\alpha_q(x,y) \ge |y|^q$ (which does not hold for $q<2$). Using Lemma~\ref{lemma:adil_et_al} we arrive at the following counterpart of \eqref{eq:lower_cL_4}:
\begin{multline*}
 \sum_{j\in J^c}\alpha_q\Bigl(  \bigl| \cA_j[\vy-\bX_J\bbeta_J^*]\bigr| , \,\cA_j[\vy-\bX_J\bbeta_J] - \cA_j[\vy-\bX_J\bbeta_J^*]\Bigr)
\\ \ge\frac 1{4}\min \biggl\{ \bigl\|  \cA[\vy-\bX_J\bbeta_J] - \cA[\vy-\bX_J\bbeta_J^*] \bigr\|_q^q , \,
\| \cA[\vy-\bX_J\bbeta_J] - \cA[\vy-\bX_J\bbeta_J^*] \bigr\|_q^2 \|\cA[\vy-\bX_J\bbeta_J^*] \|_q^{q-2}  \biggr\}.
 \end{multline*}
As in the case $q\ge 2$, we use the definition of $\cA$ and \eqref{eq:norm-cA-DM-RIP'} to get the following counterpart of \eqref{eq:lower_cL_5}:
\begin{align} \label{eq:lower_cL_10} 
 &4\sum_{j\in J^c}\alpha_q\Bigl(  \bigl| \cA_j[\vy-\bX_J\bbeta_J^*]\bigr| , \,\cA_j[\vy-\bX_J\bbeta_J] - \cA_j[\vy-\bX_J\bbeta_J^*]\Bigr)
\\ \notag
&\ge\min \biggl\{ \biggl( \frac{\sqrt N}{C_{\mathrm{DM}}C_{\mathrm{RIP}} \ell_*} \biggr)^{q} \bigl\| \Sigma_J^{1/2}\bigl( \bbeta_J-\bbeta_J^* \bigr)\bigr\|_2^q , \,
\biggl( \frac{\sqrt N}{C_{\mathrm{DM}}^{4/q-1}C_{\mathrm{RIP}}^{2/q} C_{\mathrm{noise}}^{2/q-1}\ell_*} \biggr)^{q} \bigl\| \Sigma_J^{1/2}\bigl( \bbeta_J-\bbeta_J^* \bigr)\bigr\|_2^2 (\sigma_\xi +\|\Sigma_{J^c}^{1/2} \bbeta_{J^c}^\ast \|_2 )^{q-2}  \biggr\}.
 \end{align}
Observe that $q-2 < 0$ and so above we used both the upper and lower estimates from the definitions of $\Omega_{\mathrm{DM}}$ and $\Omega_{\mathrm{RIP}}$ and the upper estimate from the definition of $\Omega_{\mathrm{noise}}$.

Moreover, instead of using \eqref{eq:norm-rho_r-vs-viii} (which does not hold when $1<q<2$), we use the second part of Lemma~\ref{lemma:norm_rho-r-upper-bound}, together with the assumption that $\krad\ge \|\bbeta_J^*\|_q$ (which follows by the definition of $r=\krad \ell_\ast/\sqrt{N}$ from \eqref{eq:def_r_star_final} in the case $1<q<2$).
 Starting again with the estimates \eqref{eq:lower_cR_1} and \eqref{eq:lower_cL_3} and splitting the boundary into two parts, we end up with the following conditions on $r$, which are satisfied if we adjust the constant $C$ in the definition \eqref{eq:def_r_star_final} again:
\[ 
	 4q r\vertiii{ \bbeta_J^*\odot|\bbeta_J^*|^{\odot(q-2)} } + qC_{\mathrm{DM}}^2C_{\mathrm{multi}}C_{\mathrm{noise}} \frac{r\sqrt{|J|}}{\ell_*^q}  \bigl(\sqrt N (\sigma_\xi +\|\Sigma_{J^c}^{1/2} \bbeta_{J^c}^\ast \|_2)  \bigr)^{q-1}
 < \frac{q-1}{q2^{q+2}} \biggl( \frac{r\sqrt N}{C_{\mathrm{DM}}C_{\mathrm{RIP}} \ell_*} \biggr)^{q} ,
\] 
\begin{align*}
	 4q r\vertiii{ \bbeta_J^*\odot|\bbeta_J^*|^{\odot(q-2)} } + qC_{\mathrm{DM}}^2C_{\mathrm{multi}}C_{\mathrm{noise}} \frac{r\sqrt{|J|}}{\ell_*^q} & \bigl(\sqrt N (\sigma_\xi +\|\Sigma_{J^c}^{1/2} \bbeta_{J^c}^\ast \|_2)  \bigr)^{q-1}
 \\ & < \frac{q-1}{q2^{q+2}} \biggl( \frac{\sqrt N} {C_{\mathrm{DM}}^{4/q-1}C_{\mathrm{RIP}}^{2/q} C_{\mathrm{noise}}^{2/q-1} \ell_*} \biggr)^{q}  r^2 (\sigma_\xi +\|\Sigma_{J^c}^{1/2} \bbeta_{J^c}^\ast \|_2 )^{q-2},
\end{align*}
and
\[
	4 q r\vertiii{ \bbeta_J^*\odot|\bbeta_J^*|^{\odot(q-2)} } + qC_{\mathrm{DM}}^2C_{\mathrm{multi}}C_{\mathrm{noise}} \frac{r\sqrt{|J|}}{\ell_*^q}  \bigl(\sqrt N (\sigma_\xi +\|\Sigma_{J^c}^{1/2} \bbeta_{J^c}^\ast \|_2)  \bigr)^{q-1}
 < \frac{q-1}{q2^{q}} \biggl(\frac{r\sqrt N}{\ell_*} \biggr)^q.
\]

Note that this time we have three inequalities (and not two as for $q\ge 2$) due to the minimum appearing on the right-hand side of \eqref{eq:lower_cL_10}.
\end{proof}


\section{Main results in the weak moment and log-concave cases}\label{sec:main_results_other_cases}

We may perform the proof strategy described in Subsection~\ref{proof-sketch} as long as we are able to show that the event $\Omega_0$ from \eqref{eq:def-Omega_0} holds with large probability. This is what we do in this section in two scenarios: `weak moment' and `log-concave'. However,  the proof showing that on $\Omega_0$,  $\hat\bbeta_J\in \bbeta_J^*+r\Sigma_J^{-1/2} B_2^J$ is purely deterministic (see the proof of Proposition~\ref{prop:hatbeta_in_set}) and other `stochastic settings' could be considered as long as $\Omega_0$ is proved to hold with large probability.

The main difficulty in the proof of our main results in the weak moment and log-concave settings lies in showing that $\Omega_{\mathrm{DM}}$ has  high probability -- this follows by Dvoretzky--Milman type results stated and proven in Section~\ref{sec:proof_DM_q}. Our proofs work under the additional assumption that $\Sigma$ is diagonal, which we were able to get rid of in the convex concentration setting. Let us be precise about the role of this hypothesis, since it is the only place where the geometry of $\Sigma$ interacts with the canonical basis beyond Remark~\ref{rem:block-diagonal}. The FSD decomposition of the estimator into \eqref{eq:beta_J-minimizer} and \eqref{eq:beta_J^c-minimizer}, the triangle inequality \eqref{eq:aim-triangle} and the whole deterministic part of the argument -- that is, Proposition~\ref{prop:hatbeta_in_set} -- only require $V_J$ and $V_{J^c}$ to be coordinate subspaces and $\Sigma$ to be block-diagonal with respect to $\bR^p=V_J\oplus^\perp V_{J^c}$, which is automatic here by Remark~\ref{rem:block-diagonal}. Diagonality of $\Sigma$ is needed \emph{only} to establish that $\Omega_{\mathrm{DM}}$ holds with large probability in the weak moment and log-concave settings, and it is therefore included in Assumptions~\ref{assumption-2} and \ref{assumption-3} rather than assumed throughout the paper. In fact, only the diagonality of $\Sigma_{J^c}$ is used, since $\Omega_{\mathrm{DM}}$ involves $\bX_{J^c}$ alone.

Let us first give the formulation of the version of our main result when the design variable $X$ has independent coordinates satisfying some mild moment assumptions and when $X$ is log-concave. Then, in Section~\ref{sec:proofs_other_cases}, we give the counterparts of Propositions~\ref{prop:probab-omega-noise}-\ref{prop:probab-omega-multi} -- i.e., the main ingredients of the proof in the convex concentration case -- in these two cases and provide their proofs.
These results are also used in the proof of Theorem~\ref{thm:weak-moment-optimal}, but in this proof we need als to perform a more delicate analysis of the noise absorption part from Section~\ref{sec:noise_absorption_part}. Recall that this allows us to replace the last condition in \eqref{eq:BO_conditions_conv_conc} by $\frac{N\sigma_\xi^2}{R_*} \to 0$, which is equivalent to the second last condition in \eqref{eq:sufficient_conditions_KZSS20}, see \eqref{eq:def_effective_rank_R_star}.

\subsection{The weak moment case}

In this section we assume that the design variable $X$ and the noise $\xi$ satisfy the following assumption.

\begin{Assumption}\label{assumption-2}
Assume that there exist an isotropic random vector $\vzeta=(\zeta_1,\dots, \zeta_p)\in \bR^p$ with independent coordinates,  constants $M\ge 1$ and $\eps\in(0,1)$, and a $p\times p$ positive definite diagonal matrix $\Sigma$ such that $X=\Sigma^{1/2}\vzeta$ and,  for every $j\le p$, $\EE |\zeta_j|^{\max\{4,2q'\}+\eps} \le M^{\max\{4, 2q'\}+\eps}$.

Assume also that the noise $\xi\in \bR$ satisfies a similar condition: $\xi$ is centered and has variance $\sigma_{\xi}^2$ and there exists a constant $M\ge 1$ such that $\EE |\xi|^{4} \le M^{4}\sigma_\xi^{4}$.
\end{Assumption}

Note that if $\xi$ is Gaussian, then it satisfies the second part of Assumption~\ref{assumption-2}. Thus, our result applies under the classical Gaussian noise assumption.

\begin{Theorem}[Weak moment case]\label{thm:independent}
Let $1<q<\infty$. Suppose $\xi$ is independent of $X$, $J\subset[p]$ is such that $X_J$ is independent of $X_{J^c}$,  and  that $\xi$ and $X$ satisfy Assumption~\ref{assumption-2}.
Let 
$\hat\bbeta$ be the minimum $\ell_q$-norm interpolant estimator in the linear regression problem.
Then there exist  constants $C=C(q,M)\ge 1$, $c=c(q,M)\in(0,1)$, $\eps_0=\eps_0(q,M)\in(0,1)$, and $ \kappa_{DM}= \kappa_{DM}(q,M)\in(0,1)$ depending only on $q$ and $M$, such that if
\begin{equation}\label{eq:threshold-N-independent}
C|J|\le N\le \kappa_{DM}
d_*\Log(|J^c|)^{-2 \cdot\1_{\{q>2\}}} 
\end{equation}
then with probability at least
\begin{equation}\label{eq:probability-independent}
1 - \frac{C}{\min\{N^{\varepsilon/2},|J|\}}
- \begin{cases}
2e^{-c  d_* } + Cd_\ast^{-(q'-2)/4} &\text{if}\quad 1< q< 2,\\
 Cd_\ast^{-c\min\{\eps,\eps_0\}} &\text{if}\quad q=2, \\
 Cd_\ast^{-c\min\{\eps,\eps_0\}} +2\Log(|J^c|) e^{- cd_\ast/\Log^{2}(|J^c|)}  &\text{if}\quad  2< q,
\end{cases}
\end{equation}
we have
\begin{equation}\label{eq:aim-independent}
	\|\langle\hat\bbeta-\bbeta^*,X\rangle\|_{L^2(\mu_X)} \le C \biggl( \Log(|J^c|)^{\frac{q}{q-1}\1_{\{q>2\}}}\, r(V_J,V_{J^c}) + \sqrt{\frac{N}{d_*}}\,\sigma_\xi +  \|\Sigma_{J^c}^{1/2}\bbeta_{J^c}^*\|_2 \biggr),
\end{equation}
 where $r=r(V_J,V_{J^c})$ is defined in \eqref{eq:def_r_star_final}.

  Moreover, if $\Sigma=\Id$, then one may omit  the logarithmic factors in \eqref{eq:threshold-N-independent} and in \eqref{eq:aim-independent}, as well as in the probability lower bound.
\end{Theorem}

The probability estimate \eqref{eq:probability-independent} is the union bound over the four events constituting $\Omega_0$: Propositions~\ref{prop:probab-omega-noise-independent}, \ref{prop:probab-omega-RIP-independent} and \ref{prop:probab-omega-multi-independent} contribute $C/N$, $CN^{-\eps/2}$ and $C/|J|$ respectively, which are all bounded by $C/\min\{N^{\eps/2},|J|\}$ because $N\ge C|J|$, while Proposition~\ref{prop:probab-omega-DM-independent} contributes the three-case expression. The logarithmic factor in \eqref{eq:aim-independent} carries the exponent $q'=q/(q-1)$; see the discussion following \eqref{eq:def-omega-DM-log-concave} for its derivation.

\begin{Remark}
If we assume that also higher moments of the random variables $\zeta_j$ are comparable with their second moments -- say $\EE|\zeta_j|^{s}\le M^{s}$ for some $s\ge \max\{4,2q'\}+1$, which holds for instance when the $\zeta_j$'s are uniformly sub-Gaussian -- then the moment-based arguments of Subsection~\ref{sect:propositions-proof-independent} may be run at a higher moment, and one obtains a statement analogous to Theorem~\ref{thm:independent} with a larger lower estimate on the probability, the polynomial terms $C/\min\{N^{\eps/2},|J|\}$ and $Cd_*^{-c\min\{\eps,\eps_0\}}$ being replaced by terms decaying faster (in the sub-Gaussian case the decay will be subexponential). The proof goes along the same lines as the proof of Theorem~\ref{thm:independent}.
We omit the details and will not go back to this case in this article.
\end{Remark}

\subsection{Log-concave case}

In this section we assume that the design variable $X$ and the noise $\xi$ satisfy the following assumption.

\begin{Assumption}\label{assumption-3}
Assume that there exist an isotropic log-concave random vector $\vzeta=(\zeta_1,\dots, \zeta_p)\in \bR^p$ and a $p\times p$ positive definite diagonal matrix $\Sigma$ such that $X=\Sigma^{1/2}\vzeta$.

Assume also that the noise $\xi$ is centered, has variance $\sigma_{\xi}^2$ and is sub-exponential with a constant proportional to $\sigma_\xi$, i.e., it satisfies
 \[
	 \PP(|\xi|\ge t) \le 2e^{-t/(C\sigma_\xi)} \quad \text{for every } t>0
\]
for some universal constant $C$.
\end{Assumption}

Note that if $\xi$ is Gaussian, then it satisfies the second part of Assumption~\ref{assumption-3}. Thus, we may apply our result with the classical Gaussian noise. Let us stress again that the log-concavity assumption does not require that the coordinates of $X$ are independent. In Theorem~\ref{thm:log-concave} below we only require the independence of  $X_J$ and $X_{J^c}$.

\begin{Theorem}[Log-concave case]\label{thm:log-concave}
Let $1<q<\infty$. Suppose $\xi$ is independent of $X$, $J\subset[p]$ is such that $X_J$ is independent of $X_{J^c}$,  and  that $\xi$ and $X$ satisfy Assumption~\ref{assumption-3}.
Let 
$\hat\bbeta$ be the minimum $\ell_q$-norm interpolant estimator in the linear regression problem.
Then there exist  constants $C=C(q)\ge 1$, $c=c(q)\in(0,1)$, and $ \kappa_{DM}= \kappa_{DM}(q)\in(0,1)$ depending only on $q$, such that if
\begin{equation}\label{eq:threshold-N-log-concave}
C|J|\le N\le \kappa_{DM} d_*\Log^{-2}(|J^c|)(\Log\Log(|J^c|))^{-1},
\end{equation}
then with probability at least
\begin{equation}\label{eq:probability-log-concave}
1-2\Log(|J^c|) \exp\Bigl(- c\frac{\sqrt{d_\ast}}{\Log(|J^c|)}\Bigr)-Ce^{-c\sqrt{|J|}},
\end{equation}
we have
\begin{equation}\label{eq:aim-log-concave}
	\|\langle\hat\bbeta-\bbeta^*,X\rangle\|_{L^2(\mu_X)} \le C \biggl(\Log^{\frac{5q}{2(q-1)}}(|J^c|) \,r(V_J,V_{J^c}) + \sqrt{\frac{N}{d_*}}\,\sigma_\xi +  \|\Sigma_{J^c}^{1/2}\bbeta_{J^c}^*\|_2 \biggr),
\end{equation}
 where $r=r(V_J,V_{J^c})$ is defined in \eqref{eq:def_r_star_final}.

Moreover, if $\vzeta$ from Assumption~\ref{assumption-3} is additionally unconditional, then one may omit the logarithmic term in \eqref{eq:aim-log-concave} and the iterated logarithm in \eqref{eq:threshold-N-log-concave}.

If on the other hand $\Sigma=\Id$, then we may omit logarithmic terms in the probability lower bound \eqref{eq:probability-log-concave}, even if $\vzeta$ is not assumed to be unconditional.
\end{Theorem}

The upper part of the threshold \eqref{eq:threshold-N-log-concave} and the first two terms in the probability bound\eqref{eq:probability-log-concave} are the ones provided by Proposition~\ref{prop:upp-log-concave-isotropic} and  Corollary~\ref{cor:log-concave-lower-DM}. The lower part of the threshold \eqref{eq:threshold-N-log-concave} comes from Proposition~\ref{prop:probab-omega-RIP-log-concave}. 
The contributions to \eqref{eq:probability-log-concave}: $2e^{-c\sqrt N}$, $2e^{-\sqrt{|J|}}$, and $2e^{-c\sqrt{|J|}}$, coming from Propositions~\ref{prop:probab-omega-noise-log-concave}, \ref{prop:probab-omega-RIP-log-concave}, and \ref{prop:probab-omega-multi-log-concave}, respectively, are all bounded by $Ce^{-c\sqrt{|J|}}$ since $N\ge C|J|$. 
The exponent $\frac{5q}{2(q-1)}=\frac52 q'$ of the logarithmic factor in \eqref{eq:aim-log-concave} is the $q'$-th power of the factor $\Log^{5/2}(|J^c|)$ appearing in the upper bound in the definition \eqref{eq:def-omega-DM-log-concave} of $\Omega_{\mathrm{DM}}$; see the discussion following equation \eqref{eq:def-omega-DM-log-concave}.


\section{Proofs in the weak moment and log-concave cases}\label{sec:proofs_other_cases}


\subsection{Overview of  new ingredients of the proof in  the weak moment and log-concave cases} \label{section:overview-cases}

Since the outline of the proof in the independent and log-concave cases is the same as the one described in Subsection~\ref{proof-sketch} in the convex concentration case, we will omit this part of proof of Theorems~\ref{thm:independent} and \ref{thm:log-concave}.
The only difference is  the definition of the set $\Omega_{\mathrm{DM}}$.
We need to add some logarithmic factors in the upper bound of its definition to be able to prove that the probability of $\Omega_{\mathrm{DM}}$ is large.
Similarly, we need some logarithmic factor in the upper bounds for $N$ in the assertions of Theorems~\ref{thm:independent} and \ref{thm:log-concave} (see Propositions~\ref{prop:probab-omega-DM-independent} and \ref{prop:probab-omega-DM-log-concave} below).
It would be interesting to skip these logarithmic factor, but our proofs do not allow this.
In the independent case we will be working with
\begin{equation} \label{eq:def-omega-DM-independent}
	\Omega_{\mathrm{DM}} =
	\begin{cases}
	\bigl\{ C_{\mathrm{DM}}^{-1} \ell_* \|\vlambda\|_2 \le \|\bX_{J^c}^\top \vlambda\|_{q'}\le  C_{\mathrm{DM}} \ell_* \|\vlambda\|_2 \quad \text{for all } \vlambda \in\bR^N \bigr\} &\text{if } q\le 2,\\
	 \bigl\{ C_{\mathrm{DM}}^{-1} \ell_* \|\vlambda\|_2 \le \|\bX_{J^c}^\top \vlambda\|_{q'}\le  C_{\mathrm{DM}} \Log(|J^c|)\ell_* \|\vlambda\|_2 \quad \text{for all } \vlambda \in\bR^N \bigr\} &\text{if } q>2,
	\end{cases}
\end{equation}
and in the log-concave case we will be working with
\begin{equation} \label{eq:def-omega-DM-log-concave}
	\Omega_{\mathrm{DM}} =
		 \bigl\{ C_{\mathrm{DM}}^{-1} \ell_* \|\vlambda\|_2 \le \|\bX_{J^c}^\top \vlambda\|_{q'}\le  C_{\mathrm{DM}} \Log^{5/2}(|J^c|)\ell_* \|\vlambda\|_2 \quad \text{for all } \vlambda \in\bR^N \bigr\}
\end{equation}
instead of the set $\Omega_{\mathrm{DM}}$ defined by \eqref{eq:def-omega-DM}. However, in the log-concave unconditional case we may omit the logartithmic term in the definition above and work with $\Omega_{\mathrm{DM}}$ from \eqref{eq:def-omega-DM}.

Let us explain how these logarithmic factors propagate to the final bounds, and denote by $L$ the extra factor appearing in the upper bound in the definition of $\Omega_{\mathrm{DM}}$, so that $L=\Log(|J^c|)$ in \eqref{eq:def-omega-DM-independent} for $q>2$ and $L=\Log^{5/2}(|J^c|)$ in \eqref{eq:def-omega-DM-log-concave}. By the duality equation \eqref{eq:duality-statement}, the upper bound in the definition of $\Omega_{\mathrm{DM}}$ produces the \emph{lower} bound on $\norm{\cA[\vmu]}_q$, while its lower bound produces the \emph{upper} bound on $\norm{\cA[\vmu]}_q$; hence \eqref{eq:def-omega-DM-cA} has to be replaced by
\begin{equation} \label{eq:def-omega-DM-cA-log}
\Omega_{\mathrm{DM}} \subset  \biggl\{ \frac{1}{C_{\mathrm{DM}}L}\frac{\norm{\vmu}_2}{\ell_*} \leq \norm{\cA[\vmu]}_q \leq C_{\mathrm{DM}} \frac{\norm{\vmu}_2}{ \ell_*} \quad \text{for all } \vmu\in\bR^N  \biggr\}.
\end{equation}
Consequently, the only places where the factor $L$ enters the proof of Proposition~\ref{prop:hatbeta_in_set} are the right-hand sides of the conditions on $r$, namely \eqref{eq:lower_cL_5} and the first branch of the minimum in \eqref{eq:lower_cL_10}, which get multiplied by $L^{-q}$, and the second branch of that minimum, which gets multiplied by $L^{-2}$. Let us stress that \eqref{eq:use-of-the-upper-bound-in-DM} and \eqref{eq:norm-cA-DM-RIP'} use the \emph{lower} bound in the definition of $\Omega_{\mathrm{DM}}$, because they combine the constraint $\norm{\bX_{J^c}^\top\vlambda^*}_{q'}\le1$ from \eqref{eq:dual_problem_cA} with $C_{\mathrm{DM}}^{-1}\ell_*\norm{\vlambda^*}_2\le\norm{\bX_{J^c}^\top\vlambda^*}_{q'}$, and therefore they carry no logarithmic factor; the same is true of \eqref{eq:noise_absorption_part_intro}, which is why the noise absorption term $\sigma_\xi\sqrt{N/d_*}$ remains free of logarithms. Solving the resulting conditions on $r$ as in the proof of Proposition~\ref{prop:hatbeta_in_set} multiplies the first two terms in \eqref{eq:def_r_star_final} by $L^{q/(q-1)}=L^{q'}$ and the last two ones by $L^{2}$; since $q'\ge 2$ when $q\le 2$ and only the first two terms occur when $q\ge2$, the value $r$ from \eqref{eq:def_r_star_final} has to be multiplied by $L^{q'}$ in all cases. This is why we obtain the factor $\Log^{\frac{q}{q-1}}(|J^c|) $ in \eqref{eq:aim-independent} for $q>2$, and the factor $\Log^{\frac{5q}{2(q-1)}}(|J^c|) $ in \eqref{eq:aim-log-concave}.

In Subsections~\ref{sect:propositions-independent} and \ref{sect:propositions-log-concave} we give the counterparts of Propositions~\ref{prop:probab-omega-noise}-\ref{prop:probab-omega-multi} in the weak moment and log-concave settings, respectively.
Recall that we do not need any counterpart of Proposition~\ref{prop:hatbeta_in_set}, since it is purely deterministic and \eqref{eq:hatbeta-lies-set} holds on $\Omega_0$ regardless of the properties satisfied by $X$.
Then, in Subsections~\ref{sect:propositions-proof-independent} and \ref{sect:propositions-proof-log-concave}, we give short justifications of these propositions, whose proofs are similar as in the convex concentration case.
The main difference lies in applying substitutions of concentration inequalities in independent and log-concave settings, and the main challenge lies in  obtaining Dvoretzky-Milman type results for independent and log-concave cases.
We postpone all the proofs concerning  Dvoretzky-Milman type inequalities to Section~\ref{sec:proof_DM_q}.


\subsection{Counterparts of Propositions~\ref{prop:probab-omega-noise}-\ref{prop:probab-omega-multi}: \texorpdfstring{$\Omega_0$ has large probability in the weak moment case}{Omega 0 has large probability in the weak moment case}}
\label{sect:propositions-independent}

In Subsections~\ref{proof-sketch} and \ref{section:overview-cases} we reduced the proof of Theorem~\ref{thm:independent} to the following four propositions.

\begin{Proposition}
\label{prop:probab-omega-noise-independent}
If $\vxi=(\xi_1, \dots, \xi_N)$ has i.i.d.\ coordinates with the same distribution as $\xi$, and $X$ and $\xi$ satisfy Assumption~\ref{assumption-2}, then for $C_{\mathrm{noise}}=\sqrt{2}$,
\[
	\bP(\Omega_{\mathrm{noise}}) \ge 1- \frac{C}{N},
\]
where $\Omega_{\mathrm{noise}}$ is defined in \eqref{eq:def-omega-noise} and $C=C(M)=8M^4$.
\end{Proposition}

\begin{Proposition}
\label{prop:probab-omega-RIP-independent}
If  $X$  satisfies Assumption~\ref{assumption-2}, then there exist constants $C_{\mathrm{RIP}},C$, depending only on $M$ and $\varepsilon$, such that if $N \ge C|J|$, then
\[
	\bP(\Omega_{\mathrm{RIP}}) \ge 1- CN^{-\varepsilon/2},
\]
where $\Omega_{\mathrm{RIP}}$ is defined in \eqref{eq:def-omega-RIP}.
\end{Proposition}

\begin{Proposition}
\label{prop:probab-omega-DM-independent}
If $X$  satisfies Assumption~\ref{assumption-2} and $q>1$, then there exist  constants $C_{\mathrm{DM}}$, $\kappa_{\mathrm{DM}}$, $\eps_0$, $c$ and  $C$ depending only on $M$ and $q$, such that
\[
	\bP(\Omega_{\mathrm{DM}})
	 \ge 1-
	\begin{cases}
2e^{-c  d_* } + Cd_\ast^{-(q'-2)/4} &\text{if}\quad 1< q< 2,\\
Cd_\ast^{-c\min\{\eps,\eps_0\}} &\text{if}\quad q=2 \\
Cd_\ast^{-c\min\{\eps,\eps_0\}} +2\Log(|J^c|) e^{- cd_\ast/\Log^{2}(|J^c|)}  &\text{if}\quad  2< q.
\end{cases}
\]
provided that $N$ satisfies \eqref{eq:threshold-N-independent},
where $\Omega_{\mathrm{DM}}$ is defined in \eqref{eq:def-omega-DM-independent}.

Moreover, if $\Sigma=\Id$, then one may omit  the logarithmic factors in \eqref{eq:threshold-N-independent} and in the probability lower bound.
\end{Proposition}

\begin{Proposition}
\label{prop:probab-omega-multi-independent}
If $\vxi=(\xi_1, \dots, \xi_N)$ has i.i.d.\ coordinates with the same distribution as $\xi$, $\bX_{J}$ is independent of $\bX_{J^c}$ and $\vxi$, and $X$ and $\xi$ satisfy Assumption~\ref{assumption-2}, then for $C_{\mathrm{multi}}=\sqrt2$,
\[
	\bP(\Omega_{\mathrm{multi}}) \ge 1- \frac{C}{|J|}.
\]
where $\Omega_{\mathrm{multi}}$ is defined in \eqref{eq:def-omega-multi}, and $C=8M^4$.
\end{Proposition}


\subsection{Proofs of Propositions~\ref{prop:probab-omega-noise-independent}-\ref{prop:probab-omega-multi-independent}: \texorpdfstring{$\Omega_0$ has large probability in the weak moment case}{Omega 0 has large probability in the weak moment case}}
\label{sect:propositions-proof-independent}

Let us begin with the sketch of the proof of Propositions~\ref{prop:probab-omega-noise-independent} and \ref{prop:probab-omega-multi-independent}.
It goes along the same lines as the proof of Propositions~\ref{prop:probab-omega-noise} and \ref{prop:probab-omega-multi}, so below we only describe how to deal with the lack of the concentration inequality from Assumption~\ref{assumption-1}.

\begin{proof}[Proof of Propositions~\ref{prop:probab-omega-noise-independent} and \ref{prop:probab-omega-multi-independent}]
	Let $U_i=|\langle Y_i,\,s\rangle|$, where  $s=\Sigma_{J^c}^{1/2}\bbeta_{J^c}^\ast$.
	As in the proof of Proposition~\ref{prop:probab-omega-noise} in Section~\ref{sec:proof-4.1} we see that $\|\bX_{J^c} \bbeta_{J^c}^\ast \|_2^2 = \sum_{i=1}^N U_i^2$.	Moreover, $\EE U_i^2 =\|\Sigma_{J^c}^{1/2} \bbeta_{J^c}^\ast\|_2^2 $ and, by Assumption~\ref{assumption-2},
	\[
	 	\EE U_i^4 = \sum_{j\in J^c} \EE Y_i(j)^4 s_j^4 + \sum_{j,j'\in J^c, j\neq j'}  \EE Y_i(j)^2 \EE Y_i(j')^2s_j^2 s_{j'}^2
		\le M^4 \bigl(\sum_{j\in J^c} s_j^2\bigr)^2 = M^4 \|\Sigma_{J^c}^{1/2}\bbeta_{J^c}^\ast\|_2^4.
	\]
	Therefore, Markov's inequality and the independence of $U_1,\dots,U_N$ yield
	\begin{align*}
		1- \PP\Bigl( \frac12 N\|\Sigma_{J^c}^{1/2}\bbeta_{J^c}^\ast\|_2^2 \le \sum_{i=1}^N U_i^2 \le 2N\|\Sigma_{J^c}^{1/2}\bbeta_{J^c}^\ast\|_2^2 \Bigr)
		& \le \PP\Bigl( \Bigl| \sum_{i=1}^N U_i^2 - \EE \sum_{i=1}^N U_i^2 \Bigr| \ge \frac12 N\|\Sigma_{J^c}^{1/2}\bbeta_{J^c}^\ast\|_2^2 \Bigr)
		\\ & \le 4\frac{\Var \bigl( \sum_{i=1}^N U_i^2 \bigr)}{N^2\|\Sigma_{J^c}^{1/2}\bbeta_{J^c}^\ast\|_2^4 }
		= \frac{4\Var(U_1^2)}{N\|\Sigma_{J^c}^{1/2}\bbeta_{J^c}^\ast\|_2^4 }
		\\ &\le  \frac{4\EE U_1^4}{N\|\Sigma_{J^c}^{1/2}\bbeta_{J^c}^\ast\|_2^4 }
		\le \frac{4M^4}{N}.
	\end{align*}
	We proceed similarly with two-sided bounds for $\|\vxi\|_2$.
	
The proof of Proposition~\ref{prop:probab-omega-multi-independent} is similar -- see the proof of its counterpart, Proposition~\ref{prop:probab-omega-multi}, in Section~\ref{sec:proof-4.4} for more details.
\end{proof}

Proposition~\ref{prop:probab-omega-DM-independent} follows by Theorem~\ref{thm:DM-independent} below, applied with $q\coloneqq q'$, $m\coloneqq |J^c|$, $A\coloneqq\Sigma_{J^c}^{1/2}$, $d_\ast(A)= d_\ast$ and $\ell_q(A)= \ell_\ast$. 
The moreover part follows by Propositions~\ref{prop:7.6} and \ref{prop:lower_DM}, see also the discussion preceding Corollary~\ref{cor:DM-log-concave_lower}.

\begin{proof}[Proof of Proposition~\ref{prop:probab-omega-RIP-independent}]

By the Hoffmann-J{\o}rgensen inequality (see \cite[Theorem 6.20]{ledouxProbabilityBanachSpaces1991}), there exists a universal constant $C_1$, such that for every unit vector $x \in S_2^J$,
\begin{displaymath}
  \E |\langle \vzeta_J,x\rangle|^4 \le C_1 \Bigl((\E |\langle \vzeta_J,x\rangle|)^4 + \E \max_{j\in J} |x_j \zeta_j|^4\Bigr) \le 2M^4C_1\|x\|_2^4 =2M^4 C_1(\bE \langle \vzeta_J,x\rangle^2)^2.
\end{displaymath}
By Lemma \ref{lem:holder} below, this implies that there exists a constant $C_2$, depending only on $M$, such that for all $x \in S_2^J$,

\begin{displaymath}
  (\E \langle \vzeta_J,x\rangle^2)^{1/2} \le C_2 \E|\langle \vzeta_J, x\rangle|.
\end{displaymath}
Thus, by \cite[Theorem 1.4]{MR3431642}, for some positive constants $C_3,c_1,c_2$, depending only on $M$, if $N \ge C_3|J|$, then with probability at least $1 - \exp(-c_2 N)$
for all $x \in V_J$,
\begin{displaymath}
\|\bY_J x\|_2 \ge c_3\sqrt{N}\|x\|_2.
\end{displaymath}
On the other hand, by \cite[Theorem 2]{MR2111932}, for some universal constant $C_4$,
\begin{displaymath}
  \E \| \bY_J\colon \ell_2^J \to \ell_2^N\| \le C_4\bigl (\sqrt{|J|} + \sqrt{N} + M(|J|N)^{1/4}\bigr).
\end{displaymath}
If $N \ge C_3|J|$, the right-hand side above does not exceed $C_5\sqrt{N}$, where $C_5$ depends only on $M$ and $C_3$. Moreover, by  \cite[Theorem~2]{adamczakFewRemarksOperator2010} (applied with $n=1$, $p=4+\eps$, and $\|X\|=\|X\colon \ell_2^J \to \ell_2^N\|$), for every $t > 0$,
\begin{displaymath}
  \PP\Bigl(\| \bY_J\colon \ell_2^J \to \ell_2^N\| \ge 2C_5\sqrt{N} + t \Bigr) \le \exp\Bigl(-\frac{t^2}{C_6 \sigma^2}\Bigr) + C_6 \frac{\E \max_{i\le N, j \in J} |Y_i(j)|^{4+\varepsilon}}{t^{4+\varepsilon}},
\end{displaymath}
where
\begin{displaymath}\sigma^2 = \sup_{\substack{y \in \RR^N, x \in V_J\\ \|x\|_2,\|y\|_2 \le 1}} \sum_{i=1}^N\sum_{j\in J} \E (Y_i(j)x_jy_i)^2 = 1
\end{displaymath}
and $C_6$ depends only on $\varepsilon$.

Taking into account that $\E \max_{i\le N, j \in J} |Y_i(j)|^{4+\varepsilon} \le |J|N M^{4+\varepsilon}$ and $|J|\le N/C_3$, we obtain that for some constant $C_8$, depending only on $M$ and $\varepsilon$,
\begin{displaymath}
  \PP(\| \bY_J\colon \ell_2^J \to \ell_2^N\| \ge C_8 \sqrt{N}) \le \frac{1}{N^{\varepsilon/2}}.
\end{displaymath}

Summarizing, with probability at least $1 - N^{-\varepsilon/2} - \exp(-c_2 N)$, for all $x \in V_J$,
\begin{displaymath}
  c_3\sqrt{N}  \|x\|_2 \le \|\bY_J x\|_2 \le C_8\sqrt{N} \|x\|_2,
\end{displaymath}
which by a substitution $x = \Sigma^{1/2}\beta$ and the relation $\bX_J \beta = \bY_J x$ (see \eqref{eq:bX-to-bY}) ends the proof.
\end{proof}


\subsection{Counterparts of Propositions~\ref{prop:probab-omega-noise}-\ref{prop:probab-omega-multi}: \texorpdfstring{$\Omega_0$ has large probability in the log-concave case}{Omega 0 has large probability in the log-concave case}}
\label{sect:propositions-log-concave}

In Subsections~\ref{proof-sketch} and \ref{section:overview-cases} we reduced the proof of Theorem~\ref{thm:log-concave} to the following four propositions.

\begin{Proposition}
\label{prop:probab-omega-noise-log-concave}
If $\vxi=(\xi_1, \dots, \xi_N)$ has i.i.d.\ coordinates with the same distribution as $\xi$, and $X$ and $\xi$ satisfy Assumption~\ref{assumption-3}, then there exist universal constants $C_{\mathrm{noise}}$ and $c$ such that
\[
	\bP(\Omega_{\mathrm{noise}}) \ge 1- 2e^{-c\sqrt{N}}
\]
where $\Omega_{\mathrm{noise}}$ is defined in \eqref{eq:def-omega-noise}.
\end{Proposition}

\begin{Proposition}
\label{prop:probab-omega-RIP-log-concave}
Assume that  $X$  satisfies Assumption~\ref{assumption-3}. There exist positive universal constants $C_{\mathrm{RIP}},C$ such that if $N \ge C|J|$, then
\[
	\bP(\Omega_{\mathrm{RIP}}) \ge 1- 2e^{-\sqrt{|J|}}
\]
where $\Omega_{\mathrm{RIP}}$ is defined in \eqref{eq:def-omega-RIP}.
\end{Proposition}

\begin{Proposition}
\label{prop:probab-omega-DM-log-concave}
If $X$  satisfies Assumption~\ref{assumption-3} and $q>1$, then there exist  constants $C_{\mathrm{DM}}$, $\kappa_{\mathrm{DM}}$, and $c$ depending only on $q$, such that
\[
	\bP(\Omega_{\mathrm{DM}}) \ge 1- 2\Log(|J^c|) \exp\Bigl(- c\frac{\sqrt{d_\ast}}{\Log(|J^c|)}\Bigr)
\]
provided that $N$ satisfies \eqref{eq:threshold-N-log-concave},
where $\Omega_{\mathrm{DM}}$ is defined in \eqref{eq:def-omega-DM-log-concave} in general log-concave case, and $\Omega_{\mathrm{DM}}$ is defined in \eqref{eq:def-omega-DM} if we additionally assume that $\vzeta$ from Assumption~\ref{assumption-3} is unconditional.

Moreover, if $\Sigma=\Id$, then we may omit logarithmic terms in the probability lower bound and the double logarithm in the assumption \eqref{eq:threshold-N-log-concave}.\end{Proposition}

\begin{Proposition}
\label{prop:probab-omega-multi-log-concave}
If $\vxi=(\xi_1, \dots, \xi_N)$ has i.i.d.\ coordinates with the same distribution as $\xi$, $\bX_{J}$ is independent of $\bX_{J^c}$ and $\vxi$, and $X$ and $\xi$ satisfy Assumption~\ref{assumption-3}, then there exist universal constants  $C_{\mathrm{multi}}$ and $c$ such that
\[
	\bP(\Omega_{\mathrm{multi}}) \ge 1- 2e^{-c\sqrt{|J|}}.
\]
where $\Omega_{\mathrm{multi}}$ is defined in \eqref{eq:def-omega-multi}.
\end{Proposition}


\subsection{Proofs of Propositions~\ref{prop:probab-omega-noise-log-concave}-\ref{prop:probab-omega-multi-log-concave}: \texorpdfstring{$\Omega_0$ has large probability in the log-concave case}{Omega 0 has large probability in the log-concave case}}
\label{sect:propositions-proof-log-concave}

As in the independent case, let us begin with the sketch of the proof of Propositions~\ref{prop:probab-omega-noise-log-concave} and \ref{prop:probab-omega-multi-log-concave}.
It goes along the same lines as the proof of Propositions~\ref{prop:probab-omega-noise} and \ref{prop:probab-omega-multi}, so below we only describe how to deal with the lack of the concentration inequality from Assumption~\ref{assumption-1}.

\begin{proof}[Proof of Propositions~\ref{prop:probab-omega-noise-log-concave} and \ref{prop:probab-omega-multi-log-concave}]
	As in the proof in the independent case, let $U_i=\langle Y_i,\,s\rangle$, where  $s=\Sigma_{J^c}^{1/2}\bbeta_{J^c}^\ast$.
	Then Assumption~\ref{assumption-3} implies that the coordinates of $V\coloneqq \frac1{\|s\|_2}(U_1,\dots,U_N)$ are independent isotropic log-concave  random variables, so $V$ is isotropic and log-concave random vector.
	Therefore, estimates of Paouris \cite{Paour06,Paour12} (see \cite[Theorems~5.2.1 and 5.3.3]{BGVV} for the exact formulations and proofs of the corollaries we use below) imply that
	\[
		\bP(\|V\|_2 \ge C\sqrt N )\le e^{-\sqrt N}
	\]
	and
	\[
		\bP(\|V\|_2 \le C^{-1}\sqrt N )\le e^{-c\sqrt N}
	\]
	where $c$ and $C$ are universal constants.
	
	To obtain high probability two-sided bounds on $\|\bxi\|_2$ we proceed similarly as in the proof of Proposition~\ref{prop:probab-omega-noise}, but we use \cite[Lemma~2.19]{APSS} instead of  Bernstein's inequality, with $r=1/2$, $Z_i=\sigma_\xi^{-2}\xi_i^2-1$, $a_i=1$,  $k=N$, and $t=N/2$. (Note that \cite[Lemma~2.19]{APSS} needs the variables $Z_i$ to be symmetric, not only centered, but one may repeat the proof of this lemma to obtain a similar result in the centered case.)
	We finish the proof of Proposition~\ref{prop:probab-omega-noise-log-concave} by using the union bound, as in the proof of  Proposition~\ref{prop:probab-omega-noise}.

Proceeding similarly as above we prove Proposition~\ref{prop:probab-omega-multi-log-concave} -- see the proof of its counterpart, Proposition~\ref{prop:probab-omega-multi}, in Section~\ref{sec:proof-4.4} for more details.
\end{proof}

Proposition~\ref{prop:probab-omega-DM-log-concave} follows by Proposition~\ref{prop:upp-log-concave-isotropic} and Corollary~\ref{cor:log-concave-lower-DM} below, applied with $q\coloneqq q'$, $m\coloneqq |J^c|$, $A\coloneqq\Sigma_{J^c}^{1/2}$, $d_\ast(A)= d_\ast$ and $\ell_q(A)= \ell_\ast$. 
The moreover part follows from Propositions~\ref{prop:upp-log-concave-isotropic-uncond} and \ref{prop:low-log-concave-isotropic}.

\begin{proof}[Proof of Proposition \ref{prop:probab-omega-RIP-log-concave}] Note that $\vzeta_J$ is an isotropic log-concave random vector in $V_J$. Thus, for some universal constant $C_1$,
\begin{displaymath}
  (\E \langle \vzeta_J,x\rangle^2)^{1/2} \le C_1 \E|\langle \vzeta_J, x\rangle|
\end{displaymath}
for all $x\in V_J$.
By \cite[Theorem 1.4]{MR3431642}, this implies that for some absolute positive constant $c_1,C_2$, if $N \ge C_2|J|$, then with probability at least $1 - \exp(-c_1 N)$,
for all $x \in V_J$,
\begin{displaymath}
\|\bY_J x\|_2 \ge c_1\sqrt{N}\|x\|_2.
\end{displaymath}

By Theorem \cite[Theorem 4.1]{MR2601042}, there exists an absolute constant $C_3$, such that if $N \ge C_3|J|$, then with probability at least $1 - \exp(-\sqrt{|J|})$,
\begin{displaymath}
\| \bY_J\colon \ell_2^J \to \ell_2^N\| \le 2\sqrt{N}.
\end{displaymath}
We finish the proof by using the union bound, the assumption $N \ge C|J|$, adjusting the constants, and substituting $x = \Sigma^{1/2}\beta$ (recall that $\bX_J \beta = \bY_J x$ by \eqref{eq:bX-to-bY}).
\end{proof}


\section{Control of the noise absorption part \texorpdfstring{$\norm{\Sigma_{J^c}^{1/2}\hat\bbeta_{J^c} }_2$}{noise absorption}}\label{sec:noise_absorption_part}

In this section, we are interested in establishing an explicit (only depending on $(\sigma_j)_{j\in J^c}$)  non-asymptotic high-probability upper bound  for $\norm{\Sigma_{J^c}^{1/2}\hat\bbeta_{J^c} }_2$ that appears in the FSD approach as one of the two statistical errors over the `noise absorption' part of the feature space -- the other error term being the deterministic alignment term $\norm{\Sigma_{J^c}^{1/2}\bbeta_{J^c}^*}_2$. The bound obtained here is strictly better than the `quick' one from \eqref{eq:noise_absorption_part_intro}: the Dvoretzky dimension $d_*$ is replaced by the effective rank $R_*$ from \eqref{eq:def_effective_rank_R_star}, which by \eqref{eq:Holder_R_versus_d} is always at least of the order of $d_*$ and is often much larger. It is this improvement which is exploited in Theorem~\ref{thm:weak-moment-optimal}.

We keep the notation of Section~\ref{sec:main_results}, which we briefly recall, and we let $\Sigma=\diag(\sigma_1,\dots,\sigma_p)$ with $\sigma_1\ge\dots\ge\sigma_p>0$ be diagonal throughout the section. The noise vector $\bxi$ has independent coordinates distributed as $\xi$, it is independent of $\bX=[X_1|\cdots|X_N]^\top$, and the only property of $\xi$ used in this section is that it is centered with $\bE\xi^2=\sigma_\xi^2<\infty$; in particular no Gaussianity and no sub-Gaussianity of the noise is required. We denote $\bX_I=[P_IX_1|\cdots|P_IX_N]^\top$ for $I\subset[p]$. We use the abbreviations $\ell_*=\ell_*(\Sigma_{J^c}^{1/2}B_q^p)$ and $d_*=d_*(\Sigma_{J^c}^{-1/2}B_{q'}^p)=(\ell_*/\diam(\Sigma_{J^c}^{1/2}B_q^p))^2$ introduced before. Since $\Sigma_{J^c}$ vanishes on $V_J$ we have $\Sigma_{J^c}^{1/2}B_q^p=\Sigma_{J^c}^{1/2}B_q^{J^c}$, so that $\ell_*=\ell_*(\Sigma_{J^c}^{1/2}B_q^{J^c})$; we will use the two expressions interchangeably. Recall finally the operator $\cA:\vmu\in\bR^N\mapsto\cA[\vmu]\in\argmin\left(\|\vnu\|_q:\, \bX_{J^c}\vnu=\vmu, \, \vnu\in V_{J^c}\right)$ from \eqref{eq:def-cA}, and the FSD decomposition \eqref{eq:beta_J-minimizer}--\eqref{eq:beta_J^c-minimizer} of $\hat\bbeta$, which gives $\hat\bbeta_{J^c} = \cA[\vy-\bX_J\hat\bbeta_J] = \cA[\bxi+\vzeta]$ with $\vzeta = \bX_J(\bbeta_J^*-\hat\bbeta_J) + \bX_{J^c}\bbeta_{J^c}^*$, so that
\begin{align*}
   \norm{\Sigma_{J^c}^{1/2}\hat\bbeta_{J^c} }_2 =  \|\Sigma_{J^c}^{1/2}\cA[\bxi+\vzeta]\|_2.
\end{align*}

To handle the term $\norm{\Sigma_{J^c}^{1/2}\hat\bbeta_{J^c} }_2$, we first need statistical guarantees on the `estimation part'; this is because $\hat\bbeta_{J^c} = \cA[\vy-\bX_J\hat\bbeta_J]$ uses the residual $\vy-\bX_J\hat\bbeta_J$ of the estimation part $\hat\bbeta_J$. Hence, we consider the following event $\Omega_{\mathrm{esti}}$ that handles the estimation part of $\hat \bbeta$: let $r_*>0$  and
\begin{align}\label{eq:upper_vzeta}
    \Omega_{\mathrm{esti}} : =\left\{\frac{1}{\sqrt N}\|\vzeta\|_2\leq r_*\right\}.
\end{align}In $\Omega_{\mathrm{esti}}$, $r_*$ plays the role of the convergence rate that is derived from the statistical analysis of the estimation part; on the event $\Omega_{\mathrm{RIP}}\cap\Omega_{\mathrm{noise}}$ and under \eqref{eq:hatbeta-lies-set} (which holds on $\Omega_0$), the triangle inequality and \eqref{eq:X_J(beta-beta*)-estimate} show that it holds with
\begin{align}\label{eq:value_of_r_star}
    r_* = C_{\mathrm{RIP}}\,r_L + C_{\mathrm{noise}}\norm{\Sigma_{J^c}^{1/2}\bbeta^*_{J^c}}_2 ,
\end{align}
where $r_L$ is defined in \eqref{eq:def_r_L}.

A second preliminary property that we need to handle the noise absorption part concerns the design matrix restricted to $V_{J^c}$. Indeed, the operator $\cA$ uses only the projected features $P_{J^c}X_1,\cdots,P_{J^c}X_N$ (i.e. the row vectors of $\bX_{J^c}$) to interpolate the data. The one condition we identified as sufficient for our analysis is the `lower side of the Dvoretzky--Milman isomorphic property': unlike in the `quick' analysis leading to \eqref{eq:noise_absorption_part_intro}, we shall not need the two-sided inequality defining $\Omega_{\mathrm{DM}}$ in \eqref{eq:def-omega-DM}, but \emph{only} its lower bound part -- and this for all the coordinate subspaces of $V_{J^c}$ obtained by deleting a subset of the coordinates. We will therefore work under the following assumption, in which $I$ ranges over the subsets of $J^c$ and $\ell_*(\Sigma_I^{1/2}B_q^I)$, $d_*(\Sigma_I^{-1/2}B_{q'}^I)$ are the quantities $\ell_*$ and $d_*$ computed on the coordinate subspace $V_I$ instead of $V_{J^c}$.

\begin{Assumption}\label{assum:LBDM}
Let $I\subset J^c$ and
    \begin{equation}
        \label{eq:LBDM_event}
        \Omega_{\mathrm{LBDM}}(I)  =\left\{\forall \boldsymbol\lambda\in\bR^N,\quad   \|\bX_I^\top\boldsymbol\lambda\|_{q'} \geq c\,\ell_{*}(\Sigma_I^{1/2} B_q^I)\|\boldsymbol\lambda\|_2 \right\}.
    \end{equation}
When $1<q\leq 2$, we assume that $\bP[\Omega_{\mathrm{LBDM}}(I)] \geq 1-C\exp(-c\,d_*(\Sigma_I^{-1/2} B_{q'}^I))$, and when $q>2$, we assume that $\bP[\Omega_{\mathrm{LBDM}}(I)]\ge 1-C\Log(|I|)\exp(-c\,d_*(\Sigma_I^{-1/2}B_{q'}^I)/\Log^2(|I|))$.
\end{Assumption}

Assumption~\ref{assum:LBDM} for sets $I$ satisfying $d_*(\Sigma_I^{-1/2}B_{q'}^p)\ge \hat C^{-1}d_*$ is implied by the lower half of the event $\Omega_{\mathrm{DM}}$ from \eqref{eq:def-omega-DM}, \eqref{eq:def-omega-DM-independent} and \eqref{eq:def-omega-DM-log-concave}, written on the coordinate subspace $V_I$, with $c=C_{\mathrm{DM}}^{-1}$; it therefore holds in each of our three stochastic scenarios, provided that $N \le \hat C^{-1}\kappa_{DM}d_*$, where $\kappa_{DM}$ is the constant from the threshold for $N$ from Section~\ref{sec:proof_DM_q}. Under the weak moment Assumption~\ref{assumption-2} -- the setting in which the results of this section will be used -- it is precisely the content of item~\emph{1.} of Theorem~\ref{thm:DM-independent} applied with $q:=q'\ge2$ when $1<q\le2$, and of item~\emph{2.} of the same theorem applied with $q:=q'<2$ when $q>2$, in both cases with $A=\Sigma_I^{1/2}=\diag(\sqrt{\sigma_j}:j\in I)$ and $m=|I|$; see Proposition~\ref{prop:probab-omega-DM-independent} and Subsection~\ref{subsec:DM-indep}. Note that the moment condition $\EE|\zeta_j|^{\max\{4,2q'\}}\le M^{\max\{4,2q'\}}$ needed there is part of Assumption~\ref{assumption-2}, and that it is inherited by every coordinate subspace, so that a single moment assumption yields Assumption~\ref{assum:LBDM} simultaneously for all $I\subset J^c$ satisfying $d_*(\Sigma_I^{-1/2}B_{q'}^p)\ge \hat C^{-1}d_*$. In the convex concentration case one may instead invoke Proposition~\ref{thm:DM}, and in the log-concave case Corollary~\ref{cor:log-concave-lower-DM}.

The Gaussian mean widths as well as the Dvoretzky dimensions appearing in Assumption~\ref{assum:LBDM} can be explicitly estimated since $\Sigma$ is diagonal: by \eqref{eq:c0'} and \eqref{eq:c0} below, applied to $A\coloneqq\Sigma_I^{1/2}$ and $q\coloneqq q'$, for all $I\subset[p]$,
\begin{align}\label{eq:estimate_ell_star}
(\bE|g|)\bigg(\sum_{j\in I}\sigma_j^{q'/2}\bigg)^{1/q'}\le \ell_*(\Sigma_{I}^{1/2}B_q^I)\le \big(\bE|g|^{q'}\big)^{1/q'}\bigg(\sum_{j\in I}\sigma_j^{q'/2}\bigg)^{1/q'},    
\end{align}
where $g$ is a standard Gaussian random variable.

Given that the FSD method and so the study of estimators over the noise absorption part are new, it is important to identify tools and events sufficient for them to proceed. Let $\delta$ denote the random selector vector and let and $\cI(\delta)$ be its selected index subset, both defined later in Section~\ref{sub:upper_bound_for_moment}. In this section, the noise absorption part is handled on the intersection $\Omega_{\mathrm{LBDM}}(J^c)\cap  \Omega_{\mathrm{LBDM}}(J^c\setminus \cI(\delta)) \cap\Omega_{\mathrm{esti}}\cap\Omega_{\mathrm{noise}}$ for $\delta\in D$, where $\Omega_{\mathrm{noise}}$ is the event \eqref{eq:def-omega-noise} and $D$ is the set defined in \eqref{eq:def-D}; $D$ is defined in such a way that $I=J^c\setminus \mathcal I(\delta)$ for $\delta \in D$ satisfies the aforementioned property $d_*(\Sigma_I^{-1/2}B_{q'}^p)\ge \hat C^{-1}d_*$, see \eqref{eq:stability_Dvoretzky_dimension} below. For each realization $\delta\in D$, we condition on $\vdelta$ and use the high conditional probability of $\Omega_{\mathrm{LBDM}}(J^c\backslash \cI(\delta))$ with respect to $\bX_{J^c}$. It means that the noise absorption part is handled when $\bX_{J^c}$ has some type of lower DM property, that we know the statistical rate of the estimator over the estimation part and we have a control of $\norm{\bxi}_2$ and $\norm{\bX_{J^c}\bbeta^*_{J^c}}_2$. Let us now show that these events are indeed sufficient for the analysis over the noise absorption part.

\subsection{A reduction from \texorpdfstring{$\|\Sigma_{J^c}^{1/2}\cA[\vzeta+\bxi]\|_2^2$}{xi + zeta} to \texorpdfstring{$\|\Sigma_{J^c}^{1/2}\cA[\bxi]\|_2^2$}{xi}.}\label{sub:a_reduction_via_gaussian_shift_inequality_}

A major challenge when controlling $\|\Sigma_{J^c}^{1/2}\hat\bbeta_{J^c}\|_2 = \|\Sigma_{J^c}^{1/2}\cA[\vzeta+\bxi]\|_2$ for $\vzeta = \bX_J(\bbeta_J^*-\hat\bbeta_J) + \bX_{J^c}\bbeta_{J^c}^*$ is the dependency between $\cA$ and $\vzeta$ that both depend on $\bX_{J^c}$. In this section, we show how to bypass this problem.

The first step of the analysis is a reduction of $\|\Sigma_{J^c}^{1/2}\cA[\bxi+\vzeta]\|_2$ to $\|\Sigma_{J^c}^{1/2}\cA[\bxi]\|_2$, that is, a quantitative continuity estimate for the nonlinear operator $\cA$ seen as a map from $(\bR^N,\norm{\cdot}_2)$ into $(V_{J^c},\|\Sigma_{J^c}^{1/2}\cdot\|_2)$. It only uses the lower side of the Dvoretzky--Milman property on $V_{J^c}$ itself.

\begin{Proposition}\label{prop:Gaussian_shift_reduc_new}
    Suppose $N\leq c_0d_*$ for some constant $0<c_0<1$. Then there exists a constant $C=C(q,c,c_0, C_{\mathrm{noise}})$, depending only on $q$, $c_0$, $C_{\mathrm{noise}}$ and $c$ from the definition of $\Omega_{\mathrm{noise}}$ and $\Omega_{\mathrm{LBDM}}$, such that on $\Omega_{\mathrm{LBDM}}(J^c)\cap\Omega_{\mathrm{esti}}\cap\Omega_{\mathrm{noise}}$ (where $\Omega_{\mathrm{noise}}$ is the event defined in \eqref{eq:def-omega-noise}, which in particular guarantees that $\|\bxi\|_2\le C_{\mathrm{noise}}\sqrt N\sigma_\xi$),
    \begin{align*}
    \|\Sigma_{J^c}^{1/2}(\mathcal A[\boldsymbol\xi+\vzeta]-\mathcal A[\boldsymbol\xi])\|_2^2\leq \begin{cases}
       C (\sigma_\xi+r_\ast)^{\frac{2(2-q)}{3-q}} r_*^{\frac{2}{3-q}}, & \mbox{ if }1<q<2,\\
       C (\sigma_\xi+r_\ast)^{\frac{2(q-2)}{q-1}} r_*^{2q'-2}, &\mbox{ if } 2\le q<\infty.
    \end{cases}
\end{align*}
\end{Proposition}
\begin{proof} For every $\boldsymbol\mu\in\mathbb R^N$ let 
\[
\overline{\boldsymbol\lambda^*}[\boldsymbol\mu]=\|\cA[\vmu]\|_q^{q-1} \vlambda^*[\vmu],
\]
where $\vlambda^*[\vmu]$ is defined in \eqref{eq:dual_problem_cA}.
By item \ref{item:calculation_subdifferential} of Lemma~\ref{lemma:operator_A}, 
we have 
\begin{equation}\label{eq:prop_loose_zeta_1}
\mathbb X_{J^c}^{\top}\overline{\boldsymbol\lambda^*}[\boldsymbol\mu]=|\mathcal A[\boldsymbol\mu]|^{q-2}\odot\mathcal A[\boldsymbol\mu],
\end{equation}
 while the definition of $\cA[\vmu]$ gives $\mathbb X_{J^c}\mathcal A[\boldsymbol\mu]=\boldsymbol\mu$.
Thus, for any $\boldsymbol\mu,\boldsymbol h\in\bR^N$,
\begin{align}\label{eq:KKT_mu_h}
    \begin{aligned}
        \mathbb X_{J^c}^{\top}\big(\overline{\boldsymbol\lambda^*}[\boldsymbol\mu+\boldsymbol h]-\overline{\boldsymbol\lambda^*}[\boldsymbol\mu]\big)=|\mathcal A[\boldsymbol\mu+\boldsymbol h]|^{q-2}\odot\mathcal A[\boldsymbol\mu+\boldsymbol h]-|\mathcal A[\boldsymbol\mu]|^{q-2}\odot\mathcal A[\boldsymbol\mu],\mbox{ and }
        \mathbb X_{J^c}\big(\mathcal A[\boldsymbol\mu+\boldsymbol h]-\mathcal A[\boldsymbol\mu]\big)=\boldsymbol h.
    \end{aligned}
\end{align}
By duality equation \eqref{eq:duality-statement}, on $\Omega_{\mathrm{LBDM}}(J^c)$,
\begin{align}
&\forall\boldsymbol\mu\in\mathbb R^N,\qquad \|\mathcal A[\boldsymbol\mu]\|_q\leq\frac{\|\boldsymbol\mu\|_2}{c\ell_*(\Sigma_{J^c}^{1/2}B_q^{J^c})}.
\label{eq:A_q_bound}
\end{align}

From $\|\Sigma_{J^c}^{1/2}\vv\|_2\leq\diam(\Sigma_{J^c}^{1/2}B_q^{J^c})\|\vv\|_q$ for any $\vv\in V_{J^c}$, we obtain $\|\Sigma_{J^c}^{1/2}(\cA[\vmu+\vh]-\cA[\vmu])\|_2^2\leq \diam(\Sigma_{J^c}^{1/2}B_q^{J^c})^2\|\cA[\vmu+\vh]-\cA[\vmu]\|_q^2$. We use the following inequality, whose proof is postponed to the appendix,
\begin{align}\label{eq:upper_u_v_q_norm}
   \forall\vu,\vv\in V_{J^c},\qquad\begin{cases}
        (q-1)\|\vu-\vv\|_q^2(\|\vu\|_q+\|\vv\|_q)^{q-2}\leq \langle |\vu|^{q-2}\odot\vu - |\vv|^{q-2}\odot\vv, \vu-\vv\rangle, & \mbox{ if } 1<q<2\\
        2^{2-q}\|\vu-\vv\|_q^q\leq \langle |\vu|^{q-2}\odot\vu - |\vv|^{q-2}\odot\vv, \vu-\vv\rangle, & \mbox{ if } 2\le q<\infty,
    \end{cases}
\end{align}applied to $\vu = \cA[\vmu+\vh]$, $\vv=\cA[\vmu]$ to obtain the upper bound for $\|\cA[\vmu+\vh]-\cA[\vmu]\|_q^2$ in the following way.
\begin{enumerate}
    \item When $1<q<2$, by \eqref{eq:KKT_mu_h},
    \begin{align*}
        &\big\langle \big|\cA[\vmu+\vh]\big|^{q-2}\odot \cA[\vmu+\vh] -  \big|\cA[\vmu]\big|^{q-2}\odot \cA[\vmu], \cA[\vmu+\vh] - \cA[\vmu] \big\rangle = \langle \overline{\boldsymbol\lambda^*}[\vmu+\vh] - \overline{\boldsymbol\lambda^*}[\vmu], \vh\rangle\\
        &\leq\|\overline{\boldsymbol\lambda^*}[\vmu+\vh] - \overline{\boldsymbol\lambda^*}[\vmu]\|_2\|\vh\|_2.
    \end{align*}On $\Omega_{\mathrm{LBDM}}(J^c)$, $\|\overline{\boldsymbol\lambda^*}[\vmu+\vh] - \overline{\boldsymbol\lambda^*}[\vmu]\|_2\leq \frac{1}{c\ell_*(\Sigma_{J^c}^{1/2}B_q^{J^c})}\|\bX_{J^c}^\top(\overline{\boldsymbol\lambda^*}[\vmu+\vh] - \overline{\boldsymbol\lambda^*}[\vmu])\|_{q'}$. By \eqref{eq:KKT_mu_h} 
    together with the following inequality, whose proof is postponed to the appendix,
    \begin{align}\label{eq:upper_q_prime}
        \forall\vu,\vv\in V_{J^c},\qquad \| |\vu|^{q-2}\odot\vu - |\vv|^{q-2}\odot \vv\|_{q'}\leq 2^{2-q}\|\vu-\vv\|_q^{q-1},
    \end{align}
there holds
    \begin{align*}
        &\|\overline{\boldsymbol\lambda^*}[\vmu+\vh] - \overline{\boldsymbol\lambda^*}[\vmu]\|_2\leq \frac{2^{2-q}}{c\ell_*(\Sigma_{J^c}^{1/2}B_q^{J^c})}\norm{\cA[\vmu+\vh] - \cA[\vmu]}_q^{q-1}.
    \end{align*}
    Combining with \eqref{eq:upper_u_v_q_norm} and \eqref{eq:A_q_bound}, we obtain that there exists some constant $\tilde C$ (with $\lim_{q\to 1^+}C_q\to\infty$), depending only on $q,c,c_0,$ and $C_{\mathrm{noise}}$, such that for any $\boldsymbol\mu,\boldsymbol h\in\mathbb R^N$,
    \begin{align*}
    &\|\Sigma_{J^c}^{1/2}(\mathcal A[\boldsymbol\mu+\boldsymbol h]-\mathcal A[\boldsymbol\mu])\|_2^2 \leq \tilde C\frac{\operatorname{diam}^2(\Sigma_{J^c}^{1/2}B_q^{J^c})}{\ell_*^2(\Sigma_{J^c}^{1/2}B_q^{J^c})}(2\|\boldsymbol\mu\|_2+\|\boldsymbol h\|_2)^{\frac{2(2-q)}{(3-q)}}\|\boldsymbol h\|_2^{\frac{2}{(3-q)}}.
    \end{align*}

    \item When $2\le q<\infty$, the argument is similar, but with \eqref{eq:upper_q_prime} replaced by
    \begin{align}\label{eq:upper_q_prime_2}
        \forall\vu,\vv\in V_{J^c},\qquad \| |\vu|^{q-2}\odot\vu - |\vv|^{q-2}\odot \vv\|_{q'}\leq (q-1)(\|\vu\|_q+\|\vv\|_q)^{q-2}\|\vu-\vv\|_q,
    \end{align}
    whose proof is postponed to the appendix.
    In this case, we obtain that there exists some constant $\tilde C$ such that for any $\vmu,\vh\in\bR^N$,
    \begin{align*}
        \|\Sigma_{J^c}^{1/2}(\mathcal A[\boldsymbol\mu+\boldsymbol h]-\mathcal A[\boldsymbol\mu])\|_2^2\leq \tilde C\frac{\operatorname{diam}^2(\Sigma_{J^c}^{1/2}B_q^{J^c})}{\ell_*^2(\Sigma_{J^c}^{1/2}B_q^{J^c})}\cdot(2\|\boldsymbol\mu\|_2+\|\boldsymbol h\|_2)^{\frac{2(q-2)}{(q-1)}}\|\boldsymbol h\|_2^{\frac{2}{(q-1)}}.
    \end{align*}
\end{enumerate}
It remains to combine the above two cases applied to $\vmu=\bxi$ and $\vh = \vzeta$. On $\Omega_{\mathrm{noise}}\cap\Omega_{\mathrm{esti}}$ we have  $\|\vzeta\|_2\le r_*\sqrt N$ and $2\|\bxi\|_2+\|\vzeta\|_2\le (2C_{\mathrm{noise}}\sigma_\xi+r_*)\sqrt N$ , while $\diam^2(\Sigma_{J^c}^{1/2}B_q^{J^c})/\ell_*^2 = d_*^{-1}\le c_0/N$ by the assumption $N\le c_0d_*$. Since $\frac{2(2-q)}{3-q}+\frac{2}{3-q}=2$ when $1<q<2$ and $\frac{2(q-2)}{q-1}+\frac{2}{q-1}=2$ when $q\ge 2$, the powers of $\sqrt N$ give exactly the factor $N^{-1}$, and we obtain that on $\Omega_{\mathrm{LBDM}}(J^c)\cap\Omega_{\mathrm{esti}}\cap\Omega_{\mathrm{noise}}$, there holds
\begin{align*}
    \|\Sigma_{J^c}^{1/2}(\mathcal A[\boldsymbol\xi+\vzeta]-\mathcal A[\boldsymbol\xi])\|_2^2\leq \begin{cases}
       C (\sigma_\xi+r_\ast)^{\frac{2(2-q)}{3-q}} r_*^{\frac{2}{3-q}}, & \mbox{ if }1<q<2,\\
       C (\sigma_\xi+r_\ast)^{\frac{2(q-2)}{q-1}} r_*^{2q'-2}, &\mbox{ if } 2\le q<\infty.
    \end{cases}
\end{align*}
The verification of \eqref{eq:upper_u_v_q_norm}, \eqref{eq:upper_q_prime} and \eqref{eq:upper_q_prime_2} is postponed to Section~\ref{sec:verification_equations}.
\end{proof}

\subsection{Upper bound for \texorpdfstring{$\|\Sigma_{J^c}^{\frac{1}{2}}\cA[\bxi]\|_2^2$}{noise part} under moment condition.}\label{sub:upper_bound_for_moment}
We now bound $\|\Sigma_{J^c}^{1/2}\cA[\bxi]\|_2^2$ under independence and moment assumptions. Let us first collect all the assumptions that we will use in our analysis.

\begin{Assumption}
    \label{assum:beyond_gaussian}
    We assume that $X = (\langle X,\be_j\rangle)_{j=1}^p$ has $p$ independent coordinates and that $\Sigma = \diag(\sigma_1, \ldots, \sigma_p)$ with $\sigma_1\geq \ldots\geq \sigma_p>0$\footnote{The ordering of the $\sigma_j$'s plays no role in the proofs, but we keep it since it is always available and it is the usual convention.} is its covariance matrix. Assume moreover that there exists $M'\ge1$ such that
    \begin{equation}
        \label{eq:moment_assumption_noise_absorption}
        \bigl\| \inr{X, \be_j} \bigr\|_{L^{\max\{2,\, 2q'-2\}}} \le M' \sqrt{\sigma_j} = M'\bigl\|\inr{X, \be_j}\bigr\|_{L^2}
        \qquad \text{for every } j\in J^c.
    \end{equation}
    Finally, the noise $\bxi=(\xi_1,\dots,\xi_N)$ has independent coordinates distributed as $\xi$, which is centered with $\bE\xi^2=\sigma_\xi^2$, and $\bxi$ is independent of $\bX$.
\end{Assumption}

Two comments are in order. First, \eqref{eq:moment_assumption_noise_absorption} is a genuine moment condition only when $q'>2$, i.e. when $q<2$; for $q\ge2$ it reduces to a tautology, and in that range the whole analysis below only uses the variances $\sigma_j$. Second, since $\max\{2,2q'-2\}\le\max\{4,2q'\}$, Assumption~\ref{assum:beyond_gaussian} is implied by Assumption~\ref{assumption-2} with $M'=M$; the same assumption therefore yields at once Assumption~\ref{assum:LBDM} and \eqref{eq:moment_assumption_noise_absorption} for sets $I$ satisfying $d_*(\Sigma_I^{-1/2}B_{q'}^p)\ge \hat C^{-1}d_*$ and provided that $N \le \hat C^{-1}\kappa_{DM}d_*$, where $\kappa_{DM}$ is the constant from the threshold for $N$ from Section~\ref{sec:proof_DM_q}. Let us emphasize that no moment assumption beyond $\bE\xi^2=\sigma_\xi^2<\infty$ is imposed on the noise: the decoupling argument below is insensitive to the distribution of $\bxi$, which is what allows us, in Section~\ref{sec:main_results_weak_moment_optimal}, to work under weak moment assumptions on \emph{both} the design vector and the noise.

By independence of the coordinates of $X$, \eqref{eq:moment_assumption_noise_absorption} implies that for every $j\in J^c$ and every deterministic $\vlambda\in\bR^N$,
\begin{equation}
    \label{eq:moment_assumption_column}
    \bigl(\bE|\langle\bX_{J^c}\be_j,\boldsymbol{\lambda}\rangle|^{2q'-2}\bigr)^{\frac{1}{2q'-2}} \le C_qM'\sqrt{\sigma_j}\,\|\boldsymbol{\lambda}\|_2 ;
\end{equation}
when $2q'-2\le 2$ this holds with $C_q=1$, because $\bE\langle\bX_{J^c}\be_j,\vlambda\rangle^2=\sigma_j\|\vlambda\|_2^2$, and when $2q'-2>2$ it follows from the Marcinkiewicz--Zygmund inequality together with \eqref{eq:moment_assumption_noise_absorption}, at the price of enlarging $M'$ by a factor depending on $q$ only.

We start with the following lemma that shows a stability result of the lower bound in the Dvoretzky--Milman theorem under coordinate random projections. For every $I\subset J^c$, set $V_{J^c\setminus I}=\Span(\be_j:j\in J^c\setminus I)$, denote by $\mathrm{Proj}_{J^c\setminus I}$ the coordinate projection onto $V_{J^c\setminus I}$, and 
\begin{align*}
&\bX_{J^c\setminus I}=[\mathrm{Proj}_{J^c\setminus I}X_1|\cdots|\mathrm{Proj}_{J^c\setminus I}X_N]^\top:V_{J^c\setminus I}\to\bR^N. 
\end{align*}We also consider the set
\begin{equation}\label{eq:def-D}
    D = \bigg\{\delta \in\{0,1\}^{J^c}: \bigg| \sum_{j\in \cI(\delta)}\sigma_j^{q'/2} - \frac{1}{2}\sum_{j\in J^c}\sigma_j^{q'/2}\bigg|\leq\frac{1}{4}\sum_{j\in J^c}\sigma_j^{q'/2}\bigg\}
\end{equation} where, for any $\delta=(\delta_j)_{j\in J^c}\in\{0,1\}^{J^c}$,  $\mathcal I(\delta)=\{j\in J^c:\delta_j=1\}$ is the set of selected indices  and the event
\begin{align*}
\Omega_{\mathrm{size}}=\big\{\boldsymbol{\delta}\in D\big\}
\end{align*}where  $\boldsymbol{\delta}:=(\delta_j)_{j\in J^c}$ is a vector of i.i.d.\ selectors (i.e., Bernoulli variables) independent of $\bX_{J^c}$ and $\bxi$, satisfying $\bP(\delta_j=1)=1/2 = \bP(\delta_j=0)$ for every $j \in J^c$. 

\begin{Lemma}[Stability of the lower side of the Dvoretzky--Milman theorem under random coordinate projections]\label{lemma:selector_stability}
Let $\kappa\in(0,1)$, $1<q\le2$ and grant Assumption~\ref{assum:LBDM} for every $I=J^c\setminus \mathcal I(\delta)$ with $\delta\in D$. Let \(g\) be a standard Gaussian random variable, and $c,C$ be the constants from Assumption~\ref{assum:LBDM}, and set $\tilde c_q=4^{-1/q'}\frac{\bE|g|}{\|g\|_{L^{q'}}}$, \(c_q = c\tilde c_q\), \(C'=C+2\), and  $\eta_q=\frac{1}{8\kappa_{DM}\|g\|_{L^{q'}}^2}$. Assume that the constant $\kappa$ is smaller than \(c\tilde c_q^2\). If \(N\leq \kappa d_*\), then
\begin{align*}
&\bP_{\bX_{J^c},\boldsymbol{\delta}}\bigg(\forall\boldsymbol{\lambda}\in\bR^N,\ \big\|\bX_{J^c\setminus\mathcal I(\boldsymbol{\delta})}^\top\boldsymbol{\lambda}\big\|_{q'}\geq c_q\ell_\ast\|\boldsymbol{\lambda}\|_2\bigg) \geq 1-C'\exp(-\eta_qN). 
\end{align*}

Similarly, when $q> 2$ and $N\leq \kappa\,\Log^{-2}(|J^c|)d_*$, the same random event holds with $\bP_{\bX_{J^c},\vdelta}$ probability at least $1-C'\Log(|J^c|)\exp\big( -\eta_q\Log^{-2}(|J^c|) N\big)$, provided that $\kappa$ is smaller than \(c\tilde{c}_q^2\).
\end{Lemma}
Let us stress that the only property of the collection $\mathcal{C}=\{J^c\setminus\cI(\delta):\delta\in D\}$ that is used is the lower bound \eqref{eq:stability_Dvoretzky_dimension} below on the corresponding Dvoretzky dimensions. As we have already mentioned, under Assumption~\ref{assumption-2}, Theorem~\ref{thm:DM-independent}, estimate \eqref{eq:stability_Dvoretzky_dimension}, and the assumption $N\le \kappa_{\mathrm{DM}} \tilde c_q^2d_*\Log(|J^c|)^{-2\cdot \1_{q>2}}$ imply that Assumption~\ref{assum:LBDM} holds simultaneously for $I=J^c$ and for every $I$ from $\mathcal{C}$, because  $N\le \kappa_{\mathrm{DM}} \tilde c_q^2d_*\Log(|J^c|)^{-2\cdot \1_{q>2}}$ and \eqref{eq:stability_Dvoretzky_dimension} (which holds for $I\in \mathcal{C}$) yield that $N\le \kappa_{\mathrm{DM}}d_*\Log(|J^c|)^{-2\cdot \1_{q>2}}$.

\begin{proof}[Proof of Lemma~\ref{lemma:selector_stability}]
We first show the result when $1<q\le 2$ (and then we will handle the $q> 2$).  By Hoeffding's inequality we get 
\begin{equation*} 
    \bP_{\boldsymbol{\delta}}(\Omega_{\mathrm{size}})\geq 1- 2\exp\bigg( -\frac{(\sum_{j\in J^c}\sigma_j^{q'/2})^2}{8\sum_{j\in J^c}\sigma_j^{q'}} \bigg) \geq 1- 2\exp\bigg( -\frac{\sum_{j\in J^c}\sigma_j^{q'/2}}{8\max_{j\in J^c}\sigma_j^{q'/2}} \bigg).
    \end{equation*}
    Since $q\leq2$, one has $\mathrm{diam}^2\bigl(\Sigma_{J^c}^{1/2}B_q^{J^c}\bigr)=\max_{j\in J^c}\sigma_j$, and, by \eqref{eq:estimate_ell_star}, $d_*(\Sigma_{J^c}^{-1/2}B_{q'}^{J^c})\leq \|g\|_{L^{q'}}^2 \big(\frac{\sum_{j\in J^c}\sigma_j^{q'/2}}{\max_{j\in J^c}\sigma_j^{q'/2}}\big)^{2/q'}$. Therefore,  since $q'/2 > 1$, 
    \begin{equation}\label{eq:Hoeffding-delta}
    \bP_\vdelta(\Omega_{\mathrm{size}})\geq 1-  2\exp\Bigl(-\frac{1}{8}\Bigl(\frac{N}{\kappa \|g\|_{L^{q'}}^2}\Bigr)^{q'/2}\Bigr) \ge  1-2\exp\biggl(-\frac{N}{8\kappa\|g\|_{L^{q'}}^2}\biggr)= 1 - 2\exp(-\eta_q N)    .
    \end{equation}

It follows from another application of \eqref{eq:estimate_ell_star} that, on $\Omega_{\mathrm{size}}$, we have
\begin{align}
\ell_\ast\bigl(\Sigma_{J^c\setminus\mathcal I(\boldsymbol{\delta})}^{1/2}B_q^{J^c\setminus\mathcal I(\boldsymbol{\delta})}\bigr)\geq \bE|g|\bigg( \sum_{j\in J^c\setminus\cI(\vdelta)}\sigma_j^{q'/2} \bigg)^{\frac{1}{q'}} \geq 4^{-\frac{1}{q'}}\bE|g|\bigg( \sum_{j\in J^c}\sigma_j^{q'/2} \bigg)^{\frac{1}{q'}} \geq  \tilde c_q \ell_*.
\label{eq:remaining_gaussian_width}
\end{align}
On the other hand, $\mathrm{diam}\bigl(\Sigma_{J^c\setminus\mathcal I(\boldsymbol{\delta})}^{1/2}B_q^{J^c\setminus\mathcal I(\boldsymbol{\delta})}\bigr)\leq \mathrm{diam}\bigl(\Sigma_{J^c}^{1/2}B_q^{J^c}\bigr).$
Combining this inequality with \eqref{eq:remaining_gaussian_width}, we obtain, on $\Omega_{\mathrm{size}}$,
\begin{align}\label{eq:stability_Dvoretzky_dimension}
d_*\bigl(\Sigma_{J^c\setminus\mathcal I(\boldsymbol{\delta})}^{-1/2}B_{q'}^{J^c\setminus\mathcal I(\boldsymbol{\delta})}\bigr) 
= \bigg(\frac{\ell_\ast\bigl(\Sigma_{J^c\setminus\mathcal I(\boldsymbol{\delta})}^{1/2}B_q^{J^c\setminus\mathcal I(\boldsymbol{\delta})}\bigr)}{\mathrm{diam}\bigl(\Sigma_{J^c\setminus\mathcal I(\boldsymbol{\delta})}^{1/2}B_q^{J^c\setminus\mathcal I(\boldsymbol{\delta})}\bigr)}\bigg)^2
\geq d_*\tilde c_q^{2}
\geq \frac{N\tilde c_q^2}{ \kappa} .
\end{align}
Thus, since we assumed that Assumption~\ref{assum:LBDM} holds for every subset $I=J^c\setminus \cI(\delta)$ with $\delta\in D$, that $\kappa\le \tilde c_q$, and $\eta_q\le 1$, we have for all $\delta\in D$ 
\begin{align*}
&\bP_{\bX_{J^c}}\bigg(\forall\boldsymbol{\lambda}\in\mathbb R^N,\ \big\|\bX_{J^c\setminus\mathcal I(\boldsymbol{\delta})}^{\top}\boldsymbol{\lambda}\big\|_{q'}\geq c\,\ell_\ast\bigl(\Sigma_{J^c\setminus\mathcal I(\boldsymbol{\delta})}^{1/2}B_q^{J^c\setminus\mathcal I(\boldsymbol{\delta})}\bigr)\|\boldsymbol{\lambda}\|_2\bigg)\geq 1- C\exp(-\eta_qN).
\end{align*}
The result follows from \eqref{eq:remaining_gaussian_width}, \eqref{eq:Hoeffding-delta}, and the union bound .

Next, we turn to the case $q > 2$. In that case, one has $\mathrm{diam}^2(\Sigma_{J^c}^{1/2} B_q^{J^c}) = \Bigl(\sum_{j \in J^c} \sigma_j^{\frac{q'}{2-q'}}\Bigr)^{\frac{2-q'}{q'}}$ and 
\begin{displaymath}
d_\ast(\Sigma_{J^c}^{-1/2}B_{q'}^{J^c}) \le \|g\|_{L^{q'}}^2\frac{\bigl(\sum_{j\in J^c} \sigma_j^{q'/2}\bigr)^{2/q'}}{\bigl(\sum_{j\in J^c} \sigma_j^{\frac{q'}{2-q'}}\bigr)^{\frac{2-q'}{q'}}} \le \|g\|_{L^{q'}}^2 
\frac{\bigl(\sum_{j\in J^c} \sigma_j^{q'/2}\bigr)^{2}}{\sum_{j\in J^c} \sigma_j^{q'}},
\end{displaymath}
where the last estimate follows from H\"older's inequality with exponents $\frac{q'}{2(q'-1)}$ and $\frac{q'}{2-q'}$:
\begin{displaymath}
    \sum_{j\in J^c} \sigma_j^{q'} = \sum_{j\in J^c} \sigma_j^{q'-1}\sigma_j \le \bigl(\sum_{j\in J^c} \sigma_j^{q'/2}\bigr)^{2 - 2/q'}\bigl(\sum_{j\in J^c} \sigma_j^{\frac{q'}{2-q'}}\bigr)^{\frac{2-q'}{q'}} ;
\end{displaymath}
this is exactly \eqref{eq:Holder_R_versus_d}. The rest of the proof is the same as above. 
\end{proof}

We are now in a position to bound the pure noise term $\|\Sigma_{J^c}^{1/2}\cA[\bxi]\|_2^2$ beyond the Gaussian case.

\begin{Proposition}\label{prop:objective_cA_xi}
Grant Assumption~\ref{assum:LBDM} for all $I=J^c\setminus \mathcal I(\delta)$ with $\delta\in D$, and Assumption~\ref{assum:beyond_gaussian}. Suppose that $0< \kappa <\tilde c_q^2$, where $\tilde c_q^2$ is the constant from Lemma~\ref{lemma:selector_stability}.
\begin{enumerate}
    \item When $1<q\le 2$ assume additionally that $N\leq \kappa d_*$. Then there exists a constant $C_q$, depending only on $c$, $C$, $q$, and $M'$, such that for every $0<\varepsilon<1$,
    \begin{align}\label{eq:objective_cA_xi}
    &\bP_{\bX_{J^c},\bxi}\bigg(\|\Sigma_{J^c}^{1/2}\cA[\bxi]\|_2^2\leq \frac{C_q}{\varepsilon}\cdot\frac{N\sigma_\xi^2\Tr(\Sigma_{J^c}^{q'})}{\ell_\ast^{2q'}}\bigg)\geq1-\varepsilon-C_q\exp(-\eta_qN),
    \end{align}
  where $\eta_q$ is another constant from Lemma~\ref{lemma:selector_stability}.
 
    \item Similarly, when $2< q<\infty$, suppose $N\le \kappa\,\Log^{-2}(|J^c|)d_*$. Then there exists a constant $C_q$, depending only on $c$, $C$, $q$, and $M'$, such that for every $0<\varepsilon<1$,
    \begin{align}\label{eq:objective_cA_xi_above_2}
    &\bP_{\bX_{J^c},\bxi}\bigg(\|\Sigma_{J^c}^{1/2}\cA[\bxi]\|_2^2\leq \frac{C_q}{\varepsilon}\cdot\frac{N\sigma_\xi^2\Tr(\Sigma_{J^c}^{q'})}{\ell_\ast^{2q'}}\bigg)\geq1-\varepsilon-4C_q\Log(|J^c|)\exp\bigl( -\eta_q \Log^{-2}(|J^c|) d_* \bigr),
    \end{align}
     where again $\eta_q$ is the constant from Lemma~\ref{lemma:selector_stability}.
\end{enumerate}
\end{Proposition}
Recalling \eqref{eq:def_effective_rank_R_star}, the bound in \eqref{eq:objective_cA_xi} and \eqref{eq:objective_cA_xi_above_2} reads $\|\Sigma_{J^c}^{1/2}\cA[\bxi]\|_2^2\le C_q\eps^{-1}N\sigma_\xi^2/R_*$, to be compared with the bound $C N\sigma_\xi^2/d_*$ coming from \eqref{eq:noise_absorption_part_intro}. Under Assumption~\ref{assumption-2}, the hypotheses of Proposition~\ref{prop:objective_cA_xi} with $\kappa=\kappa_{\mathrm{DM}}\tilde c_q^2$ follow from Theorem~\ref{thm:DM-independent}, provided that $N\le \kappa d_*$, as explained after Lemma~\ref{lemma:selector_stability} (we also use the trivial bound $\max\{2,2q'-2\}\le\max\{4,2q'\}$).

\textbf{Proof sketch.} The key idea of the proof is decoupling. Indeed, $\|\Sigma_{J^c}^{1/2}\cA[\bxi]\|_2^2 = \sum_{j\in J^c}\sigma_j\langle \cA[\bxi],\be_j\rangle^2$ is not a sum of independent random variables, and each term $\sigma_i\langle\cA[\bxi],\be_i\rangle^2$ is in general correlated with the other ones $\sigma_j\langle\cA[\bxi],\be_j\rangle^2$ where $i,j\in J^c$ and $i\neq j$; this makes a direct analysis of this sum challenging. However, we notice that
\begin{align}\label{eq:decoupling_cA}
    \bE_{\boldsymbol{\delta}}\sum_{j\in\mathcal I(\boldsymbol{\delta})}\sigma_j|\langle\cA[\bxi],\be_j\rangle|^2=\frac{1}{2}\|\Sigma_{J^c}^{1/2}\cA[\bxi]\|_2^2,\mbox{ and } \sum_{j\in\mathcal I(\boldsymbol{\delta})}\sigma_j|\langle\cA[\bxi],\be_j\rangle|^2\leq\|\Sigma_{J^c}^{1/2}\cA[\bxi]\|_2^2.
\end{align}That is, by sampling $\mathcal I(\vdelta)$ from $J^c$ through $\vdelta$, $2\sum_{j\in\mathcal I(\boldsymbol{\delta})}\sigma_j|\langle\cA[\bxi],\be_j\rangle|^2$ is an unbiased estimator of $\|\Sigma_{J^c}^{1/2}\cA[\bxi]\|_2^2$. It remains to obtain an upper bound on $\sum_{j\in \cI(\vdelta)}\sigma_j\langle\cA[\bxi],\be_j\rangle^2$. To that end, we rely on the following observation inspired by \cite{kurMinimumNormInterpolation2024}: for every $j\in J^c$, let
\begin{align*}
&\cA_{\setminus\{j\}}[\vmu] = \argmin\big(\|\vnu\|_q:\vnu\in V_{J^c\setminus\{j\}},\ \bX_{J^c\setminus\{j\}}\vnu=\vmu\big)  \in V_{J^c},
\end{align*}
and  
\[
\overline{\vlambda}^{\setminus\{j\}}[\vmu] = \|\cA_{\setminus\{j\}}[\vmu]\|_q^{q-1} {\vlambda}^{\setminus\{j\}}[\vmu] ,
\]
 where ${\vlambda}^{\setminus\{j\}}[\vmu] = \argmax (\langle \vlambda, \vmu \rangle \colon \| \bX_{J^c\setminus\{j\}}^\top\vlambda\|_{q'} \le 1)$, be the dual solution.
 Then $\overline{\vlambda}^{\setminus\{j\}}[\bxi]$ depends only on $\bX_{J^c\setminus\{j\}}$ and $\bxi$, and so is, conditionally on $\boldsymbol{\delta}$, independent of $\bX_{J^c}\be_j$.

\begin{proof}[Proof of Proposition~\ref{prop:objective_cA_xi}] We first carefully present the proof for the case $1 < q \le 2$, and then explain at the end the additional steps needed for the case $q > 2$. Let $C'$ be as in Lemma~\ref{lemma:selector_stability}, let $C_{q,0}=q^{2q'-2}(M')^{2q'-2}c_q^{-2q'}$, where $c_q$ is another constant from Lemma~\ref{lemma:selector_stability}, and set $C_q=4\max\{C_{q,0},C'\}$. Recall that $\vdelta=(\delta_j)_{j\in J^c}$ is the vector of i.i.d.\ selectors introduced above, and set
\begin{align}\label{eq:def_stab}
&\Omega_{\mathrm{stab}}=\bigg\{\forall\boldsymbol{\lambda}\in\mathbb R^N,\ \big\|\bX_{J^c\setminus\mathcal I(\boldsymbol{\delta})}^\top\boldsymbol{\lambda}\big\|_{q'}\geq c_q\ell_\ast \|\boldsymbol{\lambda}\|_2\bigg\},\quad \mbox{ and let}\quad \Omega_{\mathrm{good}}=\bigg\{\bP_{\boldsymbol{\delta}}\big(\Omega_{\mathrm{stab}}^c\big)\leq \frac14\bigg\},
\end{align}
where recall that $\ell_\ast=\ell_\ast(\Sigma_{J^c}^{1/2}B_q^{J^c})$.
By Markov's inequality and Lemma~\ref{lemma:selector_stability}, 
\begin{align}\label{eq:P_good_selector}
&\bP_{\bX_{J^c}}(\Omega_{\mathrm{good}}^c)\leq4\bP_{\bX_{J^c},\boldsymbol{\delta}}(\Omega_{\mathrm{stab}}^c)\leq C_q\exp(-\eta_qN).
\end{align}

For every $j\in J^c$ and $\vmu\in\bR^N$, let
\begin{align*}
&\cA_{\setminus\{j\}}[\vmu]=\argmin\big(\|\vnu\|_q:\vnu\in V_{J^c\setminus\{j\}},\ \bX_{J^c\setminus\{j\}}\vnu=\vmu\big) \in V_{J^c},
\end{align*}
and
\[
    \overline{\vlambda}^{\setminus\{j\}}[\vmu] = \|\cA_{\setminus\{j\}}[\vmu] \|_q^{q-1} \argmax\Bigl(\inr{\vmu, \vlambda}: \norm{\bX_{J^c\setminus \{j\}}^\top \vlambda}_{q^\prime}\leq 1\Bigr).
\]
Item \ref{item:calculation_subdifferential} of Lemma~\ref{lemma:operator_A} implies that
\begin{align}\label{eq:KKT_selector}
&\bX_{J^c\setminus\{j\}}^\top\overline{\vlambda}^{\setminus\{j\}}[\bxi]=|\cA_{\setminus\{j\}}[\bxi]|^{\odot(q-2)}\odot \cA_{\setminus\{j\}}[\bxi].
\end{align}
Since we are going to estimate $\sum_{j\in \cI(\vdelta)}\sigma_j|\langle\cA[\bxi],\be_j\rangle|^2\1_{\Omega_{\mathrm{stab}}}=\sum_{j\in J^c}\delta_j\sigma_j|\langle\cA[\bxi],\be_j\rangle|^2\1_{\Omega_{\mathrm{stab}}}$, we only need to consider those $j\in J^c$ with $\delta_j=1$, i.e., $j\in \cI(\vdelta)$. In this case, we have $J^c\setminus\cI(\vdelta)\subset J^c\setminus\{j\}$, so
\begin{align*}
&\|\bX_{J^c\setminus\{j\}}^\top\boldsymbol\lambda\|_{q'}^{q'}=\sum_{\alpha\in J^c\setminus\{j\}}|\langle\bX_{J^c}\be_\alpha,\boldsymbol\lambda\rangle|^{q'}\geq\sum_{\alpha\in J^c\setminus\cI(\vdelta)}|\langle\bX_{J^c}\be_\alpha,\boldsymbol\lambda\rangle|^{q'}=\|\bX_{J^c\setminus\cI(\vdelta)}^\top\boldsymbol\lambda\|_{q'}^{q'}.
\end{align*}
Therefore, on $\Omega_{\mathrm{stab}}\cap \{\delta_j=1\}$, for every $\boldsymbol\lambda\in\bR^N$, we have $\|\bX_{J^c\setminus\{j\}}^\top\boldsymbol\lambda\|_{q'}\geq c_q\ell_* \|\boldsymbol\lambda\|_2$, so the definition of $\overline{\vlambda}^{\setminus\{j\}}[\vmu]$ gives
\[
	\|\cA_{\setminus\{j\}}[\vmu] \|_q^{q-1}  = \norm{\bX_{J^c\setminus \{j\}}^\top  \overline{\vlambda}^{\setminus\{j\}}[\vmu]}_{q^\prime} \ge  c_q \ell_* \|\overline{\vlambda}^{\setminus\{j\}}[\vmu]\|_2.
\]
On the other hand, duality formula \eqref{eq:duality-statement} implies that on $\Omega_{\mathrm{stab}}\cap \{\delta_j=1\}$, 
\begin{align}\label{eq:lambda_selector_bound}
&\|\cA_{\setminus\{j\}}[\bxi]\|_q=\sup_{\boldsymbol\lambda\neq0}\frac{\langle\bxi,\boldsymbol\lambda\rangle}{\|\bX_{J^c\setminus\{j\}}^\top\boldsymbol\lambda\|_{q'}}\leq c_q^{-1}\frac{\|\bxi\|_2}{\ell_*},
\end{align}
so
\begin{align*}
&\|\overline{\vlambda}^{\setminus\{j\}}[\bxi]\|_2\leq c_q^{-q}\frac{\|\bxi\|_2^{q-1}}{\ell_*^q}.
\end{align*}
Because $q(q'-1)=q'$ and $(q-1)(q'-1)=1$, this implies that on the set $\{\delta_j=1\}$,
\begin{align}\label{eq:leave_one_dual_bound}
&\|\overline{\vlambda}^{\setminus\{j\}}[\bxi]\|_2^{2q'-2}\1_{\Omega_{\mathrm{stab}}}\leq c_q^{-2q'}\frac{\|\bxi\|_2^2}{\ell_*^{2q'}}\1_{\Omega_{\mathrm{stab}}}.
\end{align}
Notice that $\bX_{J^c\setminus\{j\}}\mathrm{Proj}_{J^c\setminus\{j\}}\cA[\bxi]+\langle\cA[\bxi],\be_j\rangle\bX_{J^c}\be_j=\bX_{J^c}\cA[\bxi]=\bxi$. Hence, $\mathrm{Proj}_{J^c\setminus\{j\}}\cA[\bxi]$ is feasible for the optimization problem $\min\big(\|\vnu\|_q:\bX_{J^c\setminus\{j\}}\vnu=\bxi-\langle\cA[\bxi],\be_j\rangle\bX_{J^c}\be_j\big)$, and so
\begin{align*}
&\|\cA_{\setminus\{j\}}[\bxi-\langle\cA[\bxi],\be_j\rangle\bX_{J^c}\be_j]\|_q^q\leq\|\mathrm{Proj}_{J^c\setminus\{j\}}\cA[\bxi]\|_q^q.
\end{align*}
Moreover, since $V_{J^c\setminus\{j\}}\subset V_{J^c}$ and $\bX_{J^c}\cA_{\setminus\{j\}}[\bxi]=\bX_{J^c\setminus\{j\}}\cA_{\setminus\{j\}}[\bxi]=\bxi$, the vector $\cA_{\setminus\{j\}}[\bxi]$ is feasible for $\min(\|\vnu\|_q:\bX_{J^c}\vnu=\bxi)$ and so $\|\cA[\bxi]\|_q^q\leq\|\cA_{\setminus\{j\}}[\bxi]\|_q^q$. Therefore,
\begin{align*}
&|\langle\cA[\bxi],\be_j\rangle|^q=\|\cA[\bxi]\|_q^q-\|\mathrm{Proj}_{J^c\setminus\{j\}}\cA[\bxi]\|_q^q\leq\|\cA_{\setminus\{j\}}[\bxi]\|_q^q-\|\cA_{\setminus\{j\}}[\bxi-\langle\cA[\bxi],\be_j\rangle\bX_{J^c}\be_j]\|_q^q.
\end{align*}
Thus, by item \ref{item:inequality_subdifferential_ineq} of Lemma~\ref{lemma:operator_A}, applied with $\vmu=\bxi$ and $\vh=-\langle\cA[\bxi],\be_j\rangle\bX_{J^c}\be_j$, we obtain
\begin{align*}
|\langle\cA[\bxi],\be_j\rangle|^q
\leq q|\langle\cA[\bxi],\be_j\rangle|\,|\langle\bX_{J^c}\be_j,\overline{\vlambda}^{\setminus\{j\}}[\bxi]\rangle|.
\end{align*}
Consequently,
\begin{align}\label{eq:leave_one_coordinate_bound}
&|\langle\cA[\bxi],\be_j\rangle|^2\leq q^{2q'-2}|\langle\bX_{J^c}\be_j,\overline{\vlambda}^{\setminus\{j\}}[\bxi]\rangle|^{2q'-2}.
\end{align}
Conditionally on $\boldsymbol\delta$ and $\bxi$,   the vector $\overline{\vlambda}^{\setminus\{j\}}[\bxi]\1_{\Omega_{\mathrm{stab}}}$ is independent of $\bX_{J^c}\be_j$  whenever $\delta_j=1$. Hence \eqref{eq:moment_assumption_column} may be applied conditionally with $\vlambda=\overline{\vlambda}^{\setminus\{j\}}[\bxi]\1_{\Omega_{\mathrm{stab}}}$. Using in addition \eqref{eq:leave_one_dual_bound}, the identities $(q-1)(2q'-2)=2$ and $q(2q'-2)=2q'$ again, and $\bE\|\bxi\|_2^2=N\sigma_\xi^2$, inequality \eqref{eq:leave_one_coordinate_bound} gives
\begin{align}\label{eq:selector_first_moment_bound}
 \nonumber   \bE_{\bX_{J^c},\bxi,\boldsymbol\delta}\sum_{j\in\mathcal I(\boldsymbol\delta)}\sigma_j|\langle\cA[\bxi],\, \be_j\rangle|^2\1_{\Omega_{\mathrm{stab}}}
   & = \bE_{\bxi,\boldsymbol\delta} \sum_{j\in\mathcal I(\boldsymbol\delta)} \sigma_j\bE_{\bX_{J^c}}|\langle\cA[\bxi],\,\be_j\rangle|^2\1_{\Omega_{\mathrm{stab}}} 
 \nonumber    \\ &\mathop{\le}^{\eqref{eq:leave_one_coordinate_bound}}
    q^{2q'-2} \bE_{\bxi,\boldsymbol\delta} \sum_{j\in\mathcal I(\boldsymbol\delta)}\sigma_j \bE_{\bX_{J^c}}|\langle\bX_{J^c}\be_j,\, \overline{\vlambda}^{\setminus\{j\}}[\bxi]\1_{\Omega_{\mathrm{stab}}}\rangle|^{2q'-2}
 \nonumber    \\ &   \mathop{\le}^{\eqref{eq:moment_assumption_column}} 
    q^{2q'-2} (M')^{2q'-2} \bE_{\bxi,\boldsymbol\delta} \sum_{j\in\mathcal I(\boldsymbol\delta)}\sigma_j^{q'} \bE_{\bX_{J^c}}\|\overline{\vlambda}^{\setminus\{j\}}[\bxi]\|_2^{2q'-2} \1_{\Omega_{\mathrm{stab}}}
 \nonumber     \\ &   \mathop{\le}^{\eqref{eq:leave_one_dual_bound}} 
    q^{2q'-2} (M')^{2q'-2}c_q^{-2q'} \bE \frac{  \| \bxi\|_2^2\1_{\Omega_{\mathrm{stab}}}}{\ell_*^{2q'}} \sum_{j\in\mathcal I(\boldsymbol\delta)}\sigma_j^{q'} 
 \nonumber    \\ &   \leq
    q^{2q'-2} (M')^{2q'-2}c_q^{-2q'}  \frac{\bE  \| \bxi\|_2^2}{\ell_*^{2q'}} \sum_{j\in J^c}\sigma_j^{q'}
  \nonumber   \\ & =
    C_{q,0}\,\frac{N\sigma_\xi^2\Tr(\Sigma_{J^c}^{q'})}{\ell_*^{2q'}}.
\end{align}

For fixed $\bX_{J^c},\bxi$, using
\begin{align}\tag{\ref{eq:decoupling_cA}}
&\bE_{\boldsymbol{\delta}}\sum_{j\in\mathcal I(\boldsymbol{\delta})}\sigma_j|\langle\cA[\bxi],\be_j\rangle|^2=\frac{1}{2}\|\Sigma_{J^c}^{1/2}\cA[\bxi]\|_2^2,\mbox{ and }\sum_{j\in\mathcal I(\boldsymbol{\delta})}\sigma_j|\langle\cA[\bxi],\be_j\rangle|^2\leq\|\Sigma_{J^c}^{1/2}\cA[\bxi]\|_2^2,
\end{align}
we have on $\Omega_{\mathrm{good}}$,
\begin{align}\label{eq:lower_condi_expec_selectors}
\begin{aligned}
\bE_{\boldsymbol{\delta}}\sum_{j\in\mathcal I(\boldsymbol{\delta})}\sigma_j|\langle\cA[\bxi],\be_j\rangle|^2\1_{\Omega_{\mathrm{stab}}}
&=\bE_{\boldsymbol{\delta}}\sum_{j\in\mathcal I(\boldsymbol{\delta})}\sigma_j|\langle\cA[\bxi],\be_j\rangle|^2
-\bE_{\boldsymbol{\delta}}\sum_{j\in\mathcal I(\boldsymbol{\delta})}\sigma_j|\langle\cA[\bxi],\be_j\rangle|^2\1_{\Omega_{\mathrm{stab}}^c}
\\
&\geq \frac{1}{2}\|\Sigma_{J^c}^{1/2}\cA[\bxi]\|_2^2-\|\Sigma_{J^c}^{1/2}\cA[\bxi]\|_2^2\bP_{\boldsymbol{\delta}}\big(\Omega_{\mathrm{stab}}^c\big)\geq\frac14\|\Sigma_{J^c}^{1/2}\cA[\bxi]\|_2^2.
\end{aligned}
\end{align}
Combining \eqref{eq:selector_first_moment_bound} and \eqref{eq:lower_condi_expec_selectors} we get
\begin{align*}
&\bE_{\bX_{J^c},\bxi}\big[\|\Sigma_{J^c}^{1/2}\cA[\bxi]\|_2^2\1_{\Omega_{\mathrm{good}}}\big]\leq 4C_{q,0}\frac{N\sigma_\xi^2\Tr(\Sigma_{J^c}^{q'})}{\ell_*^{2q'}}.
\end{align*}
Markov's inequality and \eqref{eq:P_good_selector} yield, for every $0<\varepsilon<1$,
\begin{align*}
&\bP_{\bX_{J^c},\bxi}\bigg(\|\Sigma_{J^c}^{1/2}\cA[\bxi]\|_2^2>\frac{C_q}{\varepsilon}\cdot\frac{N\sigma_\xi^2\Tr(\Sigma_{J^c}^{q'})}{\ell_*^{2q'}}\bigg)\leq\varepsilon+C_q\exp(-\eta_qN),
\end{align*}
which proves \eqref{eq:objective_cA_xi}.

As for the case when $q\geq 2$, the only modification is that, by the second part of Lemma~\ref{lemma:selector_stability},
\begin{align*}
&\bP_{\bX_{J^c}}(\Omega_{\mathrm{good}}^c)\leq4\bP_{\bX_{J^c},\boldsymbol\delta}(\Omega_{\mathrm{stab}}^c)\leq4C_1\Log(|J^c|)\exp\big( -c_1\Log^{-2}(|J^c|) d_* \big).
\end{align*}
All the remaining steps, from \eqref{eq:lambda_selector_bound} to \eqref{eq:lower_condi_expec_selectors}, are identical -- note that \eqref{eq:moment_assumption_column} holds for every $1<q<\infty$ -- and yield \eqref{eq:objective_cA_xi_above_2}.
\end{proof}

Gathering Propositions~\ref{prop:Gaussian_shift_reduc_new} and \ref{prop:objective_cA_xi} we obtain the main result of this section, which is the statement that will be plugged into the FSD inequality \eqref{eq:aim-triangle} in Section~\ref{sec:proof-thm-3-1}.

\begin{Corollary}\label{cor:noise_absorption}
    Let $1<q<\infty$, $0<\kappa< \tilde c_q^2$, where $\tilde c_q$ is the constant from Lemma~\ref{lemma:selector_stability}. Grant Assumption~\ref{assum:LBDM} for $I=J^c$ and for every $I=J^c\setminus\cI(\delta)$ with $\delta\in D$, and Assumption \ref{assum:beyond_gaussian}. Moreover, assume that
    \[
    N\le \kappa d_*\Log(|J^c|)^{-2\cdot\1_{\{q>2\}}} .
    \]
    Then there exist  constants $C_q\ge1$ and $\eta_q$, depending only on $q$, $M'$,$c$, $C$, and $C_{\mathrm{noise}}$ from the definition \eqref{eq:def-omega-noise} of $\Omega_{\mathrm{noise}}$, such that for every $0<\varepsilon<1$, with probability at least
    \[
    1-\varepsilon - C_q\begin{cases}
        \exp(-\eta_q N) & \text{if } 1<q\le 2,\\
        \Log(|J^c|)\exp\bigl(-\eta_q\Log^{-2}(|J^c|)d_*\bigr) & \text{if } q> 2,
    \end{cases}
    \]
    we have, on the event $\Omega_{\mathrm{esti}}\cap\Omega_{\mathrm{noise}}$,
    \begin{equation}
        \label{eq:noise_absorption_final}
        \norm{\Sigma_{J^c}^{1/2}\hat\bbeta_{J^c}}_2 \le C_q\biggl( \sigma_\xi\sqrt{\frac{N}{\varepsilon R_*}} + \Theta_q \biggr),
    \end{equation}
    where $R_*$ is the effective rank from \eqref{eq:def_effective_rank_R_star}, $r_*$ is the radius from \eqref{eq:upper_vzeta} and $\Theta_q=(\sigma_\xi+r_*)^{1-\theta_q}r_*^{\theta_q}$ is the quantity defined in \eqref{eq:def_Theta_q}.
\end{Corollary}

\begin{proof}
    Write $\|\Sigma_{J^c}^{1/2}\hat\bbeta_{J^c}\|_2 = \|\Sigma_{J^c}^{1/2}\cA[\bxi+\vzeta]\|_2\le \|\Sigma_{J^c}^{1/2}\cA[\bxi]\|_2 + \|\Sigma_{J^c}^{1/2}(\cA[\bxi+\vzeta]-\cA[\bxi])\|_2$. The first term is bounded by Proposition~\ref{prop:objective_cA_xi} together with the identity $\ell_*^{2q'}/\Tr(\Sigma_{J^c}^{q'})=R_*$ from \eqref{eq:def_effective_rank_R_star}. For the second one, note that $\frac{2(2-q)}{3-q}=2(1-\theta_q)$ and $\frac{2}{3-q}=2\theta_q$ when $1<q<2$, and that $\frac{2(q-2)}{q-1}=2(1-\theta_q)$ and $2q'-2=\frac{2}{q-1}=2\theta_q$ when $q\ge2$; hence Proposition~\ref{prop:Gaussian_shift_reduc_new} reads $\|\Sigma_{J^c}^{1/2}(\cA[\bxi+\vzeta]-\cA[\bxi])\|_2\le \sqrt{C}\,\Theta_q$. It remains to take a union bound over the events of Propositions~\ref{prop:Gaussian_shift_reduc_new} and \ref{prop:objective_cA_xi}. The former requires $\Omega_{\mathrm{LBDM}}(J^c)$, which by Assumption~\ref{assum:LBDM} fails with probability at most $C\exp(-c\,d_*)$ when $1<q\le2$ and at most $C\Log(|J^c|)\exp(-c\,d_*\Log^{-2}(|J^c|))$ when $q>2$; since $N\le\kappa d_*\le d_*$, the former quantity is bounded by $C\exp(-cN)$ and is therefore absorbed in the stated deviation term after adjusting $C_q$ and $\eta_q$.
\end{proof}

Two features of \eqref{eq:noise_absorption_final} deserve to be emphasized. First, the price for overfitting the noise is now $\sigma_\xi\sqrt{N/R_*}$ instead of $\sigma_\xi\sqrt{N/d_*}$, and by \eqref{eq:Holder_R_versus_d} this is always an improvement. Second, the extra term $\Theta_q$ is a \emph{geometric mean} between the noise level and the rate $r_*$ achieved on the estimation part: it is at most $C(\sigma_\xi+r_*)$, and it tends to zero as soon as $r_*\to0$ and $\sigma_\xi=O(1)$. In particular $\Theta_q$ does not deteriorate the set of covariance profiles for which benign overfitting occurs.


\subsection{Proof of Theorem~\ref{thm:weak-moment-optimal}}\label{sec:proof-thm-3-1}

The proof below only gathers statements established in Section~\ref{sec:proofs_other_cases} and Corollary~\ref{cor:noise_absorption}; no new argument is needed at this point.

\begin{proof}[Proof of Theorem~\ref{thm:weak-moment-optimal}]
Grant $(A_1)$--$(A_3)$ and place ourselves on the event $\Omega_0=\Omega_{\mathrm{noise}}\cap\Omega_{\mathrm{RIP}}\cap\Omega_{\mathrm{DM}}\cap\Omega_{\mathrm{multi}}$, where $\Omega_{\mathrm{DM}}$ is now the event \eqref{eq:def-omega-DM-independent}, and the other events are defined in \eqref{eq:def-omega-noise}, \eqref{eq:def-omega-RIP}, and \eqref{eq:def-omega-multi}. By Propositions~\ref{prop:probab-omega-noise-independent}--\ref{prop:probab-omega-multi-independent} and the union bound performed after Theorem~\ref{thm:independent}, $\bP(\Omega_0)$ is at least the quantity \eqref{eq:probability-independent}.

\emph{Step 1: the estimation part.} Proposition~\ref{prop:hatbeta_in_set} is purely deterministic and, as explained in Subsection~\ref{section:overview-cases}, on $\Omega_0$ it yields $\hat\bbeta_J\in\bbeta_J^*+r_L\Sigma_J^{-1/2}B_2^J$, that is,
\begin{equation}
    \label{eq:step1_optimal}
    \| \Sigma^{1/2}_J(\hat\bbeta_J-\bbeta_J^\ast) \|_2 \le r_L .
\end{equation}

\emph{Step 2: the event $\Omega_{\mathrm{esti}}$.} Recall that $\vzeta=\bX_J(\bbeta_J^*-\hat\bbeta_J)+\bX_{J^c}\bbeta_{J^c}^*$. On $\Omega_{\mathrm{RIP}}$, \eqref{eq:step1_optimal} gives $\|\bX_J(\hat\bbeta_J-\bbeta_J^*)\|_2\le C_{\mathrm{RIP}}\sqrt N r_L$ as in \eqref{eq:X_J(beta-beta*)-estimate}, while on $\Omega_{\mathrm{noise}}$ we have $\|\bX_{J^c}\bbeta_{J^c}^*\|_2\le C_{\mathrm{noise}}\sqrt N\|\Sigma_{J^c}^{1/2}\bbeta_{J^c}^*\|_2$. Hence, by the triangle inequality, $\Omega_0$ is contained in the event $\Omega_{\mathrm{esti}}$ of \eqref{eq:upper_vzeta} associated with the value of $r_*$ given in \eqref{eq:def_Theta_q}, provided the constant $C$ there is at least $C_{\mathrm{RIP}}+C_{\mathrm{noise}}$; this is \eqref{eq:value_of_r_star}.

\emph{Step 3: the noise absorption part.} By $(A_2)$ and the discussion in Subsection~\ref{subsec:assumptions_optimal}, Assumptions~\ref{assum:LBDM} and \ref{assum:beyond_gaussian} hold for any set $I$ satisfying $d_*(\Sigma_I^{-1/2}B_{q'}^p)\ge C^{-1} d_*(\Sigma_{J^c}^{-1/2}B_{q'}^p)$, and the  latter inequality holds for all $I=J^c\setminus \mathcal{I}(\vdelta)$, $\vdelta\in D$, because of \eqref{eq:stability_Dvoretzky_dimension}. Moreover, $(A_3)$ is the sample size condition required in Corollary~\ref{cor:noise_absorption}. Therefore, for every $\tau\in(0,1)$, with probability at least $1-\tau$ minus the deviation term appearing in Corollary~\ref{cor:noise_absorption}, on $\Omega_{\mathrm{esti}}\cap\Omega_{\mathrm{noise}}$ we have
\begin{equation}
    \label{eq:step3_optimal}
    \norm{\Sigma_{J^c}^{1/2}\hat\bbeta_{J^c}}_2 \le C_q\biggl( \sigma_\xi\sqrt{\frac{N}{\tau R_*}} + \Theta_q \biggr).
\end{equation}

\emph{Step 4: conclusion.} Plugging \eqref{eq:step1_optimal} and \eqref{eq:step3_optimal} into the FSD inequality \eqref{eq:aim-triangle} gives \eqref{eq:aim-weak-moment-optimal}. It remains to collect the probabilities. We have already justified that the event $\Omega_0$ fails with probability at most \eqref{eq:probability-independent}, and the event of Corollary~\ref{cor:noise_absorption} fails with probability at most $\tau+C_qe^{-\eta_qN}$ when $1<q\le2$ and at most $\tau+C_q\Log(|J^c|)\exp(-c_1\Log^{-2}(|J^c|)d_*)$ when $q>2$. Since $N\le\kappa_{\mathrm{DM}}d_*$ by $(A_3)$, the term $2e^{-cd_*}$ in \eqref{eq:probability-independent} is dominated by $e^{-cN}$, and adjusting the constants $c$ and $C$ yields \eqref{eq:probability-weak-moment-optimal}.
\end{proof}


\section{Dvoretzky--Milman type results}\label{sec:proof_DM_q}

The key property of the design matrix restricted to the `noise absorption' part $\bX_{J^c}$ needed for our analysis is a Dvoretzky-Milman $\ell_2\to \ell_{q'}$ isomorphic property as it appears in the event $\Omega_{{\rm DM}}$ from \eqref{eq:def-omega-DM}. Such a result was originally stated for the uniform measure on the sphere or the Gaussian measure \cite{pisierVolumeConvexBodies1989}. Our `\textit{search for universality}' in Statistics requires to extend it beyond these two measures. In this section, we prove Dvoretzky-Milman type $\ell_2\to \ell_{q'}$ isomorphic theorems in three different scenarios: under convex concentration, under independent coordinates satisfying weak moment condition, and under log-concavity. Let us comment that in the literature there are many results concerning Dvoretzky--Milman type theorems beyond the Gaussian case (see, e.g., \cite{MR634905,MR435822, MR2391154, MR2304336}, the list being clearly non-exhaustive). They concern mainly matrices with independent light-tailed entries (with special focus on random signs) or matrices with independent rows. To our best knowledge, the settings we consider, namely matrices with independent columns and matrices with heavy-tailed entries, have not been thoroughly explored before. 

Throughout this section we use a slightly different notation than in previous sections. First of all, one should keep in mind that $q$ plays the role of $q'$ from the previous sections.
We also use $Y$ to denote an isotropic random vector in $\bR^m$, and $Y_1,\dots,Y_N$ to denote i.i.d.\ copies of $Y$.
Let $A$ be an $m\times m$ deterministic matrix and set $X_j = AY_j$. Let $\bX = \bX (A)$ (respectively $\bY$
) be the $N\times m$ matrix with (independent) rows $X_j$ (respectively $Y_j$), and let $X_j(i)$ be the $i$'th coordinate of $X_j$, so $\bX=(X_j(i))_{j\le N, i\le m} = \bY A^\top = \bY A$. By $\|\vx\|_\rho = (\sum_i |x_i|^\rho)^{1/\rho}$ we denote the $\ell_\rho$-norm of a vector $\vx$. For $q\geq 1$, define $\ell_q(A) = \bE\| AG\|_{q}$, where $G$ is a standard Gaussian vector in $\bR^m$ and $d_q(A)=\sup(\|A\vx\|_q:\,\vx\in B_2^m)$. Set $d_*(A) = (\ell_q(A)/d_q(A))^2$.

In the convex concentration setting our Dvoretzky-Milman-type theorem follows along the same lines as its classical version in the Gaussian setting. However, the weak moment case and the log-concave case are not so straightforward and we need to exploit some new ideas and use very recent results to prove Dvoretzky--Milman-type theorems in these two settings.


\subsection{Vectors with dimension free convex concentration property}

Proceeding similarly as in the proof of the Dvoretzky--Milman theorem for standard Gaussian matrices we obtain the following result.

\begin{Proposition}\label{thm:DM}
 Let $1\le q<\infty$ and grant the first part of Assumption \ref{assumption-1} with $\Sigma^{1/2}=A$, $p=m$ and parameter $\rho$. There exist constants $C,c>0$ depending only on $q$ and $\rho$, such that if $N \le c d_\ast(A)$, then with probability at least
$1 - 2\exp(-c d_\ast(A))$, for all $\lambda \in S_2^{N-1}$
\[
 C^{-1} \ell_q(A) \le \|\bX^\top \lambda\|_{q}\le  C \ell_q(A).
\]
\end{Proposition}

\begin{Remark}
    It will be clear from the proof that Proposition~\ref{thm:DM} holds also for non-symmetric matrices $A$, i.e., if $X=AY$, $Y$ is isotropic and satisfies \eqref{eq:assumption-1} (for every 1-Lipschitz function $f$ and $t>0$), and $A$ not necessarily symmetric.
\end{Remark}

 In the proof we shall use the following standard fact.

\begin{Lemma} \label{lem:holder}
	Assume that $Z$ is such a random variable, that for some $0<r_1<r_2$ and $\kappa\ge 1$,
	\[
		\|Z\|_{L^{r_2}}\le \kappa \|Z\|_{L^{r_1}}.
	\]
	Then for every $s>0$ there exists a constant $C$ depending only on $r_1,r_2,s,$ and $\kappa$, such that
	\[
		\|Z\|_{L^{r_2}}\le C \|Z\|_{L^{s}}.
	\]
\end{Lemma}

\begin{proof}
	When $s\ge r_1$, the monotonicity of the $L^\rho$-norms gives the assertion with $C=\kappa$. Assume now that $0<s<r_1$. We may write $r_1$ as a convex combination of $s$ and $r_2$: \[
    r_1 = \theta r_2 + (1-\theta)s
    \]
    with $\theta = \frac{r_1-s}{r_2-s}$. Thus, H{\"o}lder's inequality and the $L^{r_1}/L^{r_2}$-moment equivalence assumption imply that
\[
 \|Z\|_{L^{r_1}}^{r_1}= \E |Z|^{r_1}
 \le (\E |Z|^{r_2})^{\theta}(\E |Z|^{s})^{(1-\theta)}
 \le 
 \kappa^{r_2\theta}\|Z\|_{L^{r_1}}^{\theta r_2} \|Z\|_{L^{s}}^{s(1-\theta)},
\]
so
\[
\|Z\|_{L^{r_1}}^{(1-\theta)s}  \le \kappa^{r_2\theta}\|Z\|_{L^{s}}^{(1-\theta)s},
\]
i.e.,
\[
  \|Z\|_{L^{r_1}} \le \kappa^{\frac{r_2\theta}{(1-\theta)s}}\|Z\|_{L^{s}}
\]
and finally, using one more time the $L^{r_1}/L^{r_2}$-moment equivalence assumption, we obtain
\[
  \|Z\|_{L^{r_2}} \le \kappa^{1+ \frac{r_2\theta}{(1-\theta)s}}\|Z\|_{L^{s}},
\]
which completes the proof.
\end{proof}

\begin{proof}[Proof of Proposition~\ref{thm:DM}] Our argument is a simple modification of Milman's proof of the Dvoretzky theorem. The additional ingredients involve  estimates of $q$-th moments of random variables. Note that for general matrices with independent columns we cannot hope for an almost isometric version of the Dvoretzky theorem, since $q$-th moments of linear combinations of general independent random variables are not invariant under rotations of the vector of coefficients.

Below we will allow the constants $C_i,c_i$ to depend on $q$ and $\rho$. In what follows we will assume that $d_\ast(A) \ge K^2$ for some large constant $K$, depending only on $q$,  to be chosen later. Otherwise we may choose $c$ small enough to make the asserted probability estimate  trivial.

\subsubsection*{Step 1. Bounds in expectation for fixed $\lambda$}

We will first show that for any $\lambda \in S^{N-1}$, $\EE \|\XX^T \lambda\|_q$ is up to constants comparable to $\ell_q(A)$, provided that $K$ is large enough.

Let us thus fix $\lambda \in S_2^{N-1}$. Note that  the function $\ell_2^{Nm} \ni (y_1,\ldots,y_N) \mapsto \|\sum_{j=1}^N A y_j \lambda_j\|_q$ is $d_q(A)$-Lipschitz. Indeed, the map $(y_1 \ldots,y_N)\in \ell_2^{Nm} \mapsto \|\sum_{j=1}^N \lambda_j y_j\|_2$ is 1-Lipschitz and the map $x\in\ell_2^m \mapsto \|A x\|_q$ is $d_q(A)$-Lipschitz. This, together with Assumption~\ref{assumption-1}, and \eqref{eq:bX-to-bY} imply that
\begin{align}\label{eq:q-th-norm-concentration}
  \PP\Bigl(\big|\|\XX^T \lambda\|_q - \EE \|\XX^T \lambda\|_q\big| \ge t\Bigr) \le 2\exp\Bigl(-\frac{t^2}{\rho^2 d_q(A)^2}\Bigr),
\end{align}
which via integration by parts gives
\begin{align}\label{eq:moment-estimate}
  \EE \|\XX^T \lambda\|_q^q \le (\EE \|\XX^T \lambda\|_q)^q + C_1 d_q(A)^q \le (\EE \|\XX^T \lambda\|_q)^q + C_1K^{-q} \ell_q(A)^q.
\end{align}
Let $A_1,\ldots,A_m$ be the rows of $A$. Fix an arbitrary $i \in [m]$. Let $G,G_1,\ldots,G_N$ be i.i.d.\ standard Gaussian vectors in $\RR^m$. We will consider two cases.
\paragraph{Case (i) $q \ge 2$.} We have
\begin{align*}
  \EE|(\XX^\top \lambda )_i|^q &= \EE \Bigl|\sum_{j=1}^N \langle A_i, Y_j\rangle \lambda_j\Bigr|^q
   \ge \Bigl(\EE \Big|\sum_{j=1}^N \langle A_i, Y_j\rangle \lambda_j\Bigr|^2\Big)^{q/2}
  =  \Bigl(\EE \Bigl|\sum_{j=1}^N \langle A_i, G_j\rangle \lambda_j\Bigr|^2\Bigr)^{q/2}
  \\
 & \ge c_1 \EE \Bigl|\sum_{j=1}^N \langle A_i, G_j\rangle \lambda_j\Bigr|^q
 = c_1 \EE | \langle A_i,G\rangle|^q,
\end{align*}
where in the second equality we used independence and isotropy of $Y_j$'s and in the last inequality we used comparison of Gaussian moments. Thus,
\begin{displaymath}
  \EE \|\XX^T \lambda\|_q^q \ge c_1 \EE\|A G\|_q^q \ge  c_1 \ell_q(A)^q,
\end{displaymath}
which, when combined with \eqref{eq:moment-estimate}, gives for $K^q \ge 2C_1/c_1$,
\begin{align}\label{eq:norm-lower-bound}
  \EE \|\XX^T \lambda\|_q \ge c_2\ell_q(A).
\end{align}

Let us now prove a bound from above. Observe that the Lipschitz constant of the function
\begin{displaymath}
\RR^{Nm} \ni (y_1\ldots,y_N)\mapsto \Big|\sum_{j=1}^N \langle A_i,  y_j\rangle \lambda_j\Big|
\end{displaymath}
is at most $\|A_i\|_2$. Again, by concentration, integration by parts, and the isotropy of $Y_j$'s we obtain
\begin{align}\label{eq:moments-bound-from-above}
\notag
  \Bigl\|\sum_{j=1}^N \langle A_i, Y_j\rangle \lambda_i\Bigr\|_{L^q}
 & \le \Bigl\|\sum_{j=1}^N \langle A_i, Y_j\rangle \lambda_j\Bigr\|_{L^2} + C_2 \|A_i\|_2
  = (C_2+1)\|A_i\|_2
  \\ & = (C_2+1)\|\langle A_i,G\rangle\|_{L^2}
  = (C_2+1)\frac{\|g\|_{2}}{\|g\|_q}\|\langle A_i,G\rangle\|_{L^q}
\end{align}
where $g$ is a standard Gaussian variable.
Thus
\begin{align}\label{eq:norm-upper-bound}
   \EE \|\XX^T \lambda \|_q \le (\EE \|\XX^T \lambda\|_q^q)^{1/q} \le C_3 (\EE\|A G\|_q^q)^{1/q} \le  C_4 \ell_q(A),
\end{align}
where in the last inequality we again used moment comparisons for norms of Gaussian variables.

\paragraph{Case (ii) $q \in [1,2)$.} The inequality \eqref{eq:moments-bound-from-above} gives
\begin{displaymath}
  \Bigl\|\sum_{j=1}^N \langle A_i, Y_j\rangle \lambda_i\Bigr\|_{L^4}   \le (C_2+1)\Bigl\|\sum_{j=1}^N \langle A_i, Y_j\rangle \lambda_j\Bigr\|_{L^2}.
\end{displaymath}
This, together with  Lemma~\ref{lem:holder}, implies that for some constant $C_5$,
\begin{displaymath}
  C_5 \Big\|\sum_{j=1}^N \langle A_i, Y_j\rangle \lambda_j\Big\|_{L^q} \ge \Big\|\sum_{j=1}^N \langle A_i, Y_j\rangle \lambda_j\Big\|_{L^2}
  = \Big\|\sum_{j=1}^N \langle A_i, G_j\rangle \lambda_j\Big\|_{L^2} \ge  \|\langle A_i, G\rangle \|_{L^q}.
\end{displaymath}
We can thus proceed as in the case $q \ge 2$ and obtain the lower bound \eqref{eq:norm-lower-bound} (adjusting the constant $c_2$).

As for \eqref{eq:norm-upper-bound}, observe that for $q\le 2$, the inequality in \eqref{eq:moments-bound-from-above} is valid by H\"older's inequality (even without the term $C_2\|A_i\|_2$). The equalities in \eqref{eq:moments-bound-from-above} also hold as well as the estimates in \eqref{eq:norm-upper-bound}.

\subsubsection*{Step 2. Discretization and union bound}

Consider $\varepsilon \in (0,1)$ to be specified later on (depending only on $q,\rho$) and  let $\mathcal{N}$ be an $\varepsilon$-net in $S_2^{N-1}$ of cardinality not exceeding $(3/\varepsilon)^N$ (such a net exists by standard volumetric arguments).

By \eqref{eq:q-th-norm-concentration}, the estimates \eqref{eq:norm-lower-bound}, \eqref{eq:norm-upper-bound}, and the union bound we get that with probability at least
\begin{align}\label{eq:union-bound}
1 - 2\Bigl(\frac{3}{\varepsilon}\Bigr)^N\exp\Bigl(- \frac{ c_2^2 \ell_q(A)^2}{4d_q(A)^2 \rho^2}\Bigr),
\end{align}
for all $\lambda \in \mathcal{N}$,
\begin{align}\label{eq:net}
  \frac{c_2}{2}\ell_q(A) \le \|\XX^T \lambda\|_q \le \Bigl(C_4 + \frac{c_2}{2}\Bigr)\ell_q(A)\le 2C_4 \ell_q(A).
\end{align}

If $\lambda_0 \in S_2^{N-1}$, then for some $\tilde\lambda \in \mathcal{N}$, $\|\lambda_0-\tilde\lambda\|_2\le\eps$, so
\begin{align}\label{eq:two-sided}
 \|\XX^T \tilde\lambda\|_q -\eps \sup_{\lambda\in B_2^N}\|\XX^T \lambda\|_q \le \|\XX^T \lambda_0\|_q \le \|\XX^T \tilde\lambda\|_q +\eps \sup_{\lambda\in B_2^N}\|\XX^T \lambda\|_q.
\end{align}
Let us now choose $\varepsilon = \min( 1/2, c_2/(16C_4))$.

The second inequality in \eqref{eq:two-sided} shows in particular (after taking the supremum over $\lambda_0$ and $\tilde{\lambda}$), that on the event where \eqref{eq:net} holds
\begin{displaymath}
  \sup_{\lambda\in B_2^N}\|\XX^T \lambda\|_q \le \frac{1}{1-\varepsilon} \sup_{\lambda\in \mathcal{N}}\|\XX^T \lambda\|_q \le 4C_4 \ell_q(A).
\end{displaymath}
As a consequence, by the lower bounds in \eqref{eq:net} and \eqref{eq:two-sided}, for any $\lambda_0 \in S^{N-1}$,
\begin{displaymath}
 \|\XX^T \lambda_0\|_q \ge \Bigl(\frac{c_2}{2}- 4C_4 \varepsilon \Bigr)\ell_q(A) \ge \frac{c_2}{4}\ell_q(A).
\end{displaymath}
Since for sufficiently small $c_3>0$, if $N \le c_3 d_\ast(A)$, then the expression in \eqref{eq:union-bound} is greater than $1 - 2\exp(-c_3 d_\ast(A))$, this ends the proof of the proposition.
\end{proof}


\subsection{Weak moment case}\label{subsec:DM-indep}
Let us now assume that $A = \diag(\sigma_1,\ldots,\sigma_m)$, so $X_j(i)=\sigma_i Y_{j}(i)$.
In this case
\begin{equation} \label{eq:c0'}
\ell_q(A)\le (\EE \sum_{i=1}^m |\sigma_i|^q |g_i|^q)^{1/q} = \gamma_q \|\sigma\|_q,
\end{equation}
where $\gamma_q=(\EE |g_1|^q)^{1/q}\sim \sqrt{q}$, and, by Jensen's inequality
\begin{equation} \label{eq:c0}
\ell_q(A)= \EE\|(|\sigma_i g_i|)_{i\le m} \|_q
\ge \|(\EE|\sigma_i g_i|)_{i\le m} \|_q
=c_0 \|\sigma\|_q,
\end{equation}
where $c_0=\EE|g_1|=\sqrt{2/\pi}$.

Moreover,
\begin{equation}	\label{eq:parameters-diagonal}
	d_q(A) =
	\begin{cases}
	 \|\sigma\|_\infty & \text{if }q\ge 2 \\
	  \|\sigma\|_{\frac{2q}{2-q}} & \text{if }q< 2
	 \end{cases}
	 \qquad \text{and so} \quad
	 d_\ast(A)\mathop{\sim}_q
	 \begin{cases}
	 \Bigl( \frac{\|\sigma\|_q}{ \|\sigma\|_\infty}  \Bigr)^2  & \text{if }q\ge 2\\
	 \Bigl( \frac{\|\sigma\|_q}{ \|\sigma\|_{\frac{2q}{2-q}}}  \Bigr)^2 & \text{if }q< 2 .
	 \end{cases}
\end{equation}

\begin{Theorem}\label{thm:DM-independent}
    Suppose $Y$ is 
    isotropic and has independent entries. 
Suppose that $A=\diag(\sigma_1,\dots,\sigma_m)$, and the entries of $Y$ satisfy $\EE |Y_{j}(i)|^{\max\{4,2q\}} \le M^{\max\{4, 2q\}}$.
Then there exist  constants $c,C>0$, depending only on $M$ and $q$, such that the following holds.
\begin{enumerate}
    \item If $q\ge 2$ and $N\le cd_*(A)$, then  with probability at least $1-\exp(-c d_*(A))$, $\|\bX^\top\vlambda\|_q\geq c\ell_q(A)$ for every $\vlambda\in S_2^{N-1}$.

    \item If $q\in [1,2)$ and $N\le \frac{c  d_\ast(A)} {\Log^2(m)}$, then with probability at least
    \[
    1-C\Log m\, e^{- cd_\ast(A)/\Log^{2}(m)},
    \]
      $\|\bX^\top\vlambda\|_q\geq c\ell_q(A)$  for every $\vlambda\in S_2^{N-1}$.

    \item If $q>2$ and $N\le d_\ast(A)$, then
	\begin{equation*}
		\sup_{\lambda \in \SSS^{N-1}} \|\XX^\top \lambda\|_q \le C\ell_q(A)
	\end{equation*}
	with probability at least $1-2e^{-cd_\ast(A)} - Cd_\ast(A)^{-(q-2)/4}$.
	
	\item
	 If $q\le2$, $N\le d_\ast(A)$,
     and we additionally assume that $\|Y_j(i)\|_{2q+\eps}\le M$ for some $\eps>0$, then there exists a universal constant ${\eps_0}$ such that
	\begin{equation}\label{eq:upper-indep-norm2}
		\sup_{\lambda \in \SSS^{N-1}} \|\XX^\top \lambda\|_q \le C\Log m\cdot \ell_q(A)
	\end{equation}
	with probability at least $1-2e^{-cd_\ast(A)} - Cd_\ast(A)^{ -c\min\{\eps,\eps_0\}}$.
	Moreover, if $q=2$, then one may omit the logarithmic factor in \eqref{eq:upper-indep-norm2}.
    \end{enumerate}
\end{Theorem}

The rest of this subsection is devoted to the proof of Theorem~\ref{thm:DM-independent}.
In this subsection constants $c, c_i$ and $C, C_i$ depend only on $q$ and $M$.

\subsubsection{Upper bounds}

Let us first provide upper estimates for the expectation of the operator norms of $\XX^\top$ from $\ell_2^N$ to $\ell_q^m$. We do it separately for the cases $q\ge 2$ and $q<2$.

\begin{Proposition}\label{prop:indep-upper-expectation-above2}
    Suppose that the assumptions of Theorem~\ref{thm:DM-independent} hold. 
    If $q\ge 2$, then there exists a constant $C>0$ such that for every $N\le d_*{(A)}$,
	\[
		\EE \| \XX^\top \colon \ell_2^N \to \ell_q^m \| \le C \ell_q(A).
	\]
\end{Proposition}

\begin{proof}
	Since $q\ge 2$ and $A=\diag(\sigma_1,\ldots,\sigma_m)$, $d_q(A) = \|\sigma\|_\infty$.
	Let $Y_j(i)$, for $i\le m$ and $j\le N$, be the $i$-th coordinate of $Y_j$ (the $j$-th copy of $Y$) and let $X_{j}(i)=\sigma_{i} Y_{j}(i)$, so $\XX^\top = (\sigma_{i}Y_j(i))_{i\le m, j\le N}$. By our assumptions
	\[
	\max\bigl\{(\EE |Y_{j}(i)|^{4})^{1/4}, (\EE |Y_{j}(i)|^{2q})^{1/2q}\bigr\} = (\EE |Y_{j}(i)|^{2q})^{1/(2q)} \le M.
	\]
	Thus, if $N\le d_* \lesssim_q \bigl(\frac{\|\sigma\|_q}{\|\sigma\|_\infty} \bigr)^2 $, then \cite[Corollary~14]{latalaOperatorellptoellqNorms2025} and \eqref{eq:c0} yield
	\begin{align*}
		\EE \| \XX^\top \colon \ell_2^N \to \ell_q^m \|
		& \le C_1 \bigl( \|\sigma\|_\infty N^{1/2} + M \|\sigma\|_q + M \|\sigma\|_{2q}N^{1/4} + M  \|\sigma\|_{2q}N^{1/(2q)} \bigr)
		\\
		& \le C_2 \ell_q(A) + 2C_1M   \|\sigma\|_{2q} N^{1/4}.
	\end{align*}
	Moreover, $\|\sigma\|_{2q}^{2q} \le \|\sigma\|_\infty^q \|\sigma\|_q^q$, so the AM-GM inequality and \eqref{eq:c0} imply that
\[
	2\|\sigma\|_{2q} N^{1/4}\le \|\sigma\|_q+ \|\sigma\|_\infty N^{1/2}\le 2c_0^{-1} \ell_q(A),
\]
 and the assertion follows.
\end{proof}

\begin{Proposition}	\label{prop:indep-upper-expectation-below2}
Suppose that the assumptions of Theorem~\ref{thm:DM-independent} hold. If $q< 2$, then there exists a constant $C>0$ such that for every $N\le d_\ast(A)$,
	\[
		\EE \| \XX^\top \colon \ell_2^N \to \ell_q^m \| \le C\bigl(\Log N + (2-q)\Log m\bigr) \ell_q(A).
	\]
\end{Proposition}

\begin{proof}
	Assume without the loss of generality that $\|\sigma\|_\infty =1$.
	Let $k_0$ be an integer to be fixed later,
	\[
		I_k=\{i\le m\colon 2^{-k-1}<|\sigma_i|\le 2^{-k} \} \quad \text{for } k=0,1,\dots,k_0-1,
	\]
	and
	\[
		I_{k_0}=\{i\le m\colon |\sigma_i|\le 2^{-k_0} \}.
	\]
	Define $A_k=\diag\bigl((\sigma_i\1_{\{i\in I_k\}}\bigr)_{i\le m})$ and i.i.d.\ copies $X_1^{(k)}, \dots X_N^{(k)}$ of $X^{(k)}=A_kY$, and let $\XX_k$ be the $N\times m$ matrix with rows $X_j^{(k)}$. Then $\XX^\top = \sum_{k=0}^{k_0} \XX_k^\top$, and the triangle inequality implies that
	\begin{equation} \label{eq:indep-upper-1}
		\EE \| \XX^\top \colon \ell_2^N \to \ell_q^m \| \le  (k_0+1) \max_{0\le k\le k_0} \EE \| \XX_k^\top \colon \ell_2^N \to \ell_q^m \|.
	\end{equation}

	Fix $0\le k\le k_0$ such that $I_k \neq \emptyset$. 
	Let $(2/q)^*=2/(2-q)$ denote the H{\"o}lder's conjugate to $2/q\in(1,2]$.
	Then, similarly as in the proof of Proposition~\ref{prop:indep-upper-expectation-above2}, \cite[Corollary~14]{latalaOperatorellptoellqNorms2025} yields that whenever $N^{1/2}\lesssim_{q} \frac{\|\sigma\|_q}{\|\sigma\|_{\frac{2q}{2-q}}}$, then
	\begin{align}
		\notag
		\EE \| \XX_k^\top \colon \ell_2^N \to \ell_q^m \|
		&\le
		\bigl\| \diag\bigl((|\sigma_i|^{1-q/2}\sgn(\sigma_i))_{i\in I_k}\bigr)  \colon \ell_2^{I_k} \to \ell_q^{I_k} \bigr\|\cdot \EE \bigl\|(|\sigma_i|^{q/2}X_j(i))_{\substack{i\in I_k\\ j\le N}} \colon \ell_2^N \to \ell_2^{I_k}\bigr\|
		\\  &\le
		C_3 \|(|\sigma_i|^{ (1-q/2)})_{i\in I_k}\|_{2q/(2-q)} \Bigl( \sup_{i\in I_k}|\sigma_i|^{q/2}N^{1/2} + \|(|\sigma|^{q/2})_{i\in I_k}\|_2 \Bigr)
		\label{eq:indep-upper-2a}
		\\ & \le  C_32^{-k}|I_k|^{\frac{2-q}{2q}} \frac{\|\sigma\|_q}{\|\sigma\|_{\frac{2q}{2-q}}} + C_3 \|\sigma\|_q .
		\label{eq:indep-upper-2}
	\end{align}
	
	If $k\le k_0-1$, then
	\begin{equation} \label{eq:indep-upper-later}
		\|\sigma\|_{\frac{2q}{2-q}}\ge \Bigl( \sum_{i\in I_k}|\sigma_i|^{\frac{2q}{2-q}}\Bigr)^{\frac{2-q}{2q}} \ge 2^{-(k+1)} |I_k|^{\frac{2-q}{2q}},
	\end{equation}
	so it follows by \eqref{eq:indep-upper-2} and \eqref{eq:c0} that
	\begin{equation}		\label{eq:indep-upper-3}
		\EE \| \XX_k^\top \colon \ell_2^N \to \ell_q^m \| \le 3C_3  \|\sigma\|_q \le C_4 \ell_q(A),
	\end{equation}
	provided that $N\le d_\ast(A)$.
	
	If $k=k_0 \coloneqq \lceil \log_2(\sqrt{N}m^{\frac 1q-\frac 12}) \rceil$, then \eqref{eq:indep-upper-2a} implies for every $N$,
	\begin{align}
		\EE \| \XX_{k_0}^\top \colon \ell_2^N \to \ell_q^m \|
		\le
		  C_32^{-k_0}m^{\frac{2-q}{2q}} N^{1/2}+ C_3 \|\sigma\|_q  \le C_3(\|\sigma\|_\infty+ \|\sigma\|_q) \le C_4 \ell_q(A).
		\label{eq:indep-upper-4}
	\end{align}
	Inequalities \eqref{eq:indep-upper-1},  \eqref{eq:indep-upper-3}, and  \eqref{eq:indep-upper-4} yield the assertion.
\end{proof}

Now we use Propositions~\ref{prop:indep-upper-expectation-above2} and \ref{prop:indep-upper-expectation-below2}, and some concentration bounds guaranteed by the assumption on moments of order $\max\{4,2q+\eps\}$, to obtain the following proposition, which provides items 3. and 4. of Theorem~\ref{thm:DM-independent}.

\begin{Proposition}\label{prop:7.6}
Suppose that the assumptions of Theorem~\ref{thm:DM-independent} hold.

	If $q>2$ and $N\le d_\ast(A)$, then there exist  constants $c$ and $C$, depending only on $q$ and $M$, such that
	\begin{equation*}
		\sup_{\lambda \in \SSS^{N-1}} \|\XX^\top \lambda\|_q \le C\ell_q(A)
	\end{equation*}
	with probability at least $1-2e^{-cd_\ast(A)} - Cd_\ast(A)^{-(q-2)/4}$.
	
	If $q\le2$, $N\le d_\ast(A)$ and we additionally assume that $\|Y_j(i)\|_{2q+\eps}\le M$ for some $\eps>0$, then there exist  constants ${\eps_0}$, $c$ and $C$, depending only on $q$ and $M$, such that
	\begin{equation*}
		\sup_{\lambda \in \SSS^{N-1}} \|\XX^\top \lambda\|_q \le C\Log m\cdot \ell_q(A)
	\end{equation*}
	with probability at least $1-2e^{-cd_\ast(A)} - Cd_\ast(A)^{ -c\min\{\eps,\eps_0\}}$.
	Moreover, if $q=2$, then one may omit the logarithmic term above.
\end{Proposition}

\begin{proof}
	Note  that $\sup_{\lambda \in S_2^{N-1}} \|\XX^\top \lambda\|_q =  \|\XX^\top\colon \ell_2^N \to \ell_q^m \|$.

	Let us first consider the case $q\ge 2$. By \cite[Theorem~2]{adamczakFewRemarksOperator2010} (applied with $p\coloneqq q$, $\|\cdot\|$ being the operator norm from $\ell_2^N$ to $\ell_q^m$, $X_{(i,j)}\coloneqq \sigma_iY_j(i)e_i \otimes e_j$, $\eta=1$, and $\sigma\coloneqq \Sigma$ defined as below) and Propoistion~\ref{prop:indep-upper-expectation-above2} we have for every $t\ge 0$,
	\begin{equation}	\label{eq:concentration-upper-indep-1}
		\PP \bigl( \|\XX^\top\colon \ell_2^N \to \ell_q^m \| \ge C_5\ell_q(A)+t  \bigr) \le e^{-t^2/(3\Sigma^2)}+C_5 \frac{\EE\max_{i,j}|\sigma_i Y_j(i)|^q}{t^q},
	\end{equation}
	where
	\[
		\Sigma^2 = \sup_{s\in B_{q^*}^m} \sup_{t\in B_2^N}  \sum_{i, j} \sigma_i^2 s_i^2t_j^2 = \max_{j\le N} \sup_{x\in B_{q^*/2}^m} \sum_i \sigma_i^2 x_i = \|\sigma\|_\infty^2 = d_q(A)^2.
	\]
	
	Note that for $N\le d_*(A)$,
	\begin{align*}
		\EE\max_{i,j}|\sigma_i Y_j(i)|^q
		& \le \Bigl( \EE \sum_{i,j} |\sigma_i|^{2q} |Y_j(i)|^{2q} \Bigr)^{1/2}
		\le M^q N^{1/2} \bigl(\sum_{i}|\sigma_i|^{2q} \bigr)^{1/2}
		\\ &\le M^q d_\ast(A)^{1/2}  \bigl(\sum_{i}|\sigma_i|^{q} \bigr)^{1/2} \|\sigma\|_\infty^{q/2}
		=M^q d_\ast(A)^{1/2}  \|\sigma\|_{q}^{q/2} \|\sigma\|_\infty^{q/2}.
	\end{align*}
	Thus, by taking $t=\ell_q(A)\sim_{ q} \|\sigma\|_q$ (recall that $\sim$ follows by \eqref{eq:c0'} and \eqref{eq:c0}) we get the assertion in the case $q>2$.
	We will go back to the case $q=2$ at the end of this proof.
	
	In the case $q<2$ we proceed similarly, exploiting again \cite[Theorem~2]{adamczakFewRemarksOperator2010}, Proposition~\ref{prop:indep-upper-expectation-below2} (instead of Proposition~\ref{prop:indep-upper-expectation-above2}), and the assumption $N\le d_\ast(A)$ (which implies that $N\lesssim_q m$) to get that for every $t\ge 0$,
	\[
		\PP \bigl( \|\XX^\top\colon \ell_2^N \to \ell_q^m \| \ge C_5\Log m \cdot \ell_q(A)+t  \bigr) \le e^{-t^2/(3\Sigma^2)}+C_5 \frac{\EE\max_{i,j}|\sigma_i Y_j(i)|^q}{t^q}.
	\]
	Since $(q^*/2)^*=q/(2-q)$,
	\[
		\Sigma^2 = \max_{j\le N} \sup_{x\in B_{q^*/2}^m} \sum_{i=1}^m \sigma_i^2 x_i = \|\sigma\|_{2q/(2-q)}^2= d_q(A)^2.
	\]
	
	If $q<2$, we may assume without loss of generality that $\eps>0$ is small enough to provide that $2q+\eps < \frac{2q}{2-q}$. Then $2q+\eps = \theta \frac{2q}{2-q}+(1-\theta)q$, where $\theta=(q+\eps)(2-q)/q^2 \in (0,1)$.
	Hence, we may use H{\"o}lder's inequality to obtain that
	\begin{align*}
		\EE\max_{i,j}|\sigma_i Y_j(i)|^q
		& \le \Bigl( \EE \sum_{i,j} |\sigma_i|^{2q+\eps} |Y_j(i)|^{2q+\eps} \Bigr)^{q/(2q+\eps)}
		\le M^q N^{q/(2q+\eps)} \bigl(\sum_{i}|\sigma_i|^{2q+\eps} \bigr)^{q/(2q+\eps)}
		\\ &\le M^q d_\ast(A)^{q/(2q+\eps)}   \|\sigma\|_{q}^{q-\frac{2q+2\eps}{2q+\eps}} \|\sigma\|_{2q/(2-q)}^{\frac{2q+2\eps}{2q+\eps}}
		\\ &= M^q d_\ast(A)^{q/(2q+\eps)}  \|\sigma\|_{q}^{q-\frac{2q+2\eps}{2q+\eps}} d_q(A)^{\frac{2q+2\eps}{2q+\eps}} ,
	\end{align*}
	whereas for $q=2$ we have
	\begin{align*}
		\EE\max_{i,j}|\sigma_i Y_j(i)|^q
		& \le  M^2 N^{2/(4+\eps)} \bigl(\sum_{i}|\sigma_i|^{4+\eps} \bigr)^{2/(4+\eps)}
		\le  M^2 d_\ast(A)^{2/(4+\eps)} \|\sigma\|_2^{\frac{4}{4+\eps}} \|\sigma\|_\infty^{\frac{4+2\eps}{4+\eps}}
		\\ & =  M^q d_\ast(A)^{q/(2q+\eps)} \|\sigma\|_q^{q-\frac{2q+2\eps}{2q+\eps}}d_q(A)^{\frac{2q+2\eps}{2q+\eps}}.
	\end{align*}
	We finish by taking $t=\ell_q(A)$ as before.
\end{proof}

\subsubsection{Lower bounds in the weak moment case, part~I: reduction from the identity to an arbitrary diagonal covariance matrix}

The next proposition shows that in order to provide  lower bounds from Theorem~\ref{thm:DM-independent} (up to logarithmic terms) it suffices to consider the case when $A=\diag(\sigma_1,\dots,\sigma_m)$ with $\sigma_i\in\{-1,0,1\}$; this corresponds to the case when $X$ is a projection of an isotropic vector into some coordinate subspace.
Recall that for fixed $\bY$ (with rows $Y_j$), by $\bX (A)$ we denote the $N\times m$ matrix with rows $AY_j$.
\begin{Proposition}\label{prop:reduction-to-isotropic}
	Let $m\ge N$ be integers, $q\ge 1$, $Y$ be an isotropic random vector in $\RR^m$ and $\alpha,\beta,\gamma,\theta>0$.
	Assume that for any nonempty $I\subset [m]$ and any $\eta_i \in \{-1,1\}$, $i=1,\dots, m$, with probability at least $1-2e^{-\alpha d_\ast(A(\eta,I))^\theta}$ for every $\lambda\in S_2^{N-1}$,
	\[
		 \|\XX(A(\eta,I))^\top \lambda\|_q \ge \beta \ell_q(A(\eta,I)),
	\]
	provided that $N\le \gamma d_\ast(A(\eta, I))$, where $A(\eta,I) = \diag\bigl((\eta_i\1_{\{i\in I\}})_{i\le m}\bigr)$.
	Then there exist  constants $c,C>0$, depending only on $\theta$ and $q$, such that for every $\sigma\in \RR^m\setminus \{0\}$  with probability at least $1-2\Log m \,e^{-\alpha cd_\ast(A)^\theta/\Log^{2\theta}(m)}$ for every $\lambda\in S_2^{N-1}$,
	\[
		 \|\XX(A)^\top \lambda\|_q \ge c\beta \ell_q(A),
	\]
	provided that $N\le \frac{c\gamma  d_\ast(A)} {\Log^2(m)}$, where $A = \diag(\sigma_1, \dots, \sigma_m)$.
\end{Proposition}

\begin{proof}
	Assume without the loss of generality that $\|\sigma\|_\infty =1$ if $q\ge 2$ and $\|\sigma\|_{\frac{2q}{2-q}} =1$ if $q<2$. In both cases $\|\sigma\|_\infty \le 1$.
	Let $D=\|\sigma \|_q \sim_q \ell_q(A) =\sqrt{d_\ast(A)}$,
	$k_0=\lceil \log_2(4 m^{1/q}/D) \rceil$ and, as in the proof of Proposition~\ref{prop:indep-upper-expectation-below2}, let
	\[
		I_k=\{i\le m\colon 2^{-k-1}<|\sigma_i|\le 2^{-k} \} \quad \text{for } k=0,1,\dots,k_0-1.
	\]
	and
	\[
		I_{k_0}=\{i\le m\colon |\sigma_i|\le 2^{-k_0} \},
	\]
	Note that  $ D \ge 1$, so $k_0\lesssim_q \Log m$.
	
	We consider two cases.
	\subsubsection*{Case I: $q<2$}
	Let
	\[
		K = \Bigl\{ k\in\{0,1,\dots,k_0-1\} \colon \sum_{i\in I_k}\sigma_i^2 \ge \frac{D^2 2^{-2k}}{64k_0^{2/q}}  \Bigr\}.
	\]
	Note that
	\begin{align*}
		\sum_{k\in [k_0-1]\setminus K} \sum_{i\in I_k}|\sigma_i|^q
		\le \sum_{k\in [k_0-1]\setminus K} \Bigl( \sum_{i\in I_k}\sigma_i^2\Bigr)^{q/2} |I_k|^{\frac{2-q}2}
		\le  \sum_{k=0}^{k_0-1}  \frac{D^q 2^{-kq}}{8^{q}k_0}    |I_k|^{\frac{2-q}2}.
	\end{align*}
	Moreover, if $k\le k_0-1$, then inequality \eqref{eq:indep-upper-later} shows that $|I_k|^{\frac{2-q}2}\le \|\sigma\|_{\frac{2q}{2-q}}^q 2^{q(k+1)}$, so
	\begin{equation} \label{eq:reduct-isotr-1}
		\Bigl(\sum_{k\in [k_0-1]\setminus K} \sum_{i\in I_k}|\sigma_i|^q \Bigr)^{1/q}
		\le  \frac{D }{4} .
	\end{equation}
	
	Since $k_0=\lceil \log_2(4 m^{1/q}/D) \rceil$,
	\[
	  \sum_{i\in I_{k_0}}|\sigma_i|^q \le |I_{k_0}|2^{-k_0q}\le m2^{-k_0q} \le \frac{D^q}{4^q},
	\]
	which together with \eqref{eq:reduct-isotr-1} shows that
	\[
		\|(\sigma_i)_{i\in \bigcup_{k\in K} I_k}\|_q \ge D/2,
	\]
	so it suffices to prove that for every $k\in K$ with probability at least $1-2e^{-c_1\alpha D^{2\theta}/k_0^{2\theta}}$ for every $\lambda \in S_2^{N-1}$,
	\begin{equation} \label{eq:reduct-isotr-2}
		\|\operatorname{Proj}_{I_k}(\XX(A)^\top\lambda)\|_q^q \ge \beta^q\|(\sigma_i)_{i\in I_k}\|_q^q,
	\end{equation}
	provided that $N\le c_1 \gamma D^2/k_0^2$. Since every entry of the vector $(\sigma_i)_{i\in I_k}$ is comparable to its $\ell_\infty$ norm,
	\[
	d_\ast(\diag((\sigma_i)_{i\in I_k})) \sim  \Bigl(\frac{\|(\sigma_i)_{i\in I_k}\|_q}{\|(\sigma_i)_{i\in I_k}\|_{\frac{2q}{2-q}}}\Bigr)^2 \sim \Bigl( \frac{|I_k|^{1/q}}{|I_k|^{\frac{2-q}{2q}}} \Bigr)^2= |I_k|.
	\]
	Thus, the inequality \eqref{eq:reduct-isotr-2} follows by the assumption, with probability at least $1-2e^{- c_2 \alpha |I_k|^\theta}$ and provided that $N\le  c_2 \gamma |I_k|$. Therefore, it suffices to show that $|I_k|\ge c_3D^2/k_0^2$ whenever $k\in K$.
	
	By the definition of $K$ and $I_k$ we have for every $k\in K$,
	\[
		|I_k|2^{-2k}\ge \sum_{i\in I_k} \sigma_i^2\ge \frac{D^2 2^{-2k}}{64k_0^{2/q}},
	\]
	so indeed $|I_k|\ge c_3D^2/k_0^2$.
	
	\subsubsection*{Case II: $q\ge 2$}
	This case is similar to the previous one, so we omit some details.
	Let
	\[
		K = \Bigl\{ k\in\{0,1,\dots,k_0-1\} \colon \sum_{i\in I_k}|\sigma_i|^q \ge \frac{D^q 2^{-kq}}{2\cdot4^q}  \Bigr\}.
	\]
	Then
	\begin{align*}
		\sum_{k\in [k_0-1]\setminus K} \sum_{i\in I_k}|\sigma_i|^q
		\
		\le  \sum_{k=0}^{k_0}  \frac{D^q 2^{-kq}}{2\cdot 4^{q}} \le \frac{D^q}{4^q}
	\end{align*}
	and since $k_0=\lceil \log_2(4 m^{1/q}/D) \rceil$,
	\[
	  \sum_{i\in I_{k_0}}|\sigma_i|^q \le m2^{-k_0q} \le \frac{D^q}{4^q},
	\]
	so as in the Case I it suffices to provide \eqref{eq:reduct-isotr-2} whenever $k\in K$, with appropriate probability and the upper bound for $N$. This again boils down to proving that for any $k\in K$,
	\[
		 c_4D\le \sqrt{d_\ast(\diag((\sigma_i)_{i\in I_k}))} \sim \frac{\|(\sigma_i)_{i\in I_k}\|_q}{\|(\sigma_i)_{i\in I_k}\|_{\infty}} \sim |I_k|^{1/q}.
	\]
	 The inequality $c_5 D\le   |I_k|^{1/q}$ follows by the definitions of $K$ and $I_k$, as before in  Case~I.
\end{proof}

\subsubsection{Lower bounds in the weak moment case, part~II: using selector processes}

Let $\tilde{d}_*=\bigl(\frac{\|\sigma\|_q}{\|\sigma\|_\infty}\bigr)^2$.
By $Z_1,\dots,Z_m$ we denote the rows of $\bY^\top$ (i.e., the matrix with columns $Y_1,\dots,Y_N$), so $\bX^\top\lambda = (\sigma_i\langle Z_i,\lambda\rangle)_i$. Note that the assumption
$\EE |Y_{j}(i)|^{4} \le M^{4}$
of Theorem~\ref{thm:DM-independent},  the independence of random variables $Y_{j}(i)$, and the isotropy of $(Y_j(i))_{i,j}$ yield that for every $\lambda\in S_2^{N-1}$,
\[
\EE\langle Z_i,\lambda\rangle^4
=\EE \Bigl( \sum_j Y_j(i)\lambda_j \Bigr)^4
=\sum_j \lambda_j^4\EE Y_j(i)^4 + \sum_{j\neq j'}\lambda_j^2\lambda_{j'}^2
\le M^4  \bigl(\sum_j \lambda_j^2 \bigr)^2=M^4
\]
and $\EE \langle Z_i,\lambda\rangle^2 = \|\lambda\|_2^2 =1$.
Thus, the Paley-Zygmund inequality yields that for every $\lambda\in S_2^{N-1}$,
\[
\PP\Bigl(|\langle Z_i,\lambda\rangle|\geq \frac 12\Bigr)
=\PP\Bigl(\langle Z_i,\lambda\rangle^2\geq \frac 14 \EE \langle Z_i,\lambda\rangle^2 \Bigr)
\ge (1-1/4)^2 \frac{(\EE \langle Z_i,\lambda\rangle^2)^2}{\EE \langle Z_i,\lambda\rangle^4}
\ge \frac{9}{16 M^4}.
\]
Therefore,
\begin{equation}\label{eq:small-ball-assumption}
\PP\bigl(|\langle Z_i,\lambda\rangle|\geq 1/2\bigr)\geq \kappa \qquad \text{for all } \lambda\in S_2^{N-1} \text{ and }i\le m,
\end{equation}
where $\kappa$ is the constant depending only on $M$.
The constants in this subsubsection depend only on $\kappa$ and $q$, so they depend only on $q$ and $M$.

\begin{Proposition}\label{prop:lower_DM}
    Assume that \eqref{eq:small-ball-assumption} holds and that the $Y_1,\dots, Y_N$ are independent, centered, and have independent entries of variance one.
    Then there exists a constant $c>0$, depending only on $\kappa$, such that when $N\le c\tilde{d}_*^{q/2}$, then with probability at least $1- 2\exp(-c \tilde{d}_*^{q/2})$, for any $\lambda\in S_2^{N-1}$, $\|\bX^\top\lambda\|_q\geq c\|\sigma\|_q$.
\end{Proposition}
Inequality \eqref{eq:small-ball-assumption} is referred to as the small-ball property. The small-ball method is similar to the \(\varepsilon\)-net argument in that aspect that both require a single-scale approximation on an index set, and then extend this approximation to the whole index set via a union bound. The difference from the \(\varepsilon\)-net argument is that the small-ball method does not require an upper bound on the \(\ell_2\to\ell_q\) operator norm of \(\bX^\top\), but instead relies on upper bounds for the corresponding selector process, which in many cases only requires mild moment assumptions.

\begin{proof}[Proof of Proposition~\ref{prop:lower_DM}] We aim to prove the following fact: there exist  a constant \(0<c<1\)  and a high-probability event \(\Omega_0\)  (more precisely, such that $\PP(\Omega_0)\ge 1-2e^{-c\tilde{d}_*^{q/2}}$) such that, on this event, for every \(\lambda\in S_2^{N-1}\) there exists a set of indices $J_\lambda\subset\{1,\dots,m\}$ satisfying
\[
\|\bX^\top\lambda\|_q^q = \sum_{i=1}^m|\sigma_i|^q|\langle Z_i,\lambda\rangle|^q \geq \sum_{i\in J_\lambda}|\sigma_i|^q|\langle Z_i,\lambda\rangle|^q
 \ge \sum_{i\in J_\lambda}|\sigma_i|^q
\geq c^q \|\sigma\|_q^q.
\]
This requires a single-scale approximation of the index set \(S_2^{N-1}\).
We separate  the rest of the proof into three steps.

\subsubsection*{Step 1. A union bound.}Let $\varepsilon  \in(0,1)$ to be determined later (in Step 3). Take an $\varepsilon$-net $V_\varepsilon\subset S_2^{N-1}$ of cardinality not exceeding $(1+2/\eps)^N$. For every $\lambda\in S_2^{N-1}$ let  $\pi\lambda\in V_\varepsilon$ be such that $\|\pi\lambda-\lambda\|_2\leq\varepsilon$.
Fix $\lambda\in S_2^{N-1}$ and define $\eta_i = \1_{\{|\langle Z_i,\lambda\rangle|\geq1/2\}}$, that is, $\eta_i=1$ if $|\langle Z_i,\lambda\rangle|\geq1/2$ and $\eta_i=0$ otherwise.  Let $I_\lambda = \{ i\leq m \colon \eta_i=1\} =  \{i \le m \colon |\langle Z_i,\lambda\rangle|\geq1/2\}$.  By assumption~\eqref{eq:small-ball-assumption} $\EE \eta_i \geq\kappa$.
    Moreover,  $|\sigma_i|^q|\eta_i-\EE\eta_i|\leq |\sigma_i|^q\leq \|\sigma\|_\infty^{q}$ and $\EE[|\sigma_i|^{2q}(\eta_i-\EE\eta_i)^2] = |\sigma_i|^{2q}\Var(\eta_i)\leq|\sigma_i|^{2q}$.

   Applying Bernstein's inequality \cite[Theorem 2.8.2]{vershyninHighDimensionalProbabilityIntroduction2018} for  the sum of independent variables $ \sigma_i^q(\eta_i-\EE\eta_i)$ we obtain that for a universal constant $c_1$,
\begin{align*}
    \PP\Bigl(\Bigl|\sum_{i=1}^m |\sigma_i|^q(\eta_i-\EE\eta_i)\Bigr| \geq \frac{1}{2} \kappa\|\sigma\|_q^q\Bigr)
     \leq 2\exp\biggl(-\frac{c_1 \kappa^2\|\sigma\|_q^{2q}}{\|\sigma\|_{2q}^{2q}+\kappa\|\sigma\|_\infty^q\|\sigma\|_q^q}\biggr).
\end{align*}
  Moreover, $ \sum_{i\in I_\lambda} |\sigma_i|^q=\sum_{i=1}^m |\sigma_i|^q\eta_i\geq \kappa\|\sigma\|_q^q + \sum_{i=1}^m |\sigma_i|^q(\eta_i-\EE \eta_i)$, $\|\sigma\|_{2q}^{2q}\leq \|\sigma\|_\infty^q\|\sigma\|_q^q$, and $\kappa \le 1$, so for every $\lambda\in S_2^{N-1}$,
\begin{equation} \label{eq:prop-lower_DM-2}
   \PP\Bigl(\sum_{i\in I_\lambda} |\sigma_i|^q\geq \frac{1}{2}\kappa\|\sigma\|_q^q\Bigr)
   \geq 1- 2\exp\Bigl(-\frac{c_1\kappa^2\|\sigma\|_q^{q}}{ 2\|\sigma\|_\infty^q}\Bigr) = 1- 2\exp\Bigl(-\frac{c_1\kappa^2}{ 2}\tilde{d}_*^{q/2}\Bigr).
\end{equation}

Let
\begin{equation*}
    \Omega_1 = \Bigl\{ \sum_{i\in I_{\lambda}} |\sigma_i|^q     \geq \frac{1}{2} \kappa\|\sigma\|_q^q  \quad \text{for all } \lambda\in V_\varepsilon \Bigr\}.
\end{equation*}
Recall that $|V_\varepsilon|\leq (1+2/\varepsilon)^N\le e^{2N/\eps}$. Therefore, under the assumption $\tilde{d}_*^{q/2}\frac{c_1\kappa^2}{ 4}\geq\frac{2}{\varepsilon}N$, we get by the union bound and \eqref{eq:prop-lower_DM-2} that $\PP(\Omega_1)\ge 1-2\exp\bigl(-\frac{c_1\kappa^2}{ 4}\tilde{d}_*^{q/2}\bigr)$.

\subsubsection*{Step 2. Using selector process.}
Let $\varepsilon_1,\dots,\varepsilon_m$ be  i.i.d.\ Rademacher random variables independent of other variables.
Define
\[
\phi = \phi(Z_1,\dots,Z_m)
= \frac{1}{\|\sigma\|_q^q}\sup_{\lambda\in S_2^{N-1}} \sum_{i=1}^m |\sigma_i|^q\1_{\{|\langle\lambda-\pi\lambda,Z_i\rangle|>1/4\}}.
\]
Note that $\1_{\{|\langle\lambda-\pi\lambda,\,Z_i\rangle|>1/4\}}\le 4 |\langle\lambda-\pi\lambda,\,Z_i\rangle |$.
  This, the symmetrization, and the contraction principle (see, e.g. \cite[Exercise~6.7.7]{vershyninHighDimensionalProbabilityIntroduction2018}), applied conditionally,  imply that
\begin{align*}
    \EE\phi  &= \frac{1}{\|\sigma\|_q^q}\EE\sup_{\lambda \in \SSS^{N-1}}\sum_{i=1}^m|\sigma_i|^q\1_{\{|\langle\lambda-\pi\lambda,Z_i\rangle|>\frac{1}{4}\}}
     \le \frac{4}{\|\sigma\|_q^q}\EE\sup_{\lambda \in \SSS^{N-1}}\sum_{i=1}^m|\sigma_i|^q |\langle\lambda-\pi\lambda,\,Z_i\rangle |
    \\
    & \le \frac{4\eps }{\|\sigma\|_q^q}\EE\sup_{x \in  \SSS^{N-1}} \sum_{i=1}^m|\sigma_i|^q\Bigl(  |\langle x,\,Z_i\rangle | - \EE |\langle x,\,Z_i\rangle |\Bigr)
    + \frac{4\eps }{\|\sigma\|_q^q}\sup_{x \in \SSS^{N-1}}\sum_{i=1}^m|\sigma_i|^q\EE |\langle x,\,Z_i\rangle |
     \\
    & \hspace{-0,78cm}\mathop{\leq}^{\text{symmetrization}} \frac{8\eps }{\|\sigma\|_q^q}\EE\sup_{x \in \SSS^{N-1}}\sum_{j=1}^m|\sigma_i|^q\varepsilon_i  |\langle x,\,Z_i\rangle |
   +  \frac{4\eps }{\|\sigma\|_q^q}\sup_{x \in \SSS^{N-1}} \sum_{i=1}^m|\sigma_i|^q(\EE|\langle x,Z_{i}\rangle|^2)^{1/2} \\
    & \hspace{-0,52cm}\mathop{\leq}^{\text{contraction}} \frac{8\eps}{\|\sigma\|_q^q}\EE\sup_{x \in \SSS^{N-1}}\sum_{i=1}^m|\sigma_i|^q\varepsilon_i\langle x,Z_i\rangle
       + 4\eps
    =\frac{8 \eps}{\|\sigma\|_q^q}\EE\Bigl\|\sum_{i=1}^m\varepsilon_i|\sigma_i|^qZ_i\Bigr\|_2
    + 4\eps.
\end{align*}
 Since $Z_i$ are independent and have independent entries  with second moment equal to $1$, $\EE\|\sum_{i=1}^m\varepsilon_i|\sigma_i|^qZ_i\|_2^2=\sum_{i=1}^m|\sigma_i|^{2q}\EE\|Z_{i}\|_2^2=N\|\sigma\|_{2q}^{2q} \le N \|\sigma\|_{q}^q \|\sigma\|_{\infty}^{q}$. Thus, $\EE\|\sum_{i=1}^m\varepsilon_i|\sigma_i|^qZ_i\|_2\leq \sqrt{N}\|\sigma\|_\infty^{q/2}\|\sigma\|_q^{q/2}$.
 In  consequence,
 \begin{equation}\label{eq:mean-phi}
 	\EE[\phi]\leq8\eps \sqrt{\frac{N}{\tilde{d}_*^{q/2}}}+4\eps \leq 12\eps,
 \end{equation}
  provided that $\tilde{d}_*^{q/2}\geq N$.

Notice that for any $\lambda\in S_2^{N-1}$ and $z_i,z_i'\in\RR^N$,
$| \1_{\{|\langle\lambda-\pi\lambda,z_i\rangle|>1/4\}} - \1_{\{|\langle\lambda-\pi\lambda,z_i'\rangle|>1/4\}}| \le 1$, so for any $z_1,\dots,z_{i-1}$, $z_{i+1},\dots,z_m$,
\begin{align*}
    \bigl|\phi(z_1,\dots,z_{i-1},z_i,z_{i+1},\dots,z_m) - \phi(z_1,\dots,z_{i-1},z_i',z_{i+1},\dots,z_m)\bigr|\leq\frac{|\sigma_i|^q}{\|\sigma\|_q^q}.
\end{align*}
Therefore, by McDiarmid's inequality \cite[Theorem 2.9.1]{vershyninHighDimensionalProbabilityIntroduction2018}, for any $\eps>0$,
\begin{align*}
    \PP\bigl(\phi-\EE[\phi]>\eps \bigr)\leq\exp\biggl(-\frac{2\eps^2}{\sum_{i=1}^m\frac{|\sigma_i|^{2q}}{\|\sigma\|_q^{2q}}}\biggr)\leq \exp\biggl(-\frac{2\eps^2}{\frac{\|\sigma\|_\infty^q\|\sigma\|_q^q}{\|\sigma\|_q^{2q}}}\biggr) = \exp\bigl(-2\eps^2\tilde{d}_*^{q/2}\bigr).
\end{align*}
This and \eqref{eq:mean-phi} imply that $\phi\leq 13\eps$  with probability at least $1-\exp(-2\varepsilon^2\tilde{d}_*^{q/2})$.  Hence, if we denote $B_\lambda= \{1\leq i\leq m\colon |\langle\lambda-\pi\lambda,Z_i\rangle|\le1/4\}$, $B_\lambda^c = \{1,2, \dots,m\}\setminus B_\lambda$, and
\begin{align*}
    \Omega_2 = \Bigl\{   \frac{1}{\|\sigma\|_q^q}\sum_{i\in B_\lambda^c}\sigma_i^q\leq13\eps \quad \text{for all } \lambda\in S_2^{N-1} \Bigr\},
\end{align*}
then we get $\PP(\Omega_2)\ge 1-\exp(-2\varepsilon^2\tilde{d}_*^{q/2})$ provided that $\tilde{d}_*^{q/2}\geq N$.

\subsubsection*{Step 3. The sets $J_{\lambda}$}
Let $\eps=\kappa/52$. We work on $\Omega_0=\Omega_1\cap \Omega_2$,  whose probability is not smaller than $1-2e^{-c_2\tilde{d}_*^{q/2}}$  provided that $N\le c_3\tilde{d}_*^{q/2}$ (where $c_2$ and $c_3$ depend only on $\kappa$).  For every $\lambda \in \SSS^{N-1}$ let $J_\lambda = I_{\pi\lambda}\cap B_\lambda$. Then $I_{\pi\lambda} = (I_{\pi\lambda}\cap B_\lambda)\sqcup(I_{\pi\lambda}\cap B_\lambda^c) = J_\lambda\sqcup (I_{\pi\lambda}\cap B_\lambda^c)$, so
\begin{align*}
    \sum_{i\in J_\lambda}|\sigma_i|^q\geq \sum_{i\in I_{\pi\lambda}}  |\sigma_i|^q- \sum_{i\in B_\lambda^c}|\sigma_i|^q\geq \Bigl(\frac{\kappa}{2}-13\eps \Bigr)\|\sigma\|_q^q = \frac{\kappa}{4}\|\sigma\|_q^q.
\end{align*}
Thus,  by the definitions of the sets $I_{\pi\lambda}$ and $B_{\lambda}$ we know that  on the set  $\Omega_1\cap \Omega_2$, for every $\lambda\in \SSS^{N-1}$,
\begin{align*}
\|\bX^\top\lambda\|_q & \ge \Bigl(\sum_{i\in J_\lambda}|\sigma_i|^q|\langle Z_i,\lambda\rangle|^q \Bigr)^{1/q}
\geq \Bigr(\sum_{i\in J_{\lambda}} |\sigma_i|^q|\langle Z_i,\pi \lambda\rangle|^q\Bigr)^{1/q} - \Bigl(\sum_{i\in J_{\lambda}} |\sigma_i|^q|\langle Z_i,\pi\lambda - \lambda \rangle|^q\Bigr)^{1/q}
\\
&\ge
\frac12 \Bigl( \sum_{i\in J_{\lambda}} |\sigma_i|^q\Bigr)^{1/q} - \frac14  \Bigl(\sum_{i\in J_{\lambda}} |\sigma_i|^q\Bigr)^{1/q}
 \ge
 \frac{1}{4} \Bigl( \sum_{i\in J_\lambda}|\sigma_i|^q\Bigr)^{1/q}\geq\frac{\kappa^{1/q}}{16}\|\sigma\|_q
 \ge \frac{\kappa}{16}\|\sigma\|_q.
\end{align*}
 As a consequence,
\begin{align*}
    \PP\Bigl(\|\bX^\top\lambda\|_q\geq\frac{\kappa}{16}\|\sigma\|_q \quad \text{ for all }\lambda\in S_2^{N-1}  \Bigr)\geq \PP(\Omega_1\cap \Omega_2)\ge 1- 2e^{-c_2\tilde{d}_*^{q/2}} ,
    \end{align*}
provided that $N\le c_3\tilde d_*^{q/2}$.
\end{proof}

If $q\ge 2$, then \eqref{eq:parameters-diagonal} implies that $\tilde{d}_*^{q/2}\ge c d_*(A)$ for some  constant $c$ depending only on $q$.
Unfortunately, if $q<2$, then $\tilde{d}_*^{q/2}\le  C d_*(A)$, since
\[
	\sum_{i=1}^m|\sigma_i|^{2q/(2-q)} \le \|\sigma\|_\infty^{q^2/(2-q)} \sum_{i=1}^m |\sigma_i|^q ,
\]
 and $\tilde{d}_*^{q/2}$ may be of smaller order than $d_*(A)$ (for example, if $\sigma_i = i^{(q-2)/2q}$).
However, if $A=\Id_m$, then $\tilde{d}_*^{q/2}=m\sim d_*(A)$.
Hence, Propositions~\ref{prop:lower_DM} and \ref{prop:reduction-to-isotropic} yield the following corollary, which provides items 1. and 2. of Theorem~\ref{thm:DM-independent}.

\begin{Corollary}\label{cor:DM-log-concave_lower}
	Suppose that the assumptions of Theorem~\ref{thm:DM-independent} are satisfied.

	If $q\ge 2$, then there exists a constant $c>0$, such that  with probability at least $1-\exp(-c d_*(A))$, for any $\lambda\in S_2^{N-1}$, $\|\bX^\top\lambda\|_q\geq c\ell_q(A)$, provided that $N\le cd_*(A)$.
	
	If $q< 2$, then there exists a constant $c>0$, such that  with probability at least $$1-2\Log m\, e^{- cd_\ast(A)/\Log^{2}(m)},$$ for any $\lambda\in S_2^{N-1}$, $\|\bX^\top\lambda\|_q\geq c\ell_q(A)$, provided that $N\le \frac{c  d_\ast(A)} {\Log^2(m)}$.
\end{Corollary}


\subsection{Log-concave case}

Let us now assume that $A =  \diag(\sigma_1,\ldots,\sigma_m)$ , so $X_j(i)=\sigma_i Y_{j}(i)$,  and $Y$ is an isotropic log-concave vector in $\RR^m$. In this section  constants $c,C,c_i,C_i$ depend only on $q$. The main findings of this subsection are Propostiions~\ref{prop:upp-log-concave-isotropic-uncond} and \ref{prop:upp-log-concave-isotropic}, and Corollary~\ref{cor:log-concave-lower-DM}.

\subsubsection{Upper bounds}

\begin{Proposition}\label{prop:upp-log-concave-isotropic-uncond}
	If $N\le cd_\ast(A)$  and $Y$ is unconditional, then with probability at least $1-2e^{-c\sqrt{d_\ast(A)}}$ for every $\lambda\in \SSS^{N-1}$,
	\[
		\|\XX^T \lambda\|_q \le C \ell_q(A).
	\]
\end{Proposition}

\begin{proof}
	By \cite[Theorem~2]{L-domination}, \cite[Corollary~3.4]{LS-Chevet} (applied to $r=1$), and \eqref{eq:parameters-diagonal} we have
	\begin{align*}
		\EE \|\XX^T\colon \ell_2^N \to \ell_q^m\| \le C_1\Bigl( N^{1/2}d_q(A) + \|\sigma\|_q \Bigr) \le C_2 \ell_q(A).
	\end{align*}
	Note that the Lipschitz constant of $\|\XX^T\colon \ell_2^N \to \ell_q^m\|$ as a function of $\bY$ is equal to $d_q(A)$. Therefore, by the concentration for the vector with i.i.d.\ exponential entries and \cite[Theorem~2]{L-domination}, as in the proof of Proposition~\ref{thm:DM}, we obtain that
	\begin{equation*}
		\PP \bigl( \|\XX^T\colon \ell_2^N \to \ell_q^m\| \ge 2C_2 \ell_q(A) \bigr) \le 2 e^{-c_1 \ell_q(A)/d_q(A)},
	\end{equation*}
	which is equivalent to the assertion.
\end{proof}

\begin{Proposition}\label{prop:upp-log-concave-isotropic}
	If $N\le cd_\ast(A)$, then with probability at least $1-2e^{-c\sqrt{d_\ast(A)}}$ for every $\lambda\in \SSS^{N-1}$,
	\begin{align*}
		& \|\XX^T \lambda\|_q \le C'\Log^{3/2}( m)\, \ell_q(A) \qquad \text{if } q\ge 2
		\\ \text{and} \quad & \|\XX^T \lambda\|_q \le C' \Log^{ 5/2}(m)\, \ell_q(A) \qquad \text{if } q\in [1,2).
	\end{align*}
\end{Proposition}

\begin{proof}
	Assume first that $q\ge 2$.
	Since the columns of $\YY^T$ are i.i.d.\ isotropic log-concave vectors, we may apply the dual formulation of \cite[Theorem~1.1]{St2019}, i.e., in the notation from \cite{St2019},
	\[
		\EE\|X^T\|_{q'\to p} \le C(p,q) \Bigl[ (\log n)^{1/p} \max_{1\le j \le n} \|A^{(j)}\|_q + \max_{1\le i\le m} \|A_i\|_{p}  + (\log n)^{1+1/p} \EE \max_{i,j} |X_{ij}| \Bigr],
	\]
	 with $p\coloneqq2$, $q\coloneqq q'\le 2$, $n \coloneqq N$,  $m\coloneqq m$,   $X=\XX$, $A^{(j)} = \sigma$, and $A_i = (\sigma_i, \dots, \sigma_i) \in \RR^N$. By doing so we obtain that
	 \begin{equation}	\label{eq:exp-norm-log-concave-isotropic}
		\EE\|\XX^T\colon \ell_2^N \to \ell_q^m\| \le C_3 \bigl( \Log^{1/2}(N)\|\sigma\|_q+ N^{1/2} \|\sigma\|_\infty  + \Log^{3/2}(N) \bigl(\EE \sum_{i,j} |\sigma_i Y_{j}(i)|^{2q} \bigr)^{1/2q} \bigr).
	\end{equation}
	Since a random variable $Y_j(i)$ is log-concave, $\|Y_j(i)\|_{2q}\le C_4 \|Y_j(i)\|_2 = C_4$ with $C_4$ proportional to $q$ (see, e.g. \cite[Theorem~2.4.6]{BGVV}).
	Hence, the AM-GM inequality yields
	\[
	\bigl(\EE \sum_{i,j} |\sigma_i Y_{j}(i)|^{2q} \bigr)^{1/2q} \le C_4 N^{1/(2q)} \|\sigma\|_{2q} \le C_4 N^{1/(2q)} \|\sigma\|_{q}^{1/2} \|\sigma\|_\infty^{1/2} \le N^{1/4} \|\sigma\|_{q}^{1/2} \|\sigma\|_\infty^{1/2}\le N^{1/2} \|\sigma\|_\infty + \|\sigma\|_q.
	\]
	Thus, \eqref{eq:exp-norm-log-concave-isotropic}, \eqref{eq:parameters-diagonal}, and the assumption $N\le cd_\ast(A)$ (which implies that  $N\lesssim_q m$) yield
	\begin{equation}\label{eq:exp-norm-log-concave-isotropic2}
		\EE\|\XX^T\colon \ell_2^N \to \ell_q^m\| \le C_5 \Log^{3/2}(m)\, \ell_q(A).
	\end{equation}
	
	Now we  shall deduce from \eqref{eq:exp-norm-log-concave-isotropic2} 
	the desired high probability bound, using the concentration for log-concave isotropic vectors. By \cite{E2013} it is known that the Poincar{\'e} constant $P_Y$ for the distribution of the  isotropic log-concave random vector $Y$ in $\RR^m$ is upper bounded by
	\[
		C_6 \Log m \sum_{i=1}^m \frac{S_i^2}{i},
	\]
	where $S_i$ is the thin-shell constant in dimension $i$. By the recent resolution of the thin-shell conjecture \cite{KL-thinshell} it is known that for every $i$, $S_i$ is  upper bounded by a universal constant, so $P_Y\le C_7 \Log^2 m$.
Recall that the Poincar{\'e} inequality implies the exponential concentration for Lipschitz function (see, e.g., \cite[Corollary~3.2, inequality~(1.10) and Proposition~1.8]{Ledoux}), so as in the proof of Propositions~\ref{thm:DM} or \ref{prop:upp-log-concave-isotropic-uncond} we get
	\[
		\PP  \bigl( \|\XX^T\colon \ell_2^N \to \ell_q^m\| \ge 2C_5 \Log^{3/2}(m) \ell_q(A) \bigr) \le 2 e^{-c_1  \frac{\ell_q(A)\Log^{3/2}(m)}{d_q(A)\Log m}} \le 2 e^{-c_1 \ell_q(A)/d_q(A)},
	\]
	so the assertion in the case $q\ge 2$ follows. Alternatively, we could have directly used the recent solution to the KLS conjecture, see \cite{KLS1, KLS2}, to obtain the same bound.
	
	If $q<2$, we proceed as in the proof of Proposition~\ref{prop:indep-upper-expectation-below2} (using \eqref{eq:exp-norm-log-concave-isotropic2} for $q=2$ instead of Proposition~\ref{prop:indep-upper-expectation-above2}) to obtain the following counterpart of \eqref{eq:exp-norm-log-concave-isotropic2}:
	\[
		\EE\|\XX^T\colon \ell_2^N \to \ell_q^m\|  \le C_6 \Log^{5/2}(m)\, \ell_q(A).
	\]
	The assertion follows by the concentration for Lipschitz functions as in the case $q\ge 2$.
\end{proof}

\subsubsection{Lower bounds}

\begin{Proposition}\label{prop:low-log-concave-isotropic}
	If $N\le c  \min\{d_\ast(\Id), m/\Log\Log m\}$, then with probability at least $1-2e^{-c\sqrt{d_\ast(\Id)}}$ for every $\lambda\in \SSS^{N-1}$,
	\begin{equation} \label{eq:low-log-concave-isotropic-assertion}
		\|\YY^T \lambda\|_q \ge c \ell_q(\Id).
	\end{equation}

	 Moreover, if $N\le c d_\ast(\Id)$ and $Y$ is unconditional, then \eqref{eq:low-log-concave-isotropic-assertion} also holds with probability at least $1-2e^{-c\sqrt{d_\ast(\Id)}}$ for every $\lambda\in \SSS^{N-1}$.
\end{Proposition}

\begin{proof}
	Note first that according to \eqref{eq:parameters-diagonal}, $ d_\ast(\Id) \sim_q m^{\min\{\frac 2q,1\}} \le m$ and $\ell_q(\Id) \sim_q m^{1/q}$. Hence, it suffices to prove that with probability at least $1-2e^{-c_1 \sqrt{d_\ast(\Id)}}  -e^{-m}$ for every $\lambda\in \SSS^{N-1}$,
	\[
		\|\YY^T \lambda\|_q \ge c_2 m^{1/q}
	\]
	provided that $N\le c_3m$ in the unconditional case and $N\le c  \min\{d_\ast(\Id), m/\Log\Log m\}$ in the general log-concave case.
	
	We first fix $\lambda\in \SSS^{N-1}$.
	Let $f$ be the probability density function of $\YY^T\lambda$.
	Since $Y_1, \dots, Y_N$ are independent copies of an isotropic log-concave vector and $\|\lambda\|_2=1$, $\YY^T\lambda$ is an $m$-dimensional isotropic log-concave vector.
	 Thus, by the recent solution(s) to the slicing conjecture (\cite{KL-slicing,B-slicing}) it is known that $f^{1/m}$ is upper bounded by some constant $C_3$ (cf. \cite[inequality~(5)]{KL-slicing} and \cite[Theorem~4]{F-sections}).
	Hence, 
	\begin{equation}	\label{eq:low-loc-concave-isotropic-1}
		\PP\bigl( \|\YY^T \lambda\|_q \le c_ 4 m^{1/q} \bigr)
		= \PP (\YY^T \lambda \in c_4 m^{1/q}B_q^m)
		\le (C_3c_4)^m m^{m/q} \vol_m(B_q^m) \le (C_4c_4)^m \le e^{-2m},
	\end{equation}
	provided that $c_4\le (e^2C_4)^{-1}$.
	
	 Now assume that $Y$ is unconditional.
	Let $C$ be as in Proposition~\ref{prop:upp-log-concave-isotropic-uncond}, $\eps=c_4/(2C)$, $c_3=(\Log(3/\eps))^{-1}$ and $\mathcal{N}$ be an $\eps$-net in $\SSS^{N-1}$ of size at most $(3/\eps)^N\le e^m$, where the inequality follows by the assumption that $N\le c_3m$. If  $\lambda\in \SSS^{N-1}$, then there exists $\lambda_0 \in \mathcal{N}$ such that $\|\lambda-\lambda_0\|_2\le \eps$.
	Thus, \eqref{eq:low-loc-concave-isotropic-1} implies that with probability at least $1-e^{-2m}$,
	\begin{equation*}
		\|\YY^T \lambda\|_q \ge \|\YY^T \lambda_0\|_q - \eps \sup_{\eta\in\SSS^{N-1}} \|\YY^T \eta\|_q \ge c_4 m^{1/q} - \eps \sup_{\eta\in\SSS^{N-1}} \|\YY^T \eta\|_q .
	\end{equation*}
	Therefore,  with probability at least $1-|\mathcal{N}|e^{-2m} \ge 1-e^{-m}$ for every $\lambda\in\SSS^{N-1}$,
	\begin{equation*}
		\|\YY^T \lambda\|_q  \ge c_4 m^{1/q} - \eps  \sup_{\eta\in\SSS^{N-1}} \|\YY^T \eta\|_q
		= c_4 m^{1/q} - \frac{c_4}{2C}  \sup_{\eta\in\SSS^{N-1}} \|\YY^T \eta\|_q ,
	\end{equation*}
	so the assertion follows by Proposition~\ref{prop:upp-log-concave-isotropic-uncond}.
	
	 In the case when $Y$ is not necessarily unconditional and $N\le cm/\Log\Log m$, we proceed similarly, with the only difference that $\eps = c_4/(2C'\Log^{5/2}(m))$ -- where $C'$ is as in Proposition~\ref{prop:upp-log-concave-isotropic}  -- and we use Proposition~\ref{prop:upp-log-concave-isotropic} instead of Proposition~\ref{prop:upp-log-concave-isotropic-uncond}.
	\end{proof}

Note that if a random vector $Z$ is unconditional, log-concave and isotropic, then so is its projection to any coordinate subspace.
 Also, if a random vector $Z$ is log-concave and isotropic, then so is its projection to any coordinate subspace.
Therefore,
Propositions~\ref{prop:low-log-concave-isotropic} and \ref{prop:reduction-to-isotropic} (applied with $\theta=\frac 12$), together with the inequality $d_\ast(A)\lesssim_q m$, yield the following result.

\begin{Corollary}\label{cor:log-concave-lower-DM}
	If $N\le cd_\ast(A)/(\Log^2(m)\Log\Log(m))$, then with probability at least $1-
	2\Log m\,e^{-c\sqrt{d_\ast(A)}/\Log m}$, for every $\lambda\in \SSS^{N-1}$,
	\[
		\|\XX^T \lambda\|_q \ge c \ell_q(A).
	\]
\end{Corollary}

\appendix
\section{Appendix}\label{sec:proofs}

\subsection{Verification of \texorpdfstring{\eqref{eq:upper_u_v_q_norm}, \eqref{eq:upper_q_prime}, and \eqref{eq:upper_q_prime_2}}{auxiliary inequalities}}\label{sec:verification_equations}
    Let us first note that the fundamental theorem of calculus applied to $x\mapsto |x|^{q-2}x$ gives, for every $(s,t)\in\mathbb R^2\setminus\{(0,0)\}$ and $q>1$,
        \begin{align}\label{eq:cor-FTC}
        &\big(|s|^{q-2}s-|t|^{q-2}t\big)(s-t)=(q-1)(s-t)^2\int_0^1|(1-u)t+us|^{q-2}\,\mathrm{d}u.
        \end{align}

\begin{proof}[Proof of \eqref{eq:upper_u_v_q_norm}]
        Let us begin with the case $1<q<2$. 
       Since $q-2<0$ and $|(1-u)t+us|\leq(1-u)|t|+u|s|\leq|s|+|t|$ for every $u\in[0,1]$, the integral representation \eqref{eq:cor-FTC} gives
        \begin{align*}
        &\big(|s|^{q-2}s-|t|^{q-2}t\big)(s-t)\geq(q-1)\frac{|s-t|^2}{(|s|+|t|)^{2-q}}.
        \end{align*}
        On the other hand, for $\boldsymbol u,\boldsymbol v\in V_{J^c}$, H\"older's inequality with exponents $2/q$ and $2/(2-q)$ yields
        \begin{align*}
        &\|\boldsymbol u-\boldsymbol v\|_q^2=\bigg(\sum_{\substack{j\in J^c:\,|u_j|+|v_j|>0}}\left(\frac{|u_j-v_j|^2}{(|u_j|+|v_j|)^{2-q}}\right)^{q/2}\left((|u_j|+|v_j|)^q\right)^{(2-q)/2}\bigg)^{2/q}\\
        &\leq\sum_{\substack{j\in J^c:\,|u_j|+|v_j|>0}}\frac{|u_j-v_j|^2}{(|u_j|+|v_j|)^{2-q}}\bigg(\sum_{j\in J^c}(|u_j|+|v_j|)^q\bigg)^{(2-q)/q}\\
        &\leq\sum_{\substack{j\in J^c:\,|u_j|+|v_j|>0}}\frac{|u_j-v_j|^2}{(|u_j|+|v_j|)^{2-q}}(\|\boldsymbol u\|_q+\|\boldsymbol v\|_q)^{2-q},
        \end{align*}
        where in the last inequality we applied the triangle inequality for the $\ell_q$ norm.
        Combining the above two estimates together, we obtain that
        \begin{align*}
            (q-1)\|\vu-\vv\|_q^2(\|\vu\|_q+\|\vv\|_q)^{q-2}\leq \langle |\vu|^{q-2}\odot\vu - |\vv|^{q-2}\odot\vv, \vu-\vv\rangle.
        \end{align*}

        Now we move to the case $2\leq q<\infty$. 
        Let us begin with the case $s\neq t$. Since $x\mapsto|x|^{q-2}$ is even and nondecreasing on $[0,\infty)$, its integral over an interval of length one is minimized on $[-1/2,1/2]$. Consequently,
        \begin{align*}
        &\int_0^1|(1-u)t+us|^{q-2}\,\mathrm{d}u=|s-t|^{q-2}\int_0^1\bigg|u+\frac{t}{s-t}\bigg|^{q-2}\,\mathrm{d}u \geq |s-t|^{q-2}\int_{-1/2}^{1/2}|u|^{q-2}\,\mathrm{d}u = \frac{2^{2-q}}{q-1}|s-t|^{q-2}.
        \end{align*}
        This and the integral representation \eqref{eq:cor-FTC} imply $\big(|s|^{q-2}s-|t|^{q-2}t\big)(s-t)\geq2^{2-q}|s-t|^q$. The case $s=t$ is immediate. In both cases the summation over $j\in J^c$ proves
        \begin{align*}
            2^{2-q}\|\vu-\vv\|_q^q\leq \langle |\vu|^{q-2}\odot\vu - |\vv|^{q-2}\odot\vv, \vu-\vv\rangle.
        \end{align*}
    \end{proof}

    \begin{proof}[Proof of \eqref{eq:upper_q_prime}] 
    Recall that $1<q<2$. We first consider the case when $s,t\ge 0$. Without loss of generality, we assume that $s\geq t>0$. Since $q<2$,  $s^{q-1} = (t+(s-t))^{q-1}\le t^{q-1}+(s-t)^{q-1}$, so $|s|^{q-2}s -|t|^{q-2}t=s^{q-1} - t^{q-1} \le  (s-t)^{q-1}\le 2^{2-q}(s-t)^{q-1}= 2^{2-q}|s-t|^{q-1}$. When $s,t<0$, applying the same argument to $-s,-t$ yields the same result. When $st\le 0$, assume without loss of generality, $s\ge 0\ge t$, then $||s|^{q-2}s-|t|^{q-2}t|=|s|^{q-1}+|t|^{q-1} =2\cdot \frac{1}{2}(|s|^{q-1}+|t|^{q-1})\leq 2\cdot(\frac{|s|+|t|}{2})^{q-1}=2^{2-q}(|s|+|t|)^{q-1}$. Since $st\le 0$, $|s-t|=|s|+|t|$, and so $||s|^{q-2}s-|t|^{q-2}t|\leq 2^{2-q}(|s-t|)^{q-1}$. In summary, for every $s,t\in\mathbb R$,
    \begin{align*}
    &\big||s|^{q-2}s-|t|^{q-2}t\big|\leq2^{2-q}|s-t|^{q-1}.
    \end{align*}
    Raising this inequality to the power $q'$, summing over $j\in J^c$, and using $(q-1)q'=q$ proves \eqref{eq:upper_q_prime}.
        \end{proof}

    \begin{proof}[Proof of \eqref{eq:upper_q_prime_2}] 
    Recall that  $2\leq q<\infty$. In this case the integral representation \eqref{eq:cor-FTC} yields that for every $s,t\in\mathbb R$,
    \begin{align*}
    &\big||s|^{q-2}s-|t|^{q-2}t\big|\leq(q-1)(|s|+|t|)^{q-2}|s-t|.
    \end{align*}
    The case $q=2$ is immediate. For $q>2$, H\"older's inequality with exponents $(q-1)/(q-2)$ and $q-1$ yields
    \begin{align*}
    &\big\||\boldsymbol u|^{q-2}\odot\boldsymbol u-|\boldsymbol v|^{q-2}\odot\boldsymbol v\big\|_{q'}\leq(q-1)\big\||\boldsymbol u|+|\boldsymbol v|\big\|_q^{q-2}\|\boldsymbol u-\boldsymbol v\|_q\leq(q-1)(\|\boldsymbol u\|_q+\|\boldsymbol v\|_q)^{q-2}\|\boldsymbol u-\boldsymbol v\|_q,
    \end{align*}
    which proves \eqref{eq:upper_q_prime_2}.
    \end{proof}

\subsection{Comparison of effective ranks}\label{sec:app-lemma}
Define
\begin{align}\label{eq:def_effective_rank}
    R_{\|\cdot\|_q}(\Sigma) = \left(\frac{\ell_*(\Sigma^{1/2}B_q^p)}{\bE\|\Sigma^{1/2}\vv^*\|_2}\right)^2,\mbox{ where } \vv^*\in\argmin(\|\Sigma^{1/2}\vv\|_2:\, \vv\in\partial^-\|\cdot\|_{q'}(\Sigma^{1/2}G)),
\end{align} 
$G$ is a standard Gaussian random vector in $\bR^p$, and $\partial^-\|\cdot\|_{q'}$ is the sub-differential of $\|\cdot\|_{q'}$. This is the quantity appearing in the fourth condition in \eqref{eq:sufficient_conditions_KZSS20}; by the following lemma it agrees, up to a constant depending only on $q$, with the effective rank $R_*$ defined in \eqref{eq:def_effective_rank_R_star}. 
\begin{Lemma}\label{lemma:estimate_R}
If \(1<q< \infty\), then
\begin{align*}
&\frac{\big(\sum_{j\in J^c}\sigma_j^{q'/2}\big)^2}{\sum_{j\in J^c}\sigma_j^{q'}}\lesssim_q R_{\|\cdot\|_q}(\Sigma_{J^c})\lesssim_q \frac{\big(\sum_{j\in J^c}\sigma_j^{q'/2}\big)^2}{\sum_{j\in J^c}\sigma_j^{q'}}.
\end{align*}
\end{Lemma}
\begin{proof} Let $\vG = (g_j)_{j\in J^c} = (\langle G,\be_j\rangle)_{j\in J^c}$ be a standard Gaussian random vector in $V_{J^c}$.
Since \(q'>1\), the norm \(\|\cdot\|_{q'}\) is differentiable except at the origin, so almost surely
\begin{align}\label{eq:solution_v_star}
&\boldsymbol v^*=\frac{\sum_{j\in J^c}\sigma_j^{(q'-1)/2}|g_j|^{q'-1}\operatorname{sgn}(g_j)\boldsymbol e_j}{\big(\sum_{k\in J^c}\sigma_k^{q'/2}|g_k|^{q'}\big)^{(q'-1)/q'}},\mbox{ and } \|\Sigma_{J^c}^{1/2}\boldsymbol v^*\|_2=\frac{\big(\sum_{j\in J^c}\sigma_j^{q'}|g_j|^{2q'-2}\big)^{1/2}}{\big(\sum_{j\in J^c}\sigma_j^{q'/2}|g_j|^{q'}\big)^{1/q}},
\end{align}where we have used that $\frac{q'-1}{q'}=\frac{1}{q}$.
Consequently, $\|\boldsymbol v^*\|_q=1$.

By \eqref{eq:profile_quantities}, to prove Lemma~\ref{lemma:estimate_R} it suffices to show that
\begin{equation}\label{eq:aim-R_*-exponent-rho}
C^{-1}\frac{\big(\sum_{j\in J^c}\sigma_j^{q'}\big)^{1/2}}{\big(\sum_{j\in J^c}\sigma_j^{q'/2}\big)^{1/q}}  
\le \bE \|\Sigma_{J^c}^{1/2}\boldsymbol v^* \|_2 
\le C \frac{\big(\sum_{j\in J^c}\sigma_j^{q'}\big)^{1/2}}{\big(\sum_{j\in J^c}\sigma_j^{q'/2}\big)^{1/q}}
\end{equation}
 with $C$ depending only on $q$. In this proof $C_1, C_2, \dots$ denote constants depending only on $q$.

We first prove the upper side of \eqref{eq:aim-R_*-exponent-rho}.
    Let $p_0 = \bP(|g|\geq 1)$, where $g$ is a standard Gaussian random variable, and define
\begin{align*}
&\Omega=\bigg\{\sum_{j\in J^c}\sigma_j^{q'/2}|g_j|^{q'}>\frac{p_0}{2}\sum_{j\in J^c}\sigma_j^{q'/2}\bigg\}.
\end{align*}
Hoeffding's lemma implies that
\begin{equation}\label{eq:Hoeff-appendix}
\mathbb P(\Omega^c)\le \mathbb P\bigg\{\sum_{j\in J^c}\sigma_j^{q'/2} \1_{\{|g_j|\ge 1\}}\le \frac{p_0}{2}\sum_{j\in J^c}\sigma_j^{q'/2}\bigg\}\le \exp\bigg(-\frac{p_0^2}{2}\frac{\big(\sum_{j\in J^c}\sigma_j^{q'/2}\big)^2}{\sum_{j\in J^c}\sigma_j^{q'}}\bigg).
\end{equation}
On \(\Omega\), $(\sum_{j\in J^c}\sigma_j^{q'/2}|g_j|^{q'})^{-1/q}<(\frac{2}{p_0})^{1/q}(\sum_{j\in J^c}\sigma_j^{q'/2})^{-1/q}$. Notice that $\1_\Omega\leq 1$, and by the formula for $\|\Sigma_{J^c}^{\frac{1}{2}}\vv^*\|_2$ given in \eqref{eq:solution_v_star},
\begin{align} \notag
\mathbb E\big[\|\Sigma_{J^c}^{1/2}\boldsymbol v^*\|_2 \1_\Omega\big] 
& \le \bigg(\frac{2}{p_0}\bigg)^{1/q}\frac{\mathbb E\big(\sum_{j\in J^c}\sigma_j^{q'}|g_j|^{2q'-2}\big)^{1/2}}{\big(\sum_{j\in J^c}\sigma_j^{q'/2}\big)^{1/q}} 
\le \bigg(\frac{2}{p_0}\bigg)^{1/q}\big(\mathbb E|g|^{2q'-2}\big)^{1/2} \frac{\big(\sum_{j\in J^c}\sigma_j^{q'}\big)^{1/2}}{\big(\sum_{j\in J^c}\sigma_j^{q'/2}\big)^{1/q}} 
\\ & =  C_1 \frac{\big(\sum_{j\in J^c}\sigma_j^{q'}\big)^{1/2}}{\big(\sum_{j\in J^c}\sigma_j^{q'/2}\big)^{1/q}}.
\label{eq:on-Omega-appendix}
\end{align}
If \(1<q\le 2\), then
\begin{align*}
&\sum_{j\in J^c}\sigma_j^{q'}|g_j|^{2q'-2}=\sum_{j\in J^c}\sigma_j\big(\sigma_j^{q'/2}|g_j|^{q'}\big)^{2/q} \le \max_{j\in J^c}\sigma_j\bigg(\sum_{j\in J^c}\sigma_j^{q'/2}|g_j|^{q'}\bigg)^{2/q},
\end{align*}
 hence, \(\|\Sigma_{J^c}^{1/2}\boldsymbol v^*\|_2\le \max_{j\in J^c}\sigma_j^{1/2}\le \big(\sum_{j\in J^c}\sigma_j^{q'}\big)^{1/(2q')}\). 
If \(2<q<\infty\), we use the H{\"o}lder inequality with exponents  $\frac{q}{q-2}$ and $\frac{q}{2}$ to obtain 
\[
\sum_{j\in J^c}\sigma_j|v_j|^2\leq (\sum_{j\in J^c}\sigma_j^{\frac{q}{q-2}})^{\frac{q-2}{q}}(\sum_{j\in J^c}|v_j|^q)^{\frac{2}{q}} .
\]
 Combining this with the aforementioned equation $\|\vv^*\|_q=1$ and the fact that $\|\cdot\|_{r_1}\leq\|\cdot\|_{r_2}$ for $r_1\geq r_2$, applied to $r_1 = \frac{q}{q-2}$ and $r_2 = q' = \frac{q}{q-1}$, we get
\begin{align*}
\|\Sigma_{J^c}^{1/2}\boldsymbol v^*\|_2\le \sup\bigg\{\bigg(\sum_{j\in J^c}\sigma_j|v_j|^2\bigg)^{1/2}:\sum_{j\in J^c}|v_j|^q\le 1\bigg\} \le \bigg(\sum_{j\in J^c}\sigma_j^{q/(q-2)}\bigg)^{1/2-1/q}\le \bigg(\sum_{j\in J^c}\sigma_j^{q'}\bigg)^{1/(2q')}.
\end{align*}
Thus, 
\begin{equation}\label{eq:sigma-v-l2-upper_bound_appendix}
\|\Sigma_{J^c}^{1/2}\boldsymbol v^*\|_2\le \bigg(\sum_{j\in J^c}\sigma_j^{q'}\bigg)^{1/(2q')}
\end{equation} for all $q\in (1,\infty)$.

Moreover, $0<\frac{\sum_{j\in J^c}\sigma_j^{q'}}{\big(\sum_{j\in J^c}\sigma_j^{q'/2}\big)^2}\le 1.$
Thus, \eqref{eq:Hoeff-appendix} and \eqref{eq:sigma-v-l2-upper_bound_appendix} imply that
\begin{align*}
\mathbb E\big[\|\Sigma_{J^c}^{1/2}\boldsymbol v^*\|_2\mathbf 1_{\Omega^c}\big] 
&\le \bigg(\sum_{j\in J^c}\sigma_j^{q'}\bigg)^{1/(2q')}\exp\bigg(-\frac{p_0^2}{2}\frac{\big(\sum_{j\in J^c}\sigma_j^{q'/2}\big)^2}{\sum_{j\in J^c}\sigma_j^{q'}}\bigg)\\
&= \frac{\big(\sum_{j\in J^c}\sigma_j^{q'}\big)^{1/2}}{\big(\sum_{j\in J^c}\sigma_j^{q'/2}\big)^{1/q}}\bigg(\frac{\big(\sum_{j\in J^c}\sigma_j^{q'/2}\big)^2}{\sum_{j\in J^c}\sigma_j^{q'}}\bigg)^{1/(2q)}\exp\bigg(-\frac{p_0^2}{2}\frac{\big(\sum_{j\in J^c}\sigma_j^{q'/2}\big)^2}{\sum_{j\in J^c}\sigma_j^{q'}}\bigg)
\\ & \le  \sup_{t>0}\bigl(t^{1/(2q)} e^{-tp_0/2} \bigr)  \frac{\big(\sum_{j\in J^c}\sigma_j^{q'}\big)^{1/2}}{\big(\sum_{j\in J^c}\sigma_j^{q'/2}\big)^{1/q}} 
= C_2  \frac{\big(\sum_{j\in J^c}\sigma_j^{q'}\big)^{1/2}}{\big(\sum_{j\in J^c}\sigma_j^{q'/2}\big)^{1/q}} .
\end{align*}
This, together with \eqref{eq:on-Omega-appendix}, yields the upper estimate in \eqref{eq:aim-R_*-exponent-rho}.

Now we move to the proof of the lower side of \eqref{eq:aim-R_*-exponent-rho}.
For any positive random variables $A,B$, Cauchy--Schwarz's inequality gives that, $(\bE[A])^2\leq\bE[\frac{A}{B}]\bE[AB]$. Applying it to $A = (\sum_{j\in J^c}\sigma_j^{q'}|g_j|^{2q'-2})^{\frac{1}{2}}$, and $B = (\sum_{j\in J^c}\sigma_j^{q'/2}|g_j|^{q'})^{\frac{1}{q}}$ yields
\begin{align*}
&\mathbb E\|\Sigma_{J^c}^{1/2}\boldsymbol v^*\|_2 \ge \frac{\big[\mathbb E\big(\sum_{j\in J^c}\sigma_j^{q'}|g_j|^{2q'-2}\big)^{1/2}\big]^2}{\mathbb E\big[\big(\sum_{j\in J^c}\sigma_j^{q'}|g_j|^{2q'-2}\big)^{1/2}\big(\sum_{j\in J^c}\sigma_j^{q'/2}|g_j|^{q'}\big)^{1/q}\big]}.
\end{align*}
The numerator is bounded below by Jensen's inequality applied to Euclidean norm on $V_{J^c}$, that is, 
\begin{align*}
    \mathbb E(\sum_{j\in J^c}\sigma_j^{q'}|g_j|^{2q'-2})^{1/2}\ge (\mathbb E|g|^{q'-1})(\sum_{j\in J^c}\sigma_j^{q'})^{1/2}.
\end{align*}
To upper bound the denominator for $1<q \le 2$ we use the Cauchy--Schwartz inequality and comparison of moments of norms of Gaussian vectors; we get
\begin{align*}
&\mathbb E\bigg[\bigg(\sum_{j\in J^c}\sigma_j^{q'}|g_j|^{2q'-2}\bigg)^{1/2}\bigg(\sum_{j\in J^c}\sigma_j^{q'/2}|g_j|^{q'}\bigg)^{1/q}\bigg] \le \bigg(\mathbb E\sum_{j\in J^c}\sigma_j^{q'}|g_j|^{2q'-2}\bigg)^{1/2}\bigg(\mathbb E\bigg(\sum_{j\in J^c}\sigma_j^{q'/2}|g_j|^{q'}\bigg)^{2/q}\bigg)^{1/2}\\
&\le \mathbb E|g|^{2q'-2}\bigg(\sum_{j\in J^c}\sigma_j^{q'}\bigg)^{1/2}\bigg(\sum_{j\in J^c}\sigma_j^{q'/2}\bigg)^{1/q},
\end{align*}where we have used the Minkowski's inequality to obtain $\big(\bE\big(\sum_{j\in J^c}\sigma_j^{q'/2}|g_j|^{q'}\big)^{\frac{2}{q}}\big)^{\frac{q}{2}}\leq \sum_{j\in J^c}\sigma_j^{q'/2}\big(\bE|g_j|^{\frac{2q'}{q}}\big)^{\frac{q}{2}}$ and the fact that $2\frac{q'}{q} = 2q'-2$.
If \(2<q<\infty\), then Jensen's inequality applied to the concave function \(t\mapsto t^{2/q}\) gives
\begin{align*}
&\mathbb E\bigg[\bigg(\sum_{j\in J^c}\sigma_j^{q'}|g_j|^{2q'-2}\bigg)^{1/2}\bigg(\sum_{j\in J^c}\sigma_j^{q'/2}|g_j|^{q'}\bigg)^{1/q}\bigg] \le \bigg(\mathbb E\sum_{j\in J^c}\sigma_j^{q'}|g_j|^{2q'-2}\bigg)^{1/2}\bigg(\mathbb E\bigg(\sum_{j\in J^c}\sigma_j^{q'/2}|g_j|^{q'}\bigg)^{2/q}\bigg)^{1/2}\\
&\le (\mathbb E|g|^{2q'-2})^{1/2}(\mathbb E|g|^{q'})^{1/q}\bigg(\sum_{j\in J^c}\sigma_j^{q'}\bigg)^{1/2}\bigg(\sum_{j\in J^c}\sigma_j^{q'/2}\bigg)^{1/q}.
\end{align*}
Hence, the lower bound in \eqref{eq:aim-R_*-exponent-rho}  follows.
\end{proof}

\bibliographystyle{alpha}
\bibliography{references}

@article {F-sections,
    AUTHOR = {Fradelizi, M.},
     TITLE = {Sections of convex bodies through their centroid},
   JOURNAL = {Arch. Math. (Basel)},
  FJOURNAL = {Archiv der Mathematik},
    VOLUME = {69},
      YEAR = {1997},
    NUMBER = {6},
     PAGES = {515--522},
      ISSN = {0003-889X,1420-8938},
   MRCLASS = {52A20},
  MRNUMBER = {1480519},
MRREVIEWER = {Uwe\ Schnell},
       DOI = {10.1007/s000130050154},
       URL = {https://doi.org/10.1007/s000130050154},
}

@inproceedings{kurMinimumNormInterpolation2024,
	title = {Minimum {Norm} {Interpolation} {Meets} {The} {Local} {Theory} of {Banach} {Spaces}},
	issn = {2640-3498},
	url = {https://proceedings.mlr.press/v235/kur24a.html},
	language = {en},
	urldate = {2024-10-01},
	booktitle = {Proceedings of the 41st {International} {Conference} on {Machine} {Learning}},
	publisher = {PMLR},
	author = {Kur, Gil and Abdalla, Pedro and Bizeul, Pierre and Yang, Fanny},
	month = jul,
	year = {2024},
	pages = {25726--25754},
}

@article{lecueGeometricalViewpointBenign2024,
	title = {A geometrical viewpoint on the benign overfitting property of the minimum $\ell_2$-norm interpolant estimator and its universality},
	issn = {1432-2064},
	url = {https://doi.org/10.1007/s00440-024-01336-7},
	doi = {10.1007/s00440-024-01336-7},
	language = {en},
	urldate = {2024-11-08},
	journal = {Probability Theory and Related Fields},
	author = {Lecué, Guillaume and Shang, Zong},
	month = nov,
	year = {2024},
}

@inproceedings{krizhevsky2012imagenet,
  title={Imagenet classification with deep convolutional neural networks},
  author={Krizhevsky, Alex and Sutskever, Ilya and Hinton, Geoffrey E},
  booktitle={Advances in Neural Information Processing Systems (NeurIPS)},
  volume={25},
  year={2012}
}

@inproceedings{vaswani2017attention,
  title={Attention is all you need},
  author={Vaswani, Ashish and Shazeer, Noam and Parmar, Niki and Uszkoreit, Jakob and Jones, Llion and Gomez, Aidan N and Kaiser, {\L}ukasz and Polosukhin, Illia},
  booktitle={Advances in Neural Information Processing Systems (NeurIPS)},
  volume={30},
  year={2017}
}

@article{silver2016mastering,
  title={Mastering the game of Go with deep neural networks and tree search},
  author={Silver, David and Huang, Aja and Maddison, Chris J and Guez, Arthur and Sifre, Laurent and Van Den Driessche, George and Schrittwieser, Julian and Antonoglou, Ioannis and Panneershelvam, Veda and Lanctot, Marc and others},
  journal={Nature},
  volume={529},
  number={7587},
  pages={484--489},
  year={2016},
  publisher={Nature Publishing Group}
}

@inproceedings{zhang2017understanding,
  title={Understanding deep learning requires rethinking generalization},
  author={Zhang, Chiyuan and Bengio, Samy and Hardt, Moritz and Recht, Benjamin and Vinyals, Oriol},
  booktitle={International Conference on Learning Representations (ICLR)},
  year={2017}
}

@article{belkin2019reconciling,
  title={Reconciling modern machine-learning practice and the classical bias--variance trade-off},
  author={Belkin, Mikhail and Hsu, Daniel and Ma, Siyuan and Mandal, Soumik},
  journal={Proceedings of the National Academy of Sciences (PNAS)},
  volume={116},
  number={32},
  pages={15849--15854},
  year={2019},
  publisher={National Acad Sciences}
}

@inproceedings{nakkiran2019deep,
  title={Deep double descent: Where bigger models and more data hurt},
  author={Nakkiran, Preetum and Kaplun, Gal and Bansal, Yamini and Yang, Tristan and Barak, Boaz and Sutskever, Ilya},
  booktitle={International Conference on Learning Representations (ICLR)},
  year={2020}
}

@inproceedings{gunasekar2018characterizing,
  title={Characterizing implicit bias in terms of optimization geometry},
  author={Gunasekar, Suriya and Lee, Jason D and Soudry, Daniel and Srebro, Nathan},
  booktitle={International Conference on Machine Learning (ICML)},
  pages={1832--1841},
  year={2018},
  organization={PMLR}
}

@article{muthukumar2020harmless,
  title={Harmless interpolation of noisy data in regression},
  author={Muthukumar, Vidya and Vodrahalli, Kailas and Subramanian, Vignesh and Sahai, Anant},
  journal={IEEE Journal on Selected Areas in Information Theory},
  volume={1},
  number={1},
  pages={67--83},
  year={2020},
  publisher={IEEE}
}

@article{hastie2022surprises,
  title={Surprises in high-dimensional ridgeless least squares interpolation},
  author={Hastie, Trevor and Montanari, Andrea and Rosset, Saharon and Tibshirani, Ryan J},
  journal={The Annals of Statistics},
  volume={50},
  number={2},
  pages={949--986},
  year={2022},
  publisher={Institute of Mathematical Statistics}
}

@article {KL-slicing,
    AUTHOR = {Klartag, B. and Lehec, J.},
     TITLE = {Affirmative resolution of {B}ourgain's slicing problem using
              {G}uan's bound},
   JOURNAL = {Geom. Funct. Anal.},
  FJOURNAL = {Geometric and Functional Analysis},
    VOLUME = {35},
      YEAR = {2025},
    NUMBER = {4},
     PAGES = {1147--1168},
      ISSN = {1016-443X,1420-8970},
   MRCLASS = {52A20 (60D05)},
  MRNUMBER = {4956618},
       DOI = {10.1007/s00039-025-00718-w},
       URL = {https://doi.org/10.1007/s00039-025-00718-w},
}

@misc{KLS1,
      title={Presenting a proof of the Kannan-Lovasz-Simonovits conjecture}, 
      author={Bizeul, P. and Klartag, B. and Lehec, J.},
      year={preprint, arXiv:2610.05474, (2026)},
      eprint={2610.05474},
      archivePrefix={arXiv},
      primaryClass={math.FA}
}

@misc{B-slicing,
      title={The slicing conjecture via small ball estimates}, 
      author={Bizeul, P.},
      year={preprint, arXiv:2501.06854, (2025)},
      eprint={2501.06854},
      archivePrefix={arXiv},
      primaryClass={math.FA}
}

@misc{KLS2,
      title={An O(1) Bound for the KLS Constant}, 
      author={Zhao Song and Xinzhi Zhang},
      year={preprint, arXiv:2610.01447, (2026)},
      eprint={2610.01447},
      archivePrefix={arXiv},
      primaryClass={math.PR}
}

@article {LS-Chevet,
    AUTHOR = {Lata{\l}a, R.  and Strzelecka, M.},
     TITLE = {Chevet-type inequalities for subexponential {W}eibull
              variables and estimates for norms of random matrices},
   JOURNAL = {Electron. J. Probab.},
  FJOURNAL = {Electronic Journal of Probability},
    VOLUME = {29},
      YEAR = {2024},
     PAGES = {Paper No. 97, 19},
   MRCLASS = {60B20 (46B09 60E15)},
  MRNUMBER = {4761377},
       DOI = {10.1214/24-ejp1151},
       URL = {https://doi.org/10.1214/24-ejp1151},
}

@article {APSS,
	AUTHOR = {Adamczak, R. and Prochno, J. and Strzelecka, M.
              and Strzelecki, M. },
     TITLE = {Norms of structured random matrices},
   JOURNAL = {Math. Ann.},
  FJOURNAL = {Mathematische Annalen},
    VOLUME = {388},
      YEAR = {2024},
    NUMBER = {4},
     PAGES = {3463--3527},
      ISSN = {0025-5831,1432-1807},
   MRCLASS = {60B20 (46B09 52A23 60E15 60G15)},
  MRNUMBER = {4721775},
       DOI = {10.1007/s00208-023-02599-6},
       URL = {https://doi.org/10.1007/s00208-023-02599-6},
}

@Article{L-domination,
  author     = {Lata{\l}a, R.},
  journal    = {Bull. Pol. Acad. Sci. Math.},
  title      = {On weak tail domination of random vectors},
  year       = {2009},
  issn       = {0239-7269},
  number     = {1},
  pages      = {75--80},
  volume     = {57},
  doi        = {10.4064/ba57-1-8},
  fjournal   = {Bulletin of the Polish Academy of Sciences. Mathematics},
  mrclass    = {60E15 (52A40 60G50)},
  mrnumber   = {2520452},
  mrreviewer = {Deli Li},
  url        = {https://doi.org/10.4064/ba57-1-8},
}

@article {St2019,
    AUTHOR = {Strzelecka, M.},
     TITLE = {Estimates of norms of log-concave random matrices with
              dependent entries},
   JOURNAL = {Electron. J. Probab.},
  FJOURNAL = {Electronic Journal of Probability},
    VOLUME = {24},
      YEAR = {2019},
     PAGES = {Paper No. 107, 15},
   MRCLASS = {60B20 (15B52 46B09)},
  MRNUMBER = {4017125},
       DOI = {10.1214/19-ejp365},
       URL = {https://doi.org/10.1214/19-ejp365},
}

@book {Ledoux,
    AUTHOR = {Ledoux, M.},
     TITLE = {The concentration of measure phenomenon},
    SERIES = {Mathematical Surveys and Monographs},
    VOLUME = {89},
 PUBLISHER = {American Mathematical Society, Providence, RI},
      YEAR = {2001},
     PAGES = {x+181},
      ISBN = {0-8218-2864-9},
   MRCLASS = {28C15 (28A35 46B09 60E15 82B44)},
  MRNUMBER = {1849347},
MRREVIEWER = {Werner Linde},
       DOI = {10.1090/surv/089},
       URL = {https://doi.org/10.1090/surv/089},
}

@misc{KL-thinshell,
      title={Thin-shell bounds via parallel coupling}, 
      author={Klartag, B. and Lehec, J.},
      year={preprint, arXiv:2507.15495, (2025)},
      eprint={2507.15495},
      archivePrefix={arXiv},
      primaryClass={math.PR}
}

@article {E2013,
    AUTHOR = {Eldan, R.},
     TITLE = {Thin shell implies spectral gap up to polylog via a stochastic
              localization scheme},
   JOURNAL = {Geom. Funct. Anal.},
  FJOURNAL = {Geometric and Functional Analysis},
    VOLUME = {23},
      YEAR = {2013},
    NUMBER = {2},
     PAGES = {532--569},
      ISSN = {1016-443X,1420-8970},
   MRCLASS = {60D05 (52A40 60H10)},
  MRNUMBER = {3053755},
MRREVIEWER = {Christian\ Rau},
       DOI = {10.1007/s00039-013-0214-y},
       URL = {https://doi.org/10.1007/s00039-013-0214-y},
}

@book {BGVV,
    AUTHOR = {Brazitikos, Silouanos and Giannopoulos, Apostolos and
              Valettas, Petros and Vritsiou, Beatrice-Helen},
     TITLE = {Geometry of isotropic convex bodies},
    SERIES = {Mathematical Surveys and Monographs},
    VOLUME = {196},
 PUBLISHER = {American Mathematical Society, Providence, RI},
      YEAR = {2014},
     PAGES = {xx+594},
      ISBN = {978-1-4704-1456-6},
   MRCLASS = {52-02 (28Axx 46Bxx 52Axx 60D05)},
  MRNUMBER = {3185453},
MRREVIEWER = {Julio\ Bernu\'{e}s},
       DOI = {10.1090/surv/196},
       URL = {https://doi.org/10.1090/surv/196},
}

@article {Paour12,
    AUTHOR = {Paouris, Grigoris},
     TITLE = {Small ball probability estimates for log-concave measures},
   JOURNAL = {Trans. Amer. Math. Soc.},
  FJOURNAL = {Transactions of the American Mathematical Society},
    VOLUME = {364},
      YEAR = {2012},
    NUMBER = {1},
     PAGES = {287--308},
      ISSN = {0002-9947,1088-6850},
   MRCLASS = {52A20 (46B07 60D05)},
  MRNUMBER = {2833584},
MRREVIEWER = {Guillaume\ Aubrun},
       DOI = {10.1090/S0002-9947-2011-05411-5},
       URL = {https://doi.org/10.1090/S0002-9947-2011-05411-5},
}

@article {Paour06,
    AUTHOR = {Paouris, G.},
     TITLE = {Concentration of mass on convex bodies},
   JOURNAL = {Geom. Funct. Anal.},
  FJOURNAL = {Geometric and Functional Analysis},
    VOLUME = {16},
      YEAR = {2006},
    NUMBER = {5},
     PAGES = {1021--1049},
      ISSN = {1016-443X,1420-8970},
   MRCLASS = {52A20 (46B07)},
  MRNUMBER = {2276533},
MRREVIEWER = {Keith\ Ball},
       DOI = {10.1007/s00039-006-0584-5},
       URL = {https://doi.org/10.1007/s00039-006-0584-5},
}

@article {Samson2000,
    AUTHOR = {Samson, P.-M.},
     TITLE = {Concentration of measure inequalities for {M}arkov chains and
              {$\Phi$}-mixing processes},
   JOURNAL = {Ann. Probab.},
  FJOURNAL = {The Annals of Probability},
    VOLUME = {28},
      YEAR = {2000},
    NUMBER = {1},
     PAGES = {416--461},
      ISSN = {0091-1798,2168-894X},
   MRCLASS = {60E15 (60J05)},
  MRNUMBER = {1756011},
       DOI = {10.1214/aop/1019160125},
       URL = {https://doi.org/10.1214/aop/1019160125},
}

@article{adamczakFewRemarksOperator2010,
	title = {A {Few} {Remarks} on the {Operator} {Norm} of {Random} {Toeplitz} {Matrices}},
	volume = {23},
	issn = {1572-9230},
	url = {https://doi.org/10.1007/s10959-008-0201-7},
	doi = {10.1007/s10959-008-0201-7},
	language = {en},
	number = {1},
	urldate = {2024-03-19},
	journal = {Journal of Theoretical Probability},
	author = {Adamczak, Radosław},
	month = mar,
	year = {2010},
	pages = {85--108},
}

@article {latalaOperatorellptoellqNorms2025,
    AUTHOR = {Lata{\l}a, Rafa{\l} and Strzelecka, Marta},
     TITLE = {Operator {$\ell_p\to\ell_q$} norms of {G}aussian matrices},
   JOURNAL = {Adv. Math.},
  FJOURNAL = {Advances in Mathematics},
    VOLUME = {501},
      YEAR = {2026},
     PAGES = {Paper No. 111097, 56},
      ISSN = {0001-8708,1090-2082},
   MRCLASS = {60B20 (15B52)},
  MRNUMBER = {5088300},
       DOI = {10.1016/j.aim.2026.111097},
       URL = {https://doi.org/10.1016/j.aim.2026.111097},
}

@book{koltchinskiiOracleInequalitiesEmpirical2011,
	address = {Berlin, Heidelberg},
	series = {Lecture {Notes} in {Mathematics}},
	title = {Oracle {Inequalities} in {Empirical} {Risk} {Minimization} and {Sparse} {Recovery} {Problems}: École d’Été de {Probabilités} de {Saint}-{Flour} {XXXVIII}-2008},
	volume = {2033},
	isbn = {978-3-642-22146-0 978-3-642-22147-7},
	shorttitle = {Oracle {Inequalities} in {Empirical} {Risk} {Minimization} and {Sparse} {Recovery} {Problems}},
	url = {https://link.springer.com/10.1007/978-3-642-22147-7},
	language = {en},
	urldate = {2024-01-21},
	publisher = {Springer},
	author = {Koltchinskii, Vladimir},
	year = {2011},
	doi = {10.1007/978-3-642-22147-7},
}

@inproceedings{jacotNeuralTangentKernel2018,
	address = {Red Hook, NY, USA},
	series = {{NIPS}'18},
	title = {Neural tangent kernel: convergence and generalization in neural networks},
	shorttitle = {Neural tangent kernel},
	urldate = {2023-08-22},
	booktitle = {Proceedings of the 32nd {International} {Conference} on {Neural} {Information} {Processing} {Systems}},
	publisher = {Curran Associates Inc.},
	author = {Jacot, Arthur and Gabriel, Franck and Hongler, Clément},
	year = {2018},
	pages = {8580--8589},
}

@article{ghorbaniLinearizedTwolayersNeural2021,
	title = {Linearized two-layers neural networks in high dimension},
	volume = {49},
	issn = {0090-5364, 2168-8966},
	url = {https://projecteuclid.org/journals/annals-of-statistics/volume-49/issue-2/Linearized-two-layers-neural-networks-in-high-dimension/10.1214/20-AOS1990.full},
	doi = {10.1214/20-AOS1990},
	number = {2},
	urldate = {2023-09-05},
	journal = {The Annals of Statistics},
	author = {Ghorbani, Behrooz and Mei, Song and Misiakiewicz, Theodor and Montanari, Andrea},
	month = apr,
	year = {2021},
	note = {Publisher: Institute of Mathematical Statistics},
	pages = {1029--1054},
}

@article{bachBreakingCurseDimensionality2017,
	title = {Breaking the {Curse} of {Dimensionality} with {Convex} {Neural} {Networks}},
	volume = {18},
	issn = {1533-7928},
	url = {http://jmlr.org/papers/v18/14-546.html},
	number = {19},
	urldate = {2023-05-06},
	journal = {Journal of Machine Learning Research},
	author = {Bach, Francis},
	year = {2017},
	pages = {1--53},
}

@article{bartolucciUnderstandingNeuralNetworks2023,
	title = {Understanding neural networks with reproducing kernel {Banach} spaces},
	volume = {62},
	issn = {1063-5203},
	url = {https://www.sciencedirect.com/science/article/pii/S1063520322000768},
	doi = {10.1016/j.acha.2022.08.006},
	language = {en},
	urldate = {2022-11-07},
	journal = {Applied and Computational Harmonic Analysis},
	author = {Bartolucci, Francesca and De Vito, Ernesto and Rosasco, Lorenzo and Vigogna, Stefano},
	month = jan,
	year = {2023},
	pages = {194--236},
}

@misc{bartolucciNeuralReproducingKernel2024,
	title = {Neural reproducing kernel {Banach} spaces and representer theorems for deep networks},
	url = {http://arxiv.org/abs/2403.08750},
	doi = {10.48550/arXiv.2403.08750},
	urldate = {2025-07-15},
	publisher = {arXiv},
	author = {Bartolucci, Francesca and Vito, Ernesto De and Rosasco, Lorenzo and Vigogna, Stefano},
	month = mar,
	year = {2024},
	note = {arXiv:2403.08750 [stat]},
}

@article{dvoretzky1961some,
  title={Some results on convex bodies and Banach spaces},
  author={Dvoretzky, Aryeh},
  journal={Proceedings of the International Symposium on Linear Spaces},
  pages={123--160},
  year={1961},
  publisher={Jerusalem Academic Press}
}

@article{milman1971new,
  title={A new proof of A. Dvoretzky's theorem on cross-sections of convex bodies},
  author={Milman, Vitali D},
  journal={Funktsional'nyi Analiz i ego Prilozheniya},
  volume={5},
  number={4},
  pages={28--37},
  year={1971},
  publisher={Russian Academy of Sciences, Steklov Mathematical Institute}
}

@article{tsiglerBenignOverfittingRidge2023,
	title = {Benign overfitting in ridge regression},
	volume = {24},
	url = {http://jmlr.org/papers/v24/22-1398.html},
	number = {123},
	journal = {Journal of Machine Learning Research},
	author = {Tsigler, Alexander and Bartlett, Peter L.},
	year = {2023},
	pages = {1--76},
}

@inproceedings{azizan2019stochastic,
  title={Stochastic Gradient/Mirror Descent: Minimax Optimality and Implicit Regularization},
  author={Azizan, Navid and Hassibi, Babak},
  booktitle={International Conference on Learning Representations (ICLR)},
  year={2019},
  url={https://openreview.net/forum?id=B1lvJl-O-Q}
}

@inproceedings{wang2022tight,
  title={Tight bounds for minimum $\ell_1$-norm interpolation of noisy data},
  author={Wang, Guillaume and Donhauser, Konstantin and Yang, Fanny},
  booktitle={Proceedings of The 25th International Conference on Artificial Intelligence and Statistics},
  pages={10572--10602},
  year={2022},
  organization={PMLR}
}

@article{belloni2011square,
  title={Square-root lasso: pivotal recovery of sparse signals via conic programming},
  author={Belloni, Alexandre and Chernozhukov, Victor and Wang, Lie},
  journal={Biometrika},
  volume={98},
  number={4},
  pages={791--806},
  year={2011},
  publisher={Oxford University Press}
}

@article{candes2005decoding,
  title={Decoding by linear programming},
  author={Cand{\`e}s, Emmanuel J and Tao, Terence},
  journal={IEEE Transactions on Information Theory},
  volume={51},
  number={12},
  pages={4203--4215},
  year={2005},
  publisher={IEEE}
}

@article{kur2026uniform,
  title={Minimum {Norm} {Interpolation} via {The} {Local} {Theory} of {Banach} {Spaces}: {The} {Role} of 2-{Uniform} {Convexity}},
  author={Kur, Gil and Bizeul, Pierre},
  journal={arXiv preprint arXiv:2603.28956},
  year={2026}
}

@inproceedings{thrampoulidis2015regularized,
  title={Regularized linear regression: A precise analysis of the estimation error},
  author={Thrampoulidis, Christos and Oymak, Samet and Hassibi, Babak},
  booktitle={Conference on Learning Theory (COLT)},
  pages={1683--1709},
  year={2015},
  organization={PMLR}
}

@article{paouris2017random,
  title={Random version of {Dvoretzky's} theorem in $\ell_p^n$},
  author={Paouris, Grigoris and Valettas, Petros and Zinn, Joel},
  journal={Probability Theory and Related Fields},
  volume={168},
  number={3-4},
  pages={707--727},
  year={2017},
  publisher={Springer}
}

@article{thrampoulidis2015gaussian,
  title={The {Gaussian} min-max theorem in the presence of convexity},
  author={Thrampoulidis, Christos and Oymak, Samet and Hassibi, Babak},
  journal={arXiv preprint arXiv:1408.4837},
  year={2015}
}

@incollection{gordon1988milman,
  title={On {Milman's} inequality and random spaces which escape through a mesh in $\mathbb{R}^n$},
  author={Gordon, Yehoram},
  booktitle={Geometric Aspects of Functional Analysis},
  pages={84--106},
  year={1988},
  publisher={Springer},
  series={Lecture Notes in Mathematics},
  volume={1317}
}

@inproceedings{donhauser2022fast,
  title={Fast rates for noisy interpolation require rethinking the effect of inductive bias},
  author={Donhauser, Konstantin and Ruggeri, Nicol{\`o} and Stojanovic, Stefan and Yang, Fanny},
  booktitle={International Conference on Machine Learning (ICML)},
  pages={5397--5428},
  year={2022},
  organization={PMLR}
}

@inproceedings{ju2020overfitting,
  title={Overfitting Can Be Harmless for Basis Pursuit, But Only to a Degree},
  author={Ju, Peizhong and Lin, Xiaojun and Liu, Jia},
  booktitle={Advances in Neural Information Processing Systems (NeurIPS)},
  volume={33},
  pages={7956--7967},
  year={2020},
  publisher={Curran Associates, Inc.}
}

@inproceedings{negrea2020defense,
  title={In Defense of Uniform Convergence: Generalization via Derandomization with an Application to Interpolating Predictors},
  author={Negrea, Jeffrey and Dziugaite, Gintare Karolina and Roy, Daniel},
  booktitle={Proceedings of the 37th International Conference on Machine Learning (ICML)},
  pages={7263--7272},
  year={2020},
  organization={PMLR}
}

@article {MR4474492,
    AUTHOR = {Chinot, Geoffrey and L\"offler, Matthias and van de Geer,
              Sara},
     TITLE = {On the robustness of minimum norm interpolators and
              regularized empirical risk minimizers},
   JOURNAL = {Ann. Statist.},
  FJOURNAL = {The Annals of Statistics},
    VOLUME = {50},
      YEAR = {2022},
    NUMBER = {4},
     PAGES = {2306--2333},
      ISSN = {0090-5364,2168-8966},
   MRCLASS = {62J05 (65F45)},
  MRNUMBER = {4474492},
MRREVIEWER = {Ivan\ \v Ze\v zula},
       DOI = {10.1214/22-aos2190},
       URL = {https://doi.org/10.1214/22-aos2190},
}

@article {MR2559727,
    AUTHOR = {DeVore, Ronald and Petrova, Guergana and Wojtaszczyk,
              Przemyslaw},
     TITLE = {Instance-optimality in probability with an
              {$l_1$}-minimization decoder},
   JOURNAL = {Appl. Comput. Harmon. Anal.},
  FJOURNAL = {Applied and Computational Harmonic Analysis. Time-Frequency
              and Time-Scale Analysis, Wavelets, Numerical Algorithms, and
              Applications},
    VOLUME = {27},
      YEAR = {2009},
    NUMBER = {3},
     PAGES = {275--288},
      ISSN = {1063-5203,1096-603X},
   MRCLASS = {42C15 (94A12)},
  MRNUMBER = {2559727},
MRREVIEWER = {Joseph\ D.\ Lakey},
       DOI = {10.1016/j.acha.2009.05.001},
       URL = {https://doi.org/10.1016/j.acha.2009.05.001},
}

@article {MR2591836,
    AUTHOR = {Wojtaszczyk, P.},
     TITLE = {Stability and instance optimality for {G}aussian measurements
              in compressed sensing},
   JOURNAL = {Found. Comput. Math.},
  FJOURNAL = {Foundations of Computational Mathematics. The Journal of the
              Society for the Foundations of Computational Mathematics},
    VOLUME = {10},
      YEAR = {2010},
    NUMBER = {1},
     PAGES = {1--13},
      ISSN = {1615-3375,1615-3383},
   MRCLASS = {94A20 (65K10)},
  MRNUMBER = {2591836},
       DOI = {10.1007/s10208-009-9046-4},
       URL = {https://doi.org/10.1007/s10208-009-9046-4},
}

@article {MR2914271,
    AUTHOR = {Wojtaszczyk, P.},
     TITLE = {{$\ell_1$} minimization with noisy data},
   JOURNAL = {SIAM J. Numer. Anal.},
  FJOURNAL = {SIAM Journal on Numerical Analysis},
    VOLUME = {50},
      YEAR = {2012},
    NUMBER = {2},
     PAGES = {458--467},
      ISSN = {0036-1429,1095-7170},
   MRCLASS = {94A20 (41A46 65K10 90C90)},
  MRNUMBER = {2914271},
MRREVIEWER = {Zhou\ Chen\ Lin},
       DOI = {10.1137/110833130},
       URL = {https://doi.org/10.1137/110833130},
}

@inproceedings{kur2026a,
title={A New Perspective on Minimum-Norm Interpolation Under Gaussian Covariates},
author={Gil Kur and Zong Shang and Paul Simanjuntak and Guillaume Lecu{\'e} and Reese Pathak},
booktitle={The 29th International Conference on Artificial Intelligence and Statistics},
year={2026},
url={https://openreview.net/forum?id=jpsrJrRzBS}
}

@inproceedings{koehler2021uniform,
  title={Uniform convergence of interpolators: {Gaussian} width, norm bounds and benign overfitting},
  author={Koehler, Frederic and Zhou, Lijia and Sutherland, Danica J and Srebro, Nathan},
  booktitle={Advances in Neural Information Processing Systems (NeurIPS)},
  volume={34},
  pages={20657--20668},
  year={2021}
}

@phdthesis{shang2026feature,
  author       = {Zong Shang},
  title        = {Feature Space Decomposition},
  school       = {CREST-ENSAE, Institut Polytechnique de Paris},
  year         = {2026},
  address      = {Palaiseau, France},
  month        = {April}
}

@article{bartlettBenignOverfittingLinear2020,
	title = {Benign overfitting in linear regression},
	volume = {117},
	url = {https://www.pnas.org/doi/10.1073/pnas.1907378117},
	doi = {10.1073/pnas.1907378117},
	number = {48},
	urldate = {2022-09-18},
	journal = {Proceedings of the National Academy of Sciences},
	author = {Bartlett, Peter L. and Long, Philip M. and Lugosi, Gábor and Tsigler, Alexander},
	month = dec,
	year = {2020},
	note = {Publisher: Proceedings of the National Academy of Sciences},
	pages = {30063--30070},
}

@book{vershyninHighDimensionalProbabilityIntroduction2018,
	address = {Cambridge},
	series = {Cambridge {Series} in {Statistical} and {Probabilistic} {Mathematics}},
	title = {High-{Dimensional} {Probability}: {An} {Introduction} with {Applications} in {Data} {Science}},
	isbn = {978-1-108-41519-4},
	shorttitle = {High-{Dimensional} {Probability}},
	url = {https://www.cambridge.org/core/books/highdimensional-probability/797C466DA29743D2C8213493BD2D2102},
	urldate = {2022-09-18},
	publisher = {Cambridge University Press},
	author = {Vershynin, Roman},
	year = {2018},
	doi = {10.1017/9781108231596},
}

@misc{tsiglerBenignOverfittingGeometry2025,
	title = {Benign {Overfitting} and the {Geometry} of the {Ridge} {Regression} {Solution} in {Binary} {Classification}},
	url = {http://arxiv.org/abs/2503.07966},
	doi = {10.48550/arXiv.2503.07966},
	urldate = {2025-03-17},
	publisher = {arXiv},
	author = {Tsigler, Alexander and Chamon, Luiz F. O. and Frei, Spencer and Bartlett, Peter L.},
	month = mar,
	year = {2025},
	note = {arXiv:2503.07966 [stat]},
}

@article{bartlett2021deep,
  title={Deep learning: a statistical viewpoint},
  author={Bartlett, Peter L and Montanari, Andrea and Rakhlin, Alexander},
  journal={Acta Numerica},
  volume={30},
  pages={87--201},
  year={2021},
  publisher={Cambridge University Press}
}

@article{belkin2021fit,
  title={Fit without fear: remarkable mathematical phenomena of deep learning through the prism of interpolation},
  author={Belkin, Mikhail},
  journal={Acta Numerica},
  volume={30},
  pages={203--248},
  year={2021},
  publisher={Cambridge University Press}
}

@book{ledouxProbabilityBanachSpaces1991,
	address = {Berlin, Heidelberg},
	title = {Probability in {Banach} {Spaces}},
	isbn = {978-3-642-20211-7 978-3-642-20212-4},
	url = {http://link.springer.com/10.1007/978-3-642-20212-4},
	urldate = {2022-09-18},
	publisher = {Springer Berlin Heidelberg},
	author = {Ledoux, Michel and Talagrand, Michel},
	year = {1991},
	doi = {10.1007/978-3-642-20212-4},
}

@book{pisierVolumeConvexBodies1989,
	address = {Cambridge},
	series = {Cambridge {Tracts} in {Mathematics}},
	title = {The {Volume} of {Convex} {Bodies} and {Banach} {Space} {Geometry}},
	isbn = {978-0-521-66635-0},
	url = {https://www.cambridge.org/core/books/volume-of-convex-bodies-and-banach-space-geometry/55C7AB549AA56140010097A579ACB482},
	urldate = {2022-09-18},
	publisher = {Cambridge University Press},
	author = {Pisier, Gilles},
	year = {1989},
	doi = {10.1017/CBO9780511662454},
}

@misc{bubeckHomotopyMethod$ell_p$2018,
	title = {An homotopy method for \${\textbackslash}ell\_p\$ regression provably beyond self-concordance and in input-sparsity time},
	url = {http://arxiv.org/abs/1711.01328},
	doi = {10.48550/arXiv.1711.01328},
	urldate = {2025-01-21},
	publisher = {arXiv},
	author = {Bubeck, Sébastien and Cohen, Michael B. and Lee, Yin Tat and Li, Yuanzhi},
	month = jun,
	year = {2018},
	note = {arXiv:1711.01328 [math]},
}

@article{gavrilopoulosGeometricalAnalysisKernel2025,
	title = {A geometrical analysis of kernel ridge regression and its applications},
	volume = {53},
	issn = {0090-5364, 2168-8966},
	url = {https://projecteuclid.org/journals/annals-of-statistics/volume-53/issue-6/A-geometrical-analysis-of-kernel-ridge-regression-and-its-applications/10.1214/25-AOS2556.full},
	doi = {10.1214/25-AOS2556},
	number = {6},
	urldate = {2025-12-24},
	journal = {The Annals of Statistics},
	publisher = {Institute of Mathematical Statistics},
	author = {Gavrilopoulos, Georgios and Lecué, Guillaume and Shang, Zong},
	month = dec,
	year = {2025},
	pages = {2592--2616},
}

@misc{lecueFeatureSpaceDecomposition2025,
	title = {Sharp convergence rates for {Spectral} methods via the feature space decomposition method},
	url = {http://arxiv.org/abs/2512.14473},
	doi = {10.48550/arXiv.2512.14473},
	urldate = {2025-12-17},
	publisher = {arXiv},
	author = {Lecué, Guillaume and Li, Zhifan and Shang, Zong},
	month = dec,
	year = {2025},
	note = {arXiv:2512.14473 [math]},
}

@misc{barzilaiGeneralizationKernelRegression2024,
	title = {Generalization in {Kernel} {Regression} {Under} {Realistic} {Assumptions}},
	url = {http://arxiv.org/abs/2312.15995},
	doi = {10.48550/arXiv.2312.15995},
	urldate = {2024-03-16},
	publisher = {arXiv},
	author = {Barzilai, Daniel and Shamir, Ohad},
	month = feb,
	year = {2024},
	note = {arXiv:2312.15995 [cs, stat]},
}

@article {chengDimensionFreeRidge2022,
    AUTHOR = {Cheng, Chen and Montanari, Andrea},
     TITLE = {Dimension free ridge regression},
   JOURNAL = {Ann. Statist.},
  FJOURNAL = {The Annals of Statistics},
    VOLUME = {52},
      YEAR = {2024},
    NUMBER = {6},
     PAGES = {2879--2912},
      ISSN = {0090-5364,2168-8966},
   MRCLASS = {62J07 (60B20)},
  MRNUMBER = {4842830},
       DOI = {10.1214/24-aos2449},
       URL = {https://doi.org/10.1214/24-aos2449},
}

@article {MR3431642,
    AUTHOR = {Koltchinskii, Vladimir and Mendelson, Shahar},
     TITLE = {Bounding the smallest singular value of a random matrix
              without concentration},
   JOURNAL = {Int. Math. Res. Not. IMRN},
  FJOURNAL = {International Mathematics Research Notices. IMRN},
      YEAR = {2015},
    VOLUME = {2015},
    NUMBER = {23},
     PAGES = {12991--13008},
      ISSN = {1073-7928,1687-0247},
   MRCLASS = {60B20 (15B52)},
  MRNUMBER = {3431642},
MRREVIEWER = {Oleksiy\ Khorunzhiy},
       DOI = {10.1093/imrn/rnv096},
       URL = {https://doi.org/10.1093/imrn/rnv096},
}

@Article{MR2111932,
  author     = {Lata{\l}a, Rafa{\l}},
  journal    = {Proc. Amer. Math. Soc.},
  title      = {Some estimates of norms of random matrices},
  year       = {2005},
  issn       = {0002-9939,1088-6826},
  number     = {5},
  pages      = {1273--1282},
  volume     = {133},
  doi        = {10.1090/S0002-9939-04-07800-1},
  fjournal   = {Proceedings of the American Mathematical Society},
  mrclass    = {15A52 (60E15 60G15)},
  mrnumber   = {2111932},
  mrreviewer = {H.\ T.\ Lau},
  url        = {https://doi.org/10.1090/S0002-9939-04-07800-1},
}

@article {MR2601042,
    AUTHOR = {Adamczak, Rados{\l }aw and Litvak, Alexander E. and Pajor, Alain
              and Tomczak-Jaegermann, Nicole},
     TITLE = {Quantitative estimates of the convergence of the empirical
              covariance matrix in log-concave ensembles},
   JOURNAL = {J. Amer. Math. Soc.},
  FJOURNAL = {Journal of the American Mathematical Society},
    VOLUME = {23},
      YEAR = {2010},
    NUMBER = {2},
     PAGES = {535--561},
      ISSN = {0894-0347,1088-6834},
   MRCLASS = {60B20 (46B09 52A21 60F15)},
  MRNUMBER = {2601042},
       DOI = {10.1090/S0894-0347-09-00650-X},
       URL = {https://doi.org/10.1090/S0894-0347-09-00650-X},
}

@article {grst-Kantorovich-duality,
	AUTHOR = {Gozlan, Nathael and Roberto, Cyril and Samson, Paul-Marie and
	Tetali, Prasad},
	TITLE = {Kantorovich duality for general transport costs and
	applications},
	JOURNAL = {J. Funct. Anal.},
	FJOURNAL = {Journal of Functional Analysis},
	VOLUME = {273},
	YEAR = {2017},
	NUMBER = {11},
	PAGES = {3327--3405},
	ISSN = {0022-1236},
	MRCLASS = {60E15 (26D10 49Q20 60B05)},
	MRNUMBER = {3706606},
	MRREVIEWER = {Sebastian Riedel},
	DOI = {10.1016/j.jfa.2017.08.015},
	URL = {https://doi.org/10.1016/j.jfa.2017.08.015},
}

@article {grsst-characterization,
	AUTHOR = {Gozlan, Nathael and Roberto, Cyril and Samson, Paul-Marie and
	Shu, Yan and Tetali, Prasad},
	TITLE = {Characterization of a class of weak transport-entropy
	inequalities on the line},
	JOURNAL = {Ann. Inst. Henri Poincar\'{e} Probab. Stat.},
	FJOURNAL = {Annales de l'Institut Henri Poincar\'{e} Probabilit\'{e}s et
	Statistiques},
	VOLUME = {54},
	YEAR = {2018},
	NUMBER = {3},
	PAGES = {1667--1693},
	ISSN = {0246-0203,1778-7017},
	MRCLASS = {60E15 (26D10)},
	MRNUMBER = {3825894},
	MRREVIEWER = {Oliver\ Johnson},
	DOI = {10.1214/17-AIHP851},
	URL = {https://doi.org/10.1214/17-AIHP851},
}

@article {MR1361756,
    AUTHOR = {Talagrand, Michel},
     TITLE = {Concentration of measure and isoperimetric inequalities in
              product spaces},
   JOURNAL = {Inst. Hautes \'Etudes Sci. Publ. Math.},
  FJOURNAL = {Institut des Hautes \'Etudes Scientifiques. Publications
              Math\'ematiques},
    VOLUME = {81},
      YEAR = {1995},
     PAGES = {73--205},
      ISSN = {0073-8301,1618-1913},
   MRCLASS = {60E15 (60D05)},
  MRNUMBER = {1361756},
MRREVIEWER = {G.\ Schechtman},
       URL = {http://www.numdam.org/item?id=PMIHES_1995__81__73_0},
}

@inproceedings {MR3909555,
    AUTHOR = {Adil, Deeksha and Kyng, Rasmus and Peng, Richard and Sachdeva,
              Sushant},
     TITLE = {Iterative refinement for {$\ell_p$}-norm regression},
 BOOKTITLE = {Proceedings of the {T}hirtieth {A}nnual {ACM}-{SIAM}
              {S}ymposium on {D}iscrete {A}lgorithms},
     PAGES = {1405--1424},
 PUBLISHER = {SIAM, Philadelphia, PA},
      YEAR = {2019},
      ISBN = {978-1-61197-548-2},
   MRCLASS = {15A60 (15A24 68W40)},
  MRNUMBER = {3909555},
       DOI = {10.1137/1.9781611975482.86},
       URL = {https://doi.org/10.1137/1.9781611975482.86},
}

@article {MR634905,
    AUTHOR = {Schechtman, Gideon},
     TITLE = {Random embeddings of {E}uclidean spaces in sequence spaces},
   JOURNAL = {Israel J. Math.},
  FJOURNAL = {Israel Journal of Mathematics},
    VOLUME = {40},
      YEAR = {1981},
    NUMBER = {2},
     PAGES = {187--192},
      ISSN = {0021-2172},
   MRCLASS = {46B20 (46A45)},
  MRNUMBER = {634905},
MRREVIEWER = {D.\ J. H. Garling},
       DOI = {10.1007/BF02761909},
       URL = {https://doi.org/10.1007/BF02761909},
}

@article {MR435822,
    AUTHOR = {Bennett, G. and Dor, L. E. and Goodman, V. and Johnson, W. B.
              and Newman, C. M.},
     TITLE = {On uncomplemented subspaces of {$L\sb{p},$} {$1<p<2$}},
   JOURNAL = {Israel J. Math.},
  FJOURNAL = {Israel Journal of Mathematics},
    VOLUME = {26},
      YEAR = {1977},
    NUMBER = {2},
     PAGES = {178--187},
      ISSN = {0021-2172},
   MRCLASS = {46E30 (46B05)},
  MRNUMBER = {435822},
MRREVIEWER = {H.\ E.\ Lacey},
       DOI = {10.1007/BF03007667},
       URL = {https://doi.org/10.1007/BF03007667},
}

@article {MR2391154,
    AUTHOR = {Mendelson, Shahar and Tomczak-Jaegermann, Nicole},
     TITLE = {A subgaussian embedding theorem},
   JOURNAL = {Israel J. Math.},
  FJOURNAL = {Israel Journal of Mathematics},
    VOLUME = {164},
      YEAR = {2008},
     PAGES = {349--364},
      ISSN = {0021-2172,1565-8511},
   MRCLASS = {46B09 (60B11)},
  MRNUMBER = {2391154},
       DOI = {10.1007/s11856-008-0034-1},
       URL = {https://doi.org/10.1007/s11856-008-0034-1},
}

@article {MR2304336,
    AUTHOR = {Gu\'edon, Olivier and Rudelson, Mark},
     TITLE = {{$L_p$}-moments of random vectors via majorizing measures},
   JOURNAL = {Adv. Math.},
  FJOURNAL = {Advances in Mathematics},
    VOLUME = {208},
      YEAR = {2007},
    NUMBER = {2},
     PAGES = {798--823},
      ISSN = {0001-8708,1090-2082},
   MRCLASS = {46B09 (46E30 46N30 60D05)},
  MRNUMBER = {2304336},
MRREVIEWER = {Zhong\ Gen\ Su},
       DOI = {10.1016/j.aim.2006.03.013},
       URL = {https://doi.org/10.1016/j.aim.2006.03.013},
}
\end{document}